\documentclass[12pt]{book}

\usepackage[letterpaper]{geometry}
\usepackage[utf8]{inputenc}
\usepackage[OT4]{fontenc} 
\usepackage{amsmath}
\usepackage{amsthm}
\usepackage{amsfonts}  
\usepackage{amssymb}
\usepackage{verbatim}
\usepackage{bm}
\usepackage{bbm}

\usepackage[retainorgcmds]{IEEEtrantools}
\usepackage{marvosym}
\usepackage{stmaryrd}
\usepackage{tikz}
\usepackage{hyperref}
\usepackage{mdframed}
\usepackage{xcolor}
\usepackage{hhline}
\usepackage{thmtools}
\usepackage{thm-restate}
\usepackage{mathrsfs}
\usepackage{enumerate}
\usepackage{ifthen}

\usepackage{listings}
\usepackage{pgfplots}

\pgfplotsset{compat=1.5}
\theoremstyle{plain}
\newtheorem{theorem}{Theorem}[chapter]
\newtheorem{lemma}[theorem]{Lemma}
\newtheorem{proposition}[theorem]{Proposition}
\newtheorem{corollary}[theorem]{Corollary}

\newtheorem{claim}[theorem]{Claim}
\newtheorem{problem}[theorem]{Problem}
\newtheorem{exercise}[theorem]{Exercise}

\theoremstyle{definition}
\newtheorem{definition}[theorem]{Definition}
\newtheorem{example}[theorem]{Example}
\newtheorem{remark}[theorem]{Remark}

\DeclareMathOperator*{\Var}{Var}
\DeclareMathOperator*{\Cov}{Cov}

\DeclareMathOperator{\Inf}{Inf}

\DeclareMathOperator{\sgn}{sgn}

\DeclareMathOperator{\length}{length}

\newcommand{\bits}{\{-1,1\}}

\newcommand{\Acyc}{\mathrm{Acyc}}

\newcommand{\UniqueBest}{\mathrm{UniqueBest}}
\newcommand{\adj}[2]{\ensuremath{[#1:#2]}}
\newcommand{\adji}[3]{\ensuremath{\adj{#1}{#2}_{#3}\,}}
\newcommand{\wt}{\widetilde}

\newcommand{\p}{\mathbb{P}}
\newcommand{\PP}{\mathbb{P}}
\renewcommand{\P}{\mathbb{P}}
\newcommand{\IP}{\mathbb{P}}
\newcommand{\IE}{\mathbb{E}}
\newcommand{\E}{\mathbb{E}}
\newcommand{\EE}{\mathbb{E}}

\newcommand{\bbR}{\mathbb{R}}
\newcommand{\IR}{\mathbb{R}}
\newcommand{\R}{\mathbb{R}}
\newcommand{\grad}{\nabla}

\def\F{{\cal{F}}}

\DeclareMathOperator{\Dist}{{\bf D}}
\DeclareMathOperator{\CONST}{CONST}
\DeclareMathOperator{\DICT}{DICT}
\DeclareMathOperator{\NONMANIP}{NONMANIP}
\DeclareMathOperator{\Maj}{Maj}
\DeclareMathOperator{\maj}{maj}
\DeclareMathOperator{\tp}{top}

\DeclareMathOperator{\Lg}{Lg}
\DeclareMathOperator{\Sm}{Sm}
\DeclareMathOperator{\LD}{LD}

\def\CalN{{\cal{N}}}
\def\CalM{{\cal{M}}}

\newcommand{\ignore}[1]{}

   {\medskip\begin{mdframed}\setlength{\parindent}{0cm}}%
   {\end{mdframed}\medskip}

\newcommand{\third}{{\frac{1}{3}}}
\newcommand{\half}{{\frac{1}{2}}}

\newcommand{\cube}{\{0,1\}^n}

\definecolor{DSgray}{cmyk}{0,0,0,0.7}
\definecolor{DSred}{cmyk}{0,0.7,0,0.7}

\newcommand{\eps}{{\varepsilon}}
\newcommand{\pdiff}[2]{\frac{\partial #1}{\partial #2}}
\newcommand{\pdiffII}[3]{\ifthenelse{\equal{#2}{#3}}
{\frac{\partial^2 #1}{\partial #2^2}}
{\frac{\partial^2 #1}{\partial #2 \partial #3}}
}

\usepackage{chngcntr}
 
\counterwithout{figure}{chapter}
\counterwithout{table}{chapter}

\title{Probabilistic Aspects of Voting, Intransitivity and Manipulation}
\date{}

\author{Elchanan Mossel\thanks{MIT, Cambridge MA, USA.
E-mail: {\tt{elmos@mit.edu}}.
Partially supported by NSF Grant CCF 1665252, DMS-1737944,  DOD ONR grant N00014-17-1-2598, Simons Investigator award (622132) and Vannevar Bush Faculty Fellowship ONR-N00014-20-1-2826}}

\begin{document}

\maketitle

\tableofcontents



\newpage

\paragraph{Acknowledgment}
 
Much of the material covered in this book is based on joint works with Joe Neeman, Miki Racz, Anindya De, Gil Kalai, and Olle Haggstrom. I thank them for allowing me to include revised material taken from joint papers and for fruitful collaborations. 
A core of the material presented here was given at that 2019 Saint-Flour Probability School.
Thanks to everyone who contributed and participated in the School. 
Special thanks to Christophe Bahadoran for his role in organizing such a successful school and to 
Philippe Rigollet and Nicolas Curien for wonderful lectures and great company. 
Finally, thanks to all the participants in my lectures, especially those who provided feedback.

\newpage

\chapter{Introduction}
Marquis de Condorcet, a French philosopher, mathematician, and political scientist, studied mathematical aspects of voting in the eighteenth century.  It is remarkable that already in the 18th century Condorcet was an advocate of equal rights for women and people of all races and of free and equal public education~\cite{wiki:Condorcet}. His applied interest in democratic processes led him to write an influential paper in 1785~\cite{Condorcet:85}, where in particular he was interested in voting as an aggregation procedure and where he pointed out the paradoxical nature of voting in the presence of $3$ or more alternatives. 

\section{The law of large numbers and Condorcet Jury Theorem}
In what is known as Condorcet Jury Theorem, Condorcet considered the following setup.
There are $n$ voters and two alternatives denoted $+$ (which stands for $+1$) and $-$ (which stands for $-1$). 
Each voter obtains a {\em signal} which indicates which of the alternatives is preferable. 
The assumption is that there is an a priori better alternative and that each voter independently obtains the correct information with probability $p > 1/2$ and incorrect information with probability $1-p$. The $n$ voters then take a majority vote to decide the winner. Without loss of generality we may assume 
that the correct alternative is $+$ and therefore the individual signals are i.i.d RV $x_i$ where $\IP[x_i = +] = p$ and $\IP[x_i = -] = 1-p$. Let $m$ denote the Majority function, i.e the function that returns the most popular values among its inputs. Condorcet Jury Theorem is: 

\begin{theorem}
 For every $1>p > 1/2$:
\begin{itemize}
\item $\lim_{n \mbox{ odd }  \to \infty} \IP[m(x_1,\ldots,x_n) = +] = 1$.
\item If $n_1 < n_2$ are odd then $\ \IP[m(x_1,\ldots,x_{n_1}) = +] <  \IP[m(x_1,\ldots,x_{n_2}) = +]$.
\end{itemize}
\end{theorem}

The first part of the theorem is immediate from the law of large numbers (which was known at the time), so the novel contribution was the second part. In the early days of modern democracy, Condorcet used his model to argue that the more people participating in decision making, the more likely that the correct decision is arrived at. We leave the proof of the second part of the theorem as an exercise.

\section{Condorcet Paradox and Arrow Theorem}
As hinted earlier, things are more interesting when there are $3$ or more alternatives. 
In the same 1785 paper, Condorcet proposed the following ``paradox". 
Consider three voters named $1,2$ and $3$, and three alternatives named $a,b$ and $c$. 
Each voter ranks the three alternatives in one of six linear orders.
While it is tempting to represent the orders as elements of the permutation group $S_3$, it will be more useful for us to use the following representation. 
Voter $i$ preference is given by $(x_i,y_i,z_i)$ where $x_i = +$, if she prefers $a$ to $b$ and $-$, otherwise, $y_i = +$, if she prefers $b$ to $c$ and $-$, otherwise and $z_i = +$ if she prefers $c$ to $a$ and $-$, otherwise. Each of the $6$ rankings corresponds to one of the vectors in 
$\{-1,1\}^3 \setminus \{ \pm (1,1,1) \}$. 

Condorcet considered $3$ voters, with rankings given by: $a>b>c, c>a>b, b>c>a$, or in our notation by the rows of the following matrix: 
\[
\left( \begin{matrix}
x_1 & y_1 & z_1 \\ 
x_2 & y_2 & z_2 \\ 
x_3 & y_3 & z_3 \end{matrix}\right) = 
\left( \begin{matrix}
+ & + & - \\ 
+ & - & + \\ 
- & + & + \end{matrix}\right) 
\]

How should we decide how to aggregate the individual rankings? 
If we use the majority rule to decide between each pair of preferences, then we apply the majority rule on each of the columns of the matrix and conclude that overall preference is $(+,+,+)$. 
In other words, the overall preference is that $a>b$, the overall preference is that $b>c$ and the overall preference is that $c>a$. This does not correspond to an order!  
This is what is known as Condorcet Paradox. 

Almost two hundred years later, Ken Arrow asked if perhaps the paradox is the result of using the majority function to decide between 
every pair of alternatives? 
Can we avoid paradoxes if we aggregate pairwise preferences using a different function? 

One function that never results in paradoxes is the dictator function $f(x) = x_1$ as the aggregate ranking is $(x_1,y_1,z_1) \neq \pm (1,1,1)$. 

Arrow in his famous theorem proved~\cite{Arrow:50,Arrow:63}:
\begin{theorem} \label{thm:arrow1}
Let  $f : \{-1,1\}^n\to\{-1,1\}$ and suppose that $f$ never results in a paradox, so for all 
$(x_i,y_i,z_i) \neq \pm (1,1,1)$ it holds that $(f(x),f(y),f(z)) \neq \pm (1,1,1)$. Then 
$f$ is a dictator: there exists an $i$ such that $f(x) = x_i$ for all $x$, 
or there exists an $i$ such that $f(x) = -x_i$ for all $x$. 
\end{theorem} 
With the right notation and formulation the proof of Arrow's Theorem is very short~\cite{Barbera:80,Mossel:09a}:
\begin{proof} 
First note that if $f$ is a constant function, then the outcome is always $\pm (1,1,1)$. 
Suppose that $f$ is not a dictator and not a constant, then $f$ depends on at least two coordinates.
Without loss of generality, let these coordinates be $1$ and $2$. Therefore:
\begin{align*}
  \exists x_2,x_3,\ldots,x_n: \quad f(+,x_2,x_3,\ldots,x_n) &\neq f(-,x_2,x_3,\ldots,x_n) \\ 
\exists y_1,y_3,\ldots,y_n: \quad f(y_1,+,y_3,\ldots,y_n) &\neq f(y_1,-,y_3,\ldots,y_n)
\end{align*}
We now choose $z_1 = -y_1$ and $z_i = -x_i$ for $i \geq 2$. 
This guarantees that $(x_i,y_i,z_i) \neq \pm (1,1,1)$ for all $i$ no matter what the values of $x_1$ and $y_2$ are. 
This can be also verified from the matrix in (\ref{eq:mat_arrow}).  
Now choose $x_1$ and $y_2$ so that $f(x) = f(y) = f(z)$ to arrive at a contradiction. 
\begin{equation} \label{eq:mat_arrow}
\left( \begin{matrix}
x_1 & y_1 & -y_1 \\ 
x_2 & y_2 & -x_2 \\ 
\vdots & \vdots & \vdots \\ 
x_n & y_n & -x_n \end{matrix}\right) 
\end{equation}
\end{proof}

Arrow also considered a more general setting where $f,g$ and $h$ are allowed to be different functions. In this case, there are other functions that result in $0$ probability of paradox. For example if $f = 1$ and $g = -1$ for all inputs then $h$ can be arbitrary. This corresponds to the case where the second alternative $b$ is ignored (always ranked last) and the choice between
$a$ and $c$ is determined by $h$. Arrow Theorem in this setting can be stated as follows:

\begin{theorem} \label{thm:arrow3}
Let  $f,g,h : \{-1,1\}^n\to\{-1,1\}$ and suppose that $(f,g,h)$ never results in a paradox, so for all 
$(x_i,y_i,z_i) \neq \pm (1,1,1)$ it holds that $(f(x),g(y),h(z)) \neq \pm (1,1,1)$. Then either two of the functions are constant of opposite sign, 
or there exists an $i$ such that $f,g$ and $h$ are dictator on voter $i$. 
\end{theorem} 

\begin{proof}
If two of the functions, say $f$ and $g$ take the same constant value and the third function $h$ is not constant, then clearly one can 
find $x,y,z$ such that $f(x) = g(y) = h(z)$ and $(x,y,z) \neq \pm (1,1,1)$. 
So WLOG we may assume at least two of the functions, say $f$ and $g$ are not constant.  
Let $A(f)$ denote the set of variables that may change the value of $f$ and similarly $A(g)$ and $A(h)$. 
Since $f$ and $g$ are not constant, it follows that $A(f)$ and $A(g)$ are not empty. If there exists a variable $i \in A(f)$ and a variable $i \neq j \in A(g)$, then by the same argument as in Theorem~\ref{thm:arrow1}, there exist $(x,y,z)$ resulting in a paradox. Thus, the only case remaining is where 
$A(f) = A(g) = i$ and $A(h) = i$ or $A(h) = \emptyset$. In either case, the functions $f,g$ and $h$ are all functions of variable $i$ only.
It is now easy to verify that it must be the case that $f=g=h$ is a dictator on voter $i$. 
\end{proof}

\section{Manipulation and the GS Theorem} 
A naturally desirable property of a voting system is \emph{strategyproofness} (a.k.a.\ nonmanipulability): no voter should benefit from voting strategically, i.e., voting not according to her true preferences. However, Gibbard~\cite{Gibbard:73} and Satterthwaite~\cite{Satterthwaite:75} showed that no reasonable voting system can be strategyproof. Before stating the result, let us specify the model more formally. 

The setting here is different than the setup of Arrow's Theorem: 
We consider $n$ voters electing a winner among $k$ alternatives. The voters specify their opinion by ranking the alternatives, and the winner is determined according to some predefined \emph{social choice function} (SCF) $f : S_k^n \to \left[k\right]$ of all the voters' rankings, where $S_k$ denotes the set of all possible total orderings of the $k$ alternatives. We call a collection of rankings by the voters a \emph{ranking profile}. We say that an SCF is \emph{manipulable} if there exists a ranking profile where a voter can achieve a more desirable outcome of the election according to her true preferences by voting in a way that does not reflect her true preferences. 

For example, consider Borda voting, where each candidate receives a score which is the sum of its ranks, and the candidate with the lowest score wins.
If the individual rankings are:
$(a b c d), (c a d b)$,
then $a$ is the winner, but if the second voter were to vote $(c d b a)$ instead then $c$ will become the winner, so the second voter will want to vote untruthfully.

\begin{theorem}[Gibbard-Satterthwaite]  \label{thm:GS_intro}
Any Social Choice Function which is not a dictatorship (i.e., not a function of a single voter), and which allows at least three alternatives to be elected, is manipulable. 
\end{theorem}

This theorem has contributed to the realization that it is unlikely to expect truthfulness in voting. There are many proofs of the Gibbard-Satterthwaite theorem, but all are more complex than the proof of Arrow's theorem given above. We will not provide proof of the theorem in these notes.

\section{Modern Perspectives} 
Work since the 1980s addressed novel aspects of aggregation of votes. 
Condorcet Jury Theorem assumes a probability distribution over the voters but is restricted to a specific aggregation function (majority) while Arrow's theorem considers general aggregation functions but involves no probability model. 
There are many interesting questions that can be asked by combining the two perspectives. 
First, it is natural to ask about the aggregation properties of Boolean functions. The study of aggregation properties of Boolean functions was fundamental to the development of the area of ``Analysis of Boolean Functions" since the 1980s, starting with the work of Ben-Or and Linial~\cite{BenorLinial:90} and Kahn, Kalai, and Linial~\cite{KaKaLi:88}.
Second, we can ask questions regarding the probability of manipulation and paradoxes, questions that were analyzed since the early 2000s, starting with the works by Kalai, Nisan and collaborators~\cite{Kalai:02,Kalai:04,FrKaNi:08}. 

The theory that was developed is intimately connected to the area of property testing in theoretical 
computer science and to additive combinatorics. 
Moreover, we will see that some of the main results and techniques have a discrete isoperimetric flavor. We will cover the following topics: 

We will start by studying the question of noise stability of Boolean functions that were originally studied by Benjamini Kalai and Schramm and  in the context of Percolation~\cite{BeKaSc:99} and later played an important role in analyzing the probability of paradoxes starting with the work of Kalai~\cite{Kalai:02}. This work as well as motivation from theoretical computer science~\cite{KKMO:07}, led to the proof of the Majority is Stablest Theorem by Mossel, O'Donnell and Oleszkiewicz~\cite{MoOdOl:05,MoOdOl:10}. 
The proof of these results will require some of the main analytical tools in the area, including, notably, hyper-contraction~~\cite{Bonami:70,Nelson:73,Gross:75,Beckner:75}, the invariance principle~\cite{MoOdOl:10} and Borell's Gaussian noise-stability result~\cite{Borell:85}.  

We will then devote a considerable effort to proving quantitative versions of Arrow's Theorem. 
We will follow~\cite{Mossel:12} by first proving a Gaussian version of Arrow theorem, as well as a quantitive version using reverse hyper-contraction~\cite{Borell:82}. 
Combining the two we will prove a general quantitive Arrow theorem~\cite{Mossel:12}. 
We will also present Kalai's original proof of a quantitative Arrow theorem~\cite{Kalai:02} which is less general but uses only hypercontraction via~\cite{FrKaNa:02}.  

Different tools were used in proving different quantitive versions of the manipulation theorem. 
The first proof, which applies only to $3$ alternatives, uses a reduction to a quantitive Arrow Theorem~\cite{FrKaNi:08,FKKN:11}. 
We will follow later proofs that apply in the case of a larger number of coordinates and use reverse hyper-contraction as a major tool~\cite{IsKiMo:12,MosselRacz:15}.
The classical proofs of manipulation theorems often use long paths of voting profiles. 
The most general proof in~\cite{MosselRacz:15} will quantify such arguments using geometric tools from the theory of Markov chains. 

In the last part of the notes, we will review some of the more classical results of the aggregation power and structure of  Boolean functions~\cite{BenorLinial:90,KaKaLi:88} and also discuss some basic models involving non-product distributions~\cite{HaKaMo:06}.

We note that much of the interest in quantitative social choice theory comes
from artificial intelligence and computer science, where virtual elections are now an established tool in preference aggregation (see the survey by Faliszewski and Procaccia~\cite{FaliszewskiProcaccia:10}). Many of the results in the study of social choice are negative: it is impossible to design a voting system that satisfies a few desired properties all at once. 

Recall that Gibbard-Satterthwaite theorem states that any SCF which is not a dictatorship (i.e., not a function of a single voter), and which allows at least three alternatives to be elected, is manipulable. This has contributed to the realization that it is unlikely to expect truthfulness in voting. Consequently, there have been many branches of research devoted to understanding the extent of manipulability of voting systems, and to finding ways of circumventing the negative results. 

One approach, introduced by Bartholdi, Tovey and Trick~\cite{BaToTr:89b}, suggests computational complexity as a barrier against manipulation: if it is computationally hard for a voter to manipulate, then she would just tell the truth (we refer to the survey by Faliszewski and Procaccia~\cite{FaliszewskiProcaccia:10} for a detailed history of the surrounding literature). This is a worst-case approach, and while worst-case hardness of manipulation is a desirable property for a SCF to have, this does not tell us anything about \emph{typical} instances of the problem---is it easy or hard to manipulate \emph{on average}? 

The quantitative versions of Arrow and the GS theorems we will prove show that typically ranking is paradoxical and that manipulation is easy on average. 

\section{Probability of Paradox for the Majority}\label{sec:preview}
As a preview of what's to come, we will compute the asymptotic probability of a non-transitive outcome in Condorcet setup with $3$ alternatives and where voters vote uniformly at random. 
The first reference to this computation is by Guilbaud (1952), see~\cite{CampbellTullock:66}. 

Let us denote the Majority function by $m:\{-1,1\}^n\to\{-1,1\}$ and assume that the number of voters $n$ is odd. 
Paradoxes seem more likely when there is no bias towards a particular candidate so we will consider voters who vote independently and where voter $i$ votes uniformly at random from the $6$ possible rankings.
Recall that we encode the $6$ possible rankings by vectors $(x,y,z) \in \{-1,+1\}^3 \setminus \{ \pm (1,1,1) \}$.
Here $x$ is $+1/-1$ if a voter ranks $a$ above/below $b$, 
$y$ is $+1/-1$ if voter ranks $b$ above/below $c$,
$z$ is $+1/-1$ if voter ranks $c$ above/below $a$. 

How do we analyze the probability of a paradox?
The following simple fact was used in~\cite{Kalai:02}: 
Since the binary predicate $\psi:\{-1,1\}^3\to\{0,1\}, \psi(a,b,c) = 1(a=b=c)$ can be expressed as
\[
\psi(a,b,c) = \frac{1}{4} ( 1+ ab  + ac + bc ),
\]
we can write 
\[
  \IP[m(x) = m(y) = m(z)] = \frac{1}{4} (1 + \IE[m(x) m(y)] + \IE[m(x) m(z)] + \IE[m(y) m(z)] ),
\]
which due to symmetry can be written as
\[
\IP[m(x) = m(y) = m(z)] = \frac{1}{4} (1 + 3 \IE[m(x) m(y)] ).
\]
Moreover, the uniform distribution over $\{\pm 1\}^3 \setminus \{ \pm (1,1,1) \}$ satisfies $\IE[x_i y_i] = \IE[y_iz_i] = \IE[z_ix_i] = -1/3$ and the $n$ coordinates are
independent.
As we will see shortly, the quantity $\IE[m(x) m(y)]$ is called the noise stability of $m$ with noise parameter $-1/3$. 
Its asymptotic value as $n \to \infty$ is easy to compute using a $2$-dimensional CLT to obtain:
\[
\lim_{n \to \infty} \IE[m(x) m(y)] = \IE[\sgn(X) \sgn(Y)],
\]
where $X,Y \sim N(0, \left( \begin{matrix} 1 & -\frac{1}{3} \\ -\frac{1}{3} & 1 \end{matrix} \right))$ and we can see that
\[
  \IE[\sgn(X)\sgn(Y)] = 2\IP[\sgn(X)=\sgn(Y)]-1 = 1-\frac{2\arccos(-1/3)}{\pi}
\]
and
\[
  \lim_{n\to\infty}\IP[m(x) = m(y) = m(z)] = 1 - \frac{3\arccos(-1/3)}{2\pi}
  \approx 0.088\;.
\]
In particular, the probability of paradox does not vanish as $n \to \infty$. 

\chapter{Noise Stability}

\section{Boolean Noise Stability}
Consider the following thought experiment: suppose the voters in binary voting obtain independent uniform signals: $x_i = +$ or $x_i = -$ with probability $1/2$. This is the same setting as in Condorcet Jury Theorem except the voters are completely uninformed. 

Now consider the following process that produces a vector $y$ as a {\em noisy version} of $x$. For each $i$ independently: 
let $y_i = x_i$ with probability $(1+\theta)/2$ and $y_i = -x_i$ with probability 
$(1-\theta)/2$, where $\theta \in [-1,1]$. We chose the parametrization so that $\IE[x_i y_i] = \theta$. 

How should we interpret $y$? A simple interpretation is as a noise process of voting machines. 
Suppose that when each voter votes, there is a small probability, say $0.01$ that the voting machine records the opposite vote (independently for all voters and independently of the intended vote). In this case $\theta = 0.98$. Given a voting aggregation
function $f:\{-1,1\}^n\to\{-1,1\}$,
ideally we would like the quantity 
\[
\IP[f(x) = f(y)] = \frac{1}{2}(1+\IE[f(x) f(y)])
\]
to be as large as possible if $\theta > 0$ and as small as possible if $\theta < 0$. 
The quantity $\IE[f(x) f(y)]$ is called the {\em noise stability} of $f$. More generally, following~\cite{BeKaSc:99} we define: 

\begin{definition}
For two functions $f,g : \{-1,1\}^n \to \bbR$ the ($\rho$-){\em noisy inner product} of $f$ and $g$ denoted by 
$\langle f, g \rangle_{\rho}$ 
is defined by $\EE[f(x) g(y)]$, where 
$( (x_i,y_i) : 1 \leq i \leq n)$ are i.i.d. mean $0$ ($\EE[x_i] = \EE[y_i] = 0$) and $\rho$-correlated ($\EE[x_i y_i] = \rho$). 
The {\em noise stability} of $f$ is its noisy inner product with itself: $\langle f, f \rangle_{\rho}$.
\end{definition} 
We can also write the noisy inner product in terms of the {\em noise operator} $T_{\rho}$, 
\begin{definition}
The Markov operator $T_{\rho}$ maps functions $f :  \{-1,1\}^n \to \bbR$ to functions $T_\rho f :  \{-1,1\}^n \to \bbR$.
It is defined by: 
\[
( T_{\rho} f)(x) = \EE[f(y) | x].
\]
\end{definition}  

The noise operator is also known as the Bonami-Beckner operator and plays a key role in the theory of hyper-contraction \cite{Bonami:70,Beckner:75}. Note that 
\[
\langle f, g \rangle_{\rho} = \EE[f(x) g(y)] = \EE[ f T_{\rho} g] =  \EE[ g T_{\rho} f] = \langle T f, g \rangle = \langle f , T g \rangle
\]
and that 
\[
\langle f, g \rangle := \E[f(x) g(x)] = \langle f, g \rangle_{1}.
\]
Basic properties of this operator can be revealed using its eigenfunctions, i.e., the Fourier basis. The following proposition is straightforward to prove: 
\begin{proposition}
For $S \subset [n]$, write $x_S = \prod_{i \in S} x_i$, so $x_{\emptyset} \equiv 1$. 
Then:
\begin{itemize}
\item $(x_S : S \subset [n])$ is an orthonormal basis for the space of all functions $f :  \{-1,1\}^n \to \bbR$.
\item $x_S$ is an eigenfunction of $T_{\rho}$ which corresponds to the eigenvalue $\rho^{|S|}$: $T_{\rho} x_S = \rho^{|S|} x_S$.
\end{itemize}  
\end{proposition}

The following easy result is folklore (see e.g.~\cite{MORSS:06}).
\begin{theorem} \label{thm:dict_stab}
For every $\rho > 0$, for every $n$ and for every $f,g : \{-1,1\}^n \to \{-1,1\}$ with 
$\IE[f] = \IE[g] = 0$, it holds that 
\[
 \langle f, g \rangle_{\rho} \leq  \langle x_1, x_1 \rangle_{\rho} = \rho. 
\]
\[ 
 \langle f, g \rangle_{-\rho} \geq  \langle x_1, x_1 \rangle_{-\rho} = -\rho. 
\]
Moreover, the only optimizers are dictator functions, i.e., functions of the form 
$f(x) = g(x) = x_i$ or $f(x) = g(x) = -x_i$.
\end{theorem}

\begin{proof}
Note that $\IE[f] =  \langle f, 1 \rangle =  \langle f, x_{\emptyset} \rangle$. So if $\IE[f] = 0$, then:
$f = \sum_{S \neq \emptyset} \hat{f}(S) x_S$, where $\hat{f}(S) =   \langle f, x_S \rangle$. 
Moreover,
\[
 T_{\rho} f(x) = \EE[f(y) | x] = \sum_S \rho^{|S|} \hat{f}(S) x_S,
\]
and similarly for $g$. Therefore: 
\begin{align}
  \langle f,g \rangle_\rho &=
  \langle T_{\rho} f, g \rangle = \langle \sum_{S \neq \emptyset} \rho^{|S|} \hat{f}(S) x_S,
  \sum_{S \neq \emptyset} \hat{g}(S) x_S \rangle \\
&= \sum_{S \neq \emptyset} \rho^{|S|} \hat{f}(S)\hat{g}(S) \leq 
\rho \sqrt{ \sum_{S \neq \emptyset} \hat{f}^2(S) \sum_{S \neq \emptyset} \hat{g}^2(S)} = \rho
\end{align}
where the last inequality uses Cauchy-Schwarz and Parseval's identity  
\[
\sum_S \hat{f}^2(S) = \langle f, f \rangle = 1. 
\]
In the case of equality we conclude that $f=g$ must be a linear function, so $f = \sum_{i} a_i x_i$. 
Since $f$ is Boolean, 
\[
1 \equiv f^2 = \sum_i a_i^2 + \sum_{i \neq j} a_i a_j x_i x_j 
\] 
and therefore it must be the case that $f$ is a dictator. 
The case where $\rho < 0$ is similar and is left as an exercise. 
\end{proof}

The Theorem above allows a quick proof of a version of Arrow's Theorem by Kalai~\cite{Kalai:02}: 
\begin{corollary}
In the context of Arrow Theorem if $\IE[f] = \IE[g] =\IE[h] = 0$ and 
$\IP[f(x) = g(y) = h(z)] = 0$ then $f,g$ and $h$ are all the same dictator. 
\end{corollary} 

\begin{proof} 
Use the previous theorem and: 
\[
\IP[f(x) = g(y) = h(z)] = 
\frac{1}{4} \left(1+ \langle f,g \rangle_{-1/3} + \langle g,h \rangle_{-1/3} + \langle h,f \rangle_{-1/3} \right). \qedhere
\]
\end{proof}

From the voting perspective, there is something a little unnatural about the result above. It says that if we want to maximize robustness, then a dictator should decide. From a mathematical perspective, it is disappointing that there is something special about $\IE[f] = 0$. 
In particular the following is open: 
\begin{problem}
For a generic $\rho > 0, 0 < \mu < 1$ what is the value of: 
\[
 \lim_{n \to \infty} \max \Big( \langle f, f \rangle_{\rho} : f : \{-1,1\}^n \to \{-1,1\}, \IE[f] = \mu \Big) 
\]
\end{problem}

Here are some additional examples:
\begin{itemize}
\item Similar argument to the theorem shows that if $\rho > 0$ then $\langle f, f \rangle_\rho \geq \rho^n$ for $f : \{-1,1\}^n \to \{-1,1\}$. The parity function $x_{[n]}$ achieves equality. 

\item The asymptotic noise stability of Majority is given by Sheppard formula~\cite{Sheppard:99}, i.e., $\IE[\sgn(N) \sgn(M)]$, where $(N,M)$ are $\rho$-correlated random variables: 
\[
  \IE[\sgn(M)\sgn(N)] = 2\IP[\sgn(M)=\sgn(N)]-1 = 1-\frac{2\arccos(\rho)}{\pi} := \kappa(\rho)
\]
In particular if $\rho = 1-\eps$, then $\IP[f(x) \neq f(y)]$ is of order $\sqrt{\eps}$. Compare this to a dictator where it is of order $\eps$. 
\item If we consider $n = r^2$ where $r$ is odd and the function $f$
  implements electoral college, i.e.,
\[
f(x_1,\ldots,x_n) = m \Big(m(x_1,\ldots,x_r),\ldots,m(x_{n-r+1},x_n) \Big),
\]
then it is easy to see that asymptotically the noise stability is given by $\kappa(\kappa(\rho))$. In particular if $\rho = 1-\eps$ for small $\eps$ then 
$\IP[f(x) \neq f(y)]$ is of order $\eps^{1/4}$. 
\item Let $m_r$ be majority function on $r$ voters and define $m_r^{(1)} = m_r$ and by induction:
\[
  m_r^{(h)}(x_1,\ldots,x_{r^h}) = m_r^{(h-1)} \left( m_r(x_1,\ldots,x_r),\ldots,m_r(x_{r^h-r+1},\ldots,x_{r^h}) \right).
\]
This function is called the recursive majority function.

\begin{exercise}[\cite{MosselOdonnell:03}]
Show that for every $\eps < 0.5$, if $r$ is large enough and $n_h = r^h$ and 
if $\rho =  1-n_h^{-\eps}$ then
\[
\lim_{h \to \infty} \E[m_r^{(h)}(x) m_r^{(h)}(y)] = 0,
\]
where $(x,y)$ are $\rho$-correlated. 
\end{exercise} 
\end{itemize}

\section{Gaussian Noise Stability} 

We will now take a detour of considering 
analogous quantities defined in Gaussian space. We will later see that this is quite useful in the Boolean setting. 

\begin{definition} 
  In Gaussian space, the ($\rho$-)noisy inner product of $\phi : \bbR^n \to \bbR$ and $\psi : \bbR^n \to \bbR$ denoted by $\langle \phi, \psi \rangle_{\rho}$  is
  defined as
\[
\EE[\phi(M) \psi(N)],
\]
where $N = (N_1,\ldots,N_n)$, $M = (M_1,\ldots,M_n)$, and $((M_i,N_i))_{i=1}^n$ are i.i.d.~two-dimensional Gaussian vectors,
such that $N_i,M_i$  
are standard (mean $0$, variance $1$) Gaussian random variables and $\EE[N_i M_i] = \rho$.
The {\em noise stability} of $\phi$ is its inner product with itself: $\langle \phi, \phi \rangle_{\rho}$.
\end{definition} 
We will generally use $f,g$ etc.~to denote functions over the Boolean cube and $\phi,\psi$ etc.~for functions in $L_2(\IR^n,\gamma)$. 
In particular, for $\mu \in [0,1]$ we write $\chi_{\mu}$ for the indicator of the interval $(-\infty,\Phi^{-1}(\mu))$ whose Gaussian measure is
$\mu$. 



We can now state Borell's~\cite{Borell:85} noise stability result: 
\begin{theorem} \label{thm:half_space}
For all $n \geq 1$, $\rho > 0$ and $\phi,\psi :  \IR^n \to [0,1]$, it holds that:  
\[
\langle 1-\chi_{1-\E \phi}, \chi_{\E \psi} \rangle_{\rho}  \leq \langle \phi, \psi \rangle_{\rho} \leq \langle \chi_{\E \phi}, \chi_{\E \psi} \rangle_{\rho} 
\]
\[ 
\langle 1-\chi_{1-\E \phi}, \chi_{\E \psi} \rangle_{-\rho}  \geq \langle \phi, \psi \rangle_{-\rho} \geq \langle \chi_{\E \phi}, \chi_{\E \psi} \rangle_{-\rho}
\]
\end{theorem}
Borell~\cite{Borell:85} was interested in more general functionals of the heat equations and he showed that these functionals increase with respect 
to nonincreasing spherical rearrangement. The fact that half spaces are the unique optimizers of $\rho$-noisy inner product  
was proven in~\cite{MosselNeeman:15b}, where a robust version of the theorem is also proven. Tighter robust versions were later proven by Eldan~\cite{Eldan:15}.  Other alternative proofs and generalization of Borell's result include~\cite{IsakssonMossel:12,KiKiOd:18}. 

\section{Gaussian and Boolean Noise Stability}
By applying the CLT, it is easy to check that Gaussian noise stability provides a bound on the Boolean noise stability.
\begin{proposition}
For every $\rho\in[-1,1],\mu,\nu \in [0,1]$, and for every
\[
s \in \Big[ \langle 1-\chi_{1-\mu}, \chi_{\nu} \rangle_{\rho} , \langle \chi_{\mu}, \chi_{\nu} \rangle_{\rho} \Big],
\]
there exists sequence of Boolean functions $f_n,g_n :\{-1,1\}^n \to \{0,1\}$ such that $\IE[f_n] \to \mu, \IE[g_n] \to \nu$ and 
\[
\langle f_n, g_n \rangle_{\rho} \to s
\]
\end{proposition}
Moreover, by Theorem~\ref{thm:half_space} and Theorem~\ref{thm:dict_stab} it follows that there is a strict inequality at $\mu = 1/2$. See Figure~\ref{fig:bool_gauss_half}.

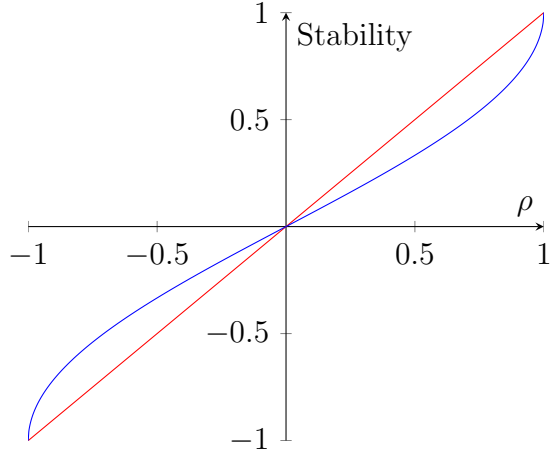
\begin{figure}
\begin{center}
\begin{tikzpicture}
\begin{axis}[
    domain=-1:1,
    samples=101,
    smooth,
    no markers,
    xlabel = $\rho$,
    ylabel = Stability,
    axis lines = middle,
    legend pos = south east
    ]
    \addplot [
        domain=-1:1, 
        samples=100, 
        color=red,
    ]
    {x};
    \addplot[
        domain = -1:1,
        samples = 1000,
        color = blue,
    ]
    {1- acos(x)/90};
\end{axis}
\end{tikzpicture}
\caption{The noise stability of dictator and Gaussian half-space of measure $0.5$. i.e., functions $\rho$ and $1-\arccos{\rho}/2\pi$. Note that for every $0 < \rho < 1$, the dictator is more stable than the corresponding half-spaces and for every $-1 < \rho < 0$, it is less stable than the corresponding half-space.}
\end{center}
\label{fig:bool_gauss_half}
\end{figure}


The proof of the proposition is standard using approximation of Gaussian random variables in terms of sums of independent Bernoullis:
\begin{proof}
To show that we can obtain the RHS of the interval, let 
\[
f_n =\chi_{\mu}(n^{-1/2} \sum_{i=1}^n x_i), \quad g_n =  \chi_{\nu}(n^{-1/2} \sum_{i=1}^n x_i),
\]
and apply the CLT. The proof of achievability of the LHS is identical. To obtain an intermediate point $s$ take the inputs of $f_n$ and $g_n$ to 
be defined on overlapping blocks of bits, e.g.:
\[
f_n =\chi_{\mu}(n^{-1/2} \sum_{i=\alpha n}^{(1+\alpha)n} x_i), \quad g_n =  \chi_{\nu}(n^{-1/2} \sum_{i=1}^n x_i).
\]
\end{proof}  


\section{Smooth Boolean Functions}
To better understand the connection between Boolean and Gaussian stability, we define two notions of smoothness, termed low influences and resilience 
for Boolean functions. 
We begin with the notion of influence, that was first defined in~\cite{BenorLinial:90} and~\cite{KaKaLi:88}. 
We will do so for general product probability spaces. 
To simplify notation we will often omit the sigma algebra and probability measure defined over a probability space $\Omega$. 
\begin{definition}
Consider a probability space $\Omega$. 
For a function $f : \Omega^n \to \bbR$, we define the $i$'th influence of $f$ as 
\[
I_i(f) = \EE \big[\Var[f | x_1,\ldots,x_{i-1},x_{i+1},\ldots,x_n] \big],
\]
where the expected value is with respect to the product measure on $\Omega^n$. 
In the Boolean case with the uniform measure $f : \{-1,1\}^n \to \bbR$, the influence is equivalently defined as 
\[
I_i(f) = \EE \big[\Var[f | x_1,\ldots,x_{i-1},x_{i+1},\ldots,x_n] \big] 
= \sum_{S : i \in S} \hat{f}^2(S). 
\]
Or as 
\[
I_i(f) = \EE[ |\partial_i f|^2],
\]
where $(\partial_i f)(x_1,\ldots,x_n) = 0.5(f(x_1,\ldots,x_{i-1},+,x_{i+1},\ldots,x_n) -  f(x_1,\ldots,x_{i-1},-,x_{i+1},\ldots,x_n))$ is the discrete $i$'th directional derivative. 
\end{definition} 

An easy corollary of the definition is that for $|\rho| < 1$, it holds that $T_{\rho} f$ is small in the sense that the sum of its influences is bounded as a function of $\rho$ only. 

\begin{lemma} \label{lem:inf_sum_rho}
Let $f : \{-1,1\}^n \to [-1,1]$ and $|\rho| < 1$ then 
\[
\sum_{i=1}^n I_i(T_{\rho}f) \leq 1/(1-|\rho|).
\]
\end{lemma}

\begin{proof}
\begin{eqnarray*}
\sum_{i=1}^n I_i(T_{\rho}f) &=& \sum_{i=1}^n \sum_{S : i \in S} \widehat{T_{\rho}f}^2(S) \\
&=& \sum_S |S| \widehat{T_{\rho}f}^2(S) = \sum_{S} |S| \rho^{|S|} \hat{f}^2(S)  \\
&\leq& \max_k |\rho|^k k \sum_S \hat{f}^2(S) \leq \max_k |\rho|^k k \leq 1/(1-|\rho|)
\end{eqnarray*}
The proof follows. 
\end{proof}

For a voting function $f : \{-1,1\}^n \to \{-1,1\}$, $I_i(f)$ is the probability that voter $i$ is the deciding voter, given all other votes.
A stronger notion of power of a voter or a small set of voters is that their vote affects the expected outcome on average. 
A function whose expectation is not affected by any small set of voters is called resilient. More formally,

\begin{definition}
We say that a function $f : \Omega^n \to \bbR$ is $(r,\alpha)$-{\em resilient} if 
\begin{align} \label{eq:cond_res_intro}
\Big| \EE \left[f  | X_S = z \right] - \EE[f] \Big| \leq \alpha. 
\end{align}
for all sets $S$ with $|S| \leq r$ and all $z\in\Omega^S$.  
\end{definition}

\begin{proposition} 
If $f : \{-1,1\}^n \to \IR$ satisfies 
\begin{equation} \label{eq:ResdefFourier}
\max (|\hat{f}(S)| : 0 < |S| \leq r) \leq 2^{-r} \alpha,
\end{equation} 
then $f$ is $(r,\alpha)$-resilient. In particular if $f$  has all influences bounded by $4^{-r} \alpha^2$ then $f$ is 
$(r,\alpha)$-resilient. 
\end{proposition}

\begin{proof}
The second statement follows from the first one immediately as for every non-empty $S$, we may choose $i \in S$ and then
\[
\hat{f}^2(S) \leq I_i(f) \leq  4^{-r} \alpha^2,
\]
as needed. For the first statement, assume~(\ref{eq:ResdefFourier}). Then: 
\[
\Big| \EE \left[f  | X_S = z \right] - \EE[f]  \Big| =  
\Big| \EE[ \sum_{T \neq \emptyset} \hat{f}(T) z_{S \cap T} x_{T \setminus S}] \Big|  =  
\Big| \sum_{\emptyset \neq T \subset S} \hat{f}(T) z_T \Big| \leq 2^{|S|} 2^{-r} \alpha \leq \alpha. \qedhere 
\]
\end{proof}

Resilient functions have long been studied in the context of pseudo-randomness, see e.g.~\cite{Chor_etal:85}.

Thus, the statement that a function has a high influence variable means that there exists a voter $i$ that can have a noticeable effect on the outcome {\em if voter $i$ has access to all other votes cast}. 
The statement that a function is not resilient implies that there is a bounded set of voters who have noticeable effect on the outcome {\em on average}, i.e., with no access to other votes cast. Consider the following examples: 
\begin{itemize}
\item Dictator has maximal influence of $1$ (and all other $0$). It is also not resilient for $r \geq 1, \alpha < 1$. 
\item Majority has all influences of order $n^{-1/2}$ and is also $(r,O(r/\sqrt{n}))$ resilient. 
\item 
An example of a resilient function with a high influence variable is the function 
\[
f(x) = x_1 \sgn(\sum_{i=2}^n x_i).
\]
Here coordinate $1$ has influence $1$ but the function is resilient. 
In terms of voting, voter $1$ has a lot of power if she has access to all other votes cast (or the majority of the votes) but without access to this information,  she is powerless. Moreover, every small set of $k$ voters can change the expected value of $f$ (by conditioning on their vote) by $O(k n^{-1/2})$. 
Another simple example is the parity function $\prod_{i=1}^n x_i$, which is 
$(r,0)$ resilient for every $r < n$, but where all influences are $1$. 

\end{itemize} 

The Majority is Stablest Theorem states that the extremal noise stability of low influence / resilient functions on the discrete cube is captured by Gaussian noise stability. Here are three increasingly stronger statements along this line:

\begin{theorem}[\cite{MoOdOl:05,MoOdOl:10}] \label{thm:MISTsimple} 
For every $\eps > 0, 0 \leq \rho < 1$, there exists a $\tau> 0$ for which the following holds. 
Let $f,g : \{-1,1\}^n \to [0,1]$ satisfy $\max(I_i(f),I_i(g)) < \tau$ for all $i$. 
Then 
\[
\langle f, g \rangle_{\rho} \leq \langle \chi_{\IE f}, \chi_{\IE g} \rangle_{\rho}+ \eps. 
\]
\end{theorem} 
This theorem is called Majority Is Stablest since
$
\langle \chi_{\IE f}, \chi_{\IE g} \rangle_{\rho} = \lim_{n \to \infty} \langle f_n , g_n \rangle_{\rho}, 
$
where $f_n(x) = \chi_{\IE f}(n^{-1/2} \sum_{i=1}^n x_i)$ and $g_n(x) = \chi_{\IE g}(n^{-1/2} \sum_{i=1}^n x_i)$.

It turns out that for two functions, it is in fact enough that one of them is low influence to obtain the same results, i.e.: 

\begin{theorem}[\cite{Mossel:10}, Prop 1.15] \label{thm:MIST}
For every $\eps > 0$ and $0 \leq \rho < 1$, there exists a $\tau(\rho,\eps) > 0$ for which the following holds. 
Let $f,g : \{-1,1\}^n \to [0,1]$ be such that $\min(I_i(f),I_i(g)) < \tau$ for all $i$. Then 
\begin{equation} \label{eq:mist}
\langle  f, g \rangle_{\rho} \leq \langle \chi_{\IE f}, \chi_{\IE g} \rangle_{\rho} + \eps, 
\end{equation} 
where one can take 
\begin{equation} \label{eq:tau_bound}
\tau = \eps^{O\left( \frac{\log(1/\eps) \log(1/(1-\rho))}{(1-\rho) \eps} \right)}. 
\end{equation} 
In particular the statement above holds when $\max_i I_i(f) < \tau$ and $g$ is {\em any} Boolean function bounded between $0$ and $1$. 
\end{theorem}

Moreover, one can replace the low influence condition by the condition that the function is resilient:
\begin{theorem} [\cite{Mossel:20resilient}] \label{thm:MIST_res}
For every $\eps > 0, 0 \leq \rho < 1$, there exist $r,\alpha> 0$ for which the following holds. 
Let $f : \{-1,1\}^n \to [0,1]$ be $(r,\alpha)$-resilient and let $g :  \{-1,1\}^n \to [0,1]$ be an arbitrary function. 
Then 
\[
\langle f, g \rangle_{\rho} \leq \langle \chi_{\IE f}, \chi_{\IE g} \rangle_{\rho}+ \eps. 
\]
One can take
\begin{equation} \label{eq:r_alpha}
r = O \left(\frac{1}{\eps^2(1-\rho)\tau} \right), \alpha = O\left(\eps 2^{-r} \right), 
\end{equation} 
where $\tau$ is given by (\ref{eq:tau_bound}).
\end{theorem} 
Note in particular that for our current bounds for $\tau$ and for 
fixed $\rho$, $r$ is exponential in a polynomial in $1/\eps$ and 
$\alpha$ is doubly exponential in a polynomial in $1/\eps$. 
Similar statements for one function were proven before by~\cite{OSTW:10} and appeared in~\cite{Jones:16}.

\section{Other Formulations for MIST} 
The original proof of the Majority is Stablest Theorem used a non-linear invariance principle to derive the theorem from Borell's result~\cite{MoOdOl:05,MoOdOl:10}. We will follow a different proof that gives an independent proof of Borell's result. This proof from~\cite{DeMoNe:13,DeMoNe:16} is based on induction on dimension in the discrete cube. It is inspired by Bobkov inductive proof of the Gaussian isoperimetric inequality~\cite{Bobkov:97b}. 

First, it will be useful, for $\rho \in [-1,1]$, to define $J_{\rho} : (0,1)^2 \rightarrow [0,1]$ by 
\[
J_{\rho}(x,y) := \langle \chi_x, \chi_y \rangle_{\rho} = \IP [ N \le \Phi^{-1}(x), M \le \Phi^{-1}(y)],
\]
where $N,M$ are jointly normally distributed random variables with the covariance matrix
$$
\mathop{\mathrm{Cov}}(N,M) = \left(\begin{array}{cc} 1 & \rho  \\\rho  & 1 \end{array}\right).
$$
Instead of just proving Borell's result, we will prove a functional form of the result, first stated and proved in~\cite{MosselNeeman:15b}. 
In~\cite{MosselNeeman:15b} it is proved that in the Gaussian setup where $\phi,\psi : \IR^n \to [0,1]$, and $N, M$ are jointly normal random variables with covariance
$\big(\begin{smallmatrix} I_n & \rho I_n \\ \rho I_n & I_n \end{smallmatrix}\big)$,
we have
\begin{equation}\label{eq:functional-borell}
\E J_\rho(\phi(N), \psi(M)) \le J_\rho(\E \phi, \E \psi).
\end{equation}
Note that $J_{\rho}(0,x) = 0$ for all $x$ and $J_{\rho}(1,1) =1$. 
Therefore for any two sets $A,B$:
\[
\E J_\rho(1_A(N), 1_B(M)) = \IP[N \in A, M \in B]. 
\]
Thus~(\ref{eq:functional-borell}) applied to indicator functions, implies Borell's inequality~\cite{Borell:85}
(Theorem~\ref{thm:half_space}). 
\begin{exercise}
Prove that (\ref{eq:functional-borell}) is in fact, equivalent to Theorem~\ref{thm:half_space}.
\end{exercise}

To prove Majority is Stablest, we would like to show an analogous inequality on the cube, with $f, g: \{-1, 1\}^n \to [0, 1]$ and
$X, Y$ correlated points on $\{-1, 1\}^n$. Our main observation is that while
inequality~\eqref{eq:functional-borell} is not true in this setup, it is
true with some extra error terms on the right hand side. These error terms will diminish in the low influence and Gaussian case. 
First, let us
define the error term:
\begin{definition}\label{def:DMS_error}
 Let $\Omega$ be a probability space and $f : \Omega^n \to \R$.
 Consider the martingale defined as 
 \[
 f_0 = f, f_1 = \IE[f | X_2,\ldots,X_n],\ldots, f_i = \IE[f | X_{i+1},\ldots,X_n],\ldots, f_n = \IE[f],
 \]
 and let 
 \[
 \Delta_m(f) = \sum_{i=1}^m \IE[|f_i-f_{i-1}|^3],
 \]
 for $m \leq n$. 
 \end{definition}
 Note that by orthogonality of martingale increments we know that 
 \[
 \sum_{i=1}^n \IE[|f_i-f_{i-1}|^2] = \Var[f],
 \]
 and by Jensen's inequality, it follows that 
 \begin{exercise}
 \[
 \IE[|f_i-f_{i-1}|^2] \leq I_i(f).
 \]
 \end{exercise}
 We therefore expect $\Delta_n(f)$ to be small when the differences $|f_i-f_{i-1}|$ are typically small, i.e., $f$ is smooth to changing one coordinate at a time in some average sense. 
 
 \begin{definition}
 Let $\Omega_1$ and $\Omega_2$ be two sets and $\mu$ be a probability
 measure on $\Omega_1 \times \Omega_2$. We say that $\mu$ has R\'enyi correlation at most $\rho$
 if for every measurable (w.r.t $\mu$) $f: \Omega_1 \to \R$
 and $g: \Omega_2 \to \R$ with $\E_\mu f = \E_\mu g = 0$,
 \[
 | \E_\mu[fg] | \le \rho \sqrt{\E_\mu[f^2] \E_\mu[g^2]}.
 \]
\end{definition}

For example, suppose that $\Omega_1 = \Omega_2$ and suppose $(X, Y)$ are
generated by the following procedure: first choose $X$ according to some
distribution $\nu$. Then, with probability
$\rho$, we set $Y = X$, and with probability $1-\rho$, $Y$ is chosen to be an independent sample from $\nu$. If $\mu$ is the distribution of $(X, Y)$, then
it is easy to check that $\mu$ has R\'enyi correlation $\rho$.
In particular, in the definition of noisy inner product, we let $(x,y) \in \{-1,1\}^2$ have $\IE[y] = \IE[y] = 0$ and $\IE[xy] = \rho$ and $(x,y)$ have  R\'enyi correlation $\rho$.

We prove the following general theorem, which we will later
use
to derive both Borell's inequality and the ``Majority is Stablest'' theorem.
\begin{theorem}\label{thm:tensorization}
For any $\epsilon > 0$ and $0 < \rho < 1$, there is $C(\rho) > 0$ such
that the following holds.
Let $\mu$ be a measure on $\Omega_1 \times \Omega_2$
with R\'enyi correlation at most $\rho$  and
let $(X_i, Y_i)_{i=1}^n$ be i.i.d.\ variables with distribution $\mu$.
Then for any measurable functions
$f: \Omega_1^n \to [\epsilon, 1-\epsilon]$ and
$g: \Omega_2^n \to [\epsilon, 1-\epsilon]$,
\[
 \E J_\rho(f(X), g(Y)) \le J_\rho(\E f, \E g) + C(\rho) \epsilon^{-C(\rho)} (\Delta_n(f) + \Delta_n(g)).
\]
\end{theorem}

To see that there isn't much difference between the case of $[0,1]$ valued functions and 
$[\eps,1-\eps]$ valued functions, we note that 

\begin{proposition}
Consider the setup of Theorem \label{thm:tensorization}.  
Let $0 < \eps < 0.5$ and let 
\[
\bar{f}(x) =  \begin{cases} x, & \eps \leq x \leq 1-\eps, \\
                                               \eps, & x \leq \eps, \\ 
				             1-\eps, & x \geq 1-\eps. \end{cases} 
\]
and similarly $\bar{g}(x)$. Then $| \IE \bar{f} - \IE[f]| \leq \eps$ and similarly for $g$. 
Moreover, 
\[
\Big| \E J_\rho(f(X), g(Y))  -  \E J_\rho(\bar{f}(X), \bar{g}(Y)) \Big| \leq 2 \eps 
\]
and $I_i(\bar{f}) \leq I_i(f)$ for all $i$ and similarly for $g$. 
\end{proposition}

\begin{proof}
The proof follows from the fact that $J_{\rho}$ is $1-Lip$ in each of its coordinates.
\end{proof}
 
 \section{The Base Case}
 
We prove Theorem~\ref{thm:tensorization} by induction on $n$.
In this section, we will prove the base case $n=1$:
\begin{claim}\label{clm:base-case}
  For any $\epsilon > 0$ and $0 < \rho < 1$, there is a $C(\rho)$
  such that for any two random variables
  $X, Y \in [\epsilon, 1-\epsilon]$ with correlation in $[-\rho, \rho]$,
  \[
   \E J_\rho(X, Y) \le J_\rho(\E X, \E Y) + C(\rho)\epsilon^{-C(\rho)}(\E |X - \E X|^3 + \E |Y - \E Y|^3).
  \]
\end{claim}

 The proof of Claim~\ref{clm:base-case} essentially
 follows from Taylor's theorem applied to the
 function $J_\rho$; the crucial point is that $J_\rho$ satisfies a certain differential equation.
Define the matrix $M_{\rho \sigma}(x,y)$ by
$$
M_{\rho \sigma}(x,y) =  \left(\begin{array}{cc} \frac{\partial^2 J_{\rho}(x,y)}{\partial x^2} & \sigma \frac{\partial^2 J_{\rho}(x,y)}{\partial x \partial y} \\\sigma \frac{\partial^2 J_{\rho}(x,y)}{\partial x \partial y}  & \frac{\partial^2 J_{\rho}(x,y)}{\partial y^2}\end{array}\right).
$$
\begin{restatable}{claim}{negsemi}
\label{clm:negative-semidefinite}
For any  $(x,y) \in (0,1)^2$ and $0 \le |\sigma| \le \rho$, $M_{\rho \sigma}(x,y)$ is a negative semidefinite matrix.  Likewise, if $|\sigma| \geq \rho$, then $M_{\rho \sigma}(x,y)$ is a positive semidefinite matrix.
\end{restatable}

We will also use the fact that the third derivatives of $J_\rho$ are bounded
(at least, away from the boundary of $[0, 1]^2$).

\begin{restatable}{claim}{thirddiff}
\label{clm:third-derivative}
For any $-1 < \rho < 1$, there exists $C(\rho) > 0$ such that
for any $i, j \ge 0$,  $i+j=3$,
$$\left|\frac{\partial^3 J_{\rho}(x,y)}{\partial x^i \partial y^j} \right| \le C(\rho) (xy(1-x) (1-y))^{-C(\rho)}. $$ Further, the function $C(\rho)$ can be chosen so that it is continuous for $\rho \in (-1,1)$.
\end{restatable}

Claims~\ref{clm:negative-semidefinite} and~\ref{clm:third-derivative} follow from elementary
calculus, and we defer their proofs (we note that Claim~\ref{clm:negative-semidefinite} is implicit in~\cite{MosselNeeman:15b}).
  Now we will use them with Taylor's theorem to prove Claim~\ref{clm:base-case}.
 \begin{proof}[Proof of Claim~\ref{clm:base-case}]
 Fix $\epsilon > 0$ and $\rho \in (0, 1)$.
 Now let $C(\rho)$ be large enough so that all third derivatives of $J_\rho$
 are uniformly bounded by $C(\rho) \epsilon^{-C(\rho)}$ on the square $[\epsilon, 1-\epsilon]^2$
 (such a $C(\rho)$ exists by Claim~\ref{clm:third-derivative}).
 Taylor's theorem then implies that for any $a, b, a+x, b+y \in [\epsilon, 1-\epsilon]$,
 \begin{multline}\label{eq:taylor}
  J_\rho(a + x, b + y) \le J_\rho(a, b) + x \pdiff {J_{\rho}}x (a, b) + y\pdiff {J_{\rho}}y (a, b) \\ +
  \frac{1}{2} (x \ y)
  \begin{pmatrix} \pdiffII {J_\rho}xx (a, b) & \pdiffII {J_\rho}xy (a, b) \\ \pdiffII {J_\rho}xy(a, b) & \pdiffII {J_\rho}yy(a, b)\end{pmatrix}
  \begin{pmatrix}x \\ y\end{pmatrix} + C(\rho)\epsilon^{-C(\rho)} (|x|^3 + |y|^3).
 \end{multline}

 Now suppose that $X$ and $Y$ are random variables taking values in
 $[\epsilon, 1-\epsilon]$. If we apply~\eqref{eq:taylor} with $a = \E X$, $b = \E Y$,
 $x = X - \E X$, and $y = Y - \E Y$, and then take expectations of both sides, we obtain
 \begin{multline}\label{eq:random-taylor}
  \E J_\rho(X, Y) \le J_\rho(\E X, \E Y) + \frac{1}{2}
  \E\left[
  (\tilde X\ \tilde Y)
  \begin{pmatrix}\pdiffII {J_\rho}xx(a, b) & \pdiffII {J_\rho}xy(a, b) \\ \pdiffII {J_\rho}xy(a, b) & \pdiffII {J_\rho}yy(a, b)\end{pmatrix}
  \begin{pmatrix}\tilde X \\ \tilde Y \end{pmatrix} \right] \\
  + C(\rho)\epsilon^{-C(\rho)} (\E |\tilde X|^3 + \E |\tilde Y|^3),
 \end{multline}
 where $\tilde X = X - \E X$ and $\tilde Y = Y - \E Y$.
 Now, if $X$ and $Y$ have correlation $\sigma \in [0, \rho]$ then
 $\E \tilde X \tilde Y = \sigma \sqrt{\E \tilde X^2 \E \tilde Y^2}$, and so
 \[
  \E\left[
  (\tilde X\ \tilde Y)
  \begin{pmatrix}\pdiffII {J_\rho}xx(a, b) & \pdiffII {J_\rho}xy(a, b) \\ \pdiffII {J_\rho}xy(a, b) & \pdiffII {J_\rho}yy(a, b)\end{pmatrix}
  \begin{pmatrix}\tilde X \\ \tilde Y \end{pmatrix} \right] \\
  = (\sigma_X\ \sigma_Y)
  \begin{pmatrix}\pdiffII {J_\rho}xx(a, b) & \sigma \pdiffII {J_\rho}xy(a, b) \\ \sigma \pdiffII {J_\rho}xy(a, b) & \pdiffII {J_\rho}yy(a, b)\end{pmatrix}
  \begin{pmatrix}\sigma_X \\ \sigma_Y \end{pmatrix},
  \]
where $\sigma_X = \sqrt{\E \tilde X^2}$ and $\sigma_Y = \sqrt{\E \tilde Y^2}$.
By Claim~\ref{clm:negative-semidefinite},
\[
  (\sigma_X\ \sigma_Y)
  \begin{pmatrix}\pdiffII {J_\rho}xx(a, b) & \sigma \pdiffII {J_\rho}xy(a, b) \\ \sigma \pdiffII {J_\rho}xy(a, b) & \pdiffII {J_\rho}yy(a, b)\end{pmatrix}
  \begin{pmatrix}\sigma_X \\ \sigma_Y \end{pmatrix}
  =
  (\sigma_X\ \sigma_Y)
  M_{\rho\sigma}(a, b)
  \begin{pmatrix}\sigma_X \\ \sigma_Y \end{pmatrix} \le 0.
\]
Applying this to~\eqref{eq:random-taylor}, we obtain
\[
 \E J_\rho(X, Y) \le J_\rho(\E X, \E Y) + C(\rho)\epsilon^{-C(\rho)} (\E |\tilde X|^3 + \E |\tilde Y|^3).
 \qedhere
\]
\end{proof}

\section{The inductive step} \label{subsec:induct}
Next, we prove Theorem~\ref{thm:tensorization} by induction.
\begin{proof}[Proof of Theorem~\ref{thm:tensorization}]
Note that the base case follows from~~\ref{clm:base-case} since $f(X_1), g(Y_1)$ have correlation between
    $-\rho$ and $\rho$ by the assumption on the Renyi correlation between the $X$'s and $Y$'s. 
    
We now prove the inductive claim:
 Assume that the theorem holds with $n$ replaced by $n-1$.
 Consider $f: \Omega_1^n \to [\epsilon, 1-\epsilon]$ and
 $g: \Omega_2^n \to [\epsilon, 1-\epsilon]$.

 Conditioning on $(X_n, Y_n)$
 we have
 \[
  \E J_\rho(f(X), g(Y)) = \E [\E [ J_\rho(f(X), g(Y)) | X_n, Y_n]]. 
 \]
 Applying the inductive hypothesis for $n-1$ conditionally on $X_n$ and $Y_n$,
 \begin{multline}\label{eq:induction-1}
 \E [ J_\rho(f(X), g(Y)) | X_n, Y_n]
 \le J_\rho(\E[f | X_n], \E[g |Y_n])\\
 + C(\rho)\epsilon^{-C(\rho)} (\Delta_{n-1}(f) + \Delta_{n-1}(g)).
 \end{multline}
 On the other hand, the base case for $n=1$ implies that
 \begin{multline}\label{eq:induction-2}
\E[ J_\rho(\E[f | X_n], \E[g |Y_n])]
  \le J_\rho(\E f, \E g)
  + C(\rho)\epsilon^{-C(\rho)} (\Delta_1 (\E[f | X_n]) + \Delta_1 (\E[g|Y_n])).
 \end{multline}
 Taking the expectation of~\eqref{eq:induction-1} over $X_n$ and $Y_n$ and combining it with~\eqref{eq:induction-2},
 we obtain
 \begin{eqnarray*}
  \E J_\rho(f(X), g(Y)) &\le& J_\rho(\E f, \E g) +
   C(\rho)\epsilon^{-C(\rho)} \big(\E[ \Delta_{n-1}(f) + \Delta_{n-1}(g)] \big) \\
 &+& C(\rho)\epsilon^{-C(\rho)} \big(\Delta_1 (\E[f| X_n]) + \Delta_1 (\E[g|Y_n])\big).
 \end{eqnarray*}
 Finally, note that the definition of $\Delta_n$ implies that the right-hand side above is just
 \[
  J_\rho(\E f, \E g) + C(\rho)\epsilon^{-C(\rho)}\big(
  \Delta_n (f) + \Delta_n(g)
  \big).
  \qedhere
 \]
\end{proof}

\section{Proof of Borell's Theorem}

We will now prove the following functional version of Borell's result. It is easy to see it implies Theorem~\ref{thm:half_space}.
\begin{theorem}\label{thm:borell_functional}
Let $\rho \ge 0$ and  $N$ and $M$ are Gaussian vectors with joint distribution 
 \[
  (N, M) \sim \mathcal{N}\left(0, \begin{pmatrix} I_d & \rho I_d \\ \rho I_d & I_d\end{pmatrix}\right).
 \]
 For any measurable $f, g: \R^d \to [0, 1]$,
 \[
  \E J_\rho(f(N), g(M)) \le J_\rho(\E f, \E g).
 \]
\end{theorem}

\begin{proof}[Proof of Theorem~\ref{thm:borell_functional}]
 Let $n = md$ and define
 \begin{align*}
  G_{1,n} &= \frac{1}{\sqrt m} \left(
  \sum_{i=1}^m x_i,
  \sum_{i=m+1}^{2m} x_i,
  \dots,
  \sum_{i=(d-1)m + 1}^{md} x_i
  \right).
 \end{align*}
 Define $G_{2,n}$ similarly by using $y$ instead of $x$. In other words, $G_{1,n}$ and $G_{2,n}$ are vectors obtained by averaging the vectors $x$ and $y$ over consecutive  blocks of size $m$. Define $Z$ as 
\[
 Z= (x_1, x_{m+1}, \ldots, x_{(d-1)m+1}, y_1, y_{m+1}, \ldots, y_{(d-1) m+1}).
 \] 
 Observe that $(G_{1,n}, G_{2,n})$ is distributed as sum of $m$ independent copies of $Z$ scaled by $1/\sqrt{m}$. Applying the Lindeberg-Feller central  limit theorem~\cite{Feller:68}, we obtain $(G_{1,n}, G_{2,n}) \to_D (N, M)$ as $m \to \infty$.

 Suppose first that $f$ and $g$ are $L$-Lipschitz functions taking values
 in $[\epsilon, 1-\epsilon]$. Define $F, G: \{-1, 1\}^n \to \mathbb{R}$ as
 $$
 F (z) = f\left(\frac{1}{\sqrt m} \left(
  \sum_{i=1}^m z_i,
  \sum_{i=m+1}^{2m} z_i,
  \dots,
  \sum_{i=(d-1)m + 1}^{md} z_i
  \right) \right), 
 $$ 
and similarly for $G$ using $g$. 
 By Theorem~\ref{thm:tensorization},
 \begin{equation}\label{eq:borell-proof-1}
  \E J_\rho(F(x), G(y)) \le J_\rho(\E F, \E G) + C(\rho) \epsilon^{-C (\rho)} (\Delta_n(F) + \Delta_n(G)).
 \end{equation}
 Since $f$ is $L$-Lipschitz,
 \[
  |F(x) - F(x \oplus e_j)| \le \frac{2L}{\sqrt m},
 \]
 for every $j$, and similarly for $g$. Therefore:
 \[
 \max( \Delta_n(F),\Delta_n(G))  \le \frac{ 8L^3 n}{m^{3/2}} = \frac{ 8L^3 d}{\sqrt m}.
 \]
 Applying this to~\eqref{eq:borell-proof-1},
 \[
  \E J_\rho(F(x), G(y)) \le J_\rho(\E F, \E G) + C(\rho) \epsilon^{-C(\rho)} \frac{16 L^3 d}{\sqrt m},
 \]
 and so the definition of $F,G$ implies
 \[
  \E J_\rho(f(G_{1,n}), g(G_{2,n})) \le J_\rho(\E f(G_{1,n}), \E g(G_{2,n})) +  C(\rho) \epsilon^{-C(\rho)} \frac{16 L^3 d}{\sqrt m}.
 \]
 Taking $m \to \infty$, the multivariate central limit theorem and using the fact that $J_\rho$ is Lipschitz
 implies that
 \begin{equation}\label{eq:borell-lipschitz}
  \E J_\rho(f(N), g(M)) \le J_\rho(\E f(N), \E g(M)).
 \end{equation}

 This establishes the theorem for functions $f$ and $g$ which are Lipschitz and take values
 in $[\epsilon, 1-\epsilon]$. But any measurable $f_1, f_2: \R^d \to [0, 1]$ can be approximated
 (say in $L^p(\R^d, \gamma_d)$) by Lipschitz functions with values in $[\epsilon, 1-\epsilon]$.
 Since neither the Lipschitz constant nor $\epsilon$ appears in~\eqref{eq:borell-lipschitz},
 the general statement of the theorem follows from the dominated convergence theorem.
\end{proof}

\section{Smoothness and hyper-contraction}
It turns out that there is an effective bound on $\Delta_n (T_{\eta}f)$ in the case where $f$ has low influences, where $T_{\eta}$ is the noise operator and $\eta$ is close to $1$. 
 In this section we show how the fact that the noise operator $T_{\eta}$ is ``hyper-contractive" allows to bound $\Delta_n (T_{\eta}f)$ when $f$ has low influences. In the next section, we will show that the smoothness of $J$ implies that only a small error results from replacing $f$ by $T_{\eta} f$ for $\eta$ close to $1$.  
 
 Hypercontractivity is a key feature of many of the proofs in the analysis of Boolean functions starting 
with the KKL paper~\cite{KaKaLi:88}. The proof will use the famous hyper-contractive theorem of Bonami and Beckner: 
\begin{theorem}\cite{Bonami:70, Beckner:75}\label{thm:BB}
Let $f : \{-1,1\}^n \rightarrow \mathbb{R}$ and $1 \le q \le p$. Then, if $\rho^2 \le \frac{q-1}{p-1}$, then: 
\[
\| T_{\rho} f \|_p \le \|  f \|_q.
\]
\end{theorem}

We will use this fact to analyze the differences  $\IE[|f_i-f_{i-1}|^3]$. In particular we will prove the following: 
\begin{lemma} \label{lem:mart_cubed}
Let $f : \{-1,1\}^n \to [0,1]$ and let $0 < \rho < 1$. Then:
\begin{itemize}
\item 
$(T_\rho f)_i = T_{\rho} (f_i)$,
\item 
\[ 
\IE [ |T_\rho (f_i - f_{i-1})|^3]  \leq  I_i(f)^{p/2},
\]
where $p = \min(3,1 + 1/\rho^2)$,
\item 
\[
\Delta_n(T_{\rho^2} f) \leq \frac{1}{1-\rho} \max I_i(f)^{p/2-1}.
\]  
\end{itemize}
\end{lemma}

\begin{proof}
The fact that $(T_\rho f)_i = T_{\rho} (f_i)$ follows immediately from the definitions. 
Let $p = \min(3,1 + 1/\rho^2)$. Then we apply Theorem~\ref{thm:BB} to obtain:
\begin{eqnarray*}
\IE [ |T_\rho (f_i - f_{i-1}) |^3] &=&  \IE [ |T_\rho (f_i - f_{i-1}) |^p |T_\rho (f_i - f_{i-1}) |^{3-p}] \\ 
&\leq&  \IE [ |T_\rho (f_i - f_{i-1})|^p] = \| T_\rho (f_i - f_{i-1}) \|_p^p \leq \| f_i - f_{i-1}\|_2^p \leq I_i^{p/2},
\end{eqnarray*} 
and the proof follows of the second statement follows. To prove the third statement observe that:
\begin{eqnarray*}
\Delta_n(T_{\rho^2} f) &=& \sum_{i=1}^n \IE [ |T_\rho^2 (f_i - f_{i-1})|^3] \\ 
                               &\leq& \sum_{i=1}^n   I_i(T_{\rho} f)^{p/2} \leq  \max I_i(T_{\rho} f)^{p/2-1}  \sum_{i=1}^n  I_i(T_{\rho} f) \\ 
                               &\leq& \frac{1}{1-\rho} \max I_i(f)^{p/2-1}.
\end{eqnarray*} 
\end{proof} 


\section{Proof of Majority is Stablest}
\begin{proof}
Fix $\eps > 0$. Our goal is to prove that:  
\begin{equation} \label{eq:MISTp1}
\langle f, g \rangle_{\rho} \leq \langle \chi_{\IE f}, \chi_{\IE g} \rangle_{\rho}+ \eps = J_{\rho}(\IE f, \IE g) + \eps, 
\end{equation} 
when $\max(I_i(f), I_i(g)) \leq \tau$ for a small enough $\tau(\eps)$. 

Let $\eps' = \eps/10$. Let $\bar{f}$ be the rounding of $f$ to the interval $[\eps',1-\eps']$ and similarly define $\bar{g}$. 
We first note that it suffices to prove the desired claim~\ref{eq:MISTp1} for $\bar{f}$ and $\bar{g}$ with error $\eps'$. 
This is because 
\[
|\langle f, g \rangle_{\rho} - \langle \bar{f}, \bar{g} \rangle_{\rho}| \leq 2 \eps', \quad |\IE f - \IE f'| \leq \eps', \quad |\IE g - \IE g'| \leq \eps
\]
and because $J_{\rho}(x,y)$ is $1$-Lip in each of the coordinates $x$ and $y$. 
Moreover, truncation reduces variance and therefore $I_i(\bar{f}) \leq I_i(f)$ for all $i$ and similarly for $g$. 
For simplicity of notation we redefine $\eps = \eps'$ and our goal is now to prove~(\ref{eq:MISTp1}) for functions bounded in the interval 
$[\eps,1-\eps]$. 

Again, define $\eps' = \eps/10$ and let $\rho = \eta^2 \rho'$ for $\eta < 1$ be chosen so that 
\[
J_{\rho'}(\IE f, \IE g) \leq J_{\rho}(\IE f, \IE g) + \eps'.
\]
Then it suffices to prove that 
\[
\langle f, g \rangle_{\rho}  =  \langle T_{\eta} f, T_{\eta} g \rangle_{\rho'} \leq J_{\rho'}(\IE f, \IE g) + \eps'.
\]
and since $J_{\rho'}(x, y) \ge xy$ it suffices to prove that 
\begin{equation} \label{eq:maj_stab_smooth}
\IE[J_{\rho'}(T_{\eta} f, T_{\eta} g)] \leq J_{\rho'}(\IE f, \IE g) + \eps'.
\end{equation}
Again we have that $I_i(T_{\eta} f) \leq I_i(f)$ for all $i$ and similarly for $g$. 
Renaming again $\rho = \rho'$ and $\eps' = \eps$ and applying Theorem~\ref{thm:tensorization} what remain to show is that 
\[
C(\rho) \eps^{-C(\rho)} (\Delta_n(T_{\eta} f) + \Delta_n(T_{\eta} g)) \leq \eps.
\]
By Lemma~\ref{lem:mart_cubed} this can be bounded by:
\[
C(\rho) \eps^{-C(\rho)} C(\eta) \max I_i(f)^{\eta} 
\]
Thus for $\tau$ small enough, this is less than $\eps$ as needed. 
\end{proof}

\section{Proof for resilient functions}
We now prove Theorem~\ref{thm:MIST_res}. 
First we note that in the inductive proof in subsection~\ref{subsec:induct}, if we stop $k$ steps before the end we obtain: 

\begin{proposition} \label{prop:part_tail}
For every $1 \leq k \leq n$, it holds that:
\begin{eqnarray*}
\E J_\rho(f(X), g(Y)) &\leq& 
\E \Big[J_\rho( f_{n-k}(X),g_{n-k}(Y)) \Big] 
\\ 
 &+& C(\rho)\epsilon^{-C(\rho)} \E \Big[ \Delta_{n-k}(f) + \Delta_{n-k}(g) \Big],
\end{eqnarray*}
where $f_{n-k} = \E[f | X_{n-k},\ldots,X_k]$ and $g_{n-k} = \E[g | Y_{n-k},\ldots,Y_n]) \Big]$. 
\end{proposition}
\begin{proof}
The proof will imitate the previous proof up to (\ref{eq:maj_stab_smooth}).
Using the fact that 
\[
\sum_{i=1}^n I_i(T_{\eta} f) \leq (1-\eta)^{-1},
\]
It follows that there are most $k = O(\tau^{-1} (1-\eta)^{-1})$ coordinates of 
$f$ and $g$ with influence more than $\tau$. Without loss of generality, assume that these are the the last $k$ coordinates. 
Then apply Proposition~\ref{prop:part_tail} to conclude that: 
\[
\IE[J_{\rho'}(T_{\eta} f, T_{\eta} g)] \leq \E [J_\rho(T_{\eta} f_{n-k}(X),T_{\eta} g_{n-k}(Y))] + \kappa,
\]
where the error term $\kappa$ is bounded as before by 
\[
C(\rho) \eps^{-C(\rho)} C(\eta) \max I_{1 \leq i \leq n-k}(f)^{\eta} \leq \eps/2. 
\]
Now if the function $f$ is $(k,\eps/100)$ resilient then so is $f_{n-k}$ and 
$T_{\eta} f_{n-k}$. Therefore 
\[
|f_{n-k} - \IE[f]| \leq \eps/100,
\]
and therefore since $J$ is Lipchitz it holds that  
\[
\E [J_\rho(T_{\eta} f_{n-k}(X),T_{\eta} g_{n-k}(Y))] \leq 
\E [J_\rho(\E f,T_{\eta} g_{n-k}(Y))] + \eps/100 \leq J_{\rho}(\IE f, \IE g) + \eps/100,
\]
where the ultimate inequality follows from the fact that $J$ is concave in each of the coordinates. 
\end{proof} 

Note that the only place where we used the resilience of $f$ in showing that 
\[
\E [J_\rho(T_{\eta} f_{n-k}(X),T_{\eta} g_{n-k}(Y))] \leq J_{\rho}(\IE f, \IE g) + \eps/100
\]
More generally:
\begin{exercise}
\begin{itemize}
\item 
Prove that if the $W_1$ distance between the joint distribution of 
$(T_{\eta} f_{n-k}(X),T_{\eta} g_{n-k}(Y))$ and the product of the marginals is bounded by $\eps$ then 
\[
\E [J_\rho(T_{\eta} f_{n-k}(X),T_{\eta} g_{n-k}(Y))] \leq J_{\rho}(\IE f, \IE g) + \eps.
\]
\item 
Show that if the $W_1$ distance between the joint distribution of 
$(T_{\eta} f_{n-k}(X),T_{\eta} g_{n-k}(Y))$ and the product of the marginals is smaller than the corresponding distance between the joint distribution of 
$( f_{n-k}(X), g_{n-k}(Y))$ and the product of their marginals. 
\item 
Show that if $f,g : \{-1,1\}^k \to \{0,1\}$ and the distance between 
 the joint distribution of  $( f(X), g(Y))$ and the product of their marginals is more than 
 $\eps$ then:
 \[
 \sum_{S \neq \emptyset} |\hat{f}(S) \hat{g}(S)| \geq 
 |\sum_{S \neq \emptyset} \rho^{|S|} \hat{f}(S) \hat{g}(S)| \geq \eps.
 \]
 \end{itemize}
\end{exercise}

This exercise implies in particular the following Theorem:

\begin{theorem}
For every $\eps > 0, 0 < \rho < 1$, there exists $k$ and $\alpha$ such that if 
$f,g : \{-1,1\}^n \to \{0,1\}$ satisfy 
\[
\langle f, g \rangle_{\rho} > J_{\rho}(\E f, \E g) + \eps,
\]
then there exists a set $S$ with $|S| \leq k$ such that $\hat{f}(S) \hat{g}(S) \geq \alpha$. 
\end{theorem}

\section{Cross Influences}
We now sketch an alternative proof of Theorem~\ref{thm:MIST} given Theorem~\ref{thm:MISTsimple}. For more details see~\cite{Mossel:10}.

\begin{proof}\mbox{}
  
\begin{itemize}
\item 
We will use the fact that if $I_i(g) \leq \tau'$, then there exists a function $g'$ that does not depend on $i$ such that $\E[|g-g'|^2] \leq \tau'$, i.e.,
$g' = \E[g | x_{-i}]$ and that if we let $f'  = \E[f | x_{-i}]$ then for all $\rho$: 
$\langle f',g' \rangle_{\rho} = \langle f,g' \rangle_{\rho}$. 
\item
As before, it suffices to prove the theorem for $T_{\gamma} f$ and $T_{\gamma} g$ for some $\gamma$ close to, but less than $1$ incurring an error of 
$0.1 \eps$. 
\item We know that if all the influences of both $f$ and $g$ are less than some
  $\tau$ then we obtain the right statement with error $0.1 \eps$.
\item By Lemma~\ref{lem:inf_sum_rho},
\[
\sum_i I_i(T_{\gamma} f) \leq \max_k k \gamma^k \leq C(\gamma),
\]
and thus there are at most $C(\gamma)/\tau$ coordinates of $f$
with influence greater than $\tau$. Let us denote this set of coordinates by $A_f$.
Let $A_g$ be the corresponding set for $g$.  
\item 
  Let us choose $\tau'$ so that $\sqrt{\tau'} C(\gamma)/\tau \leq 0.1 \eps$.  
  Suppose that $\min(I_i(f), I_i(g)) \leq \tau'$ for all $i$. Note that this implies in particular that $A_f$ and $A_g$ are disjoint. 

Let $f' = \IE[T_{\gamma} f | x_{-A_g}]$ and $g' = \IE[T_{\gamma} f | x_{-A_f}]$.  Then:
\[
| \langle f', g' \rangle - \langle T_{\gamma} f, T_{\gamma} g \rangle| \leq 0.2 \eps.
\]
Moreover, $\max(I_i(f'), I_i(g')) \leq \tau$ so we can apply Theorem~\ref{thm:MISTsimple} to conclude. 
\end{itemize} 
\end{proof} 

\section{Majority is the most stable resilient function}
We now prove Theorem~\ref{thm:MIST_res} using different ideas from additive combinatorics. 
Our proof follows~\cite{Mossel:20resilient}. Similar ideas appeared before in~\cite{OSTW:10}, see also~\cite{Jones:16}. 
The proof will use Decision Trees as well as the following standard estimates,
see e.g.~\cite[Appendix B]{MoOdOl:10} and subsection \ref{subsec:J}. 

\begin{lemma} \label{lem:gaussian_continuous} 
Assume $\rho < 1$ and $\rho_1 < \rho_2 < 1$ then 
\[
| \langle \chi_{\mu_1}, \chi_{\mu_2} \rangle_{\rho_1} - \langle  \chi_{\mu_1}, \chi_{\mu_2} \rangle_{\rho_2} | 
\leq    \frac{10 (\rho_2 - \rho_1)}{1-\rho_2}
\]

\[
| \langle  \chi_{\mu'_1}, \chi_{\mu'_2} \rangle_{\rho} - \langle  \chi_{\mu_1}, \chi_{\mu_2} \rangle_{\rho} | 
\leq     |\mu_1-\mu_1'|  +  |\mu_2-\mu_2'| 
\]
\end{lemma} 

Decision trees allow to express Boolean function by querying variables on at a time, where the order of queries may depend on the values of previous variables. In the context of low influence functions, decision trees are useful as they allow to express functions of the form $T_{\gamma} h$ as a bounded size decision tree, where almost all the leaves are low influence functions. We will use the following regularity lemma, see e.g.~\cite{MosselSchramm:08,DKN:10}. 
\begin{lemma} \label{lem:dec}
For a function $f : \{-1,1\}^n \to \IR$, let $I(f) = \sum I_i(f)$. 
Then for any $\tau > 0, \eps > 0$ and any function $f$ there exists a decision tree for $f$ of depth at most 
\[
d \leq 2+ \frac{I}{\tau \eps} 
\]
such that the probability of reaching a leaf with influence sum $\tau$ or more is bounded by $\eps$. 
\end{lemma}

\begin{proof}
The construction of the decision tree is standard. If a function $f_{x_I}$ at a certain node $x_I$ has all influences less than $\tau$ or if the node is at level $d$ do nothing. Otherwise, condition on the variable $j$ with the maximum influence in $f_{x_I}$ and create two children $f_{y_J}$ and $f_{z_K}$ where $J = K = I \cup \{ j \}$, $y_i  = z_i = x_i$ for all $i \in I$ 
and $y_j =0$ and $z_j = 1$. Since 
\[
I(f_{x_I}) = I_j(f_{x_I}) + \frac{1}{2}(I(f_{y_J}) + I(f_{z_K})),
\]
it easily follows that if $L$ is the set of leaves of the tree and if $D(\ell)$ denotes the depth of leaf $\ell$ then 
\[
I(f) \geq \tau \sum_{\ell \in L} 2^{-D(\ell)} D(\ell).
\]
Therefore if $p$ is the fraction of paths that reach level $d$ then 
\[
I(f) \geq (d-1) \tau p \implies p \leq I / (d-1) \tau,
\]
and taking $d -1$ to be the smallest integer that is greater or equal to $\frac{I}{\tau \eps}$ we obtain the desired result. 
\end{proof} 

The proof will be carried out via a sequence of reductions which will give the function $f$ more and more structure. 
Fix $\eps > 0$ and $0 \leq \rho' < 1$. 
Recall that we want to show that if $f$ is $(d(\eps,\rho'),\alpha(\eps,\rho'))$ resilient  then for all $g$ bounded between 
$0$ and $1$, 
\begin{equation} \label{eq:goal0}
\langle  f, g \rangle_{\rho'} \leq \langle \chi_{\mu_f}, \chi_{\mu_g} \rangle_{\rho'}+ \eps,
\end{equation}
where $\mu_f = \IE f$ and similarly for $g$. 
\begin{lemma} \label{lem:eta}
In order to prove~(\ref{eq:goal0}) it suffices to prove that 
\begin{equation} \label{eq:goal1}
\langle  T_{\eta} f, g \rangle_{\rho} \leq \langle  \chi_{\mu_f}, \chi_{\mu_g} \rangle_{\rho} + \eps/2. 
\end{equation}
for 
\begin{equation} \label{eq:rho_eta}
\rho = (1-0.01 \eps) \rho' + 0.01 \eps, \quad \eta = \rho'/\rho = 1- \Omega( \eps (1-\rho')). 
\end{equation} 
\end{lemma} 

\begin{proof} 
Write $\rho' = \rho \eta$, where $1-\rho \geq (1-\rho')/2$ and $\eta < 1$. Note that $\langle  T_{\eta} f, g \rangle_{\rho} = \langle  f, g \rangle_{\rho'}$.
Moreover,  $f$ and $T_{\eta} f$ have the same 
expected value. If we could establish~(\ref{eq:goal1}) 
and  
\begin{equation} \label{eq:diff_chis}
|\langle  \chi_{\mu_f}, \chi_{\mu_g} \rangle_{\rho} - \langle  \chi_{\mu_f}, \chi_{\mu_g} \rangle_{\rho'}| < \eps/2, 
\end{equation} 
then~(\ref{eq:goal0}) would follow. Note that (\ref{eq:diff_chis}) follows from Lemma~\ref{lem:gaussian_continuous} when 
\[
\frac{10 (\rho - \rho')}{1-\rho} < \eps/2.
\]
We may thus choose $\rho$ and $\eta$ as in~(\ref{eq:rho_eta}).  
\end{proof} 

\begin{lemma} \label{lem:dec_rep} 
Let $\tau$ be chosen so that (\ref{eq:mist}) holds with error $0.01 \eps$ for $\rho$.  
Then it suffices to prove (\ref{eq:goal1}) for a function $h = T_{\eta} f$ that has a decision tree of 
depth $d$ and such that for  at most $0.01 \eps$ fraction of the inputs a random path of the decision tree terminates at a node with some influence greater than $\tau$. Moreover
\[
d = O \left( \frac{1}{\eps^2 (1-\rho) \tau} \right)
\]
\end{lemma}

\begin{proof} 
We note that the function $h = T_{\eta} f$ satisfies: 
\[
I(h) := \sum I_i(h) = \sum_{S} |S| \hat{f}^2(S) \eta^{2 |S|} \leq \max_s s \eta^{2 s} = O\left( \frac{1}{\eps (1-\rho)} \right).  
\]
Apply Lemma~\ref{lem:dec} to obtain a decision tree for $h$ where for at most $0.01 \eps$ fraction of the inputs, a random path of the decision tree terminates at a node with some influence greater than $\tau$. Note that the depth of the tree satisfies
\[
d \leq C (1+  \frac{I}{\tau \eps}) \leq C(1 + \frac{1}{\eps^2 (1-\rho) \tau})
\] 
as needed.
\end{proof} 

We now conclude the proof of Theorem~\ref{thm:MIST_res}. 
\begin{proof} 
Let $h$ be a function such as in Lemma~\ref{lem:dec_rep}. Assume furthermore that 
$f$ is $(d,0.01 \eps 2^{-d})$ resilient. Note that this implies that $h  = T_{\eta} f$ is also $(d,0.01 \eps 2^{-d})$ resilient. 
This is because for every $S$, 
\[
\EE [T_{\eta} f | X_S = z] = \EE [\EE [ f | X_S = z']],
\]
where $z' ~ \sim \otimes_{i \in S} (\eta \delta_{z_i} + \frac{1-\eta}{2}(\delta_1 + \delta_{-1}))$, and therefore:
\[
\sup_{|S| \leq d, z} \Big| \EE \left[T_{\eta} f  | X_S = z \right] - \EE[f] \Big| \leq 
\sup_{|S| \leq d, z} \Big| \EE \left[f  | X_S = z \right] - \EE[f] \Big|.  
\]
Let $x,y$ be two $\rho$-correlated inputs. 
Then
\[
\EE[h(x) g(y)] = \sum_{x_I,y_I} \PP[x_I] \PP[y_I | x_I] \EE[h(x) g(y) | x_I, y_I],
\] 
where $x_I$ denotes a random leaf of the decision tree and $y_I$ is chosen after $x_I$ to be a $\rho$-correlated 
version of $x_I$. Let $A$ denote the set of $x_I$ for which $f_{x_I}$ has all influences less than $\tau$. Then:
\begin{eqnarray*}
\EE[h(x) g(y)] &=& \sum_{x_I,y_I} \PP[x_I] \PP[y_I | x_I] \EE[h(x) g(y) | x_I, y_I] 
\\ &\leq& 0.01 \eps + 
\sum_{x_I \in A} \PP[x_I] \sum_{y_I} \PP[y_I | x_I] \EE[h(x) g(y) | x_I, y_I]. 
\end{eqnarray*}
Write $\mu' = \mu_f + 0.01\eps$ and $\mu(y_I) = \EE[g(y) | y_I]$. 
Note that since $h$ is $(d,0.01 \eps 2^{-d})$-resilient 
it follows that for all leaves 
$x_I$ it holds that $\EE[h | x_I] \leq \mu'$. Thus for  $x_I \in A$ we can apply (\ref{eq:mist})  to obtain that 
\[
\EE[h(x) g(y) | x_I, y_I] \leq \langle  \chi_{\mu'},  \chi_{\mu(y_I)} \rangle_{\rho} 
+ 0.01 \eps.
\]
Plugging this back in we obtain the bound
\begin{eqnarray*}
\EE[h(x) g(y)] &\leq&  
0.02 \eps + \sum_{x_I \in A} \PP[x_I] \sum_{y_I} \PP[y_I | x_I] \langle  \chi_{\mu'},  \chi_{\mu(y_I)} \rangle_{\rho} \\ 
&\leq& 0.02 \eps + \sum_{x_I,y_I} \PP[x_I] \PP[y_I | x_I] \langle  \chi_{\mu'}, \chi_{\mu(y_I)} \rangle_{\rho} \\
&=& 0.02 \eps +  \langle  \chi_{\mu'},  (\sum_{x_I,y_I} \PP[x_I] \PP[y_I | x_I] \chi_{\mu(y_I)}) \rangle_{\rho}
\end{eqnarray*}
Note that $\psi = \sum_{x_I,y_I} \PP[x_I] \PP[y_I | x_I] \chi_{\mu(y_I)}$ is a $[0,1]$-valued function with
$\EE[\psi] = \EE[g]$. Thus by Theorem~\ref{thm:half_space} it follows that 
\[
0.02 \eps +  \langle  \chi_{\mu'}, (\sum_{x_I,y_I} \PP[x_I] \PP[y_I | x_I] \chi_{\mu(y_I)}) \rangle_{\rho} \leq 
0.02 \eps + \langle  \chi_{\mu'},  \chi_{\mu_g} \rangle_{\rho} \leq 0.04 \eps +  \langle  \chi_{\mu_f},  \chi_{\mu_g} \rangle_{\rho},
\]
where the last inequality follows from Lemma~\ref{lem:gaussian_continuous}. 
\end{proof} 

\section{Majority is Most Predictable}
Suppose $n$ voters are to make a binary decision.
Assume that the outcome of the vote is determined by a {\em social choice}
function $f : \bits^n \to \bits$, so that the outcome of the vote is
 $f(x_1,\ldots,x_n)$ where $x_i \in \bits$ is the vote of voter $i$.
We assume that the votes are
independent, each $\pm 1$ with probability $\frac{1}{2}$.
It is natural to assume that the function $f$ satisfies $f(-x) = -f(x)$, i.e.,
it does not discriminate between the two candidates. Note that this implies
that $\E[f] = 0$ under the uniform distribution.
A natural way to try and
predict the outcome of the vote is to sample a subset of the voters, by sampling each voter independently with probability $\rho$.
Conditioned on a vector $X$ of votes the distribution of $Y$,
the sampled votes, is i.i.d. where $Y_i = X_i$ with probability $\rho$ and
$Y_i = \ast$ (for unknown) otherwise.

Conditioned on $Y=y$, the vector of sampled votes, the optimal prediction of
the outcome of the vote is given by $\sgn((T f)(y))$ where
\begin{equation} \label{eq:defT1}
(T f)(y) = \E[f(X) | Y = y].
\end{equation}
This implies that the probability of
correct prediction (also called predictability) is given by
\[
\P[f = \sgn(T f)] =
\frac{1}{2}(1+\E[f \, \sgn(T f)]).
\]
For example, when
$f(x) = x_1$ is the dictator function, we have
$\E[f \, \sgn(T f)] = \rho$ corresponding to the trivial fact
that the outcome of the election is known when voter $1$ is sampled
and are $\pm 1$ with probability $1/2$ otherwise.
The notion of predictability is natural in statistical contexts.
It was also studied in a more combinatorial context in~\cite{Odonnell:02}.

Similarly to the Majority is Stablest Theorem, we can prove~\cite{Mossel:10,Mossel:20resilient}: 

\begin{theorem}[``Majority Is Most Predictable''] \label{thm:MIMP}
Let $0 \leq \rho \leq 1$ and $\eps > 0$ be given.  Then there
exists a $\tau > 0$ such that if $f : \bits^n \to [-1,1]$ satisfies
$\E[f] = 0$ and $\Inf_i(f) \leq \tau$ for all $i$, then
\begin{equation} \label{eq:mimp}
\E[f \, \sgn(T f) ] \leq {\textstyle \frac{2}{\pi}} \arcsin \sqrt{\rho} +
\eps,
\end{equation}
where $T$ is defined in~(\ref{eq:defT1}).

Similarly, in the same setup, for every $\eps > 0$, there exists $(r,\delta)$ such that if $f$ is $(r,\delta)$-resilient and 
satisfies $\E[f] = 0$ then (\ref{eq:mimp}) holds. 
\end{theorem}

We note that from the central limit theorem
it follows that
if $\Maj_n(x_1,\ldots,x_n) = \sgn(\sum_{i=1}^n x_i)$, then
\[
\lim_{n \to \infty}
\E[\Maj_n \sgn(T \Maj_n)] = {\textstyle \frac{2}{\pi}} \arcsin \sqrt{\rho}.
\]

\begin{remark}
Note that Theorem~\ref{thm:MIMP} proves a weaker statement than showing that Majority is the most predictable function. The statement only asserts that if a function has low enough influences than its predictability cannot be more than $\eps$ larger than the asymptotic predictability value achieved by the majority function when the number of voters $n \to \infty$.
This slightly inaccurate title of the theorem is inline with previous results such as the ``Majority is Stablest Theorem" (see below). Similar language may be used later when informally discussing statements of various theorems.
\end{remark}

\begin{remark}
One may wonder if for a finite $n$, among {\em all} functions $f : \{-1,1\}^n \to \{-1,1\}$ with $\E[f] = 0$, majority is the most predictable function. Note that the predictability of the dictator function $f(x) = x_1$ is given by $\rho$,
and $\frac{2}{\pi}\arcsin \sqrt{\rho} > \rho$ for $\rho \to 0$. Therefore when $\rho$ is small and $n$ is large
the majority function is more predictable than the dictator function. However, note that when $\rho \to 1$
we have $\rho > \frac{2}{\pi} \arcsin \sqrt{\rho}$
and therefore for values of $\rho$ close to $1$ and large $n$ the dictator
function is more predictable than the majority function. See Figure~\ref{fig:bool_gauss_predict}.
\end{remark}

\begin{figure}
\begin{center}
\begin{tikzpicture}
\begin{axis}[
    domain=0:1,
    samples=101,
    smooth,
    no markers,
    xlabel = $\rho$,
    ylabel = ,
    axis lines = left,
    legend pos = south east
    ]
    \addplot [
        domain=0:1, 
        samples=100, 
        color=red,
    ]
    {x};
    \addlegendentry{$\rho$}
    \addplot[
        domain = 0:1,
        samples = 1000,
        color = blue,
    ]
    {asin( sqrt(x) )/90};
    \addlegendentry{$2\arcsin(\sqrt\rho)/\pi$}
\end{axis}
\end{tikzpicture}
\end{center}
\caption{The predictability of dictator $\rho$ vs. that of Majority $2 \arcsin(\sqrt{\rho})/\pi$}. 
\label{fig:bool_gauss_predict}
\end{figure}
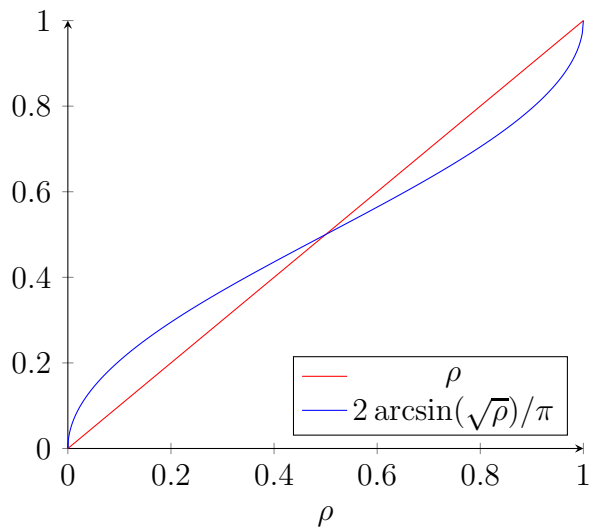

We note that the bound obtained in Theorem~\ref{thm:MIMP} is a reminiscent of
the Majority is Stablest theorem~\cite{MoOdOl:05,MoOdOl:10} as
both involve the $\arcsin$ function. However, the two theorems are quite different. The Majority is Stablest theorem
asserts that under the same condition as in Theorem~\ref{thm:MIMP} it holds
that
\[
\E[f(X) f(Y)] \leq {\textstyle \frac{2}{\pi}} \arcsin \rho + \eps.
\]
where $(X_i,Y_i) \in \bits^2$ are i.i.d. with
$\E[X_i] = \E[Y_i] = 0$ and $\E[X_i Y_i] = \rho$. Thus ``Majority is Stablest'' considers two correlated voting vectors, while ``Majority is Most Predictable'' considers a sample of one voting vector.
We note a further difference between stability and predictability:
It is well known that in the context of ``Majority is Stablest'', for {\em all} $0 < \rho < 1$,
among all boolean functions with $\E[f] = 0$ the maximum of
$\E[f(x) f(y)]$ is obtained for dictator functions of the form
$f(x) = x_i$. As discussed above, for $\rho$ close to $0$ and large $n$, the dictator is less predictable than
the majority function.

The proof of Theorem~\ref{thm:MIMP} follows the same lines of the Majority is Stablest Theorem.
The basic space $\Omega = \Omega_1 \times \Omega_2$ where $(x,y) \in \Omega$ is distributed as follows.
First $x$ is distributed uniformly in $\{\pm 1\}$. Conditioned on $x$, $y = x$ with probability $\rho$ and is equal to $\ast$ with probability 
$1-\rho$. It is easy to check that the that the Renyi correlation between $x$ and $y$ is $\sqrt{\rho}$.

\section{Facts regarding $J_\rho$} \label{subsec:J}

Here we collect various facts about the function
\[
 J_\rho(x, y) := \langle \chi_x, \chi_y \rangle_{\rho} = \Pr[X \le \Phi^{-1}(x), Y \le \Phi^{-1}(y)],
\]
where $(X, Y) \sim \mathcal{N}(0, (\begin{smallmatrix} 1 & \rho \\ \rho & 1\end{smallmatrix}))$.
As is standard, we will use $\phi$ to denote the density of the standard normal distribution. These calculations all follow from elementary calculus.

\negsemi*
\begin{proof}

Towards proving this, note that we can define $Y = \rho \cdot X + \sqrt{1-\rho^2} \cdot Z$ where $Z \sim \mathcal{N}(0,1)$ is an independent normal.  Also, let us define $\Phi^{-1}(x) = s$ and $\Phi^{-1}(y) =t$. For $s, t \in \mathbb{R}$, define $K_{\rho}(s,t)$ as
$$
K_{\rho}(s,t) = \Pr_{X,Y} [X \le s, Y \le t] = \Pr_{X,Z} [ X \le s, Z \le (t - \rho \cdot X)/\sqrt{1-\rho^2}].
$$
Note that for the aforementioned relations between $x$, $y$, $s$ and $t$, $K_{\rho}(s,t) = J_{\rho}(x,y)$.  Note that
\begin{equation}\label{eq:K}
K_{\rho}(s,t) = \int_{s'=-\infty}^{s} \phi(s') \int_{t'=-\infty}^{(t - \rho \cdot s')/\sqrt{1-\rho^2}} \phi(t') dt' ds'.
\end{equation}
This implies that
$$
\frac{\partial K_{\rho}(s,t)}{\partial s} =  \phi(s) \int_{t'=-\infty}^{(t - \rho \cdot s)/\sqrt{1-\rho^2}} \phi(t')  dt'.
$$
By chain rule, we get that
$$
\frac{\partial J_{\rho}(x,y)}{\partial x} = \frac{\partial K_{\rho}(s,t)}{\partial s} \cdot  \frac{\partial s}{\partial x}.
$$
By elementary calculus, it follows that
$$
\frac{d \Phi^{-1}(x)}{dx} = \frac{1}{\phi (\Phi^{-1}(x))} \quad \Rightarrow \quad  \frac{\partial s}{\partial x} = \frac{1}{\phi(\Phi^{-1}(x))} =\frac{1}{\phi(s)}.
$$
Thus,
$$
\frac{\partial J_{\rho}(x,y)}{\partial x} = \int_{t'=-\infty}^{(t - \rho \cdot s)/\sqrt{1-\rho^2}} \phi(t')  dt'.
$$
Thus, we next get that
\begin{align*}
  \frac{\partial^2 J_{\rho}(x,y)}{\partial x^2}
  &=\frac{\partial^2 J_{\rho}(x,y)}{\partial x \partial s}  \cdot  \frac{\partial s}{\partial x}\\
  &= \phi \left(\frac{t-\rho \cdot s}{\sqrt{1-\rho^2}} \right) \cdot \frac{-\rho}{\sqrt{1-\rho^2}} \cdot \frac{1}{\phi(s)} =\phi \left(\frac{\Phi^{-1}(y)-\rho \cdot \Phi^{-1}(x)}{\sqrt{1-\rho^2}} \right) \cdot \frac{-\rho}{\sqrt{1-\rho^2}} \cdot \frac{1}{\phi(s)}  .\\
  \frac{\partial^2 J_{\rho}(x,y)}{\partial x \partial y}
  &=\frac{\partial^2 J_{\rho}(x,y)}{\partial x \partial t} \cdot  \frac{\partial t}{\partial y} = \phi \left(\frac{\Phi^{-1}(y)-\rho \cdot \Phi^{-1}(x)}{\sqrt{1-\rho^2}} \right) \cdot \frac{1}{\sqrt{1-\rho^2}} \cdot \frac{1}{\phi(t)}.
\end{align*}
Because we know that $(X,Y) \sim (Y,X)$, by symmetry, we can conclude that
$$
\frac{\partial^2 J_{\rho}(x,y)}{\partial y^2}   =\phi \left(\frac{\Phi^{-1}(x)-\rho \cdot \Phi^{-1}(y)}{\sqrt{1-\rho^2}} \right) \cdot \frac{-\rho}{\sqrt{1-\rho^2}} \cdot \frac{1}{\phi(t)}.   $$
and likewise,
$$
\frac{\partial^2 J_{\rho}(x,y)}{\partial y \partial x} = \phi \left(\frac{\Phi^{-1}(x)-\rho \cdot \Phi^{-1}(y)}{\sqrt{1-\rho^2}} \right) \cdot \frac{1}{\sqrt{1-\rho^2}} \cdot \frac{1}{\phi(s)}.$$
It is obvious now that
$$
\pdiffII {J_\rho(x,y)}xx \cdot \pdiffII {J_\rho(x,y)}yy
- \rho^2 \left(\pdiffII {J_\rho(x,y)}xy\right)^2 = 0.
$$
Now, suppose that $|\sigma| \le |\rho|$. Then
$$
\det(M_{\rho \sigma}(x, y))
=
\pdiffII {J_\rho(x,y)}xx \cdot \pdiffII {J_\rho(x,y)}yy
- \sigma^2 \left(\pdiffII {J_\rho(x,y)}xy\right)^2 \ge 0.
$$
If $\rho \ge 0$ then the diagonal of $M_{\rho \sigma}(x,y)$ is non-positive, and it follows that
$M_{\rho \sigma}(x, y)$ is negative semidefinite. If $\rho \le 0$ then
the diagonal is non-negative and so $M_{\rho \sigma}(x,y)$ is positive semidefinite.
\end{proof}

\thirddiff*
\begin{proof}
As before, we set $\Phi^{-1}(x) =s $ and $\Phi^{-1}(y) =t$. From the proof of Claim~\ref{clm:negative-semidefinite}, we see that
$$
\frac{\partial^2 J_{\rho}(x,y)}{\partial x^2}  =\phi \left(\frac{\Phi^{-1}(y)-\rho \cdot \Phi^{-1}(x)}{\sqrt{1-\rho^2}} \right) \cdot \frac{-\rho}{\sqrt{1-\rho^2}} \cdot \frac{1}{\phi(s)}  .
$$
To compute the  third derivatives of $J$, recalling that
  $\frac{\partial s}{\partial x} = \frac{1}{\phi(s)}$ and $\frac{\partial t}{\partial y}= \frac{1}{\phi(t)}$, we have
 \begin{eqnarray}
   \frac{\partial^3 J_{\rho}(x,y)}{\partial x^3} &=& \frac{\rho}{(1-\rho^2)^{3/2}}
   \frac{\rho t + (2\rho^2 - 1) s}{\phi(s)}
   \exp\Big(-\frac{t^2 - 2\rho st + (2\rho^2 - 1) s^2}{2(1-\rho^2)}\Big) \notag \\
   &=& \frac{\sqrt{2\pi} \rho}{(1-\rho^2)^{3/2}}
   (\rho t + (2\rho^2 - 1) s)
   \exp\Big(-\frac{t^2 - 2\rho st + (3\rho^2 - 2) s^2}{2(1-\rho^2)}\Big).
   \label{eq:Jaaa}
 \end{eqnarray}
 Now, $\Phi^{-1}(x) \sim \sqrt{2 \log x}$ as $x \to 0$; hence there is a constant $C$
  such that $\Phi^{-1}(x) \le C \sqrt{\log x}$ for all $x \le \frac{1}{2}$.
  Hence, $\exp(s^2) \le x^{-C}$ for all $x \le \frac{1}{2}$; by symmetry,
  $\exp(s^2) \le (x(1-x))^{-C}$ for all $x \in (0, 1)$.
  Therefore
  \begin{align}
   \exp\Big(-\frac{t^2 - 2\rho st + (3\rho^2 - 2)s^2}{2(1-\rho^2)}\Big)
   &=
    e^{-\frac{t^2}{2(1-\rho^2)}} e^{\frac{\rho st}{1 - \rho^2}} e^{\frac{(2 - 3 \rho^2) s^2}{2(1-\rho^2}} \notag\\
   &\le
    e^{-\frac{t^2}{2(1-\rho^2)}} e^{\frac{\rho (s^2 + t^2)}{2(1 - \rho^2)}} e^{\frac{(2 - 3 \rho^2) s^2}{2(1-\rho^2}} \notag\\
   &\le
    \big(x(1-x)y(1-y)\big)^{-\frac{\rho}{2(1-\rho^2)}}
    \big(x(1-x)\big)^{-\frac{2 - 3 \rho^2}{2(1-\rho^2)}} .
\label{eq:Jaaa-bound}
  \end{align}
  Further, using $\exp(s^2) \le x^{-C}$ and $\exp(t^2) \le y^{-C}$, $\rho t + (2\rho^2 - 1) s \le 4 (xy)^{-C}$. As a consequence,  applying this to~\eqref{eq:Jaaa}, we see that there is a constant $C(\rho)>0$, 
  $$\left|\frac{\partial^3 J_{\rho}(x,y)}{\partial x^3}\right| \le C(\rho) \big(x(1-x)y(1-y)\big)^{-C(\rho)}.$$
  The other third derivatives are similar:
  \[
  \frac{\partial^3 J_{\rho}(x,y)}{\partial x^2 \partial y}
   = \frac{\sqrt{2\pi} \rho}{(1-\rho^2)^{3/2}}
   (t - 2 \rho s)
   \exp\Big(-\frac{(2\rho^2 - 1) t^2 - 2\rho st + (2\rho^2 - 1) s^2}{2(1-\rho^2)}\Big).
  \]
  By the same steps that led to~\eqref{eq:Jaaa-bound}, we get
  $$\left|\frac{\partial^3 J_{\rho}(x,y)}{\partial x^2 \partial y}\right| \le C(\rho) \big(x(1-x)y(1-y)\big)^{-C(\rho)}$$ (for a slightly different $C(\rho)$).
  The bounds on $\partial^3 J/\partial y^2 \partial x$ and $\partial^3 J/\partial x^3$ then follow because $J$ is symmetric in $x$ and $y$.
  \noindent

  The fact that $C(\rho)$ can be chosen so that it is continuous for $\rho \in (-1,1)$ is obvious from the discussion above.
 \end{proof}

 \begin{claim}\label{clm:diff-J-rho}
 For any $x, y \in (0, 1)$,
  \[
   \left|\pdiff{J_\rho(x, y)}{\rho}\right| \le (1-\rho^2)^{-3/2}.
  \]
 \end{claim}
 \begin{proof}
  We begin from~\eqref{eq:K}, but this time we differentiate with respect to $\rho$:
  \[
   \pdiff{K_\rho(s,t)}{\rho} = -\frac{1}{(1 - \rho^2)^{3/2}}
   \int_{s'=-\infty}^s \phi(s') \phi\left(\frac{t - \rho s'}{\sqrt{1-\rho^2}}\right) ds'.
  \]
 Since $\mathop{\mathrm{Range}}(\phi) \subset (0,1]$ and $\int_{s'} \phi(s') ds' = 1$, it follows that
  \[
   \left|\pdiff{K_\rho(s, t)}{\rho}\right| \le (1-\rho^2)^{-3/2}.
  \]
 Since $\pdiff{J_\rho(s,t)}{\rho} = \pdiff{K_\rho(\Phi^{-1}(x), \Phi^{-1}(y))}{\rho}$,
 the proof is complete.
 \end{proof}
 \noindent
We also state the following useful claim without a proof. The proof is obvious from the calculations in the proofs of Claim~\ref{clm:negative-semidefinite} and Claim~\ref{clm:third-derivative}.
\begin{claim}\label{clm:bounded derivatives}
For any $\rho \in (-1,1)$, $\eps>0$ there exists a continuous function $\gamma (\rho, \eps)$ such that for any $(x,y) \in [\eps,1-\eps]^2$ and $1 \le i  + j \le 3$,
$$
\left|\frac{\partial^{i+j} J_{\rho}(x,y)}{\partial^i x \partial^j y} \right| \le \gamma (\rho, \eps).
$$
\end{claim}

\chapter[Paradoxes]{Paradoxes, Noise Stability and Reverse Hyper-Contraction}
\section{Probability of Paradox}
Our next goal is to prove a quantitative version of Arrow Theorem following~\cite{Mossel:12}. 
We will only discuss the case of $3$ alternatives but will allow different function to determine different pairwise selections. 
Recall that we consider voters who vote independently and where voter $i$ votes uniformly at random from the $6$ possible rankings.
Recall that we encode the $6$ possible rankings by vectors $(x,y,z) \in \{-1,+1\}^3 \setminus \{ \pm (1,1,1) \}$.
Here $x$ is $+1/-1$ if a voter ranks $a$ above/below $b$, 
$y$ is $+1/-1$ if voter ranks $b$ above/below $c$,
$z$ is $+1/-1$ if voter ranks $c$ above/below $a$. 
We will assume further that $f,g,h : \{-1,1\}^n \to \{0,1\}$ are the aggregation functions for the $a$ vs. $b$, $b$ vs. $c$ and $c$ vs. $a$ preferences. 
We will again use the following observation used in~\cite{Kalai:02}: 
Since the binary predicate $\psi(a,b,c) = 1(a=b=c)$ for $a,b,c \in \{0,1\}$,  can be expressed as
\[
\psi(a,b,c) =  1+ ab  + ac + bc - a - b - c,
\]
we can write 
\begin{eqnarray*}
  \IP[f(x) = g(y) = h(z)] &=& 1 + \IE[f(x) g(y)] + \IE[g(y) h(z)] + \IE[h(z) f(x)]  -\IE[f] - \IE[g] - \IE[h]  \\ &=& 
  1+ \langle f, g \rangle_{-1/3} +  \langle g, h \rangle_{-1/3} +  \langle h, f \rangle_{-1/3} -\IE[f] - \IE[g] - \IE[h] 
\end{eqnarray*}
where the last equality follows from the fact that the uniform distribution over $\{\pm 1\}^3 \setminus \{ \pm (1,1,1) \}$, satisfies $E[x_i y_i] = -1/3$ and similarly for other pairs of coordinates.

To state a quantitive version, we will say that a function $f$ is $\eps$-close to a function $g$ if $\IP[f \neq g] \leq \eps$. 
The quantitative version we we wish to prove is the following:

\begin{theorem} \label{thm:quant_arrow}
For every $\eps  >  0$, there exists $\delta(\eps) > 0$ such that the following holds for every $n$: 
If
\[
\IP[f(x) = g(y) = h(z)] < \delta,
\]
then either two of the functions $f,g,h$ are $\eps$-close to constant functions of the opposite sign, or there exists a variable $i$ such that 
$f,g$ and $h$ are all $\eps$-close to the same dictator on voter $i$. 
\end{theorem} 

The main significance of Theorem~\ref{thm:quant_arrow} is that it is dimension independent. We get the same bound no matter what the dimension is. 
This shows that one cannot avoid the curse of paradoxes in voting by assuming the probability of a paradox vanishes as the number of voters grow. 
Our proof of the quantitative version will follow~\cite{Mossel:12}. It uses the Majority is Stablest Theorem along with reverse hyper-contractive inequality.

Before proving Theorem~\ref{thm:quant_arrow} we give a direct implication of the Majority is Stablest Theorem in the case where 
the functions $f=g=h$ are all balanced so $\IE[f] = \IE[g] = \IE[h] = 0$. Using Majority is Stablest Theorem we obtain: 
\begin{theorem}[\cite{Kalai:02,MoOdOl:10}]\label{thm:kalai}
For every $\eps > 0$, there exists a $\tau > 0$ such that 
if $f,g,h : \{-1,1\}^n \to \{0,1\}$ satisfy $\EE[f] = \EE[g] = \EE[h] = 1/2$ and have all influences bounded above by $\tau$ then:
\begin{equation} \label{eq:kalai}
\PP[f(x) = g(y) = h(z)] \geq  3 \langle \chi_{\half}, \chi_{\half} \rangle_{-\third} -  \half - \eps. 
\end{equation}
\end{theorem} 
Again, the right hand side of equation~(\ref{eq:kalai}) is the asymptotic probability 
that $\PP[f(x) = g(y) = h(z)]$ when $f = g = h = \chi_{\half}(n^{-\half} \sum_{i=1}^n x_i)$ are all given by the same Majority function. 
Theorem~\ref{thm:kalai} provides a surprising counter argument to Condorcet' arguments. Condorcet argued that pairwise ranking by Majority is problematic as it results in a paradox and Theorem~\ref{thm:kalai} shows that in fact Majority asymptotically minimizes the probability of a paradox among low influence functions. 
 
We also have the following strengthening of Theorem~\ref{thm:kalai}.
\begin{theorem}\label{thm:kalai2}
For every $\eps > 0$, there exist $m, \beta > 0$ such that 
if $f,g,h : \{-1,1\}^n \to \{0,1\}$ satisfy $\EE[f] = \EE[g] = \EE[h] = 1/2$ and $f,g$ and $h$ are all 
$(m,\beta)$-resilient then 
\[
\PP[f(x) = g(y) = h(z)] \geq  3 \langle \chi_{\half}, \chi_{\half} \rangle_{-\third} -  \half - \eps.
\]
\end{theorem}

\section{Reverse Hyper-Contraction}

Recall that the hyper-contractive theorem~\ref{thm:BB} states that for 
$f : \{-1,1\}^n \rightarrow \mathbb{R}$ and $1 \le q \le p$ and for any $\rho^2 \le \frac{q-1}{p-1}$, 
it holds that 
\[
\| T_{\rho} f \|_p \le \|  f \|_q.
\]

Borell~\cite{Borell:82} proved a reverse inequality: 
\begin{theorem}\cite{Borell:82}\label{thm:rev_hyp}
Let $f,g : \{-1,1\}^n \rightarrow \mathbb{R}_+$ and $1 > p > q$. 
Then, for any $0 \leq \rho^2 \le \frac{1-p}{1-q}$, 
\[
\| T_{\rho} f \|_q \ge \|  f \|_p.
\]
and for any $0 \leq \rho^2 \le (1-p)(1-q)$
\[
\langle f, g \rangle_{\rho} \geq \| f \|_p \| \| g \|_q
\]
\end{theorem}
While the inequalities may seem like a curiosity, as $p$ and $q$ norms for $p,q < 1$ are rarely used (nor are they norms), the second inequality is quite helpful in some social choice proofs. 
For a more general discussion of reverse-hyper contraction and its applications see~\cite{MORSS:06,MoOlSe:13}.

An immediate computational corollary of Theorem~\ref{thm:rev_hyp} is the following lemma  (as stated in~\cite{MORSS:06}): 

\begin{lemma} \label{lem:inv_hyp_org}
Let $x,y \in \{-1,1\}^n$ be distributed uniformly and $(x_i,y_i)$ are independent.
Assume that $\E[x(i)] = \E[y(i)] = 0$ for all $i$ and that
$\E[x(i) y(i)] = \rho \geq 0$. Let $B_1, B_2 \subset \{-1,1\}^n$ be two sets and assume that
\[
\P[B_1] \geq e^{-\alpha^2}, \quad \P[B_2] \geq e^{-\beta^2}.
\]
Then:
\[
\P[x \in B_1, y \in B_2] \geq \exp(-\frac{\alpha^2+\beta^2+2 \rho \alpha \beta}{1-\rho^2}).
\]
In particular, if $\P[B_1] \geq \eps$ and $\P[B_2] \geq \eps$, then
\[
\P[ x \in B_1, y \in B_2] \geq \eps^{\frac{2}{1-\rho}}.
\]
\end{lemma}

We will need to generalize the result above to negative $\rho$ and further to different $\rho$ values for different bits.

\begin{lemma} \label{lem:inv_hyp}
Let $x,y \in \{-1,1\}^n$ be distributed uniformly and $(x_i,y_i)$ are independent.
Assume that $\E[x(i)] = \E[y(i)] = 0$ for all $i$ and that
$|\E[x(i) y(i)]| \leq \rho$. Let $B_1, B_2 \subset \{-1,1\}^n$ be two sets and assume that
\[
\P[B_1] \geq e^{-\alpha^2}, \quad \P[B_2] \geq e^{-\beta^2}.
\]
Then:
\[
\P[x \in B_1, y \in B_2] \geq \exp(-\frac{\alpha^2+\beta^2+2 \rho \alpha \beta}{1-\rho^2}).
\]
In particular if $\P[B_1] \geq \eps$ and $\P[B_2] \geq \eps$, then:
\begin{equation} \label{eq:two_small}
\P[x \in B_1, y \in B_2] \geq \eps^{\frac{2}{1-\rho}}.
\end{equation}
\end{lemma}

\begin{proof}
Take $z$ so that $(x_i,z_i)$ are independent and $\E[z_i] = 0$ and $\E[x_i z_i] = \rho$.
It is easy to see that there exists $w_i$ independent of $x,z$ s.t.~the joint distribution of
$(x,y)$ is the same as $(x,z \cdot w)$, where $z \cdot w = (z_1 w_1,\ldots,z_n w_n)$.
Now for each fixed $w$ we have that
\[
\P[x \in B_1, z \cdot w \in B_2] =
\P[x \in B_1, z \in w \cdot B_2] \geq \exp(-\frac{\alpha^2+\beta^2+2 \rho \alpha \beta}{1-\rho^2}),
\]
where $w \cdot B_2 = \{w \cdot w' : w' \in B_2 \} $.
Therefore taking expectation over $w$ we obtain:
\[
\P[x \in B_1, y \in B_2] = \E [ \P[x \in B_1, z \cdot w \in B_2\mid w]] \geq \exp(-\frac{\alpha^2+\beta^2+2 \rho \alpha \beta}{1-\rho^2})
\]
as needed.
The conclusion~(\ref{eq:two_small}) follows by simple substitution (note the difference with Corollary 3.5 in~\cite{MORSS:06} for sets of equal size which is a typo).
\end{proof}

Applying the CLT and using~\cite{Borell:85} one obtains the same result for Gaussian random variables.
\begin{lemma} \label{lem:inv_hyp_gauss}
Let $N,M$ be $N(0,I_n)$ with $(N(i),M(i))_{i=1}^n$ independent.
Assume that
$|\E[N(i) M(i)]| \leq \rho$. Let $B_1, B_2 \subset \R^n$ be two sets and assume that
\[
\P[B_1] \geq e^{-\alpha^2}, \quad \P[B_2] \geq e^{-\beta^2},
\]
Then:
\[
\P[N \in B_1, M \in B_2] \geq \exp(-\frac{\alpha^2+\beta^2+2 \rho \alpha \beta}{1-\rho^2}).
\]
In particular if $\P[B_1] \geq \eps$ and $\P[B_2] \geq \eps$, then:
\begin{equation} \label{eq:two_small_gauss}
\P[N \in B_1, M \in B_2] \geq \eps^{\frac{2}{1-\rho}}.
\end{equation}
\end{lemma}

\begin{proof}
Fix the values of $\alpha$ and $\beta$ and assume without loss of generality that $\max_i |\E[N(i) M(i)]|$ is obtained for $i=1$.
Then by~\cite{Borell:85} (see also~\cite{Mossel:10}), the minimum of the quantity $\P[N \in B_1, M \in B_2]$ under
the constraints on the measures given by $\alpha$ and $\beta$ is obtained in one dimension, where $B_1$ and $B_2$ are intervals
$I_1, I_2$. Look at random variables $x(i), y(i)$, where $\E[x(i)] = \E[y(i)] = 0$ and
$\E[x(i) y(i)] =  \E[M_1 N_1]$. Let $X_n = n^{-1/2} \sum_{i=1} x(i)$ and $Y_n = n^{-1/2} \sum_{i=1} y^(i)$.
Then the CLT implies that
\[
\P[X_n \in I_1] \to \P[N_1 \in B_1], \quad
\P[Y_n \in I_2] \to \P[M_1 \in B_2],
\]
and
\[
\P[X_n \in I_1, Y_n \in I_2] \to \P[N_1 \in B_1, M_1 \in B_2].
\]
The proof now follows from the previous lemma.
\end{proof}

\section{The Gaussian Arrow Theorem} \label{sec:arrow_gauss}

The first step in the proof of the quantitative Arrow Theorem is to consider a Gaussian version of the problem. The Gaussian version corresponds to a situation
where the functions $f,g,h$ can only ``see" averages of large subsets of the voters.
We thus define a $3$ dimensional normal vector $N$.
The first coordinate of $N$ is supposed to represent the deviation of the number of voters where $a$ ranks above $b$ from the mean.
The second coordinate is for $b$ ranking above $c$ and the last coordinate for $c$ ranking above $a$.

Since averaging maintains the expected value and covariances, we define:
\begin{equation} \label{eq:gauss}
\E[N_1^2] = \E[N_2^2] = \E[N_3^2] = 1,\quad 
\E[N_1 N_2] = \E[N_2 N_3] = \E[N_3 N_1] =  -1/3. \nonumber
\end{equation}
We let $N(1),\ldots,N(n)$ be independent copies of $N$.
We write $\CalN = (N(1),\ldots,N(n))$ and for $1 \leq i \leq 3$ we write $\CalN_i = (N(1)_i,\ldots,N(n)_i)$.
The Gaussian version of Arrow theorem states:

\begin{theorem} \label{thm:arrow_gauss}
For every $\eps > 0$ there exists a $\delta = \delta(\eps) > 0$ such that the following hold.
Let $\phi_1,\phi_2,\phi_3 : \R^n \to \{-1,1\}$.
Assume that for all $1 \leq i \neq j \leq 3$ and all $u \in \{-1,1\}$ it holds that
\begin{equation} \label{eq:non_dict_normal}
\P[\phi_i(\CalN_i) = u] + \P[\phi_j(\CalN_j) = -u] \geq 2 \eps. 
\end{equation}
Then with the setup given in~(\ref{eq:gauss}) it holds that:
\[
\P[\phi_1(\CalN_1) = \phi_2(\CalN_2) = \phi_3(\CalN_3)] \geq \delta.
\]
Moreover, one may take $\delta = (\eps/2)^{18}$.
\end{theorem}
We note that if $\P[\phi_i(\CalN_i) = u] + \P[\phi_j(\CalN_i) = -u] \leq 2 \eps$, then one of the alternatives will be ranked the top/bottom with probability at least $1-2\eps$. Therefore the theorem states that unless this is the case, the probability of a paradox is at least $\delta$. Since the Gaussian setup excludes dictator functions in terms of the original vote, this result is to be expected in this case.

\begin{proof}
We will consider two cases: either all the functions $\phi_i$ satisfy $|\E \phi_i| \leq 1-\eps$, or there exists at least one function
with $|\E \phi_i| > 1 - \eps$.
Assume first that there exists a function $\phi_i$ with $|\E \phi_i| > 1-\eps$.
Without loss of generality assume that $\P[\phi_1 = 1] > 1-\eps/2$.
Note that by~(\ref{eq:non_dict_normal}) it follows that $\P[\phi_1 = -1] + \P[\phi_2 = 1] \geq 2 \eps$ and therefore 
$\P[\phi_2 = 1] > \eps$ and similarly $\P[\phi_3 = 1] > \eps$.
In particular, $\P[\phi_1 = 1,\phi_2 = 1] > \eps/2$
We now look at the function $\psi = 1(\phi_1 = 1, \phi_2 = 1)$. Let
\[
\CalM_1 = \frac{\sqrt{3}}{2} (\CalN_1 + \CalN_2), \quad
\CalM_2 = \frac{\sqrt{3}}{2 \sqrt{2}} (\CalN_1 - \CalN_2).
\]
Then it is easy to see that $\CalM_2(i)$ is uncorrelated with and therefore independent of $\CalN_3(i),\CalM_1(i)$ for all $i$.
Moreover, for all $i$ the covariance between $\CalM_1(i)$ and $\CalN_3(i)$ is $-1/\sqrt{3}$
and
$1-1/\sqrt{3} > 1/3$.
We may now apply Lemma~\ref{lem:inv_hyp_gauss} with the vectors
\[
(\CalN_3(1),\ldots,\CalN_3(n),Z_1,\ldots,Z_n), \quad
(\CalM_1(1),\ldots,\CalM_1(n),\CalM_2(1),\ldots,\CalM_2(n)),
\]
where $Z=(Z_1,\ldots,Z_n)$ is a normal Gaussian vector independent of anything else. We obtain:
\[
\P[\phi_1(\CalN_1) = 1, \phi_2(\CalN_2) = 1, \phi_3(\CalN_3) = 1] =
\P[\phi_3(\CalN_3,Z) = 1, \psi(\CalM_1,\CalM_2) = 1] \geq ((\eps/2))^{\frac{2}{1/3}} \geq (\eps/2)^{6}.
\]

We next consider the case where all functions satisfy $|\E \phi_i| \leq 1-\eps$.
Notice that at least two of the functions obtain the same value with probability at least $1/2$.
Let's assume that $\P[\phi_1 = 1] \geq 1/2$ and
$\P[\phi_2 = 1] \geq 1/2$. Then by Lemma~\ref{lem:inv_hyp_gauss} we obtain that
\[
\P[\phi_1 = 1, \phi_2 = 1] \geq 1/8.
\]
Again we define $\psi = 1(\phi_1 = 1, \phi_2 = 1)$. Since $\P[\phi_3 = 1] > \eps/2$, we may apply Lemma~\ref{lem:inv_hyp_gauss} and obtain that:
\[
\P[\phi_1 = 1, \phi_2 = 1, \phi_3 = 1] =
\P[\phi_1 = 1, \psi = 1] \geq (\eps/8)^{6}.
\]
This concludes the proof.
\end{proof}

\section{Arrow Theorem for Low Influence Functions} 
\begin{theorem} \label{thm:arrow_low_cross_unif}
For every $\eps > 0$ there exist a $\delta(\eps) > 0$ and a $\tau(\delta) > 0$
such that the following holds.
Let $f_1,f_2,f_3 : \{-1,1\}^n \to \{-1,1\}$.
Assume that for all $1 \leq i \neq j \leq 3$ and all $u \in \{-1,1\}$ it holds that

\begin{equation} \label{eq:non_top_bot_cross_unif}
\P[f_i = u] + \P[f_{j} = -u] \geq 4 \eps
\end{equation}
and for all $j$ it holds that
\begin{equation} \label{eq:cross_inf_unif}
|\{ 1 \leq i \leq 3 : I_j(f_i) > \tau \}| \leq 1.
\end{equation}
Then it holds that
\[
\P[f_1(x) = f_2(y) = f_3(z))] \geq \delta.
\]
Moreover, assuming the uniform distribution, one may take:
\[
\delta = \frac{1}{8}(\eps/2)^{20}, \quad \tau = \tau(\delta),
\]
where
\[
\tau(\delta) := \delta^{C \frac{\log(1/\delta)}{\delta}},
\]
for some absolute constant $C$.
\end{theorem}
\begin{proof}
Let $\phi_1,\phi_2,\phi_3 : \R \to \{-1,1\}$ be of the form $\phi_i = \sgn(x-t_i)$, where $t_i$ is chosen so that
$\E[\phi_i] = \E[f_i]$ (where the first expected value is according to the Gaussian measure). Let $N_1,N_2,N_3 \sim N(0,1)$ be jointly
Gaussian with $\E[N_i N_{i+1}] = -1/3$.
From Theorem~\ref{thm:arrow_gauss} it follows that:
\begin{equation} \label{eq:PF}
P([\phi_1(N_1) = \phi_2(N_2) = \phi_3(N_3)] > 8 \delta,
\end{equation}
and from Theorem~\ref{thm:MIST}, it follows that by choosing
$C$ in the definition of $\tau$ large enough, we have:
\[
\langle f_1, f_2 \rangle_{-1/3} \geq \langle \phi_1, \phi_2 \rangle_{-1/3} - \delta, \quad
\langle f_2, f_3 \rangle_{-1/3} \geq \langle \phi_2, \phi_3 \rangle_{-1/3} - \delta, \quad 
\langle f_3, f_1 \rangle_{-1/3} \geq \langle \phi_3, \phi_1 \rangle_{-1/3} - \delta.
\]
It now follows that:
\begin{eqnarray}
\IP[f_1(x) = f_2(y) = f_3(z)] &=& \frac{1}{4} \left(1 + \langle f_1, f_2 \rangle_{-1/3} + \langle f_2, f_3 \rangle_{-1/3} + \langle f_3, f_1 \rangle_{-1/3} \right) \\ 
                                           &\geq& 
                                           \frac{1}{4} \left(1 + \langle \phi_1, \phi_2 \rangle_{-1/3} + \langle \phi_2, \phi_3 \rangle_{-1/3} + \langle \phi_3, \phi_1  \rangle_{-1/3} \right) - 3 \delta/4 \nonumber \\ 
                                           &=&  \IP[\phi_1(N_1) = \phi(N_2) = \phi_3(N_3)] - 3 \delta/4 > 7 \delta, \nonumber
\end{eqnarray} 
as needed.
\end{proof}

\section{Arrow Theorem with at most one influential voter}
In addition to the case where all the influences are small, we need to consider the case where one voter is influential. 
This includes the case of the dictator function. The theorem below extends the case of low influences to the case where there is only 
one influential variable. 

\begin{theorem} \label{thm:arrow_one_inf}
For every $\eps > 0$ there exists a $\delta(\eps) > 0$ and a $\tau(\delta) > 0$ such that the following holds for all $n$. 
Let $f_1,f_2,f_3 : \{-1,1\}^n \to \{-1,1\}$.  
Assume that for all $1 \leq i \leq 3$ and $j > 1$ it holds that
\begin{equation} \label{eq:inf5}
I_j(f_i) <  0.5 \tau. 
\end{equation}
Then either
\begin{equation} \label{eq:paradox5}
\P[f_1(x) = f_2(y) = f_3(z)] \geq  \delta/6 , 
\end{equation}
or there exist $(f_1,f_2,f_3)$ which are $9 \eps$-close to either a dictator function $\pm (x_i,x_i,x_i)$ or a function $(g_1,g_2,g_3)$ where two of the 
$g_i$'s are constant and of different signs. 
Moreover, 
one may take:
\[
\delta = (\eps/2)^{20}, \quad \tau = \tau(\delta).
\]
\end{theorem}

\begin{proof}
Consider the functions $f_i^b$ for $1 \leq i \leq 3$ and $b \in \{-1,1\}$ defined by
\[
f_i^b(x_2,\ldots,x_n) = f_i(b,x_2,\ldots,x_n).
\]
Note that for all $b \in \{-1,1\}$, for all $1 \leq i \leq 3$ and for all $j > 1$ it holds that $I_j(f_i^{b_i}) < \tau$
and therefore we may apply
Theorem~\ref{thm:arrow_low_cross_unif}.
We obtain that for every
$b=(b_1,b_2,b_3) \notin \{(1,1,1),(-1,-1,-1)\}$ either:
\begin{equation} \label{eq:small_paradox}
\P[f_1^{b_1}(x) = f_2^{b_2}(y) = f_3^{b_3}(z)] \geq \delta,
\end{equation}
or there exist a $u(b) \in \{-1,1\}$ and an $i = i(b)$ such that
\begin{equation} \label{eq:small_constant}
\min(\P[f_i^{b_i} = u(b)], \P[f_{i+1}^{b_{i+1}} = -u(b)]) \geq 1-3\eps.
\end{equation}
Note that if there exists a vector $b = (b_1,b_2,b_3) \notin \{(1,1,1),(-1,-1,-1)\}$ for which~(\ref{eq:small_paradox}) holds
then~(\ref{eq:paradox5}) follows immediately.

It thus remains to consider the case where~(\ref{eq:small_constant}) holds for all $6$ vectors $b$.
In this case we will define new functions $g_i$ as follows. We let $g_i(b,x_2,\ldots,x_n) = u$ if
$\P[f_i^{b} = u] \geq 1-3\eps$ for $u \in \{-1,1\}$
and $g_i(b,x_2,\ldots,x_n) = f_i(b,x_2,\ldots,x_n)$ otherwise.
We let $G$ be the social choice function defined by $g_1,g_2$ and $g_3$.
From~(\ref{eq:small_constant}) it follows that for every $b = (b_1,b_2,b_3) \notin \{(1,1,1),(-1,-1,-1)\}$ there exist two functions
$g_i,g_{i+1}$ and a value $u$ s.t. $g_i(b_i,x_2,\ldots,x_n)$ is the constant function $u$ and $g_{i+1}(b_{i+1},x_2,\ldots,x_n)$
is the constant function $-u$. So
\[
P(g_1,g_2,g_3) = \P[(g_1,g_2,g_3) \in \{(1,1,1),(-1,-1,-1)\}] = 0,
\]
and therefore by Theorem~\ref{thm:arrow3} $(g_1,g_2,g_3)$ is of the required form. It is further easy to see that $\IP[f_i \neq g_i] \leq 3 \eps$ for all $i$, as needed. 
The proof follows.
\end{proof}

\section{Arrow Theorem with two influential voters} 
Our second application of reverse hyper-contraction is the following: 
\begin{lemma} \label{lem:two_inf}
Suppose that 
$I_1(f) > \eps$ and $I_2(g) > \eps$.
Let
\[
B = \{((x,y,z))_{i=3}^n : 1 \mbox{ is pivotal for } f \mbox{ and } 2 \mbox{ is pivotal for } g ) \}.
\]
Then
\[
\P[B] \geq \eps^3.
\]
\end{lemma}

\begin{proof}
Let
\[
B_1 = \{((x,y,z))_{i=3}^n : 1 \mbox{ is pivotal for } f(\cdot,\cdot,x_3,\ldots,x_n) \}, 
\]
\[
 B_2 = \{((x,y,z))_{i=3}^n : 2 \mbox{ is pivotal for } g(\cdot,\cdot,y_3,\ldots,y_n)\}.
\]
Then $\P[B_1] \geq I_1(f) > \eps$ and $\P[B_2] \geq I_2(g) > \eps$, and our goal is to obtain a bound on $\P[B_1 \cap B_2]$.
Note that the event $B_1$ is determined by $x$ and the event $B_2$ is determined by $y$.
So the proof follows immediately from Lemma~\ref{lem:inv_hyp} with $\rho = -1/3$. 
\end{proof}

We can now prove the main result of the section.
\begin{theorem} \label{thm:two_inf}
Suppose that there exist voters $i$ and $j$ such that
\[
I_i(f) > \eps, \quad I_j(g) > \eps.
\]
then $\IP[f(x) = g(y) = h(z)] > \frac{1}{36} \eps^3$.

\end{theorem}

\begin{proof}
Without loss of generality assume that $i=1$ and $j=2$. 
Let $B = B_1 \cap B_2$, where $B_1$ and $B_2$  are the events from Lemma~\ref{lem:two_inf}. By the lemma we have $\P[B] \geq \eps^3$.
Note that conditioned on any $((x,y,z)_{i=3}^n) \in B$, the functions $f,g$ and $h$ on coordinates $1$ and $2$ satisfy the condition 
of Arrow Theorem~\ref{thm:arrow3}. Thus with probability at least $1/36$, the outcome is not transitive. 
Therefore:
\[
\P[f(x) = g(y) = h(z)] \geq 
\frac{1}{36} \P[B] \geq
\frac{1}{36} \eps^3.
\]
\end{proof}

\section{Proof of the quantitative Arrow Theorem}

We are finally ready to prove Theorem~\ref{thm:quant_arrow} which we restate below: 
\begin{theorem} \label{thm:arrow3quant}
For every $\eps >  0$, there exists $\delta(\eps) > 0$ such that the following holds for every $n$: 
If
\[
\IP[f(x) = g(y) = h(z)] < \delta,
\]
then either two of the functions $f,g,h$ are $\eps$-close to constant functions of the opposite sign, or there exists a variable $i$ such that 
$f,g$ and $h$ are all $\eps$-close to the same dictator on voter $i$. 
Moreover, one can take
\begin{equation} \label{eq:del33}
\delta = \exp \left(-\frac{C}{\eps^{21}} \right).
\end{equation}
\end{theorem}

\begin{remark}
We note that the dependency of $\delta$ on $\eps$ is very far from optimal. 
The optimal dependency of $\delta = O(\eps^3)$ was obtained by Keller~\cite{Keller:12}
who considered the cases where $f,g,h$ are close to functions that have $0$ probability of paradox and applied contractive and hyper-contractive estimates.
\end{remark}

\begin{proof}
 Let $\eta$ be a small constant to be determined later.  We will consider three cases:
\begin{itemize}
\item
There exist two voters $i \neq j \in [n]$ and two different functions say $f$ and $g$ such that
\begin{equation} \label{eq:case1}
I_i(f) > \eta, \quad I_j(g) > \eta.
\end{equation}
\item
For pair of functions $k_1 \neq k_2 \in \{f,g,h\}$ and every $i \in [n]$, it holds that
\begin{equation} \label{eq:case2}
\min(I_i(k_1),I_i(k_2)) < \eta.
\end{equation}
\item
There exists a voter $j'$ such that for all $j \neq j'$
\begin{equation} \label{eq:case3}
\max(I_j(f),I_j(g),I_j(h)) < \eta.
\end{equation}
\end{itemize}

First note that each $(f,g,h)$ satisfies at least one of the three conditions (\ref{eq:case1}), (\ref{eq:case2})
or (\ref{eq:case3}). Thus it suffices to prove the theorem for each of the three cases.

In~(\ref{eq:case1}), we have by Theorem~\ref{thm:two_inf} have that
\[
\IP[f(x) = g(y) = h(z)] > \frac{1}{36} \eta^3.
\]
We thus obtain that $\IP[f(x) = g(y) = h(z)]  > \delta$ where $\delta$ is given in~(\ref{eq:del33})
by taking larger values $C'$ for $C$.

In case~(\ref{eq:case2}), by Theorem~\ref{thm:arrow_low_cross_unif} it follows that either there exists a function $(g_1,g_2,g_3)$ which
always puts a candidate at top / bottom and $(f_1,f_2,f_3)$ is $\eps$-close to $(g_1,g_2,g_3)$ (if~(\ref{eq:non_top_bot_cross_unif}) holds), or
$\IP[f(x) = g(y) = h(z)] > C \eps^{20} >> \delta$.

Similarly in the remaining case~(\ref{eq:case3}), we have by Theorem~\ref{thm:arrow_one_inf} that either $(f_1,f_2,f_3)$ is $\eps$-close 
 to $(g_1,g_2,g_3)$ with $\IP[g_1(x) = g_2(y) = g_3(z)] = 0$ or 
$\IP[f(x) = g(y) = h(z)] > C \eps^{20} >> \delta$.
The proof follows.

\end{proof}


\section{More general statements} \label{sec:k}
In this section we discuss a more general statement of Arrow Theorem closer to his original formulation and its quantitative counterpart.
This requires to introduce a number of additional definition. The reduction from the more general statements of Arrow Theorem to the 3 candidate case discussed above will be carried out in this section. 

\subsection{General Setup}
Consider $A = \{a,b,\ldots,\}$, a set of $k \geq 3$ alternatives. A {\em transitive preference} over $A$ is a ranking of the alternatives from top to bottom where ties are not allowed. Such a ranking corresponds to a {\em permutation} $\sigma$ of the elements $1,\ldots,k$ where $\sigma_i$ is the rank of alternative $i$. The set of all rankings will be denoted by $S_k$.

A {\em constitution} is a function $F$ that associates to every $n$-tuple $\sigma = (\sigma(1),\ldots,\sigma(n))$ of transitive preferences (also called a {\em profile}), and every pair of alternatives $a,b$ a preference between $a$ and $b$. Some key properties of constitutions include:
\begin{itemize}
\item
{\em Transitivity}. The constitution $F$ is {\em transitive} if $F(\sigma)$ is transitive for all $\sigma$.
In other words, for all $\sigma$ and for all three alternatives $a,b$ and $c$, if $F(\sigma)$ prefers $a$ to $b$, and prefers $b$ to $c$, it also prefers $a$ to $c$. Thus $F$ is transitive if and only if its image is a subset of the permutations on $k$ elements.
\item
{\em Independence of Irrelevant Alternatives (IIA).} The constitution $F$ satisfies the IIA property if for every pair of alternatives $a$ and $b$, the social ranking of $a$ vs. $b$ (higher or lower) depends only on their relative rankings by all voters. The IIA condition implies that the pairwise preference between any pair of outcomes depends only on the individual pairwise preferences. Thus, if $F$ satisfies the IIA property then there exists functions $f^{a>b}$ for every pair of candidates $a$ and $b$ such that
\[
F(\sigma) = ((f^{a>b}(x^{a>b}) : \{a,b\} \in {k \choose 2})
\]
\item
{\em Unanimity.} The constitution $F$ satisfies {\em Unanimity} if the social outcome ranks $a$ above $b$ whenever all individuals rank $a$ above $b$.
\item
The constitution $F$ is a {\em dictator} on voter $i$, if $F(\sigma) = \sigma(i)$, for all $\sigma$, or $F(\sigma) = -\sigma$, for all $\sigma$, where $-\sigma(i)$ is the ranking $\sigma_k(i) > \sigma_{k-1}(i) \ldots \sigma_2(i) > \sigma_1(i)$ by reversing the ranking $\sigma(i)$.
\end{itemize}

Arrow's theorem states~\cite{Arrow:50,Arrow:63} that:
\begin{theorem} \label{thm:arrow}
Any constitution on three or more alternatives which satisfies Transitivity, IIA and Unanimity is a dictatorship.
\end{theorem}
It is possible to give a characterization of all constitutions satisfying IIA and Transitivity. Results of Wilson~\cite{Wilson:72} provide a partial characterization for the case where voters are allowed to rank some alternatives as equal. In order to obtain a quantitative version of Arrow theorem, we give an explicit and complete characterization of all constitutions satisfying IIA and Transitivity in the case where all voters vote using a strict preference order.
Write $\F_k(n)$ for the set of all constitutions on $k$ alternatives and $n$ voters satisfying IIA and Transitivity.
For the characterization it is useful write $A >_F B$ for the statement that for all $\sigma$ it holds that $F(\sigma)$ ranks all alternatives in $A$ above all alternatives in $B$. We will further write $F_{A}$ for the constitution $F$ restricted to the alternatives in $A$. The IIA condition implies that $F_A$ depends only on the individual rankings of the alternatives in the set $A$. The characterization of $\F_k(n)$ we prove is the following.

\begin{theorem} \label{thm:general_arrow}
The class $\F_k(n)$ consist exactly of all constitutions $F$ satisfying the following:
There exist a partition of the set of alternatives into disjoint sets $A_1,\ldots,A_r$ such that:
\begin{itemize}
\item
\[
A_1 >_F A_2 >_F \ldots >_F A_r,
\]
\item
For all $A_s$ s.t. $|A_s| \geq 3$, there exists a voter $j$ such that $F_{A_s}$ is a dictator on voter $j$.
\item
For all $A_s$ such that $|A_s| = 2$, the constitution $F_{A_s}$ is an arbitrary non-constant function of the preferences on the alternatives in $A_s$.
\end{itemize}
\end{theorem}
We note that for all $k \geq 3$ all elements of $\F_k(n)$ are not desirable as constitutions. Indeed elements of $F_k(n)$ either have dictators whose vote is followed with respect to some of the alternatives, or they always rank some alternatives on top some other.
For a related discussion see~\cite{Wilson:72}.
The statement above follows easily from Theorem 3 in~\cite{Wilson:72}.
The exact formulation is taken from~\cite{Mossel:09a}.

The main goal of the current section is to provide a quantitative version of Theorem~\ref{thm:general_arrow} assuming voters vote
independently and uniformly at random. Note that Theorem~\ref{thm:general_arrow}
above implies that if $F \not\in \F_k(n)$ then $P(F) \geq (k!)^{-n}$.
However if $n$ is large and the probability of a non-transitive outcome is indeed as small as $(k!)^{-n}$, one may argue that a non-transitive outcome is so unlikely that in practice Arrow's theorem does not hold.

The main goal of this section is to prove the following statement: 
\begin{theorem} \label{thm:arrow_k}
For every number of alternatives $k \geq 1$ and $0.01 > \eps > 0$, there exists a $\delta = \delta(\eps)$, such that for every $n \geq 1$, if $F$ is a
constitution on $n$ voters and $k$ alternatives satisfying:
\begin{itemize}
\item
IIA and
\item
$P(F) < \delta$,
\end{itemize}
then there exists $G \in \F_k(n)$ satisfying $D(F,G) < k^2 \eps$.
Moreover, one may take:
\begin{equation} \label{eq:del_main}
\delta = \exp \left(-\frac{C}{\eps^{21}} \right),
\end{equation}
for some absolute constant $0 < C < \infty$.
\end{theorem}

We therefore obtain the following:
\begin{corollary} \label{cor:main}
For any number of alternatives $k \geq 3$ and $\eps > 0$,
there exists a $\delta = \delta(\eps)$, such that for every $n$, if $F$ is a
constitution on $n$ voters and $k$ alternatives satisfying:
\begin{itemize}
\item
IIA and
\item
$F$ is $k^2 \eps$ far from any dictator, so $D(F,G) > k^2 \eps$ for any dictator $G$,
\item
For every pair of alternatives $a$ and $b$, the probability that $F$ ranks $a$ above $b$ is at least $k^2 \eps$,
\end{itemize}
then the probability of a non-transitive outcome, $P(F)$, is at least $\delta$, where $\delta(\eps)$
may be taken as in~(\ref{eq:del_main}).
\end{corollary}

\begin{proof}
Assume by contradiction that $P(F) < \delta$. Then  by Theorem~\ref{thm:arrow_k} there exists a function $G \in \F_{n,k}$
satisfying $D(F,G) < k^2 \eps$.
Note that for every pair of alternatives $a$ and $b$ it holds that:
\[
\P[G \mbox{ ranks } a \mbox{ above } b] \geq \P[F \mbox{ ranks } a \mbox{ above } b] - D(F,G) > 0.
\]
Therefore for every pair of alternatives there is a positive probability that $G$ ranks $a$ above $b$.
Thus by Theorem~\ref{thm:arrow_k} it follows that $G$ is a dictator which is a contradiction.
\end{proof}

\begin{remark}
Note that if $G \in \F_k(n)$ and
$F$ is any constitution satisfying $D(F,G) < k^2 \eps$ then $P(F) < k^2 \eps$.
\end{remark}

\begin{remark} \label{rem:dep}
The bounds stated in Theorem~\ref{thm:arrow_k} and Corollary~\ref{cor:main} in terms of $k$ and $\eps$ is clearly not an optimal one.
We expect that the true dependency has $\delta$ which is some fixed power of $\eps$. Moreover we expect that the bound $D(F,G) < k^2 \eps$ should be improved
to $D(F,G) < \eps$.
\end{remark}

\subsection{Nisan's argument}
N. Nisan argued in his blog~\cite{Nisan:FromArrowtoFourier} that the natural way to study quantitative versions of Arrow's theorem is to look at functions from $S_k^n$ to $S_k$ and check to what extent do they satisfy the IIA property. He defines a function to be $\eta$-IIA if for every two alternatives $a$ and $b$ it holds that
$\P[F(\sigma) \neq \F(\tau)] \leq \eta$, where $\sigma$ is uniformly chosen and $\tau$ is uniformly chosen conditioned
on the $a,b$ ranking at $\tau$ being identical to that of $\sigma$ for all voters. In his blog Nisan sketches how a  quantitative Arrow's theorem proven for the definition used here implies a quantitative Arrow's theorem for his definition. We briefly repeat the argument with some minor modifications and corrections.

Fixing alternatives $a,b$ and writing $p_{a,b} : \{0,1\}^n \to \{0,1\}$ for the probability that given a vector of $n$ binary preferences between $a$ and $b$, $F$ ranks $a$ above $b$. If $F$ satisfies the IIA property then $p_{a,b} \in \{0,1\}$ a.s. If $F$ is $\eta$-IIA then $\E[2 p_{a,b}(1-p_{a,b})] \leq \eta$, and therefore $\E[\min(p_{a,b},1-p_{a,b})] \leq \eta$.

Assume a quantitative Arrow theorem such as the one proven here with parameters $\eps,\delta$ and
suppose by contradiction that $F : S_k^n \to S_k$ is $\eta$-IIA and $\eps$ far from $\F_k(n)$ for some small $\eta$ to be determined later.
Define a function $G$ as follows. Let $G(\sigma)$ rank $a$ above $b$ if for the majority of $\tau$ which agree with $\sigma$ in the $a,b$ orderings it holds that $F(\tau)$ ranks $a$ above $b$. We note that
for every pair of alternatives $a,b$ it holds that
\[
\P[F(\sigma), G(\sigma) \mbox{ have different order on } a,b] = \E[\min(p_{a,b},1-p_{a,b})] \leq \eta.
\]
By taking a union bound on all pairs of alternatives, this implies that $D(F,G) \leq {k \choose 2} \eta \leq k^2 \eta/2$.
Note further that $G$ satisfies the IIA property by definition.
Since $F$ is transitive and from the quantitative Arrow theorem proven here we conclude that
\[
D(F,G) \geq \P[P(G)] \geq \delta.
\]
and a contradiction is implied unless $k^2 \eta / 2 \geq \delta$.  Thus the Arrow theorem for the $\eta$-IIA definition holds with $\eta(\eps) = 2 \delta/k^2$.(We briefly note that moving from $F$ to $G$ does not preserve the property of the function being balanced so in the setting of Kalai's theorem an additional argument is needed)

\subsection{Proof of Theorem~~\ref{thm:arrow_k}}

\begin{proof}
The proof follows by applying Theorem~\ref{thm:arrow3quant} to triplets of alternatives.
Assume $P(F) < \delta(\eps)$.

Note that if $g_1,g_2 : \{-1,1\}^n \to \{-1,1\}$ are two different function each of which is either
a dictator or a constant function than $D(g_1,g_2) \geq 1/2$. Therefore for all $a,b$ it holds that
$D(f^{a>b},g) < \eps/10$ for at most one function $g$ which is either a dictator or a constant function.
In case there exists such function we let $g^{a>b} = g$, otherwise, we let $g^{a>b} = f^{a>b}$.

Let $G$ be the social choice function defined by the functions $g^{a>b}$. Clearly:
\[
D(F,G) < {k \choose 2} \eps < k^2 \eps.
\]
The proof would follow if we could show $P(G) = 0$ and therefore $G \in \F_k(n)$.

To prove that $G \in F_k(n)$ is suffices to show that for every set $A$ of three alternatives, it holds that
$G_A \in \F_3(n)$. Since $P(F) < \delta$ implies $P(F_A) < \delta$, Theorem~\ref{thm:arrow3quant} implies that there exists a function $H_A \in F_3(n)$ s.t. $D(H_A,F_A) < \eps$.
There are two cases to consider:
\begin{itemize}
\item
$H_A$ is a dictator. This implies that $f^{a>b}$ is $\eps$ close to a dictator for each $a,b$ and therefore $f^{a>b} = g^{a>b}$ for all pairs $a,b$, so $G_A = H_A \in \F_3(n)$.
\item
There exists an alternative (say $a$) that $H_A$ always ranks at the top/bottom.
In this case we have that $f^{a>b}$ and $f^{c>a}$ are at most $\eps$ far from the constant functions $1$ and $-1$ (or $-1$ and $1$).
The functions $g^{a>b}$ and $g^{c>a}$ have to take the same constant values and therefore again we have that $G_A \in \F_3(n)$.
\end{itemize}

The proof follows.

\end{proof}

\begin{remark} \label{rem:3tok}
Note that this proof is generic in the sense that it takes the quantitative Arrow's result for $3$ alternatives as a black box and
produces a quantitative Arrow result for any $k \geq 3$ alternatives.
\end{remark}

\subsection{Other probability measures}
Given the right analytic tools, it is not hard to generalize the proof of 
Theorem~\ref{thm:arrow3quant} and Theorem~\ref{thm:arrow_k} to other product distribution. This is done in~\cite{MoOlSe:13} where some of the tools related to reverse-hyper-contraction are developed. \cite{MoOlSe:13} obtains the following extension. 

\begin{theorem}[Quantitative Arrow's theorem for general distribution] \label{thm:arrow_general_product}
Let $\varrho$ be general distribution on $S_k$ with $\varrho$ assigning positive probability to each element of $S_k$. Let $\PP$ denote the distribution $\varrho^{\otimes n}$ on $S_k^n$. Then
for any number of alternatives $k \geq 3$ and $\eps > 0$,
there exists $\delta = \delta(\eps, \rho)>0$, such that for every $n$, if $F : S_k^n \to \{-1,1\}^{k \choose 2}$
satisfies
\begin{itemize}
\item
IIA and
\item
$\PP \{F(\sigma) \mbox{ is transitive}\} \geq 1-\delta$.
\end{itemize}
Then there exists a function $G$ which is transitive and satisfies the IIA property and
$\PP\{F(\sigma) \neq G(\sigma)\} \leq \eps$
\end{theorem}

\section{Other Variants}
In the concluding section of this chapter, we will present Kalai's original approach for a quantitative Arrow Theorem and discuss the optimal low influence function for voting in the case of $k>3$ alternatives. 

\subsection{Kalai's proof for the balanced case}
The special case, where all the functions are balanced was the first where a quantitative Arrow Theorem was proved by 
Kalai~\cite{Kalai:02}. In this short section we provide the statement and the proof of this special case.

\begin{theorem} \label{thm:kalai3}
There exists a constant $C$ such that 
if $f,g,h : \{-1,1\}^n \to \{-1,1\}$ satisfy $\EE[f] = \EE[g] = \EE[h] = 0$ and 
\[
\PP[f(x) = g(y) = h(z)] \leq \eps,
\]
then there exists a dictator function $d$, such that 
\[
\IP[f(x) \neq d(x)] \leq C \eps, \quad \IP[g(y) \neq d(y)] \leq C \eps, \quad \IP[h(z) \neq d(z)] \leq C \eps.
\] 
\end{theorem}

The proof of the Theorem will use the FKN Theorem~\cite{FrKaNa:02} which will be proved at the end of the section. 
 \begin{theorem} \label{theorem:fkn}
 There exists a constant $C$ such that if 
$f : \{-1,1\}^n \to \{-1,1\}$ satisfies $\IE[f] = 0$ and 
 \begin{equation} \label{eq:fkn}
 \sum_{S : |S| = 1} \hat{f}^2(S) = 1-\eps,
 \end{equation} 
 Then there exists a dictator a $d$  such that $\IP[f(x) \neq d(x)] \leq C \eps$. 
 \end{theorem} 

We now prove Theorem~\ref{thm:kalai3}.
\begin{proof} 
Recalling that 
\[
\IP[f(x) = g(y) = h(z)] = 
\frac{1}{4} \left(1+ \langle f,g \rangle_{-1/3} + \langle g,h \rangle_{-1/3} + \langle h,f \rangle_{-1/3} \right)
\]
and Theorem~\ref{thm:dict_stab}, it is clear that to prove Theorem~\ref{thm:kalai3} it suffices to prove that:
there exists a constant $C$ such that 
if $f,g : \{-1,1\}^n \to \{-1,1\}$ satisfy $\EE[f] = \EE[g] = 0$ and 
\begin{equation} \label{eq:arrow_eps9}
\langle f,g \rangle_{-1/3} \leq -\frac{1}{3}+\eps
\end{equation}
then there exists a dictator $d$, such that 
\begin{equation} \label{eq:almost_dictator}
\IP[f(x) \neq d(x)] \leq C \eps, \quad \IP[g(y) \neq d(y)] \leq C \eps.
\end{equation}
Note that 
\[
\langle f,g \rangle_{-1/3} \geq (-1/3)   \sum_{S : |S| = 1} \hat{f}(S) \hat{g}(S) +
 (-1/9) \sum_{S : |S| > 1} |\hat{f}(S) \hat{g}(S)| \geq -\frac{1}{3} \gamma  -\frac{1}{9}(1-\gamma), 
 \]
 where $\gamma = \sum_{S : |S| = 1}| \hat{f}(S) \hat{g}(S)|$.  Thus if~(\ref{eq:arrow_eps9}) holds then 
 \begin{equation} \label{eq:level1}
 \sum_{S : |S| = 1} |\hat{f}(S) \hat{g}(S)| \geq 1-\eps. 
 \end{equation}
 It therefore suffices to show that if~(\ref{eq:level1}) holds then~(\ref{eq:almost_dictator}) holds. 
 By CS,
 \[
 \sum_{S : |S| = 1} \hat{f}^2(S) \geq (1-\eps)^2 \geq 1-2\eps,
 \]
 Therefore by Theorem FKN, $f$ is $C_{FKN} \eps$ close to a dictator $d_1$. 
Similarly $g$ is is $C_{FKN} \eps$ close to a dictator $d_2$. Note that the statement of the theorem is trivial if $C \geq 8$ and $\eps \leq 1/8$, so assume $\eps \leq 1/8$. 
It remains to show that $d_1 = d_2$. Note that if $d_1 \neq d_2$ then 
 \[
 \IP[f(x) \neq g(x)] \geq \IP[d_1(x) \neq d_2(x)] - \IP[d_1(x) \neq f(x)] - \IP[d_2(x) \neq g(x)] \geq \frac{1}{2} - 2 \frac{1}{8} = 1/4.
 \] 
 However, by~(\ref{eq:level1}) and using the fact that $\sum_{S} |\hat{f}(S) \hat{g}(S)| \leq 1$
 \[
 \IE[f(x)g(x)] = \sum_{S} \hat{f}(S) \hat{g}(S) \geq 1-2\eps,
 \]
 and therefore $\IP[f(x) \neq g(x)] \leq \eps$. The proof follows.
 \end{proof}  
 
We will now prove the FKN Theorem. The proof is similar to the proof in~\cite{ODonnell:14}. 
For an alternative proof see~\cite{JeOlWo:15}.
We will use the following corollary of hyper-contraction.
\begin{lemma} \label{lem:deg2_hyp}
Let $q(x) = \sum_{i < j} q_{i,j} x_i x_j$. Then 
\[
\IE[q^4] \leq 81 \IE[q^2]^2.
\] 
\end{lemma}

\begin{proof} 
 By hyper-contraction if $\eta^2 \leq 1/3$ then, 
\[
\| T_{\eta} q \|_4 \leq \| q \|_2 \implies \IE[(T_{\eta} q)^4] \leq \IE[q^2]^2.
\]
On the other hand,
\[
\IE[(T_{\eta} q)^4] = \IE[(\sum_{i < j} \eta^2 q_{i,j} x_i x_j)^4] = \eta^8 \IE[q^4],
\]
So by choosing $\eta = 1/\sqrt{3}$, we get 
\[
\IE[q^4] \leq \eta^{-8} \IE[q^2]^2 \qed
\]
\end{proof} 

The proof will also use the  Paley-Zygmund inequality stating that for a positive random variable $Z$ and for $0 \leq \theta \leq 1$ it holds that:
\[
\IP[Z \geq \theta \IE[Z]] \geq (1-\theta)^2 \frac{\IE[Z]^2}{\IE[Z^2]}
\]
\begin{exercise} 
Prove this using CS.
\end{exercise}
Applying this inequality for $Z = q^2$ and using Lemma~\ref{lem:deg2_hyp} implies that 
\begin{corollary}
For $0 \leq \theta \leq 1$,
\begin{equation} \label{eq:pzq}
\IP[q^2 \geq \theta \IE[q^2]] \geq \frac{(1-\theta)^2}{81}.
\end{equation} 
\end{corollary}

 \begin{proof}
Let 
\[
\ell(x) = \sum_{|S| \leq 1} \hat{f}(S)x_S, \quad h(x) = \sum_{|S| > 1} \hat{f}(S) x_S, \quad  
q(x) = 2 \sum_{i < j} \hat{f}(\{i\}) \hat{f}(\{j\}) x_i x_j,
\]
 so that 
\[
\ell^2(x) = \sum_{i} \hat{f}^2(i) + q(x) = 1-\eps+q(x).
\]
Note that $f^2 = 1$ implies that  $\ell^2 + h(2f-h) = 1$. 
Moreover using the fact that $|f| = 1$, for all $c \geq 1$ and sufficiently small $\eps$,
\[
\IP[h(2h-f) \geq 2 c \sqrt{\eps}] \leq \IP[|h|(2|h|+1) \geq 2 c \sqrt{\eps}] \leq \IP[|h(x)| \geq c \sqrt{\eps}] \leq \frac{\IE[h^2]}{c^2 \eps} \leq \frac{1}{c^2}.
\]
Therefore,  
\[
\IP[|q(x)| \geq (2 c + 1) \sqrt{\eps}] = \IP[|\ell^2-1+\eps| \geq (2 c+1) \sqrt{\eps}] \leq \IP[|h(2f-h)| + \eps  \geq (2 c+1) \sqrt{\eps}] \leq \frac{1}{c^2}.
\]
In particular for $c=10$, we obtain:
\[
\IP[q^2(x) \geq 500 \eps] \leq \frac{1}{100}
\]
On the other hand, applying (\ref{eq:pzq}) with $\theta = 0.05$, implies that 
\[
\IP[q^2 \geq  \IE[q^2]/20] > \frac{1}{100},
\]
and therefore $\IE[q^2] \leq C \eps$, with $C = 10000$.
Now: 
\[
C \eps \geq \IE[q^2] = 4 \sum_{i < j} \hat{f}^2(i) \hat{f}^2(j) = 
2 \left( \sum_i \hat{f}^2(i) \right)^2 - 2 \sum_i \hat{f}^4(i) =
2(1-\eps)^2 - 2 \sum_i \hat{f}^4(i)
\]
So we obtain that 
\[
\max_i \hat{f}^2(i) \geq \sum_i \hat{f}^4(i) \geq (1-\eps)^2 - C \eps = 1 - O(\eps), 
\]
as needed.

 \end{proof}

\subsection{Low Influence Optimality for $k \geq 3$ alternatives} 

When we are considering $k \geq 3$ alternatives, we want to define more formally the possible outcome in Arrow voting.
Since for every two alternatives, a winner is decided,  the aggregation results in a {\em tournament} $G_k$ on the set $[k]$.
Recall that $G_k$ is a {\em tournament} on $[k]$ if it is a
directed graph on the vertex set $[k]$ such that for all $a,b \in [k]$
either $(a > b) \in G_k$ or $(b > a) \in G_k$.
Given individual
rankings $(\sigma_i)_{i=1}^n$ the tournament $G_k$ is defined as follows.

Let $x^{a > b}(i) = 1$, if $\sigma_i(a) > \sigma_i(b)$, and $x^{a > b}(i) = -1$ if
$\sigma_i(a) < \sigma_i(b)$. Note that $x^{b > a } = -x^{a > b}$.

The binary decision between each pair of candidates is performed via
a anti-symmetric function $f : \bits^n \to \bits$ so that
$f(-x) = -f(x)$ for all $x \in \bits$.
The tournament $G_k = G_k(\sigma; f)$ is then defined
by letting $(a > b) \in G_k$ if and only if $f(x^{a > b}) = 1$.

Note that there are $2^{k \choose 2}$tournaments while there are only
$k! = 2^{\Theta(k \log k)}$ linear rankings. For the purposes of
social choice, some tournaments make more sense than others.

\begin{definition}
We say that a tournament $G_k$ is {\em linear} if it is acyclic.
We will write $\Acyc(G_k)$ for the logical statement that $G_k$ is acyclic.
Non-linear tournaments are often referred to as non-rational in economics
as they represent an order where there are $3$ candidates $a,b$ and $c$ such
that $a$ is preferred to $b$, $b$ is preferred to $c$ and
$c$ is preferred to $a$.\\

We say that the tournament $G_k$ is a {\em unique max tournament} if there is a
candidate $a \in [k]$ such that for all $b \neq a$ it holds that
$(a > b) \in G_k$.
We write $\UniqueBest(G_k)$ for the logical statement that $G_k$ has a unique max.
Note that the unique max property is weaker than linearity.
It corresponds to the fact that there is a candidate that dominates all
other candidates.
\end{definition}

A generalization of Borell's result along with a general invariance principle~\cite{Mossel:08,Mossel:10} allows to prove the following~\cite{IsakssonMossel:12}. 
\begin{theorem}
  \label{thm:condorcet}
  For any $k \ge 1$ and $\epsilon>0$
  there exists a $\tau(\epsilon, k)>0$
  such that for any anti-symmetric $f:\{-1,1\}^n \rightarrow \{0,1\}$
  satisfying $\max_i \Inf_i f \le \tau$,
  \begin{equation}
    \IP[\UniqueBest_k(f)]
    \le
    \lim_{n \rightarrow \infty}
    \IP[\UniqueBest_k(\Maj_n)] + \epsilon
  \end{equation}
\end{theorem}

An alternative proof can be derived via multi-dimensional generalization of the inductive Majority is Stablest Theorem using a general notion of $\rho$-concavity which generalizes 
Claim~\ref{clm:negative-semidefinite}~\cite{Neeman:14,Ledoux:14}. 

It is not hard to show that  
\begin{equation} \label{eq:uniq_max_k}
\IP[\UniqueBest_k(\Maj_n)] = k^{-1+o(1)}.
\end{equation}
 
\begin{exercise}
Prove (\ref{eq:uniq_max_k}) by using a multi-dimensional CLT and representing the advantage of candidate $1$ over candidate $i$ as 
\[
\sqrt{\frac{1}{3}} X + \sqrt{\frac{2}{3}} Z_i,
\]
where $X,Z_2,\ldots,Z_k$ are i.i.d. $N(0,1)$ random variables. 
\end{exercise}

Other than the case $k=3$, where the notions of unique-max and linear tournaments coincide, very little is known about which function maximizes the probability of a linear order. Even computing this probability for Majority provides a surprising result~\cite{Mossel:10}: 
\begin{proposition} \label{prop:maj_linear}
We have
\begin{equation} \label{eq:maj_linear}
\lim_{n \to \infty} \P[\Acyc(G_k(\sigma; \Maj_n))] =
\exp(-\Theta(k^{5/3})).
\end{equation}
\end{proposition}
We find this asymptotic behavior quite surprising.
Indeed, given the previous results that the probability that there is a unique max is $k^{-1+o(1)}$, one may expect that the probability that the order
is linear would be
\[
k^{-1+o(1)} (k-1)^{-1+o(1)} \ldots = (k!)^{-1+o(1)}.
\]
However, it turns out that there is a strong negative correlation between the
event that there is a unique maximum among the $k$ candidates and that among
the other candidates there is a unique max.

\begin{proof}
We use the multi-dimensional CLT. Let
\[
X_{a > b} = \frac{1}{\sqrt{n}} \left(
|\{\sigma : \sigma(a) > \sigma(b)\}| - |\{\sigma : \sigma(b) > \sigma(a) \}|
\right)
\]
By the CLT at the limit the collection of variables $(X_{a > b})_{a \neq b}$
converges to a joint Gaussian vector $(N_{a > b})_{a \neq b}$ satisfying
for all distinct $a,b,c,d$:
\[
N_{a > b} = - N_{b > a}, \quad
\Cov[N_{a > b}, N_{a > c}] = \frac{1}{3}, \quad
\Cov[N_{a > b}, N_{c > d}] = 0.
\]
and $N_{a > b} \sim N(0,1)$ for all $a$ and $b$.

We are interested in providing bounds on
\[
P[\forall a > b: N_{a > b} > 0]
\]
as the probability that the resulting tournament is an order is obtained
by multiplying by a $k! = \exp(\Theta(k \log k))$ factor.

We claim that there exist independent $N(0,1)$ random variables $X_a$ for $1 \leq a \leq k$ and
$Z_{a > b}$ for $1 \leq a \neq b \leq k$ such that
\[
N_{a > b} = \frac{1}{\sqrt{3}} (X_a - X_b + Z_{a > b})
\]
(where $Z_{a > b} = - Z_{b > a}$).
This follows from the fact that the joint distribution of Gaussian random
variables is determined by the covariance matrix (this is noted
in the literature in~\cite{NiemiWeisberg:68}).

We now prove the upper bound. Let $\alpha$ be a constant to be chosen later.
Note that for all $\alpha$ and large enough $k$ it holds that:
\[
P[|X_a| > k^{\alpha}] \leq \exp(-\Omega(k^{2 \alpha})).
\]
Therefore the probability that for at least half of the $a$'s in the interval
$[k/2,k]$ it holds that $|X_a| > k^{\alpha}$ is at most
\[
\exp(-\Theta(k^{1 + 2 \alpha})).
\]

Let's assume that at least half of the $a$'s in the interval $[k/2,k]$ satisfy
that $|X_a| < k^{\alpha}$. We claim that in this case the number 
of pairs $a > b$ such that
$X_a, X_b \in -[k^{\alpha},k^{\alpha}]$ and $X_a - X_b < 1$ is $\Omega(k^{2-\alpha})$.

For the last claim partition the interval
$[-k^{\alpha},k^{\alpha}]$ into sub-intervals of
length $1$ and note that at least $\Omega(k)$ of the points belong to
sub-intervals which contain at least $\Omega(k^{1-\alpha})$ points.
This implies that the number of pairs $a > b$ satisfying
$|X_a - X_b| < 1$ is $\Omega(k^{2-\alpha})$.

Note that for such pair $a > b$ in order that $N_{a > b} > 0$ we need
that $Z_{a>b} > -1$ which happens with constant probability.

We conclude that given that half of the $X$'s fall in $[-k^{\alpha},k^{\alpha}]$
the probability of a linear order is bounded by
\[
\exp(-\Omega(k^{2-\alpha})).
\]
Thus overall we have bounded the probability by
\[
\exp(-\Omega(k^{1 + 2 \alpha})) + \exp(-\Omega(k^{2-\alpha})).
\]
The optimal exponent is $\alpha=1/3$ giving the desired upper bound.

For the lower bound we condition on $X_a$ taking value in
$(a,a+1) k^{-2/3}$. Each probability is at least $\exp(-O(k^{2/3}))$ and therefore the probability that all $X_a$ take
such values is
\[
\exp(-O(k^{5/3})).
\]
Moreover, conditioned on $X_a$ taking such values the probability that
\[
Z_{a > b} > X_b - X_a,
\]
for all $a > b$ is at least
\[
\left( \prod_{i=0}^{k-1} \Phi(i)^{k^{2/3}} \right)^k \geq \left( \prod_{i=0}^{\infty} \Phi(i) \right)^{k^{5/3}} =
\exp(-O(k^{5/3})).
\]
This proves the required result.
\end{proof}

\chapter[Manipulation]{Manipulation and Isoperimetry} 

\section{Quantitive Manipulation} 
The GS Theorem, Theorem~\ref{thm:GS_intro}, states that under natural conditions, there exist profiles of voters such that at least one voter can manipulate.
The basic approach for the proof is to view the theorem as an isoperimetric theorem. 
In classical Isoperimetric theory, the goal is to find conditions that establish large boundary between sets. 
In the context of manipulation we can consider a voter who can manipulate as a special boundary point and our goal is to prove that there are many boundary points. 
   
It is natural to consider the following graph where the vertex set is $S_k^n$ --- the set of all voting profiles and there are edges between voting profiles that differ at a single voter. The statement of the GS theorem can be interpreted in terms of this graph: for certain natural partitions of $S_k^n$ into $k$ parts there is an edge of the graph between two different parts that corresponds to a manipulation. 
We will see that the existence of many edges between different parts of the graph follows from classical Isoperimetric theory. 

Thus one may consider quantitative statements of the GS theorem as isoperimetric statements: It is not only the case that there are many edges between different parts of the partition, but it is also the case that many of these edges correspond to manipulation by one of the voters.   

In the classical setup, isoperimetry and concentration of measure are closely related. In particular, standard concentration of measure results imply that for any set of fractional size at least 
$\eps$ in $S_k^n$, the set of profiles at graph distance at most $C(\eps) \sqrt{n}$ contains almost the whole graph. However, it is not known under what conditions typically a small coalition can manipulate. 

It may be useful to consider the following examples: 
\begin{itemize}
\item 
Consider the plurality function with $q \geq 3$ alternatives and $n$ voters. For a voter to be able to manipulate, it has to be that the difference between the top candidate and the second to top is at most $1$. This implies that the probability that there exists a voter who can manipulate is $O(1/\sqrt{n})$. 

A second question we can ask is what is the minimal size $s$ of a coalition $S \subset [n]$ that can manipulate with probability close to $1$. Since the difference between the top candidate and the second candidate is typically of order $\sqrt{n}$, it is clear that the size has to be at least order $\sqrt{n}$ and in fact it is easy to see that if 
$s/\sqrt{n} \to \infty$ that for every set $S$ of size $s$, with probability $1-o(1)$ there exists a subset $T \subset S$ that can change the outcome of the elections by manipulating. Indeed if $a$, $b$ is the second from the top and $c$ is a candidate different than $a$ and $b$, we may take $T$ to be all the voters in $S$ that 
rank $c$ above $a$ above $b$. 
\item 
Consider the case $q=2$ and the function $g(x) = -f(x)$ where $f$ is the tribes function.
In this case a voter can manipulate if and only if they are influential. Therefore, on general, we cannot expect that the probability that an individual voter can manipulate is higher than $O(\log n /n)$.
\end{itemize}

\subsection{A quantitive GS Theorem} 
Our goal in this section is to prove a quantitative version of the manipulation Theorem.  We will mostly follow~\cite{MosselRacz:12,MosselRacz:15} who proved a pretty general average manipulation theorem for a single voter. Some special cases of the Theorem were known before, in particular in the case of 
$3$ alternatives this was proved by  Friedgut, Kalai, Keller and Nisan~\cite{FrKaNi:08,FKKN:11}.

In this section we will prove that: if $k \geq 3$ and the SCF $f$ is $\eps$-far from the family of nonmanipulable functions, then the probability of a ranking profile being manipulable is bounded from below by a polynomial in $1/n$, $1/k$, and $\eps$. We continue by first presenting our results, then discussing their implications, and finally we conclude this section by commenting on the techniques used in the proof.

\subsection{Definitions and formal Statements} 

Recall that our basic setup consists of $n$ voters electing a winner among $k$ alternatives via
  An SCF $f : S_k^n \to \left[k\right]$. We now define manipulability in more detail:
\begin{definition}[Manipulation points]\label{def:manip}
Let $\sigma \in S_k^n$ be a ranking profile. Write $a \stackrel{\sigma_i}{>} b$ to denote that alternative $a$ is preferred over $b$ by voter $i$. A SCF $f \colon S_k^n \to [k]$ is \emph{manipulable} at the ranking profile $\sigma \in S_k^n$ if there exists a $\sigma' \in S_k^n$ and an $i\in \left[n\right]$ such that $\sigma$ and $\sigma'$ only differ in the $i^{\text{th}}$ coordinate and
\begin{equation*}
    f(\sigma') \stackrel{\sigma_i}{>} f(\sigma).
\end{equation*}
In this case we also say that $\sigma$ is a {\em manipulation point} of $f$, and that $(\sigma,\sigma')$ is {\em a manipulation pair} for $f$. We say that $f$ is {\em manipulable} if it is manipulable at some point $\sigma$. We also say that $\sigma$ is an $r$-{\em manipulation point} of $f$ if $f$ has a manipulation pair $(\sigma,\sigma')$ such that $\sigma'$ is obtained from $\sigma$ by permuting (at most) $r$ adjacent alternatives in one of the coordinates of $\sigma$. (We allow $r > k$---any manipulation point is an $r$-manipulation point for $r > k$.)

Let $M \left( f \right)$ denote the set of manipulation points of the SCF $f$, and for a given $r$, let $M_r \left( f \right)$ denote the set of $r$-manipulation points of $f$. When the SCF is obvious from the context, we write simply $M$ and $M_r$.
\end{definition}
We first recall Gibbard and Satterthwaite theorem (stated as Theorem~\ref{thm:GS_intro} in the introduction): 
\begin{theorem}[Gibbard-Satterthwaite~\cite{Gibbard:73,Satterthwaite:75}]\label{thm:GS}
Any SCF $f \colon S_k^n \to \left[k\right]$ which takes at least three values and is not a dictator (i.e., not a function of only one voter) is manipulable.
\end{theorem}
This theorem is tight in the sense that \emph{monotone} SCFs which are dictators or only have two possible outcomes are indeed nonmanipulable (a function is non-monotone, and clearly manipulable, if for some ranking profile a voter can change the outcome from, say, $a$ to $b$ by moving $a$ ahead of $b$ in her preference). It is useful to introduce a refined notion of a dictator before defining the set of nonmanipulable SCFs.
\begin{definition}[Dictator on a subset]\label{def:dict_subset}
For a subset of alternatives $H \subseteq \left[k\right]$, let $\tp_H$ be the SCF on one voter whose output is always the top ranked alternative among those in $H$.
\end{definition}
\begin{definition}[Nonmanipulable SCFs]
We denote by $\NONMANIP \equiv \NONMANIP \left( n, k \right)$ the set of nonmanipulable SCFs, which is the following:
\begin{align*}
\NONMANIP \left( n, k\right) &= \left\{ f : S_k^n \to \left[k\right] \mid f \left( \sigma \right) = \tp_H \left( \sigma_i \right) \text{ for some } i \in \left[n\right], H \subseteq \left[k\right], H \neq \emptyset \right\} \\
&\cup \left\{ f : S_k^n \to \left[k\right] \mid f \text{ is a monotone function taking on exactly two values} \right\}.
\end{align*}
When the parameters $n$ and $k$ are obvious from the context, we omit them.
\end{definition}
Another useful class of functions, which is larger than $\NONMANIP$, but which has a simpler description, is the following.
\begin{definition}
Define, for parameters $n$ and $k$ that remain implicit:
\[
  \overline{\NONMANIP} = \{f \colon S_k^n \to [k] \mid f \text{ only depends on one coordinate or takes at most two values}\}.
\]
\end{definition}
The notation should be thought of as ``closure'' rather than ``complement''. 

As discussed previously, our goal is to study manipulability from a quantitative viewpoint, and in order to do so we need to define the distance between SCFs.
\begin{definition}[Distance between SCFs]
The distance $\Dist(f,g)$ between two SCFs $f,g \colon S_k^n \to \left[ k \right]$ is defined as the fraction of inputs on which they differ: $\Dist\left(f,g\right) = \p\left(f\left(\sigma\right) \neq g\left(\sigma\right)\right)$, where $\sigma \in S_k^n$ is uniformly selected. For a class $G$ of SCFs, we write $\Dist\left(f,G\right) = \min_{g \in G} \Dist\left(f,g\right)$.
\end{definition}


The concepts of anonymity and neutrality of SCFs will be important to us, so we define them here.
\begin{definition}[Anonymity]\label{def:anonymity}
 A SCF is \emph{anonymous} if it is invariant under changes made to the names of the voters. More precisely, a SCF $f : S_k^n \to \left[k\right]$ is anonymous if for every $\sigma = \left( \sigma_1, \dots, \sigma_n \right) \in S_k^n$ and every $\pi \in S_n$,
\[
 f\left( \sigma_1, \dots, \sigma_n \right) = f\left( \sigma_{\pi\left( 1 \right)}, \dots, \sigma_{\pi \left( n \right)} \right).
\]
\end{definition}
\begin{definition}[Neutrality]\label{def:neutrality}
 A SCF is \emph{neutral} if it commutes with changes made to the names of the alternatives. More precisely, a SCF $f : S_k^n \to \left[k\right]$ is neutral if for every $\sigma = \left( \sigma_1, \dots, \sigma_n \right) \in S_k^n$ and every $\pi \in S_k$,
\[
 f \left( \pi \circ \sigma_1, \dots, \pi \circ \sigma_n \right) = \pi \left( f \left( \sigma \right) \right).
\]
\end{definition}

Our goal is to sketch the proof of the following Theorem:
\begin{theorem}\label{cor:k_refined_truenonmanip}
Suppose we have $n \geq 1$ voters, $k \geq 3$ alternatives, and a SCF $f : S_k^n \to \left[k\right]$ satisfying $\Dist \left( f, \NONMANIP \right) \geq \eps$. Then
\begin{equation}\label{eq:manip_refined_true}
\p\left( \sigma \in M \left( f \right) \right)\\
\geq \p \left( \sigma \in M_4 \left( f \right) \right)\geq p \left( \eps, \frac{1}{n}, \frac{1}{k} \right)
\end{equation}
for some polynomial $p$, where $\sigma \in S_k^n$ is selected uniformly.

An immediate consequence is that
\[
\p \left( \left( \sigma, \sigma' \right) \text{ is a manipulation pair for } f \right) \geq q \left( \eps, \frac{1}{n}, \frac{1}{k} \right)
\]
for some polynomial $q$, where $\sigma \in S_k^n$ is uniformly selected, and $\sigma'$ is obtained from $\sigma$ by uniformly selecting a coordinate $i \in \left\{1, \dots, n \right\}$, uniformly selecting $j \in \left\{1, \dots, n-3 \right\}$, and then uniformly randomly permuting the following four adjacent alternatives in $\sigma_i$: $\sigma_i \left( j \right), \sigma_i \left( j + 1 \right), \sigma_i \left( j + 2 \right)$, and $\sigma_i \left( j + 3 \right)$. 
\end{theorem}

\subsection{Proof Ideas}
We first present our techniques that achieve a lower bound for the probability of manipulation that involves factors of $\frac{1}{k!}$ 
and then describe how a refined approach leads to a lower bound which has inverse polynomial dependence on $k$. 

\medskip

\noindent\textbf{Rankings graph and applying the original Gibbard-Satterthwaite theorem.}  Consider the graph $G = \left( V, E \right)$ having vertex set $V = S_k^n$, the set of all ranking profiles, and let $\left( \sigma, \sigma' \right) \in E$ if and only if $\sigma$ and $\sigma'$ differ in exactly one coordinate. The SCF $f : S_k^n \to \left[k \right]$ naturally partitions $V$ into $k$ subsets. Since every manipulation point must be on the boundary between two such subsets, we are interested in the size of such boundaries.

For two alternatives $a$ and $b$, and voter $i$, denote by $B_i^{a,b}$ the boundary between $f^{-1} \left( a \right)$ and $f^{-1}\left( b \right)$ in voter $i$. 
A simple lemma 
tells us that at least two of the boundaries are large; in the following assume that these are $B_1^{a,b}$ and $B_2^{a,c}$. Now if a ranking profile $\sigma$ lies on \emph{both} of these boundaries, then applying the original Gibbard-Satterthwaite theorem to the restricted SCF on two voters where we fix all coordinates of $\sigma$ except the first two, we get that there must exist a manipulation point which agrees with $\sigma$ in all but the first two coordinates. Consequently, if we can show that the \emph{intersection} of the boundaries $B_1^{a,b}$ and $B_2^{a,c}$ is large, then we have many manipulation points.

\medskip

\noindent\textbf{Fibers and reverse hypercontractivity.} In order to have more ``control'' over what is happening at the boundaries, we partition the graph further---this idea is due to Friedgut et al.~\cite{FrKaNi:08,FKKN:11}. Given a ranking profile $\sigma$ and two alternatives $a$ and $b$, $\sigma$ induces a \emph{vector of preferences} $x^{a,b}\left( \sigma \right) \in \left\{-1,1\right\}^{n}$ between $a$ and $b$. For a vector $z^{a,b} \in \left\{-1,1\right\}^n$ we define the \emph{fiber with respect to preferences between $a$ and $b$}, denoted by $F\left( z^{a,b} \right)$, to be the set of ranking profiles for which the vector of preferences between $a$ and $b$ is $z^{a,b}$. We can then partition the vertex set $V$ into such fibers, and work inside each fiber separately. Working inside a specific fiber is advantageous, because it gives us the extra knowledge of the vector of preferences between $a$ and $b$.


We distinguish two types of fibers: large and small. We say that a fiber w.r.t.\ preferences between $a$ and $b$ is \emph{large} if almost all of the ranking profiles in this fiber lie on the boundary $B_1^{a,b}$, and \emph{small} otherwise. Now since the boundary $B_1^{a,b}$ is large, either there is big mass on the large fibers w.r.t.\ preferences between $a$ and $b$ or big mass on the small fibers. This holds analogously for the boundary $B_2^{a,c}$ and fibers w.r.t.\ preferences between $a$ and $c$.

Consider the case when there is big mass on the large fibers of both $B_1^{a,b}$ and $B_2^{a,c}$. Notice that for a ranking profile $\sigma$, being in a fiber w.r.t.\ preferences between $a$ and $b$ only depends on the vector of preferences between $a$ and $b$, $x^{a,b}\left( \sigma \right)$, which is a uniform bit vector. Similarly, being in a fiber w.r.t.\ preferences between $a$ and $c$ only depends on $x^{a,c}\left( \sigma \right)$. Moreover, we know the exact correlation between the coordinates of $x^{a,b} \left( \sigma \right)$ and $x^{a,c} \left( \sigma \right)$, and it is in exactly this setting where \emph{reverse hypercontractivity} applies 
and shows that the \emph{intersection} of the large fibers of $B_1^{a,b}$ and $B_2^{a,c}$ is also large. Finally, by the definition of a large fiber it follows that the intersection of the \emph{boundaries} $B_1^{a,b}$ and $B_2^{a,c}$ is large as well, and we can finish the argument using the Gibbard-Satterthwaite theorem as above.

To deal with the case when there is big mass on the \emph{small} fibers of $B_1^{a,b}$ we use various isoperimetric techniques, including the canonical path method. 
In particular, we use the fact that for a small fiber for $B_1^{a,b}$, the size of the boundary of $B_1^{a,b}$ in the small fiber is comparable to the size of $B_1^{a,b}$ in the small fiber itself, up to polynomial factors.

\medskip

\noindent\textbf{A refined geometry.} Using this approach with the rankings graph above, our bound includes $\frac{1}{k!}$ factors.
In order to obtain inverse polynomial dependence on $k$ 
we use a refined approach. Instead of the rankings graph outlined above, we use an underlying graph with a different edge structure: $\left( \sigma, \sigma' \right) \in E$ if and only if $\sigma$ and $\sigma'$ differ in exactly one coordinate, and in this coordinate they differ by a single adjacent transposition. In order to prove the refined result, we need to show that the geometric and combinatorial quantities such as boundaries and manipulation points are roughly the same in the refined graph as in the original rankings graph. In particular, this is where we need to analyze carefully functions of one voter, and ultimately 
prove a quantitative Gibbard-Satterthwaite theorem for one voter.

\subsection{Isoperimetric Results}\label{subsec:gs:iso} 


\subsubsection{Boundaries and influences} 

For a general graph $G = \left( V, E \right)$, and a subset of the vertices $A \subseteq G$, we define the \emph{edge boundary} of $A$ as
\[
\partial_{e} \left( A \right) = \left\{ \left( u, v \right) \in E : u \in A, v \notin A \right\}.
\]
We also define the \emph{boundary} (or vertex boundary) of  a subset of the vertices $A \subseteq G$ to be the set of vertices in $A$ which have a neighbor that is not in $A$:
\[
\partial \left( A \right) = \left\{ u \in A : \text{ there exists } v \notin A \text{ such that } \left( u, v \right) \in E \right\}.
\]
If $u \in \partial \left( A \right)$, we also say that $u$ is \emph{on} the
edge boundary of $A$.


\begin{definition}[Boundaries]
For a given SCF $f$ and a given alternative $a \in \left[k \right]$, we define
\[
H^{a} \left( f \right) = \left\{ \sigma \in S_k^n : f\left( \sigma \right) = a \right\},
\]
the set of ranking profiles where the outcome of the vote is $a$. The edge boundary of this set is denoted by $B^a \left( f \right)$ : $B^a \left( f \right) = \partial_e \left( H^a \left( f \right) \right)$. This boundary can be partitioned: we say that the edge boundary of $H^a \left( f \right)$ in the direction of the $i^{\text{th}}$ coordinate is
\[
B_i^{a} \left( f \right) = \left\{ \left( \sigma, \sigma' \right) \in B^a \left( f \right) :  \sigma_i \neq \sigma'_i \right\}.
\]
The boundary $B^a \left( f \right)$ can be therefore written as $B^a \left( f \right) = \cup_{i=1}^{n} B_i^a \left( f \right)$. We can also define the boundary between two alternatives $a$ and $b$ in the direction of the $i^{\text{th}}$ coordinate:
\[
B_i^{a,b} \left( f \right) = \left\{ \left( \sigma, \sigma' \right) \in B_i^a \left( f \right) : f\left( \sigma' \right) = b \right\}.
\]
We also say that $\sigma \in B_i^a \left( f \right)$ is \emph{on} the boundary $B_i^{a,b} \left( f \right)$ if there exists $\sigma'$ such that $\left( \sigma, \sigma' \right) \in B_i^{a,b} \left( f \right)$.
\end{definition}

We will need to generalize the definition of influences as follows: 
\begin{definition}[Influences]
We define the \emph{influence} of the $i^{\text{th}}$ coordinate on $f$ as
\[
\Inf_i \left( f \right) = \p \left( f\left( \sigma \right) \neq f \left( \sigma^{\left( i \right)} \right) \right) = \p \left( \left( \sigma, \sigma^{\left( i \right)} \right) \in \cup_{a = 1}^{k} B_i^a \left( f \right) \right),
\]
where $\sigma$ is uniform on $S_{k}^{n}$ and $\sigma^{\left( i \right)}$ is obtained from $\sigma$ by rerandomizing the $i^{\text{th}}$ coordinate. Similarly, we define the influence of the $i^{\text{th}}$ coordinate with respect to a single alternative $a \in \left[k \right]$ or a pair of alternatives $a,b \in \left[k\right]$ as
\[
\Inf_i^{a} \left( f \right) = \p \left( f\left( \sigma \right) = a, f \left( \sigma^{\left( i \right)} \right) \neq a \right) = \p \left( \left( \sigma, \sigma^{\left( i \right)} \right) \in B_i^a \left( f \right) \right) ,
\]
and
\[
\Inf_i^{a,b} \left( f \right) = \p \left( f\left( \sigma \right) = a, f \left( \sigma^{\left( i \right)} \right) = b \right) = \p \left( \left( \sigma, \sigma^{\left( i \right)} \right) \in B_i^{a,b} \left( f \right) \right),
\]
respectively.
\end{definition}
Clearly
\[
\Inf_i \left( f \right) = \sum_{a=1}^{k} \Inf_i^a \left( f \right) = \sum_{a,b \in \left[k \right]: a\neq b} \Inf_{i}^{a,b} \left( f \right).
\]

Most of the time the specific SCF $f$ will be clear from the context, in which case we omit the dependence on $f$, and write simply $B^a \equiv B^a \left( f \right)$, $B_i^a \equiv B_i^a \left( f \right)$, etc.

\subsection{Large boundaries} 

The following standard proposition bounds the total influence with
respect to a given candidate from below by the variance with respect to that
candidate.
\begin{proposition}
  \label{prop:sumInfVarBound}
  For any $f \colon S_k^n \rightarrow [k]$ and $a \in [k]$,
  \begin{equation}
    \sum_{i=1}^n \Inf_i^a (f)
    \ge
    \Var [ 1_{\{f(X) = a\}} ]
  \end{equation}
  where $X \in S_k^n$ is uniformly selected.
\end{proposition}
\begin{proof}
  Create a random walk
  $X=X^{(0)}, \ldots, X^{(n)}=Y$
  from $X$ by re-randomizing the $i$:th coordinate in the $i$:th step,
  i.e. for $i \in [n]$,
  $X^{(i)} \in S_k^n$ is obtained by re-randomizing the $i$:th coordinate
  of $X^{(i-1)}$.
  Letting $g(x)=1_{\{f(x) = a\}}$ and using that $X,Y$ are independent
  and that if $g(X) \neq g(Y)$
  then the value of $g$ has to change at some edge on the path we have
  \begin{eqnarray*}
    2 \Var [ 1_{\{f(X) = a\}} ]
    &=&
    2 \Var g(X)
    =
    \P(g(X) \neq g(Y))
    \le
    \\
    &\le&
    \P(\cup_{i \in [n]} \{g(X^{(i-1)}) \neq g(X^{(i)})\})
    \le
    \sum_{i=1}^n 2 \Inf_i^a(f)
  \end{eqnarray*}
\end{proof}

Further, if a function is far from all constants not all such variances can be small:
\begin{lemma}
  \label{lem:constDist}
  For any $f \colon S_k^n \rightarrow [k]$,
  \begin{equation}
    \Dist(f, \CONST)
    \le
    \frac{k}{2} \sum_{a=1}^k
    \Var [ 1_{\{f(X) = a\}} ]
  \end{equation}
\end{lemma}
\begin{proof}
  For $a \in [k]$, let $\mu_a = \P(f(X)=a)$
  and assume w.l.o.g. that $\mu_1 \ge \mu_2 \ge \ldots \ge \mu_k$.
  Then,
  \begin{align*}
    \Dist(f, \CONST)
    & =
    (1-\mu_1)
    \le
    k \mu_1 (1-\mu_1)
    =
    \frac{k}{2} \left(1 - \mu_1^2 - (1-\mu_1)^2\right)
    \le
    \\
    & \le
    \frac{k}{2} \left(1- \sum_{a=1}^k \mu_a^2\right)
    =
    \frac{k}{2} \sum_{a=1}^k \mu_a - \mu_a^2
    =
    \frac{k}{2} \sum_{a=1}^k \Var [ 1_{\{f(X) = a\}} ]
  \end{align*}
\end{proof}

The following lemma shows that there are some boundaries which are large 
\begin{lemma} \label{lem:boundaries1}
  Fix $k \ge 3$ and
  $f \colon S_k^n \to \left[k\right]$
  satisfying $\Dist(f, \overline{\NONMANIP}) \ge \eps$.
  Then there exist distinct $i,j \in [n]$
  and $\{a,b\},\{c,d\} \subseteq [k]$ such that $c \notin \{a,b\}$ and
  \begin{equation}
    \Inf_i^{a,b} (f) \ge \frac{2\eps}{n k^2 (k-1)}
    \quad \text{ and } \quad
    \Inf_j^{c,d} (f) \ge \frac{2\eps}{n k^2 (k-1)}.
  \end{equation}
\end{lemma}

\begin{proof}
  For $a \neq b$ let
  $
  A^{a,b} =
  \left\{
    i\in [n] \mid \Inf_i^{a,b} \ge \frac{2\eps}{n k^2 (k-1)}
  \right\}
  $.

The proof of the lemma would follow if we could show that 1. $|\cup_{a,b} A^{a,b}| \ge 2$, and 
   2. for all $\{a,b\}$ there exists $\{c,d\}$ such that $\{c,d\}
  \neq \{a,b\}$ and $A^{c,d} \neq \emptyset$.
  
  For the second claim, note that $f$ being $\eps$-far from taking two values implies that we can find
  a $c \notin \{a,b\}$
  such that $1-\frac{\eps}{q} \ge \P(f(X) = c) \ge \frac{\eps}{k-2} \ge \frac{\eps}{k}$.
  But then by Proposition~\ref{prop:sumInfVarBound} it holds that 
  \begin{equation*}
    \sum_{d \neq c} \sum_{i=1}^n \Inf_i^{c,d} (f)
    =
    \sum_{i=1}^n \Inf_i^c (f)
    \ge
    \Var [ 1_{\{f(X) = c\}} ]
    \ge
    \frac{\eps(1-\eps/k)}{k}
    \ge
    \frac{\eps (k-1)}{k^2}
  \end{equation*}
  hence there must exist some $d \neq c$ and $i \in [n]$ such that
  $\Inf_i^{c,d} \ge \frac{\eps}{n k^2} \ge \frac{2\eps}{n k^2(k-1)}$,
  and thus $A^{c,d} \neq \emptyset$.

It remains to show that 
  \begin{equation}
    \label{eq:unionContains}
    |\cup_{a,b} A^{a,b}| \ge 2
  \end{equation}
  To see this, assume the contrary, i.e.
  $\cup_{a,b} A^{a,b} \subseteq \{i\}$ for some $i \in [n]$.
  Then for all $j \neq i$ it holds that
  \begin{equation}
    \label{eq:nonInflInflBound}
    \Inf_j (f) = \sum_{c,d} \Inf_j^{c,d}(f) < \frac{k(k-1)}{2}
    \frac{2\eps}{n k^2 (k-1)} = \frac{\eps}{n k}
  \end{equation}
  For $\sigma \in S_n$,
  let $f_\sigma(x) = f(x_1, \ldots, x_{i-1}, \sigma, x_{i+1}, \ldots, x_n)$
  and note that for $j \neq i$,
  \begin{equation}
    \Inf_j(f) = \frac{1}{k!} \sum_{\sigma \in S_n} \Inf_j(f_\sigma)
  \end{equation}
  while $\Inf_i(f_\sigma) = 0$.
  Hence, by \eqref{eq:nonInflInflBound}, we have
  \begin{equation*}
    \eps
    >
    k \sum_{j \neq i} \Inf_j (f)
    =
    \frac{k}{k!} \sum_{j=1}^n \sum_{\sigma } \Inf_j(f_\sigma)
    \ge
    \frac{2}{k!} \sum_{\sigma} \Dist(f_\sigma, \CONST)
    =
    2 \Dist(f,\DICT_i)
  \end{equation*}
  where the second inequality follows from Lemma~\ref{lem:constDist}
  and Proposition~\ref{prop:sumInfVarBound}.
  But this means that $f$ is $\eps/2$-close to a dictator, contradicting
  the assumption that $\Dist(f,\NONMANIP) \ge \eps$.

  Hence \eqref{eq:unionContains} holds. The proof follows. 
\end{proof}

\subsection{General isoperimetric results} 

Our rankings graph is the Cartesian product of $n$ complete graphs on $k!$ vertices. 
The edge-isoperimetric problem on the product of complete graphs was originally solved by Lindsey~\cite{Lindsey:64}, implying the following result.
\begin{corollary}\label{cor:isoLindsey}
If $A \subseteq K_{\ell} \times \dots \times K_{\ell}$ ($n$ copies of the complete graph $K_{\ell}$) and $\left| A \right| \leq \left( 1 - \frac{1}{\ell} \right) \ell^n$, then $\left| \partial_{e} \left( A \right) \right| \geq \left| A \right|$.
\end{corollary}

\subsection{Fibers}\label{sec:fibers} 

In our proof we need to partition the graph even further---this idea is due to Friedgut, Kalai, Keller, and Nisan~\cite{FrKaNi:08,FKKN:11}.
\begin{definition}\label{def:xab}
For a ranking profile $\sigma \in S_k^n$ define the vector
\[
x^{a,b} \equiv x^{a,b} \left( \sigma \right) = \left( x_1^{a,b} \left( \sigma \right), \dots, x_n^{a,b} \left( \sigma \right) \right)
\]
of preferences between $a$ and $b$, where $x_i^{a,b} \left( \sigma \right) = 1$ if $a \stackrel{\sigma_i}{>} b$ and $x_i^{a,b} \left( \sigma \right) = -1$ otherwise.
\end{definition}
\begin{definition}[Fibers]\label{def:fibers}
For a pair of alternatives $a,b \in \left[k \right]$ and a vector $z^{a,b} \in \left\{-1,1\right\}^n$, write
\[
F \left( z^{a,b} \right) := \left\{ \sigma : x^{a,b} \left( \sigma \right) = z^{a,b} \right\}.
\]
We call the $F \left( z^{a,b} \right)$ \emph{fibers} with respect to preferences between $a$ and $b$.
\end{definition}
So for any pair of alternatives $a,b$, we can partition the ranking profiles according to its fibers:
\[
S_k^n = \bigcup_{z^{a,b} \in \left\{ - 1, 1 \right\}^{n}} F \left( z^{a,b} \right).
\]

Given a SCF $f$, for any pair of alternatives $a,b \in \left[ k \right]$ and $i \in \left[ n \right]$, we can also partition the boundary $B_i^{a,b} \left( f \right)$ according to its fibers. There are multiple, slightly different ways of doing this, but for our purposes the following definition is most useful. Define
\[
B_i \left( z^{a,b} \right) := \left\{ \sigma \in F \left( z^{a,b} \right) : f \left( \sigma \right) = a, \text{ and there exists } \sigma' \text{ s.t.\ } \left( \sigma, \sigma' \right) \in B_i^{a,b} \right\},
\]
where we omit the dependence of $B_i \left( z^{a,b} \right)$ on $f$. So $B_i \left( z^{a,b} \right) \subseteq F \left( z^{a,b} \right)$ is the set of vertices on the given fiber for which the outcome is $a$ and which lies on the boundary between $a$ and $b$ in direction $i$. We call the sets of the form $B_i \left( z^{a,b} \right)$ \emph{fibers for the boundary $B_i^{a,b}$} (again omitting the dependence on $f$ of both sets).

We now distinguish between small and large fibers for the boundary $B_i^{a,b}$.
\begin{definition}[Small and large fibers]\label{def:lg_fibers}
We say that the fiber $B_i \left( z^{a,b} \right)$ is \emph{large} if
\begin{equation}\label{eq:lg_fbr_def}
\p \left( \sigma \in B_i \left( z^{a,b} \right) \, \middle| \, \sigma \in F \left( z^{a,b} \right) \right) \geq 1 - \frac{\eps^{3}}{4 n^3 k^9},
\end{equation}
and \emph{small} otherwise.

We denote by $\Lg \left( B_i^{a,b} \right)$ the union of large fibers for the boundary $B_i^{a,b}$, i.e.,
\[
\Lg \left( B_i^{a,b} \right) := \left\{ \sigma : B_i \left( x^{a,b} \left( \sigma \right) \right) \text{ is a large fiber, and } \sigma \in B_i \left( x^{a,b} \left( \sigma \right) \right) \right\}
\]
and similarly, we denote by $\Sm \left( B_i^{a,b} \right)$ the union of small fibers.
\end{definition}
We remark that what is important is that the fraction appearing on the right hand side of \eqref{eq:lg_fbr_def} is a polynomial of $\frac{1}{n}$, $\frac{1}{k}$ and $\eps$---the specific polynomial in this definition is the end result of the computation in the proof.

Finally, for a voter $i$ and a pair of alternatives $a,b \in \left[k\right]$, we define
\[
F_i^{a,b} := \left\{ \sigma : B_i \left( x^{a,b} \left( \sigma \right) \right) \text{ is a large fiber} \right\}.
\]
So this means that
\begin{equation}\label{eq:Aiab}
\p \left( \sigma \in \cup_{z^{a,b}} B_i \left( z^{a,b} \right) \, \middle| \, \sigma \in F_i^{a,b} \right) \geq 1 - \frac{\eps^{3}}{4 n^3 k^9}.
\end{equation}

\subsection{Boundaries of boundaries} 

Finally, we also look at boundaries of boundaries. In particular, for a given vector $z^{a,b}$ of preferences between $a$ and $b$, we can think of the fiber $F \left( z^{a,b} \right)$ as a subgraph of the original rankings graph. When we write $\partial \left( B_i \left( z^{a,b} \right) \right)$, we mean the boundary of $B_i \left( z^{a,b} \right)$ in the subgraph of the rankings graph induced by the fiber $F \left( z^{a,b} \right)$. That is,
\begin{align*}
  \partial \left( B_i \left( z^{a,b} \right) \right)
  &= \{ \sigma \in B_i \left( z^{a,b} \right) : \exists\ \pi \in F\left( z^{a,b} \right) \setminus B_i \left( z^{a,b} \right) \text{ s.t. }\\
  &\qquad\qquad\sigma \text{ and } \pi \text{ differ in exactly one coordinate} \}.
\end{align*}

\subsection{Dictators and miscellaneous definitions}\label{sec:dict_misc_def} 

For a ranking profile $\sigma = \left( \sigma_1, \dots, \sigma_n \right)$ we sometimes write $\sigma_{-i}$ for the collection of all coordinates except the $i^{\text{th}}$ coordinate, i.e., $\sigma = \left( \sigma_i, \sigma_{-i} \right)$. Furthermore, we sometimes distinguish two coordinates, e.g., we write $\sigma = \left( \sigma_1, \sigma_i, \sigma_{-\left\{1, i\right\}} \right)$.

\begin{definition}[Induced SCF on one coordinate]\label{def:induced_SCF}
Let $f_{\sigma_{-i}}$ denote the SCF on one voter induced by $f$ by fixing all voter preferences except the $i^{\text{th}}$ one according to $\sigma_{-i}$. I.e.,
\[
f_{\sigma_{-i}} \left( \cdot \right) := f\left( \cdot, \sigma_{-i} \right).
\]
\end{definition}
Recall Definition~\ref{def:dict_subset} of a dictator on a subset.
\begin{definition}[Ranking profiles giving dictators on a subset]
For a coordinate $i$ and a subset of alternatives $H \subseteq \left[k \right]$, define
\[
D_i^H := \left\{ \sigma_{-i} : f_{\sigma_{-i}} \left( \cdot \right) \equiv \tp_H \left( \cdot \right) \right\}.
\]

Also, for a pair of alternatives $a$ and $b$, define
\[
D_i \left( a,b \right) := \bigcup_{H : \left\{a,b \right\} \subseteq H, \left|H \right| \geq 3} D_i^H.
\]
\end{definition}

\section{Proof for fixed number of alternatives}\label{sec:k_bdd} 

We prove here the following theorem (Theorem~\ref{thm:k_bdd} below), which is weaker than our main theorem, Theorem~\ref{cor:k_refined_truenonmanip}, in two aspects: first, the condition $\Dist \left( f, \NONMANIP \right) \geq \eps$ is replaced with the stronger condition $\Dist \left( f, \overline{\NONMANIP} \right) \geq \eps$, and second, we allow factors of $\frac{1}{k!}$ in our lower bounds for $\p \left( \sigma \in M \left( f \right) \right)$. The advantage is that the proof of this statement is relatively simpler. We move on to getting a lower bound with polynomial dependence on $k$ in the following sections, and finally we replace the condition $\Dist \left( f, \overline{\NONMANIP} \right) \geq \eps$ with $\Dist \left( f, \NONMANIP \right) \geq \eps$. 

\begin{theorem}\label{thm:k_bdd}
Suppose we have $n \geq 2$ voters, $k \geq 3$ alternatives, and a SCF $f : S_k^n \to \left[k\right]$ satisfying $\Dist \left( f, \overline{\NONMANIP} \right) \geq \eps$. Then
\begin{equation}\label{eq:manip_general}
\p \left( \sigma \in M \left( f \right) \right)\geq p \left( \eps, \frac{1}{n}, \frac{1}{k!} \right),
\end{equation}
for some polynomial $p$, where $\sigma \in S_k^n$ is selected uniformly. 

An immediate consequence is that
\begin{equation*}
\p \left( \left( \sigma, \sigma' \right) \text{ is a manipulation pair for } f \right)\geq q \left( \eps, \frac{1}{n}, \frac{1}{k!} \right),
\end{equation*}
for some polynomial $q$, where $\sigma \in S_k^n$ is selected uniformly, and $\sigma'$ is obtained from $\sigma$ by uniformly selecting a coordinate $i \in \left\{1, \dots, n \right\}$ and resetting the $i^{\text{th}}$ coordinate to a random preference. In particular, the specific lower bound for $\p \left( \sigma \in M \left( f \right) \right)$ implies that we can take $q \left( \eps, \frac{1}{n}, \frac{1}{k} \right) = \frac{\eps^5}{4 n^8 k^{12} \left(k!\right)^5}$.
\end{theorem}

First we provide an overview of the proof of Theorem~\ref{thm:k_bdd} in Section~\ref{sec:proof_outline_1}. In this overview we use adjectives such as ``big'', and ``not too small'' to describe probabilities---here these are all synonymous with ``has probability at least an inverse polynomial of $n$, $k!$, and $\eps^{-1}$''.

\subsection{Overview of proof}\label{sec:proof_outline_1} 

The tactic in proving Theorem~\ref{thm:k_bdd} is roughly the following:
\begin{itemize}
\item By Lemma~\ref{lem:boundaries1}, we know that there are at least two boundaries which are big. W.l.o.g.\ we can assume that these are either $B_1^{a,b}$ and $B_2^{a,c}$, or $B_1^{a,b}$ and $B_2^{c,d}$ with $\left\{a,b\right\} \cap \left\{ c,d \right\} = \emptyset$. Our proof works in both cases, but we continue the outline of the proof assuming the former case which is more difficult. 
\item We partition $B_1^{a,b}$ according to its fibers based on the preferences between $a$ and $b$ of the $n$ voters.
 Similarly for $B_2^{a,c}$ and preferences between $a$ and $c$.
\item 
We can distinguish small and large fibers for these two boundaries. Now since $B_1^{a,b}$ is big, either the mass of small fibers, or the mass of large fibers is big. Similarly for $B_2^{a,c}$.
\item Suppose first that there is big mass on large fibers in both $B_1^{a,b}$ and $B_2^{a,c}$. In this case the probability of our random ranking $\sigma$ being in $F_1^{a,b}$ is big, and similarly for $F_2^{a,c}$, where $F_1^{a,b}$ is the union of large fibers which depends on the vector $x^{a,b} \left( \sigma \right)$ of preferences between $a$ and $b$, and similarly being in $F_2^{a,c}$ only depends on the vector $x^{a,c}\left( \sigma \right)$ of preferences between $a$ and $c$. We know the correlation between $x^{a,b}\left( \sigma \right)$ and $x^{a,c} \left( \sigma \right)$ and hence we can apply reverse hypercontractivity (Lemma~\ref{lem:inv_hyp_org}), which tells us that the probability that $\sigma$ lies in both $F_1^{a,b}$ and $F_2^{a,c}$ is big as well. If $\sigma \in F_1^{a,b}$, then voter 1 is pivotal between alternatives $a$ and $b$ with big probability, and similarly if $\sigma \in F_2^{a,c}$, then voter 2 is pivotal between alternatives $a$ and $c$ with big probability. So now we have that the probability that both 
voter 1 is pivotal between $a$ and $b$ and voter 2 is pivotal between $a$ and $c$ is big, and in this case the Gibbard-Satterthwaite theorem tells us that there is a manipulation point which agrees with this ranking profile in all except for perhaps the first two coordinates. So there are many manipulation points.
\item Now suppose that the mass of small fibers in $B_1^{a,b}$ is big. By isoperimetric theory, the size of the boundary of every small fiber is comparable (same order up to $\text{poly}^{-1} \left( \eps^{-1}, n, k! \right)$ factors) to the size of the small fiber. Consequently, the total size of the boundaries of small fibers is comparable to the total size of small fibers, which in this case has to be big.

We then distinguish two cases: either we are on the boundary of a small fiber in the first coordinate, or some other coordinate. If $\sigma$ is on the boundary of a small fiber in some coordinate $j \neq 1$, then the Gibbard-Satterthwaite theorem tells us that there is a manipulation point which agrees with $\sigma$ in all coordinates except perhaps in coordinates 1 and $j$. If our ranking profile $\sigma$ is on the boundary of a small fiber in the first coordinate, then either there exists a manipulation point which agrees with $\sigma$ in all coordinates except perhaps the first, or the SCF on one voter that we obtain from $f$ by fixing the votes of voters 2 through $n$ to be $\sigma_{-1}$ must be a dictator on some subset of the alternatives. So either we get sufficiently many manipulation points this way, or for many votes of voters 2 through $n$, the restricted SCF obtained from $f$ by fixing these votes is a dictator on coordinate 1 on some subset of the alternatives.

Finally, to deal with dictators on the first coordinate, we look at the boundary of the dictators. Since $\Dist \left( f, \overline{\NONMANIP} \right) \geq \eps$, the boundary is big, and we can also show that there is a manipulation point near every boundary point.
\item If the mass of small fibers in $B_2^{a,c}$ is big, then we can do the same thing for this boundary.
\end{itemize}

\subsection{Division into cases} 

For the remainder of Section~\ref{sec:k_bdd}, let us fix the number of voters $n \geq 2$, the number of alternatives $k \geq 3$, and the SCF $f$, which satisfies $\Dist \left( f, \overline{\NONMANIP} \right) \geq \eps$. Accordingly, we typically omit the dependence of various sets (e.g., boundaries between two alternatives) on $f$.

Our starting point is Lemma~\ref{lem:boundaries1}. W.l.o.g.\ we may assume that the two boundaries that the lemma gives us have $i=1$ and $j=2$, so the lemma tells us that
\begin{equation*}
\p \left( \left( \sigma, \sigma^{\left( 1 \right)} \right) \in B_1^{a,b} \right) \geq \frac{2 \eps}{n k^3},
\end{equation*}
where $\sigma$ is uniform on the ranking profiles, and $\sigma^{\left(1\right)}$ is obtained by rerandomizing the first coordinate. This also means that
\begin{equation*}
\p \left( \sigma \in \cup_{z^{a,b}} B_1 \left( z^{a,b} \right) \right) \geq \frac{2 \eps}{n k^3},
\end{equation*}
and similar inequalities hold for the boundary $B_2^{c,d}$. The following lemma is an immediate corollary.
\begin{lemma}\label{lem:cases_k}
Either
\begin{equation}\label{eq:sm_fbr}
\p \left( \sigma \in \Sm \left( B_1^{a,b} \right) \right) \geq \frac{\eps}{n k^3}
\end{equation}
or
\begin{equation}\label{eq:lg_fbr}
\p \left( \sigma \in \Lg \left( B_1^{a,b} \right) \right) \geq \frac{\eps}{n k^3},
\end{equation}
and the same can be said for the boundary $B_2^{c,d}$.
\end{lemma}
We distinguish cases based upon this: either \eqref{eq:sm_fbr} holds, or \eqref{eq:sm_fbr} holds for the boundary $B_2^{c,d}$, or \eqref{eq:lg_fbr} holds for both boundaries. We only need one boundary for the small fiber case, and we need both boundaries only in the large fiber case. So in the large fiber case we must differentiate between two cases: whether $d\in \left\{a,b\right\}$ or $d \notin \left\{ a, b \right\}$. 
We will see that if $d \notin \left\{a,b\right\}$ then the large fiber case cannot occur---so this method of proof works as well.

In the rest of the section we first deal with the large fiber case, and then with the small fiber case.

\subsection{Big mass on large fibers}\label{sec:k_lg_fbr} 

We now deal with the case when
\begin{equation}\label{eq:k_lg_fbr_ab}
\p \left( \sigma \in \Lg \left( B_1^{a,b} \right) \right) \geq \frac{\eps}{n k^3}
\end{equation}
and also
\begin{equation}\label{eq:k_lg_fbr_cd}
\p \left( \sigma \in \Lg \left( B_2^{c,d} \right) \right) \geq \frac{\eps}{n k^3}.
\end{equation}
As mentioned before, we must differentiate between two cases: whether $d\in \left\{a,b\right\}$ or $d \notin \left\{ a, b \right\}$.

\subsubsection{Case 1} 

Suppose $d\in \left\{a,b\right\}$, in which case we may assume w.l.o.g.\ that $d = a$.
\begin{lemma}\label{lem:lg_fbr_1}
If
\begin{equation}\label{eq:k_lg_fbr_ab2}
\p \left( \sigma \in \Lg \left( B_1^{a,b} \right) \right) \geq \frac{\eps}{n k^3} \qquad \text{ and } \qquad \p \left( \sigma \in \Lg \left( B_2^{a,c} \right) \right) \geq \frac{\eps}{n k^3},
\end{equation}
then
\begin{equation}\label{eq:k_lg_fbr_manip}
\p \left( \sigma \in M \right) \geq \frac{\eps^{3}}{2 n^{3} k^{9} \left(k!\right)^2}.
\end{equation}
\end{lemma}
\begin{proof}
By \eqref{eq:k_lg_fbr_ab2} we have that 
\[
 \p \left( \sigma \in F_1^{a,b} \right) \geq \frac{\eps}{n k^3} \qquad \text{ and } \qquad \p \left( \sigma \in F_2^{a,c} \right) \geq \frac{\eps}{n k^3}.
\]
We know that $\left|\E\left( x_i^{a,b} \left( \sigma \right) x_i^{a,c} \left( \sigma \right) \right)\right| = 1/3$, and so by reverse hypercontractivity (Lemma~\ref{lem:inv_hyp_org}) we have that
\begin{equation}\label{eq:inv_hyp_contr_k}
\p \left( \sigma \in F_1^{a,b} \cap F_2^{a,c} \right) \geq \frac{\eps^{3}}{n^3 k^9}.
\end{equation}
Recall that we say that $\sigma$ is \emph{on} the boundary $B_1^{a,b}$ if there exists $\sigma'$ such that $\left( \sigma, \sigma' \right) \in B_1^{a,b}$. If $\sigma \in F_1^{a,b}$, then with big probability $\sigma$ is on the boundary $B_1^{a,b}$, and if $\sigma \in F_2^{a,c}$, then with big probability $\sigma$ is on the boundary $B_2^{a,c}$. Using this and \eqref{eq:inv_hyp_contr_k} we can show that the probability of $\sigma$ lying on both the boundary $B_1^{a,b}$ and the boundary $B_2^{a,c}$ is big. Then we are done, because if $\sigma$ lies on both $B_1^{a,b}$ and $B_2^{a,c}$, then by the Gibbard-Satterthwaite theorem there is a $\hat{\sigma}$ which agrees with $\sigma$ on the last $n-2$ coordinates, and which is a manipulation point, and there can be at most $\left( k! \right)^{2}$ ranking profiles that give the same manipulation point. Let us do the computation:
\begin{multline*}
\p \left( \sigma \text{ on } B_1^{a,b}, \sigma \text{ on } B_2^{a,c} \right) \geq \p \left( \sigma \text{ on } B_1^{a,b}, \sigma \text{ on } B_2^{a,c}, \sigma \in F_1^{a,b} \cap F_2^{a,c} \right)\\
\geq \p \left( \sigma \in F_1^{a,b} \cap F_2^{a,c} \right) - \p \left( \sigma \in F_1^{a,b} \cap F_2^{a,c}, \sigma \text{ not on } B_1^{a,b} \right)- \p \left( \sigma \in F_1^{a,b} \cap F_2^{a,c}, \sigma \text{ not on } B_2^{a,c} \right).
\end{multline*}
The first term is bounded below via \eqref{eq:inv_hyp_contr_k}, while the other two terms can be bounded using \eqref{eq:Aiab}:
\[
\p \left( \sigma \in F_1^{a,b} \cap F_2^{a,c}, \sigma \text{ not on } B_1^{a,b} \right) \leq \p \left( \sigma \in F_1^{a,b}, \sigma \text{ not on } B_1^{a,b} \right) \leq \p \left( \sigma \text{ not on } B_1^{a,b} \, \middle| \, \sigma \in F_1^{a,b} \right)\leq \frac{\eps^{3}}{4 n^3 k^9},
\]
and similarly for the other term. Putting everything together gives us
\begin{equation*}
\p \left( \sigma \text{ on } B_1^{a,b}, \sigma \text{ on } B_2^{a,c} \right) \geq \frac{\eps^{3}}{2 n^3 k^9},
\end{equation*}
which by the discussion above implies \eqref{eq:k_lg_fbr_manip}.
\end{proof}

\subsubsection{Case 2}\label{sec:lg_fbr_gen_case2} 

\begin{lemma}\label{lem:lg_fbr_2}
If $d\notin \left\{a,b\right\}$, then \eqref{eq:k_lg_fbr_ab} and \eqref{eq:k_lg_fbr_cd} cannot hold simultaneously.
\end{lemma}
\begin{proof}
Suppose on the contrary that \eqref{eq:k_lg_fbr_ab} and \eqref{eq:k_lg_fbr_cd} do both hold. Then
\[
 \p \left( \sigma \in F_1^{a,b} \right) \geq \frac{\eps}{n k^3} \qquad \text{ and } \qquad \p \left( \sigma \in F_2^{c,d} \right) \geq \frac{\eps}{n k^3}
\]
as before. Since $\left\{ a, b \right\} \cap \left\{ c, d \right\} = \emptyset$, $\left\{ \sigma \in F_1^{a,b} \right\}$ and $\left\{ \sigma \in F_2^{c,d} \right\}$ are independent events, and so
\[
\p \left( \sigma \in F_1^{a,b} \cap F_2^{c,d} \right) = \p \left( \sigma \in F_1^{a,b} \right) \p \left( \sigma \in F_2^{c,d} \right) \geq  \frac{\eps^{2}}{n^2 k^6}.
\]
In the same way as before, by the definition of large fibers this implies that
\begin{equation*}
\p \left( \sigma \text{ on } B_1^{a,b}, \sigma \text{ on } B_2^{c,d} \right) \geq \frac{\eps^{2}}{2 n^2 k^6} > 0,
\end{equation*}
but it is clear that
\begin{equation*}
\p \left( \sigma \text{ on } B_1^{a,b}, \sigma \text{ on } B_2^{c,d} \right) = 0,
\end{equation*}
since $\sigma$ on $B_1^{a,b}$ and on $B_2^{c,d}$ requires $f\left( \sigma \right) \in \left\{a,b\right\} \cap \left\{c,d\right\} = \emptyset$. So we have reached a contradiction.
\end{proof}

\subsection{Big mass on small fibers}\label{sec:k_sm_fbr} 

We now deal with the case when \eqref{eq:sm_fbr} holds, i.e., when we have a big mass on the small fibers for the boundary $B_1^{a,b}$.
We formalize the ideas of the outline described in Section~\ref{sec:proof_outline_1} in a series of statements.

First, we want to formalize that the boundaries of the boundaries are big, when we are on a small fiber.
\begin{lemma}\label{lem:comparable_boundaries_k}
Fix coordinate 1 and the pair of alternatives $a,b$. Let $z^{a,b}$ be such that $B_1 \left( z^{a,b} \right)$ is a small fiber for $B_1^{a,b}$. Then, writing $B \equiv B_1 \left( z^{a,b} \right)$, we have
\[
\left| \partial_e \left( B \right) \right| \geq \frac{\eps^3}{4 n^3 k^9} \left| B \right|
\]
and
\begin{equation}\label{eq:iso_small_fib}
\p \left( \sigma \in \partial \left( B \right) \right) \geq \frac{\eps^3}{2 n^4 k^9 k!} \p \left( \sigma \in B \right),
\end{equation}
where both the edge boundary $\partial_e \left( B \right)$ and the boundary $\partial \left( B \right)$ are with respect to the induced subgraph $F\left( z^{a,b} \right)$, which is isomorphic to $K_{k!/2}^{n}$, the Cartesian product of $n$ complete graphs of size $k!/2$.
\end{lemma}
\begin{proof}
We use Corollary~\ref{cor:isoLindsey} with $\ell = k! / 2$ and the set $A$ being either $B$ or $B^c := F \left( z^{a,b} \right) \setminus B$. Suppose first that $\left| B \right| \leq \left( 1 - \frac{2}{k!} \right)\left( k! / 2 \right)^n$. Then $\left| \partial_e \left( B \right) \right| \geq \left| B \right|$. Suppose now that $\left| B \right| > \left( 1 - \frac{2}{k!} \right)\left( k! / 2 \right)^n$. Since we are in the case of a small fiber, we also know that $\left| B \right| \leq \left( 1 - \frac{\eps^3}{4 n^3 k^9} \right) \left(k! / 2 \right)^n$. Consequently, we get
\[
\left| \partial_e \left( B \right) \right| = \left| \partial_e \left( B^c \right) \right| \geq \left| B^c \right| \geq \frac{\eps^3}{4 n^3 k^9} \left| B \right|,
\]
which proves the first claim.

A ranking profile in $F \left( z^{a,b} \right)$ has $\left( k! / 2 - 1 \right) n \leq n k! / 2$ neighbors in $F \left( z^{a,b} \right)$, which then implies \eqref{eq:iso_small_fib}.
\end{proof}

\begin{corollary}\label{cor:comparable_boundaries_k}
If \eqref{eq:sm_fbr} holds, then
\[
\p\left( \sigma \in \bigcup_{z^{a,b}} \partial \left( B_1 \left( z^{a,b} \right) \right) \right) \geq \frac{\eps^4}{2 n^5 k^{12} k!}.
\]
\end{corollary}

\begin{proof}
Using the previous lemma and \eqref{eq:sm_fbr} we have
\begin{align*}
\p \left( \sigma \in \bigcup_{z^{a,b}} \partial \left( B_1 \left( z^{a,b} \right) \right) \right) &= \sum_{z^{a,b}} \p \left( \sigma \in \partial \left( B_1 \left( z^{a,b} \right) \right) \right)\geq \sum_{z^{a,b} : B_1 \left( z^{a,b} \right) \subseteq \Sm \left( B_1^{a,b} \right)} \p \left( \sigma \in \partial \left( B_1 \left( z^{a,b} \right) \right) \right)\\
&\geq \sum_{z^{a,b} : B_1 \left( z^{a,b} \right) \subseteq \Sm \left( B_1^{a,b} \right)} \frac{\eps^3}{2 n^4 k^9 k!} \p \left( \sigma \in  B_1 \left( z^{a,b} \right) \right) = \frac{\eps^3}{2 n^4 k^9 k!} \p \left( \sigma  \in \Sm \left( B_1^{a,b} \right) \right)\\
&\geq \frac{\eps^4}{2 n^5 k^{12} k!}. \qedhere
\end{align*}
\end{proof}

Next, we want to find manipulation points on the boundaries of boundaries.

\begin{lemma}\label{lem:manip_on_bdry}
Suppose the ranking profile $\sigma$ is on the boundary of a fiber for $B_1^{a,b}$, i.e.,
\[
\sigma \in \bigcup_{z^{a,b}} \partial \left( B_1 \left( z^{a,b} \right) \right).
\]
Then either $\sigma_{-1} \in D_1 \left( a, b \right)$, or there exists a manipulation point $\hat{\sigma}$ which differs from $\sigma$ in at most two coordinates, one of them being the first coordinate.
\end{lemma}

\begin{proof}
First of all, by our assumption that $\sigma$ is on the boundary of a fiber for $B_1^{a,b}$, we know that $\sigma \in B_1 \left( z^{a,b} \right)$ for some $z^{a,b}$, which means that there exists a ranking profile $\sigma' = \left( \sigma'_1, \sigma_{-1} \right)$ such that $\left( \sigma, \sigma' \right) \in B_1^{a,b}$. We may assume $a \stackrel{\sigma_1}{>} b$ and $b \stackrel{\sigma'_1}{>} a$, or else either $\sigma$ or $\sigma'$ is a manipulation point.

Now since $\sigma \in \partial \left( B_1 \left( z^{a,b} \right) \right)$ we also know that there exists a ranking profile $\pi = \left( \pi_j, \sigma_{-j} \right) \in F \left( z^{a,b} \right) \setminus B_1 \left( z^{a,b} \right)$ for some $j \in \left[k \right]$. We distinguish two cases: $j \neq 1$ and $j=1$.

\textbf{Case 1:} $\mathbf{j \neq 1}$\textbf{.} What does it mean for $\pi = \left( \pi_j, \sigma_{-j} \right)$ to be on the same fiber as $\sigma$, but for $\pi$ to not be in $B_1 \left( z^{a,b} \right)$? First of all, being on the same fiber means that $\sigma_j$ and $\pi_j$ both rank $a$ and $b$ in the same order. Now $\pi \notin B_1 \left( z^{a,b} \right)$ means that
\begin{itemize}
\item either $f\left( \pi \right) \neq a$;
\item or $f \left( \pi \right) = a$ and $f\left( \pi'_1, \pi_{-1} \right) \neq b$ for every $\pi'_1 \in S_k$.
\end{itemize}
If $f\left( \pi \right) = b$, then either $\sigma$ or $\pi$ is a manipulation point, since the order of $a$ and $b$ is the same in both $\sigma_j$ and $\pi_j$ (since $\sigma$ and $\pi$ are on the same fiber).

Suppose $f\left( \pi \right) = c \notin \left\{ a, b \right\}$. Then we can define a SCF function on two coordinates by fixing all coordinates except coordinates 1 and $j$ to agree with the respective coordinates of $\sigma$---letting coordinates 1 and $j$ vary we get a SCF function on two coordinates which takes on at least three values ($a$, $b$, and $c$), and does not only depend on one coordinate. Now applying the Gibbard-Satterthwaite theorem we get that this SCF on two coordinates has a manipulation point, which means that our original SCF $f$ has a manipulation point which agrees with $\sigma$ in all coordinates except perhaps in coordinates 1 and $j$.

So the final case is that $f \left( \pi \right) = a$ and $f\left( \pi'_1, \pi_{-1} \right) \neq b$ for every $\pi'_1 \in S_k$. In particular for $\tilde{\pi} := \left( \sigma'_1, \pi_{-1} \right) = \left( \pi_j, \sigma'_{-j} \right)$ we have $f\left( \tilde{\pi} \right) \neq b$. Now if $f\left( \tilde{\pi} \right) = a$ then either $\sigma'$ or $\tilde{\pi}$ is a manipulation point, since the order of $a$ and $b$ is the same in both $\sigma'_j = \sigma_j$ and $\pi_j$. Finally, if $f\left( \tilde{\pi} \right) = c \notin \left\{ a, b \right\}$, then we can apply the Gibbard-Satterthwaite theorem just like in the previous paragraph.

\textbf{Case 2:} $\mathbf{j = 1}$\textbf{.} We can again ask: what does it mean for $\pi = \left( \pi_1, \sigma_{-1} \right)$ to be on the same fiber as $\sigma$, but for $\pi$ to not be in $B_1 \left( z^{a,b} \right)$? First of all, being on the same fiber means that $\sigma_1$ and $\pi_1$ both rank $a$ and $b$ in the same order (namely, as discussed at the beginning, ranking $a$ above $b$, or else we have a manipulation point). Now $\pi \notin B_1 \left( z^{a,b} \right)$ means that
\begin{itemize}
\item either $f\left( \pi \right) \neq a$;
\item or $f \left( \pi \right) = a$ and $f\left( \pi'_1, \pi_{-1} \right) \neq b$ for every $\pi'_1 \in S_k$.
\end{itemize}
However, we know that $f \left( \sigma' \right) = b$ and that $\sigma'$ is of the form $\sigma' = \left( \sigma'_1, \sigma_{-1} \right) = \left( \sigma'_1, \pi_{-1} \right)$, and so the only way we can have $\pi \notin B_1 \left( z^{a,b} \right)$ is if $f\left( \pi \right) \neq a$.

If $f\left( \pi \right) = b$, then $\pi$ is a manipulation point, since $a \stackrel{\pi_1}{>} b$ and $f\left( \sigma \right) = a$.

So the remaining case is if $f\left( \pi \right) = c \notin \left\{a,b\right\}$. This means that $f_{\sigma_{-1}}$ (see Definition~\ref{def:induced_SCF}) takes on at least three values. Denote by $H \subseteq \left[ k \right]$ the range of $f_{\sigma_{-1}}$. Now either $\sigma_{-1} \in D_1^{H} \subseteq D_1 \left( a,b \right)$, or there exists a manipulation point $\hat{\sigma}$ which agrees with $\sigma$ in every coordinate except perhaps the first. 
\end{proof}

Finally, we need to deal with dictators on the first coordinate.

\begin{lemma}\label{lem:dictators}
Assume that $\Dist \left( f, \overline{\NONMANIP} \right) \geq \eps$. We have that either
\[
\p \left( \sigma_{-1}  \in D_1 \left( a, b \right) \right) \leq \frac{\eps^4}{4 n^5 k^{12} k!},
\]
or
\begin{equation}\label{eq:manip_with_dictators}
\p \left( \sigma \in M \right) \geq \frac{\eps^{5}}{4 n^7 k^{12} \left(k!\right)^4}.
\end{equation}
\end{lemma}

\begin{proof}
Suppose $\p \left( \sigma_{-1} \in D_1 \left( a,b \right) \right) \geq \frac{\eps^{4}}{4 n^5 k^{12} k!}$, which is the same as
\begin{equation}\label{eq:lot_of_dictators}
\sum_{H : \left\{a,b\right\} \subseteq H, \left| H \right| \geq 3} \p \left( \sigma_{-1} \in D_1^{H} \right) \geq  \frac{\eps^{4}}{4 n^5 k^{12} k!}.
\end{equation}
Note that for every $H \subseteq \left[k \right]$ we have
\[
\eps \leq \Dist \left( f, \overline{\NONMANIP} \right) \leq \p \left( f \left( \sigma \right) \neq \tp_H \left( \sigma_1 \right) \right) \leq 1 - \p \left( D_1^H \right), 
\]
and so 
\begin{equation}\label{eq:D1H_not_too_big}
\p \left( D_1^H \right) \leq 1 - \eps.
\end{equation}

The main idea is that \eqref{eq:D1H_not_too_big} implies that the size of the boundary of $D_1^H$ is comparable to the size of $D_1^H$, and if we are on the boundary of $D_1^H$, then there is a manipulation point nearby.

So first let us establish that the size of the boundary of $D_1^H$ is comparable to the size of $D_1^H$. This is done along the same lines as the proof of Lemma~\ref{lem:comparable_boundaries_k}.

Notice that $D_1^{H} \subseteq S_k^{n-1}$, where $S_k^{n-1}$ should be thought of as the Cartesian product of $n-1$ copies of the complete graph on $S_k$. We apply Corollary~\ref{cor:isoLindsey} with $\ell = k!$ and with $n-1$ copies, and we see that if $\eps \geq \frac{1}{k!}$, then $\left| \partial_e \left( D_1^H \right) \right| \geq \left| D_1^H \right|$. If $\eps < \frac{1}{k!}$ and $1 - \frac{1}{k!} \leq \p \left( D_1^H \right) \leq 1-\eps$ then
\[
\left| \partial_e \left( D_1^H \right) \right| = \left| \partial_e \left( \left( D_1^H \right)^c \right) \right| \geq \left| \left(D_1^H\right)^c \right| \geq \eps \left| D_1^H \right|.
\]
So in any case we have $\left| \partial_e \left( D_1^H \right) \right| \geq \eps \left| D_1^H \right|$. Since $\sigma_{-1}$ has $\left( n - 1 \right) \left(k! - 1\right) \leq n k!$ neighbors in $S_{k}^{n-1}$, we have that
\[
\p \left( \sigma_{-1} \in \partial \left( D_1^H \right) \right) \geq \frac{\eps}{n k!} \p \left( \sigma_{-1} \in D_1^H \right).
\]
Consequently, by \eqref{eq:lot_of_dictators}, we have
\begin{align*}
\p \left( \sigma_{-1} \in \bigcup_{H : \left\{a,b\right\} \subseteq H, \left|H \right| \geq 3} \partial \left( D_1^H \right) \right) &= \sum_{H : \left\{a,b\right\} \subseteq H, \left|H \right| \geq 3} \p \left( \sigma_{-1} \in \partial \left( D_1^H \right) \right)\\
&\geq \sum_{H : \left\{a,b\right\} \subseteq H, \left|H \right| \geq 3} \frac{\eps}{n k!} \p \left( \sigma_{-1} \in D_1^H \right) \geq \frac{\eps^{5}}{4 n^6 k^{12} \left(k!\right)^2}.
\end{align*}

Next, suppose $\sigma_{-1} \in \partial \left( D_1^H \right)$ for some $H$ such that $\left\{ a, b \right\} \subseteq H, \left| H \right| \geq 3$. We want to show that then there is a manipulation point ``close'' to $\sigma_{-1}$ in some sense. To be more precise: for the manipulation point $\hat{\sigma}$, $\hat{\sigma}_{-1}$ will agree with $\sigma_{-1}$ in all except maybe one coordinate.

If $\sigma_{-1} \in \partial \left( D_1^H \right)$, then there exist  $j \in \left\{2, \dots, n \right\}$ and $\sigma'_j$ such that $\sigma'_{-1} := \left( \sigma'_{j}, \sigma_{-\left\{1, j\right\}} \right) \notin D_1^H$. That is, $f_{\sigma'_{-1}} \left( \cdot \right) \not\equiv \tp_{H} \left( \cdot \right)$. There can be two ways that this can happen---the two cases are outlined below. Denote by $H' \subseteq \left[k\right]$ the range of $f_{\sigma'_{-1}}$.

\textbf{Case 1:} $\mathbf{H' = H}$\textbf{.} In this case we automatically know that there exists a manipulation point $\hat{\sigma}$ such that $\hat{\sigma}_{-1} = \sigma'_{-1}$, and so $\hat{\sigma}_{-1}$ agrees with $\sigma_{-1}$ in all coordinates except coordinate $j$.

\textbf{Case 2:} $\mathbf{H' \neq H}$\textbf{.} W.l.o.g.\ suppose $H' \setminus H \neq \emptyset$, and let $c \in H' \setminus H$. (The other case when $H \setminus H' \neq \emptyset$ works in exactly the same way.) First of all, we may assume that $f_{\sigma'_{-1}} \left( \cdot \right) \equiv \tp_{H'} \left( \cdot \right)$, because otherwise we have a manipulation point just like in Case 1.
 
We can define a SCF on two coordinates by fixing all coordinates except coordinate 1 and $j$ to agree with $\sigma_{-1}$, and varying coordinates 1 and $j$. We know that the outcome takes on at least three different values, since $\sigma_{-1} \in D_1^H$, and $\left| H \right| \geq 3$. 

Now let us show that this SCF is not a function of the first coordinate. Let $\sigma_1$ be a ranking which puts $c$ first, and then $a$. Then $f \left( \sigma_1, \sigma_{-1} \right) = a$, but $f\left( \sigma_1, \sigma'_{-1} \right) = c$, which shows that this SCF is not a function of the first coordinate (since a change in coordinate $j$ can change the outcome).

Consequently, the Gibbard-Satterthwaite theorem tells us that this SCF on two coordinates has a manipulation point, and therefore there exists a manipulation point $\hat{\sigma}$ for $f$ such that $\hat{\sigma}_{-1}$ agrees with $\sigma_{-1}$ in all coordinates except coordinate $j$.

Putting everything together yields \eqref{eq:manip_with_dictators}.
\end{proof}

\subsection{Proof of Theorem~\ref{thm:k_bdd} concluded} 

\begin{proof}[Proof of Theorem~\ref{thm:k_bdd}]
If \eqref{eq:k_lg_fbr_ab} and \eqref{eq:k_lg_fbr_cd} hold, then we are done by Lemmas~\ref{lem:lg_fbr_1} and~\ref{lem:lg_fbr_2}.

If not, then either \eqref{eq:sm_fbr} holds, or \eqref{eq:sm_fbr} holds for the boundary $B_2^{c,d}$; w.l.o.g.\ assume that \eqref{eq:sm_fbr} holds.

By Corollary~\ref{cor:comparable_boundaries_k}, we have
\[
\p\left( \sigma \in \bigcup_{z^{a,b}} \partial \left( B_1 \left( z^{a,b} \right) \right) \right) \geq \frac{\eps^4}{2 n^5 k^{12} k!}.
\]

We may assume that $\p \left( \sigma_{-1} \in D_1 \left( a, b \right) \right) \leq \frac{\eps^4}{4 n^5 k^{12} k!}$, since otherwise we are done by Lemma~\ref{lem:dictators}. Consequently, we then have
\[
\p\left( \sigma \in \bigcup_{z^{a,b}} \partial \left( B_1 \left( z^{a,b} \right) \right), \sigma_{-1} \notin D_1 \left( a, b \right) \right) \geq \frac{\eps^4}{4 n^5 k^{12} k!}.
\]
We can then finish our argument using Lemma~\ref{lem:manip_on_bdry}:
\[
\p \left( \sigma \in M \right) \geq \frac{1}{n \left( k! \right)^{2}} \p\left( \sigma \in \bigcup_{z^{a,b}} \partial \left( B_1 \left( z^{a,b} \right) \right), \sigma_{-1} \notin D_1 \left( a, b \right) \right) \geq \frac{\eps^4}{4 n^6 k^{12} \left(k!\right)^{3}}. \qedhere
\]
\end{proof}

\section{An overview of the refined proof}\label{sec:k_refined_overview} 

In order to improve on the result of Theorem~\ref{thm:k_bdd}---in particular to get rid of the factor of $\frac{1}{\left(k!\right)^4}$---we need to refine the methods used in the previous section. 


The key to the refined method is to consider the so-called \emph{refined rankings graph} instead of the general rankings graph studied in Section~\ref{sec:k_bdd}. The vertices of this graph are again ranking profiles (elements of $S_k^n$), and two vertices are connected by an edge if they differ in exactly one coordinate, and by an adjacent transposition in that coordinate. Again, the SCF $f$ naturally partitions the vertices of this graph into $k$ subsets, depending on the value of $f$ at a given vertex. Clearly a 2-manipulation point can only be on the edge boundary of such a subset in the refined rankings graph, and so it is important to study these boundaries.

One of the important steps of the proof in Section~\ref{sec:k_bdd} is creating a configuration where we fix all but two coordinates, and the SCF $f$ takes on at least three values when we vary these two coordinates---then we can define another SCF on two voters and $k$ alternatives which must have a manipulation point by the Gibbard-Satterthwaite theorem. The advantage of the refined rankings graph is that we can create a configuration where we fix all but two coordinates, and in these two coordinates we also fix all but constantly many adjacent alternatives, and the SCF takes on at least three values when we vary these constantly many adjacent alternatives in the two coordinates. Then we can define another SCF on two voters and $r$ alternatives, where $r$ is a small constant, which must have a manipulation point by the Gibbard-Satterthwaite theorem. Since $r$ is a constant, we only lose a constant factor in our estimates, not factors of $\frac{1}{k!}$.

We state the refined result in Theorem~\ref{thm:k_refined}, which we also prove in Section~\ref{sec:manip_ref}. The proof of Theorem~\ref{thm:k_refined} follows the outline of the proof of Theorem~\ref{thm:k_bdd}: we know that there are at least two refined boundaries which are big

The difficulty is dealing with the case when we are on the boundary of a small fiber in the first coordinate. Suppose $\sigma = \left( \sigma_1, \sigma_{-1} \right)$ is on such a boundary. We know that there are $k!$ ranking profiles which agree with $\sigma$ in coordinates 2 through $n$. The difficulty comes from the fact that---in order to obtain a polynomial bound in $k$---we are only allowed to look at a polynomial number (in $k$) of these ranking profiles when searching for a manipulation point. If there is an $r$-manipulation point among them for some small constant $r$, then we are done. If this is not the case then $\sigma$ is what we call a \emph{local dictator} on some subset of the alternatives in coordinate 1. We say that $\sigma$ is a local dictator on some subset $H \subseteq \left[k\right]$ of the alternatives in coordinate 1 if the alternatives in $H$ are adjacent in $\sigma_1$, and permuting the alternatives in $H$ in every possible way in the first coordinate, the outcome of the SCF $f$ is 
always the top-ranked alternative in $H$.

So instead of dealing with dictators on some subset in coordinate 1, as in Section~\ref{sec:k_bdd}, we have to deal with \emph{local dictators} on some subset in coordinate 1. This analysis involves essentially only the first coordinate, in essence proving a quantitative Gibbard-Satterthwaite theorem for one voter. This has not been studied in the literature before, and, moreover, we were not able to utilize previous quantitative Gibbard-Satterthwaite theorems to solve this problem easily. Hence we separate this argument from the rest of the proof of Theorem~\ref{thm:k_refined} and formulate a quantitative Gibbard-Satterthwaite theorem for one voter, Theorem~\ref{thm:quant_GS_1voter}, which is proven in Section~\ref{sec:1voter}. This proof forms the backbone for the proof of Theorem~\ref{thm:k_refined}, which is then proven in Section~\ref{sec:manip_ref}. 

\begin{theorem}\label{thm:quant_GS_1voter}
Suppose $f : S_k \to \left[ k \right]$ is a SCF on $n=1$ voter and $k \geq 3$ alternatives which satisfies $\Dist \left( f, \NONMANIP \right) \geq \eps$. Then
\begin{equation}\label{eq:main}
\p \left( \sigma \in M \left( f \right) \right) \geq \p \left( \sigma \in M_3 \left( f \right) \right) \geq p \left( \eps, \frac{1}{k} \right),
\end{equation}
for some polynomial $p$, where $\sigma \in S_k$ is selected uniformly. In particular, we show a lower bound of $\frac{\eps^3}{10^5 k^{16}}$.
\end{theorem}

\section{Refined rankings graph}\label{sec:k_refined_prelims} 


\subsection{Transpositions, boundaries, and influences} 

\begin{definition}[Adjacent transpositions]
Given two elements $a,b \in \left[k \right]$, the adjacent transposition $\left[a:b\right]$ between them is defined as follows. If $\sigma \in S_k$ has $a$ and $b$ adjacent, then $\left[a:b\right] \sigma$ is obtained from $\sigma$ by exchanging $a$ and $b$. Otherwise $\left[a:b\right] \sigma = \sigma$.

We let $T$ denote the set of all $k \left( k - 1 \right) / 2$ adjacent transpositions.

For $\sigma \in S_k^n$, we let $\left[a:b\right]_i \sigma$ denote the ranking profile obtained by applying $\left[a:b\right]$ on the $i^{\text{th}}$ coordinate of $\sigma$ while leaving all other coordinates unchanged.
\end{definition}

\begin{definition}[Boundaries]
For a given SCF $f$ and a given alternative $a \in \left[k \right]$, we define
\[
H^{a} \left( f \right) = \left\{ \sigma \in S_k^n : f\left( \sigma \right) = a \right\},
\]
the set of ranking profiles where the outcome of the vote is $a$. The edge boundary of this set (with respect to the underlying refined rankings graph) is denoted by $B^{a;T} \left( f \right)$ : $B^{a;T} \left( f \right) = \partial_e \left( H^a \left( f \right) \right)$. This boundary can be partitioned: we say that the edge boundary of $H^a \left( f \right)$ in the direction of the $i^{\text{th}}$ coordinate is
\[
B_i^{a;T} \left( f \right) = \left\{ \left( \sigma, \sigma' \right) \in B^{a;T} \left( f \right):  \sigma_i \neq \sigma'_i \right\}.
\]
The boundary $B^a \left( f \right)$ can be therefore written as $B^{a;T} \left( f \right)= \cup_{i=1}^{n} B_i^{a;T} \left( f \right)$. We can also define the boundary between two alternatives $a$ and $b$ in the direction of the $i^{\text{th}}$ coordinate:
\[
B_i^{a,b;T} \left( f \right) = \left\{ \left( \sigma, \sigma' \right) \in B_i^{a;T} \left( f \right) : f\left( \sigma' \right) = b \right\}.
\]
Moreover, we can define the boundary between two alternatives $a$ and $b$ in the direction of the $i^{\text{th}}$ coordinate with respect to the adjacent transposition $z \in T$:
\[
B_i^{a,b;z} \left( f \right) = \left\{ \left( \sigma, \sigma' \right) \in B_i^{a;T} \left( f \right) : \sigma' = z_i \sigma, f\left( \sigma' \right) = b \right\}.
\]
We also say that $\sigma$ is \emph{on} the boundary $B_i^{a,b;z} \left( f \right)$ if $\left( \sigma, z_i \sigma \right) \in B_i^{a,b;z} \left( f \right)$.
Clearly we have
\[
B_i^{a,b;T} \left( f \right) = \bigcup_{z \in T} B_i^{a,b;z} \left( f \right).
\]
\end{definition}

\begin{definition}[Influences]
Given $z \in T$, we define
\begin{align*}
\Inf_i^{a,b;z} \left( f \right) &= \p \left( f \left( \sigma \right) = a, f \left( \sigma^{\left( i \right)} \right) = b \right)\\
\Inf_i^{a;z} \left( f \right) &= \p \left( f \left( \sigma \right) = a, f \left( \sigma^{\left( i \right)} \right) \neq a \right)\\
\Inf_i^{a,b;T} \left( f \right) &= \sum_{z \in T} \Inf_i^{a,b;z} \left( f \right),
\end{align*}
where $\sigma$ is uniformly distributed in $S_k^n$ and $\sigma^{\left(i\right)}$ is obtained from $\sigma$ by rerandomizing the $i^{\text{th}}$ coordinate $\sigma_i$ in the following way: with probability 1/2 we keep it as $\sigma_i$, and otherwise we replace it by $z \sigma_i$.
\end{definition}
Note that for $a \neq b$,
\[
\Inf_i^{a,b;z} \left( f \right) = \frac{1}{2} \p \left( f\left( \sigma \right) = a, f\left( z_i \sigma \right) = b \right) = \frac{1}{2} \frac{\left| B_i^{a,b;z} \left( f \right) \right|}{\left( k! \right)^{n}}.
\]

Again, most of the time the specific SCF $f$ will be clear from the context, in which case we omit the dependence on $f$.

\subsection{Canonical Paths and Group
  Actions}\label{sec:canon-paths-group}

In order to derive the more refined result, we will need to consider in more
detail the properties of the permutation group $L_q$ with respect to adjacent
transpositions. Again we use canonical paths arguments. We state the
arguments in a more general setup.

\begin{definition}
Let $L$ be a graph. 
\begin{itemize}
\item Let $P_L(\ell)$ denote the set of paths of length at most $\ell$ in $L$
  and $P_L=\cup_{\ell \in \mathbb{N}} P_L(l)$ the set of paths of finite length.
\item
  Let $L_1,L_2 \subseteq L$.
  A {\em canonical path map} on $L$  from $L_1$ to $L_2$ of {\em length} $\ell$ is a map $\Gamma \colon L_1 \times L_2 \to
  P_L(\ell)$ which satisfies that $\Gamma(x,y)$ begins at $x$ and ends at $y$
  for all $(x,y) \in L_1 \times L_2$.
\item
  Given a canonical path map $\Gamma \colon L_1 \times L_2 \to
  P_L(\ell)$
  and $0 \leq i \le \ell$ we define the inverse image mapping of the
  $i$'th vertex, $\Gamma_i^{-1}:L \rightarrow 2^{L_1 \times L_2}$ as
  \[ \Gamma_i^{-1}(z) = \{ (x,y)
  \mid \length(\Gamma(x,y)) \ge i, \Gamma(x,y)_i = z\}. \]
  Further, we let
  \begin{equation*}
    \Gamma^{-1}(z) = \cup_{i=0}^{\ell} \Gamma_i^{-1}(z)
  \end{equation*}
\item Given a
  group $H$ acting on $L$ we say that a canonical path map
  $\Gamma \colon L_1 \times L_2 \to P_L(\ell)$ is $H$-{\em invariant} if
  $H L_1 = L_1$ and $H L_2 = L_2$ and \[ \Gamma(h x, h y) = h \Gamma(x, y), \]
  for all $h \in H$ and all $(x,y) \in L_1 \times L_2$.
\end{itemize}
\end{definition}

We will use the following proposition. Recall that a group $H$ acting on $L$ is
called {\em fixed-point-free} if for all $x \in L$ and all $h \in H$ different than the
identity it holds that $h x \neq x$.

\begin{proposition} \label{prop:canon_sym}
Let $H$ be a fixed-point-free group acting on $L$ and let $\Gamma \colon L_1 \times L_2 \to P_L(\ell)$ be a canonical path map that is $H$-invariant.
Then for all $z\in L$ and $0\le i \le l$ it holds that
\begin{equation} \label{eq:canon_sym_i}
  |\Gamma^{-1}_i(z)| \le \frac{|L_1| |L_2|} {|H|}
\end{equation}
and
\begin{equation} \label{eq:canon_sym}
  |\Gamma^{-1}(z)| \leq \frac{(\ell+1) |L_1| |L_2|} {|H|}
\end{equation}
\end{proposition}

\begin{proof}
  Note that for all $i$, \[
  |L_1 \times L_2| \ge \sum_{w} |\Gamma^{-1}_i(w)| \ge \sum_{h \in
    H} |\Gamma^{-1}_i(h z)| = |H| |\Gamma^{-1}_i(z)|, \] where the
  first inequality follows since the value of the $i$'th vertex partitions the set of paths of length at least $i$,
  the second inequality since $H$ is fixed-point-free, and the final equality from the path being $H$-invariant.
  We thus obtain: \[
  |\Gamma^{-1}(z)| \leq \sum_{i=0}^{\ell} |\Gamma_i^{-1}(z)| \leq
  \frac{(\ell+1) |L_1| |L_2|} {|H|}, \] as needed.
\end{proof}

Two applications of the result above will be given for adjacent transpositions.
\begin{definition}
 Given two elements $a,b \in [k]$ the {\em adjacent transposition} $\adj a b$
  between them is defined as follows. If $\sigma \in S_k$ has $a$ and $b$ adjacent,
  then $\adj{a}{b} \sigma$ is obtained from $\sigma$ be exchanging $a$ and $b$. Otherwise,
  $\adj{a}{b} \sigma = \sigma$.

  We let $T$ denote the set of all $q(q-1)/2$ adjacent transpositions. Given $\mu
  \in T$, we define \begin{eqnarray}
    \Inf_i^{a,b ; \mu}(f) &=& \P(f(X) = a, f(X^{(i)})=b)
    \\
    \Inf_i^{a ; \mu}(f) &=& \P(f(X) = a, f(X^{(i)}) \neq a)
    \\
  \Inf_i^{a,b ; T}(f) &=& \sum_{\mu \in T} \Inf_i^{a,b;\mu}(f)
\end{eqnarray}
where $X^{(i)}$ is obtained from $X$ by re-randomizing the $i$:th coordinate $X_i$ in the
following way: with probability $1/2$ we keep it as $X_i$ and otherwise we replace
it by $ \mu X_i$.

Finally for $\sigma \in S_k^n$ we will let $\adji{\sigma}{b}{i} \sigma$ denote the element
obtained by applying $\adj{a}{b}$ on the $i$:th coordinate of $\sigma$ while leaving
all other coordinates unchanged. \end{definition}

\begin{proposition} \label{prop:canon1}
  There exists a canonical path map
  $\Gamma \colon S_k \times S_k \to P_{S_k}(\ell)$
  of length at most $\ell = k(k-1)/2 < k^2/2$, all of whose edges are
  adjacent transpositions such that for all $\mu$ it holds that:
  \begin{equation}
    \label{eq:canon1}
    |\Gamma^{-1}(\mu)| \leq \frac{k^2 k!}{2}
  \end{equation}
\end{proposition}

\begin{proof}
  Given $\sigma,\pi \in S_k$ consider the following canonical path starting at $\sigma$ and
  ending at $\pi$. Take the element $\pi(1)$ ranked at the top for $\pi$ and bubble it
  to the top by performing adjacent transpositions. Then take the element $\pi(2)$
  ranked second for $\pi$ and bubble it to the second position etc. Clearly the
 length of the path is at most $k(k-1)/2$.
  Let $H=\{\sigma \mapsto \tau \sigma \mid \tau \in S_k\}$ be the group of compositions with
  all possible permutations of the candidates.
  Since $H$ is a fixed-point-free group acting on $S_k$ and the described canonical path map is $H$-invariant
  the result follows from Proposition~\ref{prop:canon_sym}. \end{proof}

\begin{corollary}~\label{cor:sumInfVarBound2}
 For any $f \colon S_k^n \rightarrow [k]$, $a \in [k]$ and $i \in [n]$ it holds that
  \begin{equation}
   \sum_{\mu \in T} \Inf_i^{a; \mu}(f) \geq \frac{1}{k^2} \Inf_i^a(f),
  \end{equation}
  where $T$ is the set of all adjacent transpositions.
\end{corollary}

\begin{proof}
  This is a standard canonical path argument. Since both sides of the desired
  inequality involve averaging over all coordinates but the $i$'th coordinate,
  it follows that it suffices
  to prove the claim in the case where $i=n=1$.
  Let $B = \{(u,v) \in S_k \times S_k \mid f(u) = a \neq f(v), \exists \mu \in T: v = \mu u \}$ and
  note that
  \begin{equation}
    \label{eq:sumInfVarBound2a}
    \sum_{\mu \in T} \Inf_1^{a; \mu}(f) = \frac{|B|}{2k!},
  \end{equation}
  Consider the canonical path map $\Gamma$ constructed in Proposition~\ref{prop:canon1}.
  Note that each canonical path
  between an element in $A:=\{\sigma \in L_q \mid f(\sigma)=a\}$ and an element in $A^c$ must pass via one of the edges in $B$.
  Define $h:A\times A^C \to B$ by letting $h(\sigma,\pi)$ be the first edge in $B$ which $\Gamma(\sigma,\pi)$ passes through.
  Then by \eqref{eq:canon1}, for any $(u,v) \in B$,
  \begin{equation}
    |h^{-1}((u,v))| \le |\Gamma^{-1}(u)| \le \frac{k^2 k!}{2}
  \end{equation}
  Thus
  \begin{equation}
    \label{eq:sumInfVarBound2c}
    |B| \geq \frac{|A| |A^c|}{k^2 k! / 2}
  \end{equation}
  Combining \eqref{eq:sumInfVarBound2a} and \eqref{eq:sumInfVarBound2c} we obtain:
  \begin{equation*}
    \sum_{\mu \in T} \Inf_1^{a; \mu}(f) \geq \frac{1}{2k!}
    \frac{|A| |A^c|}{k^2 k! / 2} = \frac{1}{k^2} \frac{|A|}{k!}
    \frac{|A^c|}{k!} =
    \frac{1}{k^2} \Inf_1^a(f)
  \end{equation*}
\end{proof}

A second application of Proposition~\ref{prop:canon1} is the following.

\begin{proposition} \label{prop:canon2}
  Fix two elements $a, b \in [k]$ and let
  $B \subseteq S_k$ denote the set of all permutations where $a$ is ranked above
  $b$. Then there exists a canonical path map $\Gamma:B\times B \to P_{B}(k^2)$
  consisting of adjacent transpositions such that
  all permutations along the path satisfy that $a$ is ranked above $b$. Moreover
  for all $\mu$ it holds that:
  \[
  |\Gamma^{-1}(\mu)|
  \leq
  k^4 k!
  \]
\end{proposition}

\begin{proof}
$\Gamma(\sigma,\pi)$ is defined as follows.
We look at all elements different than $a,b$, starting with the top one of $\pi$, and bubble each of them upwards to its position in $\pi$ ignoring $a,b$.
After we have done so, we have all elements but $a,b$ ordered as in $\pi$, followed by $a$, followed by $b$. We now bubble $a$ to its location in $y\pi$ and then bubble $b$.
Note that the length of the path so defined is at most
\[
\frac{k(k-1)}{2} + 2 (k-1) = \frac{(k+4)(k-1)}{2} < k^2
\]
The proof now follows from Proposition~\ref{prop:canon_sym} by considering
the group $H$ which acts by permuting arbitrary all elements but those labeled by $a$ and $b$:
\begin{equation*}
  |\Gamma^{-1}(\mu)|
  \le \frac{k^2 |B|^2}{|H|}
  = \frac{k^2 (k!/2)^2}{(k-2)!}
  \le k^4 k!
\end{equation*}
\end{proof}

\section{Refined Boundaries}\label{sec:refined-boundaries}

Similarly to the previous construction we now define
the $i$:th $a$-$b$ boundary with respect to an adjacent swap $z \in T$ as
\[
B_i^{a,b ; z} (f) = \{(x,y) \mid f(x)=a, f(y)=b, x_i = z y_i, \forall j \neq i:x_j=y_j\},
\]
and the boundary with respect to arbitrary adjacent swaps on the $i$:th coordinate as
\[
B_i^{a,b ; T} (f) = \bigcup_{z \in T} B_i^{a,b;z} (f)
\]
Note that for $a \neq b$,
\begin{equation}\label{eq:inf2Boundary}
  \Inf_i^{a,b;z} (f)
  = \frac{1}{2} \P(f(X)=a,f(zX)=b)
  = \frac{1}{2} \frac{|B_i^{a,b;z}(f)|}{(k!)^n}
\end{equation}

\subsection{Manipulation points on refined boundaries} 

The following three lemmas from Isaksson, Kindler and Mossel~\cite{IsKiMo:10,IsKiMo:12} identify manipulation points on (or close to) these refined boundaries.

\begin{lemma}
  \label{lem:nonManipBoundary} 
  Fix $f : S_k^n \to [k]$,
  distinct $a,b \in [k]$
  and $(\sigma, \pi) \in B_i^{a,b ; T}$.
  Then either $\sigma_i=\adj{a}{b} \pi_i$,
  or
  one of $\sigma$ and $\pi$ is a $2$-manipulation point for $f$.
\end{lemma}

\begin{proof}
  Suppose $\sigma_i=\adj{c}{d} \pi_i$ where $\{c,d\} \neq \{a,b\}$.
  Then an adjacent transposition of $c$ and $d$ will not change the order of $a$
  and $b$. Hence
  $
  b \stackrel{\sigma_i}{>} a
  \text{ iff }
  b \stackrel{\pi_i}{>} a
  $.
  But then either
  \begin{enumerate}
  \item $f(\pi) =b \stackrel{x_i}{>} a=f(\sigma)$ and $\sigma$ is a 2-manipulation point
  or
  \item $f(\sigma)=a \stackrel{y_i}{>} b=f(\pi)$ and $\pi$ is a 2-manipulation point.
  \end{enumerate} 
\end{proof}

\begin{lemma}
  \label{lem:nonManipTriple} 
  Fix $f : S_k^n \to [k]$
  and points
  $\sigma, \pi, \mu \in S_k^n$ such that
  $(\sigma, \pi) \in B_i^{a,b; T}$, 
  $(\mu, \pi) \in B_j^{c,b; T}$
  where $a,b,c$ are distinct and $i \neq j$.
  Then there exists a $3$-manipulation point $\nu \in S_k^n$ for $f$
  such that $\nu_{\ell}=\pi_{\ell}$ for $\ell \notin \{i,j\}$
  and $\nu_i$ is equal to $\sigma_i$ or $\pi_i$ except that the position of $c$ may be
  shifted arbitrarily
  and $\nu_j$ is equal to $\mu_j$ or $\pi_j$ except that the position of $a$ may be
  shifted arbitrarily.
\end{lemma}

  \begin{proof}
  By Lemma~\ref{lem:nonManipBoundary} we must have
  $\sigma_i = \adj{a}{b} \pi_i$
  and
  $\mu_j = \adj{c}{b} \pi_j$,
  or $\sigma$, $\pi$ or $\mu$ is a 2-manipulation point in which case we are done.

  Now create a new triple $(\sigma',\pi',\mu')$ by starting from $(\sigma,\pi,\mu)$ and
  simultaneously in the $i$:th coordinate of $\sigma$, $\pi$ and $\mu$, 
  bubbling $c$ towards the pair $ab$ until it becomes adjacent to the pair.
  Since $c$ is never swapped with $a$ or $b$ during this process
  Lemma~\ref{lem:nonManipBoundary} implies that
  for any intermediate triple $(\wt{\sigma}, \wt{\pi}, \wt{\mu})$ we have
  $f(\wt{\sigma})=a$, $f(\wt{\pi})=b$ and $f(\wt{\mu}) \notin \{a,b\}$,
  or one of $\wt \sigma$, $\wt \pi$ and $\wt \mu$ is a 2-manipulation point.
  But since we also have $\wt{\mu} = \adj{c}{b}_j \wt{\pi}$, we must actually have
  $f(\wt{\mu}) = c$, or either $\wt \pi$ or $\wt \mu$ is a 2-manipulation point.

  Similarly bubbling $a$ towards the pair $bc$ in coordinate $j$
  starting from $(\sigma',\pi',\mu')$
  gives us $\sigma'',\pi'',\mu''$ all having $a,b,c$ adjacent in
  coordinates $i$ and $j$ such that
  $(\sigma'',\pi'') \in B_i^{a,b; [a:b]}$ and
  $(\mu',\pi'') \in B_j^{c,b; [c:b]}$.
  Note that $\sigma'', \pi'', \mu''$ are equal except for a reordering of the blocks containing $a,b,c$ in coordinates $i$ and $j$.

  Now arbitrary adjacent swapping of $a,b,c$ in these coordinates of
  $\sigma'',\pi''$ and $\mu''$ will keep the value of $f$ in $\{a,b,c\}$,
  or give rise to a 2-manipulation point by Lemma~\ref{lem:nonManipBoundary}.
  Thus we can define a social choice function with 2 voters and 3 candidates
  $f':S_{\{a,b,c\}}^2 \rightarrow \{a,b,c\}$ by letting $f'(\nu)=f(g(\nu))$, where
  $g(\nu) \in S_q^n$ is obtained from $\sigma''$ by simply reordering the two blocks of elements $a,b,c$ in coordinates $i$ and $j$ to match $\nu_1$ and $\nu_2$, respectively.
  Since $f'$ takes three values and is not a dictator,
  Gibbard-Satterthwaite (Theorem~\ref{thm:GS}) implies that
  $f'$ has a manipulation point and hence
  $f$ has a 3-manipulation point satisfying our requirements.
\end{proof}

\subsection{Large refined boundaries} 

 The following lemma, again from Isaksson, Kindler and Mossel~\cite{IsKiMo:10,IsKiMo:12}, shows that there are large refined boundaries (or else we have a lot of 2-manipulation points automatically).
\begin{lemma} \label{lem:boundaries2} 
  Fix $k \ge 3$ and
  $f : S_k^n \to \left[k\right]$
  satisfying $\Dist(f, \overline{\NONMANIP}) \ge \eps$.
  Let $\sigma$ be uniformly selected from $S_k^n$.
  Then either
  \begin{equation}
    \label{eq:neutralPairs3ManipProb}
    \p\left( \sigma \in M_2 \left( f \right) \right) \ge \frac{4\eps}{n k^7},
  \end{equation}
  or there exist distinct $i,j \in [n]$
  and $\{a,b\},\{c,d\} \subseteq [k]$ such that $c \notin \{a,b\}$ and
  \begin{equation}\label{eq:k_inf_ref}
   \Inf_i^{a,b;[a:b]} (f) \ge \frac{2\eps}{n k^7} \quad
    \text{ and } \quad
   \Inf_j^{c,d;[c:d]} (f) \ge \frac{2\eps}{n k^7}.
  \end{equation}
\end{lemma}

\begin{proof}
  First, suppose that $\Inf_i^{a,b;\nu} \ge \frac{2\eps}{nk^7}$
  for some $i$, $a \neq b$ and $\nu \neq [a:b]$.
  Then by Lemma~\ref{lem:nonManipBoundary}
  for any point $(\sigma,\sigma') \in B_i^{a,b;\nu}(f)$
  at least one of $\sigma$ or $\sigma'=\nu \sigma$ is a $2$-manipulation point.
  Let $\wt{M}$ be the set of all such 2-manipulation points.
  Then
  \begin{equation}
   |\wt{M}| \ge |B_i^{a,b;\nu}(f)|
    =
    2 (k!)^n \Inf_i^{a,b;\nu} (f)
    \ge
    \frac{4\eps}{n k^7} (k!)^n
  \end{equation}
  Dividing with $(k!)^n$ gives \eqref{eq:neutralPairs3ManipProb}.
  Thus, for the remainder of the proof we may assume that
  \begin{equation}
    \label{eq:boundariesAssump}
    \Inf_i^{a,b;\nu} < \frac{2\eps}{nk^7} \quad,\quad
    \forall i\in [n], \{a,b\} \subseteq [k], \nu \neq [a:b]
  \end{equation}
  Now, for $a \neq b$ let
  $
  A^{a,b} =
  \left\{
    i\in [n] \mid \Inf_i^{a,b;[a:b]} \ge \frac{2\eps}{n k^7}
  \right\}
  $.

  We first claim that for all $\{a,b\}$ there exists $\{c,d\}$ such that $\{c,d\} \neq \{a,b\}$ and $A^{c,d} \neq \emptyset$.
  Note that $f$ being $\eps$-far from taking two values asserts that we can find
  a $c \notin \{a,b\}$
  such that $1-\frac{\eps}{k} \ge \P(f(X) = c) \ge \frac{\eps}{q-2} \ge \frac{\eps}{k}$.
  But then, by Corollary~\ref{cor:sumInfVarBound2} and Proposition~\ref{prop:sumInfVarBound},
  \begin{equation*}
    \sum_{\mu \in T} \sum_{d \neq c} \sum_{i=1}^n \Inf_i^{c,d;\mu} (f)
    =
    \sum_{ \mu \in T} \sum_{i=1}^n \Inf_i^{c;\mu} (f)
    \ge
    \frac{1}{k^2} \Var [ 1_{\{f(X) = c\}} ]
    \ge
    \frac{\eps (k-1)}{k^4}
  \end{equation*}
  hence there must exist some $\mu \in T$, $d \neq c$ and $i \in [n]$ such that
  $\Inf_i^{c,d;\mu} \ge \frac{\eps}{n k^6}$.
  But by \eqref{eq:boundariesAssump} we must have $w=[c:d]$, hence
  $A^{c,d} \neq \emptyset$.

  We next claim that
  \begin{equation}
    \label{eq:unionContains2}
    |\cup_{a,b} A^{a,b}| \ge 2
  \end{equation}
  To see this, assume the contrary, i.e.
  $\cup_{a,b} A^{a,b} \subseteq \{i\}$ for some $i \in [n]$.
  Then, by Corollary~\ref{cor:sumInfVarBound2}, for all $j \neq i$ it holds that
  \begin{equation}
    \label{eq:nonInflInflBound2}
    \Inf_j (f)
    \le k^2 \sum_{\nu \in T} \sum_a \Inf_j^{a;\nu}(f)
    = k^2 \sum_{\nu \in T, a, b>a} \Inf_j^{a,b;\nu}(f)
    \le \frac{k^6}{2} \frac{2\eps}{nk^7}
    = \frac{\eps}{nk}
  \end{equation}
  For $\bar{\sigma} \in L_q$,
  let $f_{\bar{\sigma(x)}} = f(\sigma_1, \ldots, \sigma_{i-1}, \bar{\sigma}, \sigma_{i+1}, \ldots, \sigma_n)$
  and note that for $j \neq i$,
  \begin{equation}
    \Inf_j(f) = \frac{1}{k!} \sum_{\bar{\sigma} \in S_k} \Inf_j(f_{\bar{\sigma}})
  \end{equation}
  while $\Inf_i(f_{\bar{\sigma}}) = 0$.
  Hence, by \eqref{eq:nonInflInflBound2}, we have
  \begin{equation*}
    \eps
    \ge
    k \sum_{j \neq i} \Inf_j (f)
    =
    \frac{k}{k!} \sum_{j=1}^n \sum_{\bar{\sigma} } \Inf_j(f_{\bar{\sigma}})
    \ge
    \frac{2}{k!} \sum_{\bar{\sigma}} \Dist(f_{\bar{\sigma}}, \CONST)
    =
    2 \Dist(f,\DICT_i)
  \end{equation*}
  where the second inequality follows from Lemma~\ref{lem:constDist}
  and Proposition~\ref{prop:sumInfVarBound}.
  But this means that $f$ is $\eps/2$-close to a dictator, contradicting
  the assumption that $\Dist(f,\NONMANIP) \ge \eps$.

  Hence \eqref{eq:unionContains2} holds.
  Therefore we can either find $i \neq j$ and
  $\{a,b\} \neq \{c,d\}$ such that $i \in A^{a,b}$ and $j \in A^{c,d}$
  which proves the theorem,
  or we must have $|A^{a,b}| \ge 2$ for some $\{a,b\}$ while
  $A^{c,d} = \emptyset$ for any $\{c,d\} \neq \{a,b\}$. However, this
  contradicts the first claim in the proof. The result follows.
\end{proof}

\subsection{Fibers} 

We again use fibers $F\left( z^{a,b} \right)$ as defined in Definition~\ref{def:fibers}. However, we need more than this. We note that the following definitions only apply in Section~\ref{sec:manip_ref}, i.e., when we have at least two voters; in Section~\ref{sec:1voter}, when we have only one voter, things are simpler.

Given the result of Lemma~\ref{lem:boundaries2}, our primary interest is in the boundary $B_i^{a,b;\left[a:b\right]}$. For ranking profiles on this boundary, we know that the alternatives $a$ and $b$ are adjacent in coordinate $i$---so we know more than just the preference between $a$ and $b$ in coordinate $i$. Consequently we would like to divide the set of ranking profiles with $a$ and $b$ adjacent in coordinate $i$ according to the preferences between $a$ and $b$ in all coordinates except coordinate $i$. The following definitions make this precise.

As done in Section~\ref{sec:dict_misc_def} for ranking profiles, we can write $x_{-i}^{a,b} \equiv x_{-i}^{a,b} \left( \sigma \right)$ for the vector of preferences between $a$ and $b$ for all coordinates except coordinate $i$, i.e., the whole vector of preferences between $a$ and $b$ is $x^{a,b} \left( \sigma \right) = \left( x_{i}^{a,b} \left( \sigma \right), x_{-i}^{a,b} \left( \sigma \right) \right)$.

We can define $F\left( z_{-i}^{a,b} \right)$ analogously to $F\left( z^{a,b} \right)$:
\[
F\left( z_{-i}^{a,b} \right) := \left\{ \sigma : x_{-i}^{a,b} \left( \sigma \right) = z_{-i}^{a,b} \right\}.
\]
We also define the subset of $F\left( z_{-i}^{a,b} \right)$ where $a$ and $b$ are adjacent in coordinate $i$, with $a$ above $b$:
\[
\bar{F}\left( z_{-i}^{a,b} \right) := \left\{ \sigma \in F\left( z_{-i}^{a,b} \right) : a \text{ and } b \text{ are adjacent in coordinate } i, \text{ with } a \text{ above } b \right\}.
\]

Given a SCF $f$, for any pair of alternatives $a, b \in \left[k \right]$ and coordinate $i \in \left[n\right]$, we can also partition the boundary $B_i^{a,b} \left( f \right)$ according to its fibers. There are multiple, slightly different ways of doing this, but for our purposes the following definition is most useful.

Define
\[
B_i \left( z_{-i}^{a,b} \right) := \left\{ \sigma \in \bar{F}\left( z_{-i}^{a,b} \right) : f\left( \sigma \right) = a, f\left( \left[a:b\right]_{i} \sigma \right) = b \right\},
\]
where we omit the dependence of $B_i \left( z_{-i}^{a,b} \right)$ on $f$. We call sets of the form $B_i \left( z_{-i}^{a,b} \right) \subseteq \bar{F} \left( z_{-i}^{a,b} \right)$ \emph{fibers for the boundary $B_i^{a,b;\left[a:b\right]}$}.

We now distinguish between small and large fibers for the boundary $B_i^{a,b;\left[a:b\right]}$.
\begin{definition}[Small and large fibers]
We say that the fiber $B_i \left( z_{-i}^{a,b} \right) \subseteq \bar{F} \left( z_{-i}^{a,b} \right)$ is \emph{large} if
\[
\p \left( \sigma \in B_i \left( z_{-i}^{a,b} \right) \, \middle| \, \sigma \in \bar{F} \left( z_{-i}^{a,b} \right) \right) \geq 1 - \gamma,
\]
where $\gamma = \frac{\eps^3}{10^3 n^3 k^{24}}$, and \emph{small} otherwise.

As before, we denote by $\Lg \left( B_i^{a,b;\left[a:b\right]} \right)$ the union of large fibers for the boundary $B_i^{a,b;\left[a:b\right]}$, i.e.,
\[
\Lg \left( B_i^{a,b;\left[a:b\right]} \right) := \bigcup_{B_i \left( z_{-i}^{a,b} \right) \text{ is a large fiber}} B_i \left( z_{-i}^{a,b} \right),
\]
and similarly, we denote by $\Sm \left( B_i^{a,b;\left[a:b\right]} \right)$ the union of small fibers.
\end{definition}
As in Definition~\ref{def:lg_fibers}, we remark that what is important is that $\gamma$ is a polynomial of $\frac{1}{n}$, $\frac{1}{k}$ and $\eps$---the specific polynomial in this definition is the end result of the computation in the proof.

The following definition is used in Section~\ref{sec:lg_fbr_ref} in dealing with the large fiber case in the refined setting.
\begin{definition}
For a coordinate $i$ and a pair of alternatives $a$ and $b$, define $F_i^{a,b}$ to be the set of ranking profiles $\sigma$ such that $x^{a,b} \left( \sigma \right)$ satisfies
\[
\p \left( f \left( \tilde{\sigma} \right) = \tp_{\left\{a,b\right\}} \left( \tilde{\sigma}_i \right) \, \middle| \, \tilde{\sigma} \in F \left( x_{-i}^{a,b} \left( \sigma \right) \right) \right) \geq 1 - 2k \gamma.
\]
\end{definition}
Clearly $F_i^{a,b}$ is the union of fibers of the form $F \left( z^{a,b} \right)$, and also $F \left( \left(1, x_{-i}^{a,b} \right) \right) \subseteq F_i^{a,b}$ if and only if $F \left( \left(-1, x_{-i}^{a,b} \right) \right) \subseteq F_i^{a,b}$.

\subsection{Boundaries of boundaries}\label{sec:bdry_of_bdry} 

In the refined graph setting, just like in the general rankings graph setting, we also look at boundaries of boundaries.

For a given vector $z_{-i}^{a,b}$ of preferences between $a$ and $b$, we can think of $\bar{F} \left( z_{-i}^{a,b} \right)$ as a subgraph of the original refined rankings graph $S_k^n$, i.e., two ranking profiles in $\bar{F} \left( z_{-i}^{a,b} \right)$ are adjacent if they differ by one adjacent transposition in exactly one coordinate. Since both of the ranking profiles are in $\bar{F} \left( z_{-i}^{a,b} \right)$, this adjacent transposition keeps the order of $a$ and $b$ in all coordinates, and moreover it keeps $a$ and $b$ adjacent in coordinate $i$.

We choose to slightly modify this graph: the vertex set is still $\bar{F} \left( z_{-i}^{a,b} \right)$, but we modify the edge set by adding new edges. Suppose $\sigma \in \bar{F} \left( z_{-i}^{a,b} \right)$ and
\[
\sigma_i = 
\begin{pmatrix}
\vdots\\
c\\
a\\
b\\
d\\
\vdots
\end{pmatrix};
\qquad\quad
\sigma'_i =
\begin{pmatrix}
\vdots\\
a\\
b\\
c\\
d\\
\vdots
\end{pmatrix};
\qquad\quad
\sigma''_i =
\begin{pmatrix}
\vdots\\
c\\
d\\
a\\
b\\
\vdots
\end{pmatrix}.
\]
Define in this way $\sigma' = \left( \sigma'_i, \sigma_{-i} \right)$ and $\sigma'' = \left( \sigma''_i, \sigma_{-i} \right)$, and add $\left( \sigma, \sigma' \right)$ and $\left( \sigma, \sigma'' \right)$ to the edge set. So basically, we consider the block of $a$ and $b$ in coordinate $i$ as a single element, and connect two ranking profiles in $\bar{F} \left( z_{-i}^{a,b} \right)$ if they differ in an adjacent transposition in a single coordinate, allowing this transposition to move the block of $a$ and $b$ in coordinate $i$. We call this graph $G\left( z_{-i}^{a,b} \right) = \left( \bar{F} \left( z_{-i}^{a,b} \right), E \left( z_{-i}^{a,b} \right) \right)$, where $E \left( z_{-i}^{a,b} \right)$ is the edge set.

When we write $\partial_e \left( B_i \left( z_{-i}^{a,b} \right) \right)$, we mean the edge boundary of $B_i \left( z_{-i}^{a,b} \right)$ in the graph $G \left( z_{-i}^{a,b} \right)$, and similarly when we write $\partial \left( B_i \left( z_{-i}^{a,b} \right) \right)$, we mean the vertex boundary of $B_i \left( z_{-i}^{a,b} \right)$ in the graph $G \left( z_{-i}^{a,b} \right)$.

%

\subsection{Local dictators, conditioning and miscellaneous definitions} 

In the general rankings graph setting we defined a dictator on a subset of the alternatives, but in the refined rankings graph setting we need to define so-called \emph{local dictators}.

\begin{definition}[Local dictators]
For a coordinate $i$ and a subset of alternatives $H \subseteq \left[k\right]$, define $\LD_i^H$ to be the set of ranking profiles $\sigma$ such that the alternatives in $H$ form an adjacent block in $\sigma_i$, and permuting them among themselves in any order, the outcome of the SCF $f$ is always the top ranked alternative among those in $H$.
If $\sigma \in \LD_i^{H}$, then we call $\sigma$ a local dictator on $H$ in coordinate $i$.

Also, for a pair of alternatives $a$ and $b$, define
\[
\LD_i \left( a, b\right) := \bigcup_{c \notin \left\{ a, b \right\}} \LD_i^{\left\{a,b,c\right\}},
\]
the set of local dictators on three alternatives, two of which are $a$ and $b$, in coordinate $i$.
\end{definition}

In dealing with local dictators, we will condition on the top of a particular coordinate being fixed. We therefore introduce the following notation.
\begin{definition}[Conditioning]\label{def:cond}
For any coordinate $i \in \left[n\right]$ and any vector $\mathbf{v}$ of alternatives we define
\[
\p_i^{\mathbf{v}} \left( \cdot \right) := \p \left( \cdot \, \middle| \, \left( \sigma_i \left( 1 \right), \dots, \sigma_i \left(\left|\mathbf{v} \right| \right) \right) = \mathbf{v} \right),
\]
where $\left| \mathbf{v} \right|$ denotes the length of the vector $\mathbf{v}$. E.g., $\p_1^{\left(a \right)} \left( \cdot \right) = \p \left( \cdot \, \middle| \, \sigma_1 \left( 1 \right) = a \right)$ and
\[
 \p_1^{\left( a, b, c \right)} = \p \left( \cdot \, \middle| \, \left( \sigma_1 \left( 1 \right), \sigma_1 \left( 2 \right), \sigma_1 \left( 3 \right) \right) = \left( a, b, c \right) \right).
\]
\end{definition}

We use the following notation in the proof of Theorem~\ref{thm:TRUENONMANIP}:
\begin{theorem}\label{thm:TRUENONMANIP}
Suppose $f$ is a SCF on $n$ voters and $k \geq 3$ alternatives for which $\Dist \left( f, \overline{\NONMANIP} \right) \leq \alpha$. Then
either
\begin{equation}\label{eq:true_NONMANIP}
\Dist \left( f, \NONMANIP \right) < 100 n^4 k^8 \alpha^{1/3}
\end{equation}
or
\begin{equation}\label{eq:many_manip}
\p \left( \sigma \in M \left( f \right) \right) \geq \p \left( \sigma \in M_3 \left( f \right) \right) \geq \alpha.
\end{equation}
\end{theorem}

\begin{definition}[Majority function]\label{def:maj}
For a function $f$ whose domain $X$ is finite and whose range is the set $\left\{a,b\right\}$, define $\Maj \left( f \right)$ by
\begin{equation*}
\Maj \left( f \right) =
\begin{cases}
a & \text{ if } \quad \# \left\{ x \in X : f\left( x \right) = a \right\} \geq \# \left\{ x \in X : f\left( x \right) = b \right\},\\
b & \text{ if } \quad \# \left\{ x \in X : f\left( x \right) = a \right\} < \# \left\{ x \in X : f\left( x \right) = b \right\}.
\end{cases}
\end{equation*}
\end{definition}

\section{Proof for one voter}\label{sec:1voter} 

In this section we prove our quantitative Gibbard-Satterthwaite theorem for one voter, Theorem~\ref{thm:quant_GS_1voter}. As mentioned before, we present this proof before proving Theorem~\ref{thm:k_refined}, because the proof of Theorem~\ref{thm:k_refined} follows the lines of this proof, with slight modifications needed to deal with having $n > 1$ coordinates.

For the remainder of this section, let us fix the number of voters to be 1, the number of alternatives $k\geq 3$, and the SCF $f$, which satisfies $\Dist \left( f, \NONMANIP \right) \geq \eps$. Accordingly, we typically omit the dependence of various sets (e.g., boundaries between two alternatives) on $f$. 

An additional notational remark: since our SCF is on one voter only, we omit the subscripts that denote the coordinate we are on. E.g., we write simply $\Inf^{a,b}$ instead of $\Inf_1^{a,b}$, etc.\

We present the proof in several steps.

\subsection{Large boundary between two alternatives} 

The first thing we have to establish is a large boundary between two alternatives. This can be done just like in Lemma~\ref{lem:boundaries2}, except there are two small differences. On the one hand, the assumption of the lemma, namely that $\Dist \left( f, \NONMANIP \right) \geq \eps$, is weaker than that of the original lemma. On the other hand, here we only need one big boundary, unlike in Lemma~\ref{lem:boundaries2}, where 
there are two big boundaries in two different coordinates. The following lemma formulates what we need.

\begin{lemma}\label{lem:big_ab_bdry}
Recall that $f$ is a SCF on 1 voter and $k \geq 3$ alternatives which satisfies $\Dist \left( f, \NONMANIP \right) \geq \eps$. Let $\sigma \in S_k$ be selected uniformly. Then either
\begin{equation}\label{eq:2manip_pt}
\p \left( \sigma \in M_2 \right) \geq \frac{4\eps}{k^6}
\end{equation}
or there exist alternatives $a,b \in \left[ k \right]$, $a\neq b$ such that
\begin{equation}\label{eq:big_ab_bdry}
\Inf^{a,b;\left[a:b\right]} \geq \frac{2\eps}{k^6}.
\end{equation}
\end{lemma}

\begin{proof}
The proof is just like the proof of Lemma~\ref{lem:boundaries2}. First, suppose that $\Inf^{a,b;z} \geq \frac{2\eps}{k^6}$ for some pair of alternatives $a \neq b$, and transposition $z \neq \left[a:b \right]$. Then by Lemma~\ref{lem:nonManipBoundary}, for any point $\left( \sigma, \sigma' \right) \in B^{a,b;z}$, at least one of $\sigma$ or $\sigma' = z \sigma$ is a 2-manipulation point. Then
\[
\left| M_2 \right| \geq \left| B^{a,b;z} \right| = 2 \cdot k! \cdot \Inf^{a,b;z} \geq \frac{4\eps}{k^6} k!,
\]
and dividing with $k!$ gives \eqref{eq:2manip_pt}. So for the remainder of the proof we may assume that $\Inf^{a,b;z} < \frac{2\eps}{k^6}$ for every $a \neq b$ and $z \neq \left[a:b \right]$.

For every $a \in \left[k\right]$, $\Dist \left( f, \tp_{\left\{a\right\}} \right) \geq \eps$, so $\p \left( f \left( \sigma \right) = a \right) \leq 1 - \eps$. On the other hand, there exists an alternative, say $a \in \left[ k \right]$, such that $\p \left( f\left( \sigma \right) = a \right) \geq \frac{1}{k}$. So for this alternative we have
\[
\Var \left( \mathbf{1} \left[ f\left( \sigma \right) = a \right] \right) \geq \frac{\eps}{k},
\]
and consequently using~Corollary ~\ref{cor:sumInfVarBound2} and Proposition~\ref{prop:sumInfVarBound} we have
\[
\sum_{w \in T} \sum_{b \neq a} \Inf^{a,b;w} = \sum_{w \in T} \Inf^{a;w} \geq \frac{1}{k^2} \Var \left( \mathbf{1} \left[ f\left( \sigma \right) = a \right] \right) \geq \frac{\eps}{k^3}.
\]
Hence there must exist some $w \in T$ and $b \neq a$ such that $\Inf^{a,b;w} \geq \frac{2\eps}{k^6}$, but by our assumption we must have $w = \left[a:b\right]$.
\end{proof}

If \eqref{eq:2manip_pt} holds, then we are done, so in the following we assume that \eqref{eq:big_ab_bdry} holds.

We know that $\sigma$ is on $B^{a,b;\left[a:b\right]}$ if $f\left( \sigma \right) = a$ and $f \left( \left[a:b\right] \sigma \right) = b$. We know that if $b \stackrel{\sigma}{>} a$, then $\sigma$ is a 2-manipulation point, so if this happens in more than half of the cases when $\sigma$ is on $B^{a,b;\left[a:b\right]}$, then we have 
\[
\p \left( \sigma \in M_2 \right) \geq \frac{2\eps}{k^6},
\]
in which case we are again done. So we may assume in the following that
\begin{equation}\label{eq:B_1vot}
\p \left( \sigma \in B \right) \geq \frac{2\eps}{k^6},
\end{equation}
where
\[
B := \left\{ \sigma : f\left( \sigma \right) = a, f\left( \left[a:b \right] \sigma \right) = b, a \stackrel{\sigma}{>} b \right\}.
\]

\subsection{Division into cases}\label{sec:cases_1vot} 

We again divide into two cases. 

We introduce the set $\bar{F}$ of permutations where $a$ is directly above $b$:
\[
\bar{F} := \left\{ \sigma \in S_k : a \stackrel{\sigma}{>} b \text{ and } b \stackrel{\sigma'}{>} a, \text{ where } \sigma' = \left[a:b\right] \sigma \right\}.
\]
One of the following two cases must hold.

\textbf{Case 1: Small fiber case.} We have
\begin{equation}\label{eq:sm_fbr_case}
\p \left( \sigma \in B \, \middle| \, \sigma \in \bar{F} \right) \leq 1 - \frac{\eps}{4k}.
\end{equation}

\textbf{Case 2: Large fiber case.} We have
\begin{equation}\label{eq:lg_fbr_case}
\p \left( \sigma \in B \, \middle| \, \sigma \in \bar{F} \right) > 1 - \frac{\eps}{4k}.
\end{equation}

\subsection{Small fiber case}\label{sec:sm_fbr_1vot} 

In this section we assume that \eqref{eq:sm_fbr_case} holds.

We first formalize that the boundary $\partial \left( B \right)$ of $B$ is big (recall the definition of $\partial\left( B \right)$ from Section~\ref{sec:bdry_of_bdry}). The proof uses the canonical path method again.

\begin{lemma}\label{lem:comparable_bdries_1vot}
If \eqref{eq:sm_fbr_case} holds, then
\begin{equation}\label{eq:comp_bdry_of_B}
\p \left( \sigma \in \partial \left( B \right) \right) \geq \frac{\eps}{2 k^4} \p \left( \sigma \in B \right).
\end{equation}
\end{lemma}

\begin{proof}
 Let $B^c = \bar{F} \setminus B$. For every $\left( \sigma, \sigma' \right) \in B \times B^c$, we define a canonical path from $\sigma$ to $\sigma'$, which has to pass through at least one edge in $\partial_e \left( B \right)$. Then if we show that every edge in $\partial_e \left( B \right)$ lies on at most $r$ canonical paths, then it follows that $\left| \partial_e \left( B \right) \right| \geq \left|B \right| \left|B^c \right| / r$.

So let $\left( \sigma, \sigma' \right) \in B \times B^c$. We apply the path construction of Proposition~\ref{prop:canon1}, but considering the block formed by $a$ and $b$ as a single element. Since this path goes from $\sigma$ (which is in $B$) to $\sigma'$ (which is in $B^c$), it must pass through at least one edge in $\partial_e \left( B \right)$.

For a given edge $\left( \pi, \pi' \right) \in \partial_e \left( B \right)$, at most how many possible $\left( \sigma, \sigma' \right) \in B \times B^c$ pairs are there such that the canonical path between $\sigma$ and $\sigma'$ defined above passes through $\left( \pi, \pi' \right)$? We learn from Proposition~\ref{prop:canon1} that there are at most $\left( k - 1\right)^2 \left( k - 1 \right)! / 2 < k^2 \left( k - 1 \right)! / 2$ possibilities for the pair $\left( \sigma, \sigma' \right)$.

Recall that $\left| \bar{F} \right| = \left( k - 1 \right)!$. By our assumption we have $\left| B \right| \leq \left( 1 - \frac{\eps}{4k} \right) \left( k - 1 \right)!$, and so $\left| B^c \right| \geq \frac{\eps}{4k} \left( k - 1 \right)!$. Therefore
\[
\left| \partial_e \left( B \right) \right| \geq \frac{\left|B \right| \left| B^c \right|}{\frac{k^2}{2} \left(k-1\right)!} \geq \frac{\eps}{2 k^3} \left| B \right|.
\]
Now in $G$ every ranking profile has $k - 2 < k$ neighbors, which implies \eqref{eq:comp_bdry_of_B}.
\end{proof}

\begin{corollary}\label{cor:big_bdry_of_B}
If \eqref{eq:sm_fbr_case} holds, then
\begin{equation}\label{eq:big_bdry_of_B}
\p \left( \sigma \in \partial \left( B \right) \right) \geq \frac{\eps^2}{k^{10}}.
\end{equation}
\end{corollary}

\begin{proof}
Combine Lemma~\ref{lem:comparable_bdries_1vot} and \eqref{eq:B_1vot}.
\end{proof}

Next we want to find manipulation points on the boundary $\partial \left( B \right)$. The next lemma tells us that if we are on the boundary $\partial \left( B \right)$, then either we can find manipulation points easily, or we are at a local dictator on three alternatives.

\begin{lemma}\label{lem:bdry_of_B}
Suppose that $\sigma \in \partial \left( B \right)$. Then
\begin{itemize}
\item either $\sigma \in \LD \left( a,b \right)$,
\item or there exists $\hat{\sigma} \in M_3$ such that $\hat{\sigma}$ is equal to $\sigma$ or $\left[a:b\right] \sigma$ except that the position of a third alternative $c$ might be shifted arbitrarily.
\end{itemize}
\end{lemma}

\begin{proof}
Since $\sigma \in \partial \left( B \right) \subseteq B$, we know that $f \left( \sigma \right) = a$, and if $\sigma' = \left[a:b\right] \sigma$, then $f\left( \sigma' \right) = b$. Let $\pi \in B^c$ denote the ranking profile such that $\left( \sigma, \pi \right) \in \partial_e \left( B \right)$, and let $\pi' = \left[a:b \right] \pi$. Since $\pi \notin B$, $\left( f\left( \pi \right), f \left( \pi' \right) \right) \neq \left( a, b \right)$. Then, by Lemma~\ref{lem:nonManipBoundary}, if $f\left( \pi \right) \neq f\left( \pi' \right)$, then one of $\pi$ and $\pi'$ is a 2-manipulation point. So assume $f\left( \pi \right) = f \left( \pi' \right)$.

There are two cases to consider: either $\sigma$ and $\pi$ differ by an adjacent transposition not involving the block of $a$ and $b$, or they differ by an adjacent transposition that moves the block of $a$ and $b$.

In the former case, it is not hard to see that one of $\sigma$, $\sigma'$, $\pi$, $\pi'$ is a 2-manipulation point, by Lemma~\ref{lem:nonManipBoundary}.

If $\sigma$ and $\pi$ differ by an adjacent transposition that involves the block of $a$ and $b$, then there are again two cases to consider: either this transposition moves the block of $a$ and $b$ up in the ranking, or it moves it down.

If the block of $a$ and $b$ is moved up to get from $\sigma$ to $\pi$, then we must have $f\left( \pi \right) = a$, or else $\sigma$ or $\pi$ is a 3-manipulation point. Then we must have $f \left( \pi' \right) = f\left( \pi \right) = a$, in which case $\pi'$ is a 3-manipulation point, since $f\left( \sigma' \right) = b$.

The final case is when the block of $a$ and $b$ is moved down to get from $\sigma$ to $\pi$, and a third alternative, call it $c$, is moved up, directly above the block of $a$ and $b$. Now if $f\left( \pi \right) = d \notin \left\{a, b, c \right\}$, then $\sigma$ or $\pi$ is a 3-manipulation point. If $f\left( \pi \right) = f \left( \pi' \right) = a$, then $\pi'$ is a 3-manipulation point, whereas if $f\left( \pi \right) = f \left( \pi' \right) = b$, then $\pi$ is a 3-manipulation point. The remaining case is when $f\left( \pi \right) = f \left( \pi' \right) = c$. Now if $f\left( \left[b:c\right] \sigma \right) \neq a$ or $f\left( \left[a:c\right] \sigma' \right) \neq b$, then we again have a 3-manipulation point close to $\sigma$. Otherwise $\sigma \in \LD \left( a, b \right)$.
\end{proof}

The following corollary then tells us that either we have found many 3-manipulation points, or we have many local dictators on three alternatives.

\begin{corollary}\label{cor:manip_or_loc_dict}
If \eqref{eq:sm_fbr_case} holds, then either
\begin{equation}\label{eq:lot_of_loc_dict}
\sum_{c \notin \left\{ a,b \right\}} \p \left( \sigma \in \LD^{\left\{a,b,c\right\}} \right) = \p \left( \sigma \in \LD \left( a, b \right) \right) \geq \frac{\eps^2}{2 k^{10}}
\end{equation}
or
\[
\p \left( \sigma \in M_3 \right) \geq \frac{\eps^2}{4 k^{12}}.
\]
\end{corollary}

\subsection{Dealing with local dictators}\label{sec:loc_dict_1vot} 

So the remaining case we have to deal with in this small fiber case is when \eqref{eq:lot_of_loc_dict} holds, i.e., we have many local dictators on three alternatives.

\begin{lemma}\label{lem:loc_dict_abc_to_top_1vot}
Suppose $\sigma \in \LD^{\left\{a,b,c\right\}}$ for some alternative $c\notin \left\{a,b\right\}$. Let $\sigma'$ be equal to $\sigma$ except that the block of $a$, $b$ and $c$ is moved to the top of the coordinate. Then
\begin{itemize}
\item either $\sigma' \in \LD^{\left\{a,b,c\right\}}$,
\item or there exists a 3-manipulation point $\hat{\sigma}$ which agrees with $\sigma$ except that the positions of $a$, $b$ and $c$ might be shifted arbitrarily.
\end{itemize}
\end{lemma}

\begin{proof}
W.l.o.g.\ we may assume that in $\sigma$ alternative $a$ is ranked above $b$, which is ranked above $c$. Now move $a$ to the top using a sequence of adjacent transpositions, all involving $a$; we call this procedure ``bubbling'' $a$ to the top. If at any point during this the outcome of $f$ is not $a$, then we have found a 2-manipulation point. Now bubble up $b$ to right below $a$, and then bubble up $c$ to be right below $b$. Again, if at any point during this the outcome of $f$ is not $a$, then there is a 2-manipulation point. Otherwise we now have $a$, $b$ and $c$ at the top (in this order), with the outcome of $f$ being $a$. Now permuting alternatives $a$, $b$ and $c$ at the top, we either have a 3-manipulation point, or $\sigma' \in \LD^{\left\{a,b,c\right\}}$.
\end{proof}

\begin{corollary}\label{cor:many_dict_at_top_1vot}
If \eqref{eq:lot_of_loc_dict} holds, then either
\begin{equation}\label{eq:loc_dict_abc_on_top_1vot}
\sum_{c \notin \left\{ a, b \right\}} \p \left( \sigma \in \LD^{\left\{a,b,c\right\}},  \left\{ \sigma \left( 1 \right), \sigma \left( 2 \right), \sigma \left( 3 \right) \right\} = \left\{a,b,c\right\} \right) \geq \frac{\eps^2}{4 k^{11}}
\end{equation}
or
\[
\p \left( \sigma \in M_3 \right) \geq \frac{\eps^2}{4 k^{13}}.
\]
\end{corollary}

\begin{proof}
Lemma~\ref{lem:loc_dict_abc_to_top_1vot} tells us that when we move the block of $a$, $b$, and $c$ up to the top, we either encounter a 3-manipulation point, or we get a local dictator on $\left\{ a, b, c \right\}$ at the top.

If we get a 3-manipulation point, by the describtion of this manipulation point in the lemma, there can be at most $k^3$ ranking profiles that give the same manipulation point.

If we arrive at a local dictator at the top, then there could have been at most $k$ different places where the block of $a$, $b$ and $c$ could have come from.
\end{proof}

Now \eqref{eq:loc_dict_abc_on_top_1vot} is equivalent to
\begin{equation}\label{eq:loc_dict_abc_on_top2_1vot}
\sum_{c \notin \left\{ a, b \right\}} \p \left( \sigma \in \LD^{\left\{a,b,c\right\}}, \left( \sigma \left( 1 \right), \sigma \left( 2 \right), \sigma \left( 3 \right) \right) = \left(a,b,c\right) \right) \geq \frac{\eps^2}{24 k^{11}}.
\end{equation}
We know that
\[
\p \left( \left( \sigma \left( 1 \right), \sigma \left( 2 \right), \sigma \left( 3 \right) \right) = \left(a,b,c\right) \right) = \frac{1}{k \left( k-1 \right) \left( k-2 \right)} \leq \frac{6}{k^3},
\]
and so \eqref{eq:loc_dict_abc_on_top2_1vot} implies (recall Definition~\ref{def:cond})
\begin{equation}\label{eq:loc_dict_abc_on_top3_1vot}
\sum_{c \notin \left\{ a, b \right\}} \p^{\left(a,b,c\right)} \left( \sigma \in \LD^{\left\{a,b,c\right\}}  \right) \geq \frac{\eps^2}{144 k^{8}}.
\end{equation}

Now fix an alternative $c \notin \left\{a,b\right\}$ and define the graph $G_{\left( a, b, c \right)} = \left( V_{\left( a, b, c \right)}, E_{\left( a, b, c \right)} \right)$ to have vertex set
\[
V_{\left( a, b, c \right)} := \left\{ \sigma \in S_k : \left( \sigma \left( 1 \right), \sigma \left( 2 \right), \sigma \left( 3 \right) \right) = \left(a,b,c\right) \right\}
\]
and for $\sigma, \pi \in V_{\left( a, b, c \right)}$ let $\left( \sigma, \pi \right) \in E_{\left( a, b, c \right)}$ if and only if $\sigma$ and $\pi$ differ by an adjacent transposition. So $G_{\left( a, b, c \right)}$ is the subgraph of the refined rankings graph induced by the vertex set $V_{\left( a, b, c \right)}$. (If $k = 3$ or $k=4$, then this graph consists of only one vertex, and no edges.)

Let
\[
T \left( a, b, c \right) := V_{\left( a, b, c \right)} \cap \LD^{\left\{a,b,c\right\}},
\]
and let $\partial_e \left( T \left( a, b, c \right) \right)$ and $\partial \left( T \left( a, b, c \right) \right)$ denote the edge- and vertex-boundary of $T \left( a, b, c \right)$ in $G_{\left( a, b, c \right)}$, respectively.

The next lemma shows that unless $T \left( a, b, c \right)$ is almost all of $V_{\left( a, b, c \right)}$, the size of the boundary $\partial \left( T \left( a, b, c \right) \right)$ is comparable to the size of $T \left( a, b, c \right)$. The proof uses a canonical path argument, just like in Lemma~\ref{lem:comparable_bdries_1vot}.

\begin{lemma}\label{lem:lg_bdry_for_loc_dict_1vot}
Let $c \notin \left\{a,b \right\}$ be arbitrary. Write $T \equiv T \left( a, b, c \right)$ for simplicity. If $\p^{\left(a,b,c\right)} \left( \sigma \in T \right) \leq 1 - \delta$, then
\begin{equation}\label{eq:lg_bdry_for_loc_dict_1vot}
\p^{\left( a, b, c \right)} \left( \sigma \in \partial \left( T \right) \right) \geq \frac{\delta}{k^3} \p^{\left( a, b, c \right)} \left( \sigma \in T \right).
\end{equation}
\end{lemma}

\begin{proof}
Let $T^c = V_{\left( a, b, c \right)} \setminus T \left( a, b, c \right)$. For every $\left( \sigma, \sigma' \right) \in T \times T^c$, we define a canonical path from $\sigma$ to $\sigma'$, which has to pass through at least one edge in $\partial_e \left( T \right)$. Then if we show that every edge in $\partial_e \left( T \right)$ lies on at most $r$ canonical paths, then it follows that $\left| \partial_e \left( T \right) \right| \geq \left| T \right| \left| T^c \right| / r$.

So let $\left( \sigma, \sigma' \right) \in T \times T^c$. We apply the path construction of Proposition~\ref{prop:canon1}, but only to alternatives $\left[k \right] \setminus \left\{ a, b, c \right\}$.

The analysis of this construction is done in exactly the same way as in Lemma~\ref{lem:comparable_bdries_1vot}; in the end we get that there are at most $k^2 \left( k - 3 \right)!$ paths that pass through a given edge in $\partial_e \left( T \right)$.

Recall that $\left| V_{\left(a,b,c\right)} \right| = \left( k - 3 \right)!$ and that by our assumption $\left| T \right| \leq \left( 1 - \delta \right) \left( k - 3 \right)!$, so $\left| T^c \right| \geq \delta \left( k - 3 \right)!$. Therefore
\[
\left| \partial_e \left( T \right) \right| \geq \frac{\left| T \right| \left| T^c \right|}{k^2  \left( k - 3 \right)!} \geq \frac{\delta}{k^2} \left| T \right|.
\]
Now every vertex in $V_{\left(a,b,c\right)}$ has $k - 4 < k$ neighbors, which implies \eqref{eq:lg_bdry_for_loc_dict_1vot}.
\end{proof}

The next lemma tells us that if $\sigma$ is on the boundary of a set of local dictators on $\left\{a,b,c \right\}$ for some alternative $c \notin \left\{a,b\right\}$, then there is a 2-manipulation point $\hat{\sigma}$ which is close to $\sigma$.

\begin{lemma}\label{lem:manip_pts_on_bdry_of_loc_dict_1vot}
Suppose $\sigma \in \partial \left( T \left( a,b, c \right) \right)$ for some $c\notin \left\{a,b\right\}$. Then there exists $\hat{\sigma} \in M_2$ which equals $z \sigma$ for some adjacent transposition $z$ that does not involve $a$, $b$ or $c$, except that the order of the block of $a$, $b$ and $c$ might be rearranged.
\end{lemma}

\begin{proof}
Let $\pi$ be the ranking profile such that $\left( \sigma, \pi \right) \in \partial_e \left( T \left( a, b , c \right) \right)$, and let $z$ be the adjacent transposition in which they differ, i.e., $\pi = z \sigma$. Since $\pi \notin T \left( a, b, c \right)$, there exists a reordering of the block of $a$, $b$, and $c$ at the top of $\pi$ such that the outcome of $f$ is not the top ranked alternative. Call the resulting vector $\pi'$. W.l.o.g.\ let us assume that $\pi' \left( 1 \right) = a$. Let us also define $\sigma' := z \pi'$. Now $\pi'$ is a 2-manipulation point, since $f\left( \sigma' \right) = a$.
\end{proof}

The next corollary puts together Corollary~\ref{cor:many_dict_at_top_1vot} and Lemmas~\ref{lem:lg_bdry_for_loc_dict_1vot} and~\ref{lem:manip_pts_on_bdry_of_loc_dict_1vot}.

\begin{corollary}\label{cor:manip_by_bdry_loc_dict_1vot}
Suppose \eqref{eq:loc_dict_abc_on_top_1vot} holds. Then if for every $c \notin \left\{a,b\right\}$ we have $\p^{\left(a,b,c\right)} \left( \sigma \in T \left( a,b, c \right) \right) \leq 1 - \frac{\eps}{100 k}$, then
\[
\p \left( \sigma \in M_2 \right) \geq \frac{\eps^3}{10^5 k^{16}}.
\]
\end{corollary}

\begin{proof}
We know that \eqref{eq:loc_dict_abc_on_top_1vot} implies
\[
\sum_{c \notin \left\{a,b \right\}} \p^{a,b,c} \left( \sigma \in T \left( a,b,c\right) \right) \geq \frac{\eps^2}{144 k^8}.
\]
Now using the assumptions, Lemma~\ref{lem:lg_bdry_for_loc_dict_1vot} with $\delta = \frac{\eps}{100k}$, and Lemma~\ref{lem:manip_pts_on_bdry_of_loc_dict_1vot}, we have
\begin{align*}
\p \left( \sigma \in M_2 \right) &\geq \sum_{c \neq \left\{a,b \right\}} \frac{1}{k^3} \p^{\left(a,b,c \right)} \left( \sigma \in M_2 \right) \geq \sum_{c\notin \left\{a,b\right\}} \frac{1}{6 k^4} \p^{\left( a, b, c \right)} \left( \sigma \in \partial \left( T \left( a,b,c \right) \right) \right)\\
&\geq \sum_{c \notin \left\{a,b\right\}} \frac{\eps}{600 k^{8}} \p^{\left( a, b , c \right)} \left( \sigma \in T \left( a, b, c \right) \right)\geq \frac{\eps^3}{86400 k^{16}} \geq \frac{\eps^3}{10^5 k^{16}}. \qedhere
\end{align*}
\end{proof}

So again we are left with one case to deal with: if there exists an alternative $c \notin \left\{ a, b \right\}$ such that $\p^{\left(a,b,c\right)} \left( \sigma \in T \left( a,b, c \right) \right) > 1 - \frac{\eps}{100 k}$. Define a subset of alternatives $K \subseteq \left[k \right]$ in the following way:
\[
K := \left\{a,b \right\} \cup \left\{ c \in \left[k \right] \setminus \left\{a,b \right\} : \p^{\left(a,b,c\right)} \left( \sigma \in T \left( a,b, c \right) \right) > 1 - \frac{\eps}{100 k} \right\}.
\]
In addition to $a$ and $b$, $K$ contains those alternatives that whenever they are at the top with $a$ and $b$, they form a local dictator with high probability.

So our assumption now is that $\left| K \right| \geq 3$.

Our next step is to show that unless we have many manipulation points, for any alternative $c \in K$, conditioned on $c$ being at the top, the outcome of $f$ is $c$ with probability close to 1.

\begin{lemma}\label{lem:cond_on_top_1vot}
Let $c \in K$. Then either
\begin{equation}\label{eq:cond_on_top1_1vot}
\p^{\left( c \right)} \left( f \left( \sigma \right) = c \right) \geq 1 - \frac{\eps}{50 k},
\end{equation}
or
\begin{equation}\label{eq:else_2manip_1vot}
\p \left( \sigma \in M_2 \right) \geq \frac{\eps}{100 k^4}.
\end{equation}
\end{lemma}

\begin{proof}
First assume that $c \notin \left\{ a, b \right\}$.

Let $\sigma$ be uniform according to $\p^{\left( c \right)}$, i.e., uniform on $S_k$ conditioned on $\sigma \left( 1 \right) = c$. Define $\sigma'$, where $\sigma'$ is constructed from $\sigma$ by first bubbling up alternative $a$ to just below $c$, using adjacent transpositions, and then bubbling up $b$ to just below $a$. Clearly $\sigma'$ is distributed according to $\p^{\left(c,a,b\right)}$, i.e., it is uniform on $S_k$ conditioned on $\left( \sigma \left( 1 \right), \sigma \left( 2 \right), \sigma \left( 3 \right) \right) = \left( c, a, b \right)$.

Since $c \in K$, we know that $\p^{\left(c,a,b \right)} \left( \sigma \in  \LD^{\left\{a,b,c\right\}} \right) > 1 - \frac{\eps}{100k}$. This also means that
\[
\p^{\left(c \right)} \left( \sigma' \in \LD^{\left\{a,b,c\right\}} \right) > 1 - \frac{\eps}{100k}.
\]
Now we can partition the ranking profiles into three parts, based on the outcome of the SCF $f$ at $\sigma$ and $\sigma'$:
\begin{align*}
I_1 &= \left\{ \sigma : f\left( \sigma \right) = c, f \left( \sigma' \right) = c \right\}\\
I_2 &= \left\{ \sigma : f\left( \sigma \right) \neq c, f \left( \sigma' \right) = c \right\}\\
I_3 &= \left\{ \sigma : f \left( \sigma' \right) \neq c \right\}.
\end{align*}
If $\p^{\left( c \right)} \left( I_1 \right) \geq 1 - \frac{\eps}{50k}$, then \eqref{eq:cond_on_top1_1vot} holds. Otherwise we have $\p^{\left(c\right)} \left( I_2 \cup I_3 \right) \geq \frac{\eps}{50k}$, and since $\p^{\left( c \right)} \left( I_3 \right) \leq \frac{\eps}{100 k}$, we have $\p^{\left( c \right)} \left( I_2 \right) \geq \frac{\eps}{100k}$.

Now if $\sigma \in I_2$, then we know that there is a 2-manipulation point along the way as we go from $\sigma$ to $\sigma'$. I.e., to every $\sigma \in I_2$ there exists $\hat{\sigma} \in M_2$ such that $\hat{\sigma}$ is equal to $\sigma$ except perhaps $a$ and $b$ are shifted arbitrarily. So there can be at most $k^2$ ranking profiles $\sigma$ giving the same 2-manipulation point $\hat{\sigma}$, and so we have
\[
\p\left( \sigma \in M_2 \right) \geq \frac{1}{k} \p^{\left(c\right)} \left( \sigma \in M_2 \right) \geq \frac{1}{k^3} \p^{\left( c \right)} \left( I_2 \right) \geq \frac{\eps}{100 k^4},
\]
showing \eqref{eq:else_2manip_1vot}.

Now suppose $c \in \left\{a,b\right\}$, w.l.o.g.\ assume $c = a$. We know that $\left|K \right| \geq 3$ and so there exists an alternative $d\in K \setminus \left\{a,b \right\}$. We can then do the same thing as above, but we now bubble up $b$ and $d$.
\end{proof}

We now deal with alternatives that are not in $K$: either we have many manipulation points, or for any alternative $d \notin K$, the outcome of $f$ is \emph{not} $d$ with probability close to 1.

\begin{lemma}\label{lem:d_notin_K_1vot}
Let $d \notin K$. If $\p \left( f \left( \sigma \right) = d \right) \geq \frac{\eps}{4k}$, then
\[
\p \left( \sigma \in M_2 \right) \geq \frac{\eps^2}{10^6 k^{9}}.
\]
\end{lemma}

\begin{proof}
Let $\sigma$ be such that $f \left( \sigma \right) = d$. Bubble up $d$ to the top, and call this ranking profile $\sigma'$. Now if $f\left( \sigma' \right) \neq d$, then we know that there exists a 2-manipulation point $\hat{\sigma}$ along the way, i.e., a 2-manipulation $\hat{\sigma}$ which agrees with $\sigma$ except perhaps $d$ is shifted arbitrarily. Consequently, either
\[
\p \left( \sigma \in M_2 \right) \geq \frac{\eps}{8k^2},
\]
in which case we are done, or 
\[
\p \left( \sigma: f\left( \sigma \right) = f\left( \sigma' \right) = d \right) \geq \frac{\eps}{8k}.
\]

Next, let us bubble up $a$ to just below $d$, and then bubble up $b$ to just below $d$. Denote this ranking profile by $\sigma^{\left(d,b,a\right)}$, and analogously define $\sigma^{\left(d,a,b\right)}, \sigma^{\left(a,b,d\right)}, \sigma^{\left(a,d,b\right)}, \sigma^{\left(b,a,d\right)}$, and $\sigma^{\left(b,d,a\right)}$. Either we encounter a 2-manipulation point $\hat{\sigma}$ along the way of bubbling up to $\sigma^{\left(d,b,a\right)}$ ($\hat{\sigma}$ agrees with $\sigma$ except $d$ is at the top, and $a$ and $b$ might be arbitrarily shifted), or the outcome of the SCF $f$ is $d$ all along. So we have that either
\[
\p \left( \sigma \in M_2 \right) \geq \frac{\eps}{16 k^3},
\]
in which case we are done, or
\[
\p \left( \sigma : f\left( \sigma \right) = f\left( \sigma' \right) = f\left( \sigma^{\left(d,b,a\right)} \right) = f\left( \sigma^{\left(d,a,b\right)} \right) = d \right) \geq \frac{\eps}{16 k}.
\]

Now start from $\sigma^{\left(d,a,b\right)}$. First swap $a$ and $d$ to get $\sigma^{\left(a,d,b\right)}$, then swap $d$ and $b$ to get $\sigma^{\left(a,b,d\right)}$, and finally bubble $d$ and $b$ down to their original positions in $\sigma$, except for the fact that $a$ is now at the top of the coordinate. Call this profile $\bar{\sigma}$. Since $\sigma$ is uniformly distributed, $\bar{\sigma}$ is distributed according to $\p_1^{\left( a \right)}$, i.e., uniformly conditional on $\bar{\sigma} \left( 1 \right) = a$. Now note that one of the following three events has to happen. (These events are not mutually exclusive.)
\begin{align*}
I_1 &= \left\{ f\left( \sigma^{\left(a,d,b\right)} \right) = f\left( \sigma^{\left(a,b,d\right)} \right) = a \right\}\\
I_2 &= \left\{ f \left( \bar{\sigma} \right) \neq a \right\}\\
I_3 &= \{ \sigma: \exists\ \hat{\sigma} \in M_2 \text{ which is equal to } \sigma \text{ except } a \text{ is shifted}\\
&\qquad \ \text{ to the top, and }b \text{ and } d \text{ may be shifted arbitrarily}\}.
\end{align*}
Since $a\in K$, we know by Lemma~\ref{lem:cond_on_top_1vot} that (unless we already have enough manipulation points by the lemma) we must have
\[
\p \left( f\left( \bar{\sigma} \right) \neq a \right) = \p^{\left( a \right)} \left( f\left( \bar{\sigma} \right) \neq a \right) \leq \frac{\eps}{50k}.
\]
Consequently 
\[
\p \left( I_1 \cup I_3, f\left( \sigma \right) = f\left( \sigma' \right) = f\left( \sigma^{\left(d,b,a\right)} \right) = f\left( \sigma^{\left(d,a,b\right)} \right) = d  \right) \geq \frac{\eps}{16 k} - \frac{\eps}{50 k} = \frac{17 \eps}{400 k},
\]
and so either
\[
\p \left( \sigma \in M_2 \right) \geq \frac{17 \eps}{800 k^3},
\]
in which case we are done, or
\[
\p \left( \sigma: f\left( \sigma^{\left(d,b,a\right)} \right) = f\left( \sigma^{\left(d,a,b\right)} \right) = d, f\left( \sigma^{\left(a,b,d\right)} \right) = f\left( \sigma^{\left(a,d,b\right)} \right) = a \right) \geq \frac{17 \eps}{800 k}.
\]

Next, we can do the same thing with $b$ on top, and we ultimately get that either
\[
\p \left( \sigma \in M_2 \right) \geq \frac{\eps}{1600 k^3},
\]
in which case we are done, or
\begin{equation}\label{eq:loc_dict_abd_1vot}
\p^{\left( a, b, d \right)} \left( \sigma^{\left(a,b,d\right)} \in \LD^{\left\{a,b,d\right\}} \right) = \p \left( \sigma: \sigma^{\left( a, b, d \right)} \in \LD^{\left\{a,b,d\right\}} \right) \geq \frac{\eps}{1600 k}.
\end{equation}
Define $G_{\left(a,b,d\right)}$ and $T_{\left( a, b, d \right)}$ analogously to $G_{\left( a, b, c \right)}$ and $T_{\left( a, b, c \right)}$, respectively.

Suppose that \eqref{eq:loc_dict_abd_1vot} holds. We also know that $d\notin K$, so Lemma~\ref{lem:lg_bdry_for_loc_dict_1vot} applies, and then Lemma~\ref{lem:manip_pts_on_bdry_of_loc_dict_1vot} shows us how to find manipulation points. We can put these arguments together, just like in the proof of Corollary~\ref{cor:manip_by_bdry_loc_dict_1vot}, to show what we need:
\begin{align*}
\p \left( \sigma \in M_2 \right) &\geq \frac{1}{k^3} \p^{\left(a,b,d \right)} \left( \sigma \in M_2 \right) \geq \frac{1}{6 k^4} \p^{\left( a, b, d \right)} \left( \sigma \in \partial \left( T \left( a,b,d \right) \right) \right)\\
&\geq \frac{\eps}{600 k^{8}} \p^{\left( a, b , d \right)} \left( \sigma \in T \left( a, b, d \right) \right) \geq \frac{\eps^2}{10^6 k^{9}}. \qedhere
\end{align*}
\end{proof}

Putting together the results of the previous lemmas, there is only one case to be covered, which is covered by the following final lemma. Basically, this lemma says that unless there are enough manipulation points, our function is close to a dictator on the subset of alternatives $K$.
\begin{lemma}\label{lem:final_loc_dict_1vot}
Recall that we assume that $\Dist \left( f, \NONMANIP \right) \geq \eps$. Furthermore assume that $\left| K \right| \geq 3$, for every $c \in K$ we have
\begin{equation}\label{eq:dict_cond_on_top_1vot}
\p^{\left( c \right)} \left( f \left( \sigma \right) = c \right) \geq 1 - \frac{\eps}{50k},
\end{equation}
and for every $d \notin K$ we have
\[
\p \left( f\left( \sigma \right) = d \right) \leq \frac{\eps}{4k}.
\]
Then
\begin{equation}\label{eq:final_manip_1vot}
\p \left( \sigma \in M_2 \right) \geq \frac{\eps}{4k^2}.
\end{equation}
\end{lemma}

\begin{proof}
First note that
\[
\p \left( f \left( \sigma \right) \neq \tp_K \left( \sigma \right) \right) = \p \left( f \left( \sigma \right) \notin K \right) + \p \left( f\left( \sigma \right) \neq \tp_K \left( \sigma \right), f\left( \sigma \right) \in K \right).
\]
We know that
\[
\eps \leq \Dist \left( f, \NONMANIP \right) \leq \p \left( f \left( \sigma \right) \neq \tp_K \left( \sigma \right) \right)
\]
and also that
\[
\p\left( f\left( \sigma \right) \notin K \right) \leq \left( k - \left| K \right| \right) \frac{\eps}{4k} \leq \frac{\eps}{2},
\]
which together imply that
\[
\p \left( f\left( \sigma \right) \neq \tp_K \left( \sigma \right), f\left( \sigma \right) \in K \right) \geq \frac{\eps}{2}.
\]
Let $\sigma$ be such that $ f\left( \sigma \right) \neq \tp_K \left( \sigma \right)$ and $f\left( \sigma \right) \in K$. Now bubble $\tp_K \left( \sigma \right)$ up to the top in $\sigma$, call this ranking profile $\bar{\sigma}$. Clearly then $\tp_K \left( \bar{\sigma} \right) = \tp_K \left( \sigma \right)$.

There are two cases to consider. If $f\left( \sigma \right) \neq f\left( \bar{\sigma} \right)$, then there is a 2-manipulation point along the way from $\sigma$ to $\bar{\sigma}$, i.e., a 2-manipulation point $\hat{\sigma}$ such that $\hat{\sigma}$ agrees with $\sigma$ except perhaps some alternative $c$ is arbitrarily shifted. Otherwise $f \left( \sigma \right) = f\left( \bar{\sigma} \right)$, and so $f \left( \bar{ \sigma} \right) \neq \tp_K \left( \bar{\sigma} \right)$.

Consequently we have that either \eqref{eq:final_manip_1vot} holds, or that
\begin{equation}\label{eq:last_eq_1vot}
\p \left( \sigma : f \left( \bar{ \sigma} \right) \neq \tp_K \left( \bar{\sigma} \right) \right) \geq \frac{\eps}{4}.
\end{equation}
By the construction of $\bar{\sigma}$, we know that $\bar{\sigma}$ is uniformly distributed conditional on $\bar{\sigma} \left( 1 \right) \in K$. Consequently, by \eqref{eq:dict_cond_on_top_1vot}, we have that
\[
\p \left( \sigma : f \left( \bar{ \sigma} \right) \neq \tp_K \left( \bar{\sigma} \right) \right) \leq \frac{\eps}{50 k},
\]
which contradicts with \eqref{eq:last_eq_1vot} since $\frac{\eps}{50k} < \frac{\eps}{4}$.
\end{proof}
This concludes the proof of the small fiber case.

\subsection{Large fiber case}\label{sec:lg_fbr_case_1vot} 

In this section we assume that \eqref{eq:lg_fbr_case} holds. We show that we either have a lot 2-manipulation points or we have a lot of local dictators on three alternatives.

Our first step towards this is the following lemma.

\begin{lemma}\label{lem:cond_ab_top_still_lg_fbr_1vot}
Suppose \eqref{eq:lg_fbr_case} holds. Then
\begin{equation}\label{eq:cond_ab_top_still_lg_fbr_1vot}
\p^{\left(a,b\right)} \left( \sigma \in B \right) \geq 1 - \frac{\eps}{4}.
\end{equation}
\end{lemma}

\begin{proof}
Let $B^c = \bar{F} \setminus B$. Our assumption \eqref{eq:lg_fbr_case} implies that $\p \left( \sigma \in B^c \, \middle| \, \sigma \in \bar{F} \right) \leq \frac{\eps}{4k}$, which means that $\left| B^c \right| \leq \frac{\eps \left( k - 1 \right)!}{4k}$, and so
\[
\p^{\left( a, b \right)} \left( \sigma \notin B \right) \leq \frac{\eps \left( k - 1 \right)!}{ 4k \left( k - 2 \right)!} < \frac{\eps}{4},
\]
which is equivalent to \eqref{eq:cond_ab_top_still_lg_fbr_1vot}.
\end{proof}

The next lemma (together with Section~\ref{sec:loc_dict_1vot}) concludes the proof in the large fiber case.

\begin{lemma}\label{lem:lg_fbr_final_1vot}
Suppose \eqref{eq:lg_fbr_case} holds and recall that our SCF $f$ satisfies $\Dist \left( f, \NONMANIP \right) \geq \eps$. Then either
\begin{equation}\label{eq:2manip_again_1vot}
\p \left( \sigma \in M_2 \right) \geq \frac{\eps}{4k^2}
\end{equation}
or
\begin{equation}\label{eq:loc_dict_again_1vot}
\p \left( \sigma \in \LD \left(a,b\right) \right) \geq \frac{\eps}{4k^2}.
\end{equation}
\end{lemma}

\begin{proof}
By Lemma~\ref{lem:cond_ab_top_still_lg_fbr_1vot} we know that \eqref{eq:cond_ab_top_still_lg_fbr_1vot} holds.

Let $\sigma \in S_k$ be uniform. Define $\sigma'$ by being the same as $\sigma$ except alternatives $a$ and $b$ are moved to the top of the coordinate: $\sigma' \left( 1 \right) = a$ and $\sigma' \left( 2 \right) = b$. Clearly $\sigma'$ is distributed according to $\p^{\left(a,b\right)} \left( \cdot \right)$. Also define $\sigma'' = \left[a:b \right] \sigma'$.

We partition the set of ranking profiles $S_k$ into three parts:
\begin{align*}
I_1 &:= \left\{ \sigma \in S_k : f\left( \sigma \right) = \tp_{\left\{a,b \right\}} \left( \sigma \right), \left( f \left( \sigma' \right), f \left( \sigma'' \right) \right) = \left( a, b \right) \right\}\\
I_2 &:= \left\{ \sigma \in S_k : f\left( \sigma \right) \neq \tp_{\left\{a,b \right\}} \left( \sigma \right), \left( f \left( \sigma' \right), f \left( \sigma'' \right) \right) = \left( a, b \right) \right\}\\
I_3 &:= \left\{ \sigma \in S_k : \left( f \left( \sigma' \right), f \left( \sigma'' \right) \right) \neq \left( a, b \right) \right\}.
\end{align*}

By \eqref{eq:cond_ab_top_still_lg_fbr_1vot} we know that $\p \left( \sigma \in I_3 \right) \leq \frac{\eps}{4}$. We also know that $\p \left( \sigma \in I_1 \right) \leq 1 - \eps$, since $\Dist \left( f, \NONMANIP \right) \geq \eps$. Therefore we must have
\[
\p \left( \sigma \in I_2 \right) \geq \frac{3\eps}{4} > \frac{\eps}{2}.
\]

Let us partition $I_2$ further, and write it as $I_2 = I_2' \cup \left( \cup_{c \notin \left\{a,b \right\}} I_{2,c} \right)$, where 
\[
I_2' := \left\{ \sigma \in I_2 : f \left( \sigma \right) \neq \tp_{\left\{a,b\right\}} \left( \sigma \right), f\left( \sigma \right) \in \left\{a,b\right\} \right\}
\]
and for any $c \notin \left\{a,b \right\}$,
\[
I_{2,c} := \left\{ \sigma \in I_2 : f \left( \sigma \right) = c \right\}.
\]

Suppose $\sigma \in I_2'$. W.l.o.g.\ let us assume that $a$ is ranked higher than $b$ by $\sigma$, and therefore $f\left( \sigma \right) = b$, since $\sigma \in I_2'$. Then we can get from $\sigma$ to $\sigma'$ by first bubbling up $a$ to the top, and then bubbling up $b$ to just below $a$. Since $f\left( \sigma \right) = b$ and $f\left( \sigma' \right) = a$, there must be a 2-manipulation point $\hat{\sigma}$ along the way, which is equal to $\sigma$ except perhaps the positions of $a$ and $b$ are arbitrarily shifted.

Now suppose that $\sigma \in I_{2,c}$ for some $c \notin \left\{a,b \right\}$. We distinguish two cases: either $c$ is ranked above both $a$ and $b$ in $\sigma$, or it is not.

If not, then say $a$ is ranked above $c$ in $\sigma$. Bubble $a$ all the way to the top, and then bubble $b$ as well, all the way to the top, just below $a$. Since $f\left( \sigma \right) = c$ and $f\left( \sigma' \right) = a$, there must be a 2-manipulation point $\hat{\sigma}$ along the way, which is equal to $\sigma$ except perhaps the positions of $a$ and $b$ are arbitrarily shifted.

If $c$ is ranked above both $a$ and $b$ in $\sigma$, then the argument is similar. First bubble up $a$ and $b$ to just below $c$, and denote this ranking profile by $\tilde{\sigma}$, then permute these three alternatives arbitrarily, and then bubble $a$ and $b$ to the top. It is not hard to think through that either there is a 2-manipulation $\hat{\sigma}$ along the way, which is then equal to $\sigma$ except perhaps the positions of $a$ and $b$ are arbitrarily shifted, or else $\tilde{\sigma} \in \LD^{\left\{a,b,c\right\}}$.

Combining these cases we see that either \eqref{eq:2manip_again_1vot} or \eqref{eq:loc_dict_again_1vot} must hold.
\end{proof}

So if \eqref{eq:2manip_again_1vot} holds then we are done, and if \eqref{eq:loc_dict_again_1vot} holds, then we refer back to Section~\ref{sec:loc_dict_1vot}, where we deal with the case of local dictators on three alternatives.

\subsection{Proof of Theorem~\ref{thm:quant_GS_1voter} concluded} 

\begin{proof}[Proof of Theorem~\ref{thm:quant_GS_1voter}]
Our starting point is Lemma~\ref{lem:big_ab_bdry}, which implies that \eqref{eq:B_1vot} holds (unless we already have many 2-manipulation points). We then consider two cases, as indicated in Section~\ref{sec:cases_1vot}.

We deal with the small fiber case---when \eqref{eq:sm_fbr_case} holds---in Section~\ref{sec:sm_fbr_1vot}. First, Lemma~\ref{lem:comparable_bdries_1vot}, Corollary~\ref{cor:big_bdry_of_B}, Lemma~\ref{lem:bdry_of_B} and Corollary~\ref{cor:manip_or_loc_dict} show that either there are many 3-manipulation points, or there are many local dictators on three alternatives. We then deal with the case of many local dictators in Section~\ref{sec:loc_dict_1vot}. Lemma~\ref{lem:loc_dict_abc_to_top_1vot}, Corollary~\ref{cor:many_dict_at_top_1vot}, Lemmas~\ref{lem:lg_bdry_for_loc_dict_1vot} and~\ref{lem:manip_pts_on_bdry_of_loc_dict_1vot}, Corollary~\ref{cor:manip_by_bdry_loc_dict_1vot}, and Lemmas~\ref{lem:cond_on_top_1vot},~\ref{lem:d_notin_K_1vot} and~\ref{lem:final_loc_dict_1vot} together show that there are many 3-manipulation points if there are many local dictators on three alternatives, and the SCF is $\eps$-far from the family of nonmanipulable functions.

We deal with the large fiber case---when \eqref{eq:lg_fbr_case} holds---in Section~\ref{sec:lg_fbr_case_1vot}. Here Lemma~\ref{lem:lg_fbr_final_1vot} shows that either we have many 2-manipulation points, or we have many local dictators on three alternatives. In this latter case we refer back to Section~\ref{sec:loc_dict_1vot} to conclude the proof.
\end{proof}

\section{Proof for unbounded number of Alternatives}\label{sec:manip_ref} 

In this section we prove the theorem below, which is the same as our main theorem, Theorem~\ref{cor:k_refined_truenonmanip}, except that the condition of $\Dist \left( f, \NONMANIP \right) \geq \eps$ from  Theorem~\ref{cor:k_refined_truenonmanip} is replaced with the stronger condition $\Dist \left( f, \overline{\NONMANIP} \right) \geq \eps$.

\begin{theorem}\label{thm:k_refined}
Suppose we have $n \geq 2$ voters, $k \geq 3$ alternatives, and a SCF $f : S_k^n \to \left[k\right]$ satisfying $\Dist \left( f, \overline{\NONMANIP} \right) \geq \eps$. Then
\begin{equation}\label{eq:manip_refined2}
\p\left( \sigma \in M \left( f \right) \right)\geq \p \left( \sigma \in M_4 \left( f \right) \right)\geq p \left( \eps, \frac{1}{n}, \frac{1}{k} \right),
\end{equation}
for some polynomial $p$, where $\sigma \in S_k^n$ is selected uniformly. In particular, we show a lower bound of $\frac{\eps^5}{10^9 n^7 k^{46}}$.

An immediate consequence is that
\[
\p \left( \left( \sigma, \sigma' \right) \text{ is a manipulation pair for } f \right) \geq q \left( \eps, \frac{1}{n}, \frac{1}{k} \right),
\]
for some polynomial $q$, where $\sigma \in S_k^n$ is uniformly selected, and $\sigma'$ is obtained from $\sigma$ by uniformly selecting a coordinate $i \in \left\{1, \dots, n \right\}$, uniformly selecting $j \in \left\{1, \dots, n-3 \right\}$, and then uniformly randomly permuting the following four adjacent alternatives in $\sigma_i$: $\sigma_i \left( j \right), \sigma_i \left( j + 1 \right), \sigma_i \left( j + 2 \right)$, and $\sigma_i \left( j + 3 \right)$. In particular, the specific lower bound for $\p \left( \sigma \in M_4 \left( f \right) \right)$ implies that we can take $q \left( \eps, \frac{1}{n}, \frac{1}{k} \right) = \frac{\eps^5}{10^{11} n^8 k^{47}}$.
\end{theorem}

For the remainder of the section, let us fix the number of voters $n \geq 2$, the number of alternatives $k \geq 3$, and the SCF $f$, which satisfies $\Dist \left( f, \overline{\NONMANIP} \right) \geq \eps$. Accordingly, we typically omit the dependence of various sets (e.g., boundaries between two alternatives) on $f$.

\subsection{Division into cases}\label{sec:cases_refined} 

Our starting point in proving Theorem~\ref{thm:k_refined} is Lemma~\ref{lem:boundaries2}. Clearly if \eqref{eq:neutralPairs3ManipProb} holds then we are done, so in the rest of Section~\ref{sec:manip_ref} we assume that this is not the case. Then Lemma~\ref{lem:boundaries2} tells us that \eqref{eq:k_inf_ref} holds, and w.l.o.g.\ we may assume that the two boundaries that the lemma gives us have $i=1$ and $j=2$. I.e., we have
\[
\p \left( \sigma \text{ on } B_1^{a,b;\left[a:b\right]} \right) \geq \frac{4\eps}{n k^7} \quad \text{and} \quad \p \left( \sigma \text{ on } B_2^{c,d;\left[c:d\right]} \right) \geq \frac{4\eps}{n k^7},
\]
where recall that $\sigma$ is on $B_1^{a,b;\left[a:b\right]}$ if $f\left( \sigma \right) = a$ and $f\left( \left[a:b\right]_1 \sigma \right) = b$. If $\sigma$ is on $B_1^{a,b;\left[a:b\right]}$ and $b \stackrel{\sigma_1}{>} a$, then $\sigma$ is a 2-manipulation point, so if this happens in more than half of the cases when $\sigma$ is on $B_1^{a,b;\left[a:b\right]}$, then we have
\[
\p\left( \sigma \in M_2 \right) \ge \frac{2\eps}{n k^7},
\]
and we are done. Similarly in the case of the boundary between $c$ and $d$ in coordinate 2. So we may assume from now on that
\[
\p \left( \sigma \in \cup_{z_{-1}^{a,b}} B_1 \left( z_{-1}^{a,b} \right) \right) \geq \frac{2\eps}{n k^7} \quad \text{and} \quad \p \left( \sigma \in \cup_{z_{-2}^{c,d}} B_2 \left( z_{-2}^{c,d} \right) \right) \geq \frac{2\eps}{n k^7}.
\]
The following lemma is an immediate corollary.
\begin{lemma}\label{lem:cases_ref}
Either
\begin{equation}\label{eq:sm_fbr_ref}
\p \left( \sigma \in \Sm \left( B_1^{a,b;\left[a:b\right]} \right) \right) \geq \frac{\eps}{n k^7}
\end{equation}
or
\begin{equation}\label{eq:lg_fbr_ref}
\p \left( \sigma \in \Lg \left( B_1^{a,b;\left[a:b\right]} \right) \right) \geq \frac{\eps}{n k^7},
\end{equation}
and the same can be said for the boundary $B_2^{c,d;\left[c:d\right]}$.
\end{lemma}

We distinguish cases based upon this: either \eqref{eq:sm_fbr_ref} holds, or \eqref{eq:sm_fbr_ref} holds for the boundary $B_2^{c,d;\left[c:d\right]}$, or \eqref{eq:lg_fbr_ref} holds for both boundaries. We only need one boundary for the small fiber case, and we need both boundaries only in the large fiber case. So in the large fiber case we must differentiate between two cases: whether $d\in \left\{ a, b \right\}$ or $d \notin \left\{a, b \right\}$. 
We will see that our method of proof works in both cases.

In the rest of the section we first deal with the small fiber case, and then with the large fiber case.

\subsection{Small fiber case}\label{sec:sm_fbr_ref} 

We now deal with the case when \eqref{eq:sm_fbr_ref} holds. We formalize the ideas of the outline in a series of statements.

First, we want to formalize that the boundaries of the boundaries are big in this refined graph setting as well, when we are on a small fiber. The proof uses the canonical path method, as successfully adapted to this setting by Isaksson, Kindler and Mossel~\cite{IsKiMo:12}, and is very similar to the proof of Lemma~\ref{lem:comparable_bdries_1vot}, with some necessary modifications due to the fact that we now have $n$ coordinates.
\begin{lemma}\label{lem:comparable_bdries_ref}
Fix a coordinate and a pair of alternatives---for simplicity we choose coordinate 1 and alternatives $a$ and $b$, but we note that this lemma holds in general, we do not assume anything special about these choices. Let $z_{-1}^{a,b}$ be such that $B_1 \left( z_{-1}^{a,b} \right)$ is a small fiber for $B_1^{a,b;\left[a:b\right]}$. Then, writing $B \equiv B_1 \left( z_{-1}^{a,b} \right)$ for simplicity, we have
\begin{equation}\label{eq:comp_bdries_ref}
\p \left( \sigma \in \partial \left( B \right) \right) \geq \frac{\gamma}{2nk^5} \p \left( \sigma \in B \right).
\end{equation}
\end{lemma}

\begin{proof}
Let $B^c = \bar{F} \left( z_{-1}^{a,b} \right) \setminus B$. For every $\left( \sigma, \sigma' \right) \in B \times B^c$, we define a canonical path from $\sigma$ to $\sigma'$, which has to pass through at least one edge in $\partial_e \left( B \right)$. Then if we show that every edge in $\partial_e \left( B \right)$ lies on at most $r$ canonical paths, then it follows that $\left| \partial_e \left( B \right) \right| \geq \left|B \right| \left|B^c \right| / r$.

So let $\left( \sigma, \sigma' \right) \in B \times B^c$. We define a path from $\sigma$ to $\sigma'$ by applying a path construction in each coordinate one by one, and then concatenating these paths: first in the first coordinate we get from $\sigma_1$ to $\sigma'_1$, while leaving all other coordinates unchanged, then in the second coordinate we get from $\sigma_2$ to $\sigma'_2$, while leaving all other coordinates unchanged, and so on, finally in the last coordinate we get from $\sigma_n$ to $\sigma'_n$. In the first coordinate we apply the path construction of Proposition~\ref{prop:canon1}, but considering the block formed by $a$ and $b$ as a single element; in all other coordinates we apply the path construction of Proposition~\ref{prop:canon2}. Since this path goes from $\sigma$ (which is in $B$) to $\sigma'$ (which is in $B^c$), it must pass through at least one edge in $\partial_e \left( B \right)$.

For a given edge $\left( \pi, \pi' \right) \in \partial_e \left( B \right)$, at most how many possible $\left( \sigma, \sigma' \right) \in B \times B^c$ pairs are there such that the canonical path between $\sigma$ and $\sigma'$ defined above passes through $\left( \pi, \pi' \right)$? Let us differentiate two cases.

Suppose $\pi$ and $\pi'$ differ in the first coordinate. Then coordinates 2 through $n$ of $\sigma$ must agree with the respective coordinates of $\pi$, while coordinates 2 through $n$ of $\sigma'$ can be anything (up to the restriction given by $\sigma' \in B^c \subseteq \bar{F} \left( z_{-1}^{a,b} \right)$), giving $\left( \frac{k!}{2} \right)^{n-1}$ possibilities. Now fixing all coordinates except the first,Proposition~\ref{prop:canon1} tells us that there are at most $\left( k - 1 \right)^2 \left( k -1 \right)! / 2 < k^2 \left( k - 1 \right)!$ possibilities for the pair $\left( \sigma_1, \sigma'_1 \right)$. So altogether there are at most $k^2 \left( k- 1 \right)! \left( \frac{k!}{2} \right)^{n-1}$ paths that pass through a given edge in $\partial_e \left( B \right)$ in this case.

Suppose now that $\pi$ and $\pi'$ differ in the $i^{\text{th}}$ coordinate, $i \neq 1$. Then the first $i-1$ coordinates of $\sigma'$ must agree with the first $i-1$ coordinates of $\pi$, while coordinates $i+1, \dots, n$ of $\sigma$ must agree with the respective coordinates of $\pi$. The first $i-1$ coordinates of $\sigma$, and coordinates $i+1, \dots, n$ of $\sigma'$ can be anything (up to the restriction given by $\sigma, \sigma' \in \bar{F} \left( z_{-1}^{a,b} \right)$), giving $\left( k - 1 \right)! \left( \frac{k!}{2} \right)^{n-2}$ possibilities. Now fixing all coordinates except the $i^{\text{th}}$ coordinate,~\cite[Proposition 6.6.]{IsKiMo:12} tells us that there are at most $k^4 k!$ possibilities for the pair $\left( \sigma_i, \sigma'_i \right)$. So altogether there are at most $2k^4 \left( k - 1 \right)! \left( \frac{k!}{2} \right)^{n-1}$ paths that pass through a given edge in $\partial_e \left( B \right)$ in this case.

So in any case, there are at most $2k^4 \left( k - 1 \right)! \left( \frac{k!}{2} \right)^{n-1}$ paths that pass through a given edge in $\partial_e \left( B \right)$.

Recall that $\left| \bar{F} \left( z_{-1}^{a,b} \right) \right| = \left( k - 1 \right)! \left( \frac{k!}{2} \right)^{n-1}$, and also $\left| B^c \right| \geq \gamma \left( k - 1 \right)! \left( \frac{k!}{2} \right)^{n-1}$ since $B$ is a small fiber. Therefore
\[
\left| \partial_e \left( B \right) \right| \geq \frac{\left| B \right| \left| B^c \right|}{ 2 k^4 \left( k - 1 \right)! \left( \frac{k!}{2} \right)^{n-1}} \geq \frac{\gamma}{2 k^4} \left| B \right|.
\]
Now in $G\left( z_{-1}^{a,b} \right)$ every ranking profile has no more than $n k$ neighbors, which implies \eqref{eq:comp_bdries_ref}.
\end{proof}

\begin{corollary}\label{cor:comparable_bdries_ref}
If \eqref{eq:sm_fbr_ref} holds, then
\[
\p \left( \sigma \in \bigcup_{z_{-1}^{a,b}} \partial \left( B_1 \left( z_{-1}^{a,b} \right) \right) \right) \geq \frac{\gamma \eps}{2 n^2 k^{12}}.
\]
\end{corollary}

\begin{proof}
Using the previous lemma and \eqref{eq:sm_fbr_ref} we have
\begin{align*}
\p \left( \sigma \in \bigcup_{z_{-1}^{a,b}} \partial \left( B_1 \left( z_{-1}^{a,b} \right) \right) \right) &= \sum_{z_{-1}^{a,b}} \p \left( \sigma \in \partial \left( B_1 \left( z_{-1}^{a,b} \right) \right) \right)\\
&\geq \sum_{z_{-1}^{a,b} : B_1 \left( z_{-1}^{a,b} \right)\subseteq \Sm \left( B_1^{a,b;\left[a:b\right]} \right) } \p \left( \sigma \in \partial \left( B_1 \left( z_{-1}^{a,b} \right) \right) \right)\\
&\geq \sum_{z_{-1}^{a,b} : B_1 \left( z_{-1}^{a,b} \right) \subseteq \Sm \left(  B_1^{a,b;\left[a:b\right]} \right) } \frac{\gamma}{2 n k^5} \p \left( \sigma \in  B_1 \left( z^{a,b} \right) \right)\\
&= \frac{\gamma}{2 n k^5} \p \left( \sigma \in \Sm \left( B_1^{a,b} \right) \right) \geq \frac{\gamma \eps}{2 n^2 k^{12}}. \qedhere
\end{align*}
\end{proof}

Next, we want to find manipulation points on the boundaries of boundaries.

Before we do this, let us divide the boundaries of the boundaries according to which direction they are in. If $\sigma \in \partial \left( B_1 \left( z_{-1}^{a,b} \right) \right)$ for some $z_{-1}^{a,b}$, then we know that there exists a ranking profile $\pi$ such that $\left( \sigma, \pi \right) \in \partial_{e} \left( B_1 \left( z_{-1}^{a,b} \right) \right)$. We know that $\sigma$ and $\pi$ differ in exactly one coordinate, say coordinate $j$; in this case we say that $\sigma$ is on the boundary of $B_1 \left( z_{-1}^{a,b} \right)$ in direction $j$, and we write $\sigma \in \partial_j \left( B_1 \left( z_{-1}^{a,b} \right) \right)$. (This notation should not be confused with that of the edge boundary.)

We can write the boundary of $B_1 \left( z_{-1}^{a,b} \right)$ as a union of boundaries in the different directions:
\[
\partial \left( B_1 \left( z_{-1}^{a,b} \right) \right) = \cup_{j=1}^{n} \partial_j \left( B_1 \left( z_{-1}^{a,b} \right) \right),
\]
but note that this is not (necessarily) a disjoint union, as a ranking profile $\sigma$ for which $\sigma \in \partial \left( B_1 \left( z_{-1}^{a,b} \right) \right)$ might lie on the boundary in multiple directions.

In particular, we differentiate between the boundary in direction 1 and the boundary in all other directions. To this end we introduce the notation
\[
\partial_{-1} \left( B_1 \left( x_{-1}^{a,b} \right) \right) := \cup_{j=2}^{n} \partial_j \left( B_1 \left( x_{-1}^{a,b} \right) \right).
\]
With this notation we have the following corollary of Corollary~\ref{cor:comparable_bdries_ref}.

\begin{corollary}\label{cor:cases_bdry_dir_ref}
If \eqref{eq:sm_fbr_ref} holds, then either
\begin{equation}\label{eq:lg_bdry_in_dir!=1}
\p \left( \sigma \in \cup_{z_{-1}^{a,b}} \partial_{-1} \left( B_1 \left( z_{-1}^{a,b} \right) \right) \right) \geq \frac{\gamma \eps}{4n^2 k^{12}}
\end{equation}
or
\begin{equation}\label{eq:lg_bdry_in_dir=1}
\p \left( \sigma \in \cup_{z_{-1}^{a,b}} \partial_{1} \left( B_1 \left( z_{-1}^{a,b} \right) \right) \right) \geq \frac{\gamma \eps}{4n^2 k^{12}}.
\end{equation}
\end{corollary}

\begin{lemma}\label{lem:manip_on_bdry_of_bdry_in_dir!=1_ref}
Suppose the ranking profile $\sigma$ is on the boundary of a fiber for $B_1^{a,b;\left[a:b\right]}$ in direction $j \neq 1$, i.e.,
\[
\sigma \in \cup_{z_{-1}^{a,b}} \partial_{-1} \left( B_1 \left( z_{-1}^{a,b} \right) \right).
\]
Then there exists a 3-manipulation point $\hat{\sigma}$ which agrees with $\sigma$ in all coordinates except perhaps coordinate 1 and some coordinate $j \neq 1$; furthermore $\hat{\sigma}_1$ is equal to $\sigma_1$ or $\left[a:b\right]\sigma_1$, except that the position of a third alternative $c$ might be shifted arbitrarily, and $\hat{\sigma}_j$ is equal to $\sigma_j$ or $z \sigma_j$ for some adjacent transposition $z \in T$, except the position of $b$ might be shifted arbitrarily.
\end{lemma}

\begin{proof}
Suppose $x_{-1}^{a,b} \left( \sigma \right) = z_{-1}^{a,b}$.
Since $\sigma \in \partial \left( B_1 \left( z_{-1}^{a,b} \right) \right) \subseteq B_1 \left( z_{-1}^{a,b} \right)$, we know that $f \left( \sigma \right) = a$, and if $\sigma' = \left[a:b\right]_1 \sigma$, then $f\left( \sigma' \right) = b$.

Let $\pi = \left( \pi_j, \sigma_{-j} \right)$ denote the ranking profile such that $\left( \sigma, \pi \right) \in \partial_{e} \left( B_1 \left( z_{-1}^{a,b} \right) \right)$. Let $\pi' := \left[a:b\right]_1 \pi$. Since $\pi \notin B_1 \left( z_{-1}^{a,b} \right)$, $\left( f\left( \pi \right), f \left( \pi' \right) \right) \neq \left( a, b \right)$. Then, by Lemma~\ref{lem:nonManipBoundary}, if $f\left( \pi \right) \neq f\left( \pi' \right)$, then one of $\pi$ and $\pi'$ is a 2-manipulation point.

So let us suppose that $f\left( \pi \right) = f \left( \pi' \right)$. If $f\left( \pi' \right) = a$, then one of $\sigma'$ and $\pi'$ is a 2-manipulation point by Lemma~\ref{lem:nonManipBoundary}, since $\pi' = z_j \sigma'$ for some adjacent transposition $z \neq \left[a:b\right]$. If $f\left( \pi \right) = b$, then similarly one of $\sigma$ and $\pi$ is a 2-manipulation point.

Finally let us suppose that $f\left(\pi \right) = c$ for some $c\notin \left\{ a, b \right\}$. In this case Lemma~\ref{lem:nonManipTriple} tells us that there exists an appropriate 3-manipulation point $\hat{\sigma}$.
\end{proof}

\begin{corollary}\label{cor:lg_bdry_in_dir!=1_gives_manip}
If \eqref{eq:lg_bdry_in_dir!=1} holds, then
\begin{equation}\label{eq:lg_bdry_in_dir!=1_gives_manip}
\p \left( \sigma \in M_3 \right) \geq \frac{\gamma \eps}{8 n^3 k^{16}}.
\end{equation}
\end{corollary}

\begin{proof}
Lemma~\ref{lem:manip_on_bdry_of_bdry_in_dir!=1_ref} tells us that for every ranking profile $\sigma$ which is on the boundary of a fiber for $B_1^{a,b;\left[a:b\right]}$ in some direction $j\neq 1$, there is a 3-manipulation point $\hat{\sigma}$ ``nearby''; the lemma specifies what ``nearby'' means.

How many ranking profiles $\sigma$ may give the same $\hat{\sigma}$? At most $2 n k^4$, which comes from the following: $\sigma$ and $\hat{\sigma}$ agree in all coordinates except maybe two, one of which is the first coordinate; there are $n-1 < n$ possibilities for the other coordinate; in the first coordinate, $\hat{\sigma}_1$ is either $\sigma_1$ or $\left[a:b\right] \sigma_1$ (giving 2 possibilities), while some alternative $c$ ($k-2 < k$ possibilities) might be shifted arbitrarily (at most $k$ possibilities); in the other coordinate $j \neq 1$, $\hat{\sigma}_j$ is equal to $\sigma_j$ or $z \sigma_j$ for some adjacent transposition $z\in T$ (at most $k$ possibilities), except $b$ might be shifted arbitrarily ($k$ possibilities).

So putting this result from Lemma~\ref{lem:manip_on_bdry_of_bdry_in_dir!=1_ref} together with \eqref{eq:lg_bdry_in_dir!=1} yields \eqref{eq:lg_bdry_in_dir!=1_gives_manip}.
\end{proof}

The remaining case we have to deal with is when \eqref{eq:lg_bdry_in_dir=1} holds.

\begin{lemma}\label{lem:bdry_of_bdry_in_coord1_ref}
Suppose the ranking profile $\sigma$ is on the boundary of a fiber for $B_1^{a,b;\left[a:b\right]}$ in direction 1, i.e.,
\[
\sigma \in \cup_{z_{-1}^{a,b}} \partial_{1} \left( B_1 \left( z_{-1}^{a,b} \right) \right).
\]
Then either $\sigma \in \LD_1 \left( a, b \right)$, or there exists a 3-manipulation point $\hat{\sigma}$ which agrees with $\sigma$ in all coordinates except perhaps in coordinate 1; furthermore $\hat{\sigma}_1$ is equal to $\sigma_1$, or $\left[a:b\right] \sigma_1$ except that the position of a third alternative $c$ might be shifted arbitrarily.
\end{lemma}

\begin{proof}
Just like the proof of Lemma~\ref{lem:bdry_of_B}.
\end{proof}

The following corollary then tells us that either we have found many 3-manipulation points, or we have many local dictators on three alternatives in coordinate 1.

\begin{corollary}\label{cor:sm_fbr_ref_last_cor}
Suppose \eqref{eq:lg_bdry_in_dir=1} holds. Then either
\begin{equation}\label{eq:many_loc_dict}
\sum_{c\notin\left\{a,b\right\}} \p \left( \sigma \in \LD_1^{\left\{a,b,c\right\}} \right) = \p\left( \sigma \in \LD_1 \left( a, b \right) \right) \geq \frac{\gamma \eps}{8 n^2 k^{12}}
\end{equation}
or
\[
\p \left( \sigma \in M_3 \right) \geq \frac{\gamma \eps}{16 n^2 k^{14}}.
\]
\end{corollary}

\subsection{Dealing with local dictators}\label{sec:loc_dict} 

So the remaining case we have to deal with in this small fiber case is when \eqref{eq:many_loc_dict} holds, i.e., we have many local dictators in coordinate 1.

\begin{lemma}\label{lem:loc_dict_abc_to_top}
Suppose $\sigma \in \LD_1^{\left\{a,b,c\right\}}$ for some alternative $c\notin \left\{a,b\right\}$. Define $\sigma' := \left( \sigma'_1, \sigma_{-1} \right)$ by letting $\sigma'_1$ be equal to $\sigma_1$ except that the block of $a$, $b$ and $c$ is moved to the top of the coordinate. Then
\begin{itemize}
\item either $\sigma' \in \LD_1^{\left\{a,b,c\right\}}$,
\item or there exists a 3-manipulation point $\hat{\sigma}$ which agrees with $\sigma$ in all coordinates except perhaps in coordinate 1; furthermore $\hat{\sigma}_1$ is equal to $\sigma_1$ except that the position of $a$, $b$ and $c$ might be shifted arbitrarily.
\end{itemize}
\end{lemma}

\begin{proof}
Just like the proof of Lemma~\ref{lem:loc_dict_abc_to_top_1vot}.
\end{proof}

\begin{corollary}\label{cor:many_dict_at_top}
If \eqref{eq:many_loc_dict} holds, then either
\begin{equation}\label{eq:loc_dict_abc_on_top}
\sum_{c \notin \left\{ a, b \right\}} \p \left( \sigma \in LD_1^{\left\{a,b,c\right\}}, \left\{ \sigma_1 \left( 1 \right), \sigma_1 \left( 2 \right), \sigma_1 \left( 3 \right) \right\} = \left\{a,b,c\right\} \right) \geq \frac{\gamma \eps}{16 n^2 k^{13}}
\end{equation}
or
\[
\p \left( \sigma \in M_3 \right) \geq \frac{\gamma \eps}{16 n^2 k^{15}}.
\]
\end{corollary}

\begin{proof}
Just like the proof of Corollary~\ref{cor:many_dict_at_top_1vot}.
\end{proof}

Now \eqref{eq:loc_dict_abc_on_top} is equivalent to
\begin{equation}\label{eq:loc_dict_abc_on_top2}
\sum_{c \notin \left\{ a, b \right\}} \p \left( \sigma \in LD_1^{\left\{a,b,c\right\}}, \left( \sigma_1 \left( 1 \right), \sigma_1 \left( 2 \right), \sigma_1 \left( 3 \right) \right) = \left(a,b,c\right) \right) \geq \frac{\gamma \eps}{96 n^2 k^{13}}.
\end{equation}
We know that
\[
\p \left( \left( \sigma_1 \left( 1 \right), \sigma_1 \left( 2 \right), \sigma_1 \left( 3 \right) \right) = \left(a,b,c\right) \right) = \frac{1}{k \left( k-1 \right) \left( k-2 \right)} \leq \frac{6}{k^3},
\]
and so \eqref{eq:loc_dict_abc_on_top2} implies (recall Definition~\ref{def:cond})
\begin{equation}\label{eq:loc_dict_abc_on_top3}
\sum_{c \notin \left\{ a, b \right\}} \p_1^{\left(a,b,c\right)} \left( \sigma \in LD_1^{\left\{a,b,c\right\}}  \right) \geq \frac{\gamma \eps}{576 n^2 k^{10}}.
\end{equation}

Now fix an alternative $c \notin \left\{a,b\right\}$ and define the graph $G_{\left( a, b, c \right)} = \left( V_{\left( a, b, c \right)}, E_{\left( a, b, c \right)} \right)$ to have vertex set
\[
V_{\left( a, b, c \right)} := \left\{ \sigma \in S_k^n : \left( \sigma_1 \left( 1 \right), \sigma_1 \left( 2 \right), \sigma_1 \left( 3 \right) \right) = \left(a,b,c\right) \right\}
\]
and for $\sigma, \pi \in V_{\left( a, b, c \right)}$ let $\left( \sigma, \pi \right) \in E_{\left( a, b, c \right)}$ if and only if $\sigma$ and $\pi$ differ in exactly one coordinate, and by an adjacent transposition in this coordinate. So $G_{\left( a, b, c \right)}$ is the subgraph of the refined rankings graph induced by the vertex set $V_{\left( a, b, c \right)}$.

Let
\[
T_1 \left( a, b, c \right) := V_{\left( a, b, c \right)} \cap LD_1^{\left\{a,b,c\right\}},
\]
and let $\partial_e \left( T_1 \left( a, b, c \right) \right)$ and $\partial \left( T_1 \left( a, b, c \right) \right)$ denote the edge and vertex boundary of $T_1 \left( a, b, c \right)$ in $G_{\left( a, b, c \right)}$, respectively.

The next lemma shows that unless $T_1 \left( a, b, c \right)$ is almost all of $V_{\left( a, b, c \right)}$, the size of the boundary $\partial \left( T_1 \left( a, b, c \right) \right)$ is comparable to the size of $T_1 \left( a, b, c \right)$. 

\begin{lemma}\label{lem:lg_bdry_for_loc_dict}
Let $c \notin \left\{a,b \right\}$ be arbitrary. Write $T \equiv T_1 \left( a, b, c \right)$ for simplicity. If $\p_1^{\left(a,b,c\right)} \left( \sigma \in T \right) \leq 1 - \delta$, then
\begin{equation}\label{eq:lg_bdry_for_loc_dict}
\p_1^{\left( a, b, c \right)} \left( \sigma \in \partial \left( T \right) \right) \geq \frac{\delta}{n k^3} \p_1^{\left( a, b, c \right)} \left( \sigma \in T \right).
\end{equation}
\end{lemma}

\begin{proof}
The proof is essentially the same as the proof of Lemma~\ref{lem:lg_bdry_for_loc_dict_1vot}, with a slight modification to deal with $n$ coordinates. Let $T^c = V_{\left( a, b, c \right)} \setminus T \left( a, b, c \right)$. For every $\left( \sigma, \sigma' \right) \in T \times T^c$ we define a canonical path from $\sigma$ to $\sigma'$ by applying a path construction in each coordinate one by one, and then concatenating these paths. In all coordinates we apply the path construction of~\cite[Proposition 6.4.]{IsKiMo:12}, but in the first coordinate we only apply it to alternatives $\left[k \right] \setminus \left\{ a, b, c \right\}$.

The analysis of this construction is done in exactly the same way as in Lemma~\ref{lem:comparable_bdries_ref}; in the end we get that $\left| \partial_e \left( T \right) \right| \geq \frac{\delta}{k^2} \left| T \right|$. Now every vertex in $V_{\left(a,b,c\right)}$ has no more than $nk$ neighbors, which implies \eqref{eq:lg_bdry_for_loc_dict}.
\end{proof}

The next lemma tells us that if $\sigma$ is on the boundary of a set of local dictators on $\left\{a,b,c \right\}$ for some alternative $c \notin \left\{a,b\right\}$ in coordinate 1, then there is a 4-manipulation point $\hat{\sigma}$ which is close to $\sigma$. The proof is similar to that of Lemma~\ref{lem:manip_pts_on_bdry_of_loc_dict_1vot}, but we have to take care of all $n$ coordinates.

\begin{lemma}\label{lem:manip_pts_on_bdry_of_loc_dict}
Suppose $\sigma \in \partial \left( T_1 \left( a,b, c \right) \right)$ for some $c\notin \left\{a,b\right\}$. We distinguish two cases, based on the number of alternatives.

If $k=3$, then there exists a (3-)manipulation point $\hat{\sigma}$ which differs from $\sigma$ in at most two coordinates, one of them being the first coordinate.

If $k \geq 4$, then there exists a 4-manipulation point $\hat{\sigma}$ which differs from $\sigma$ in at most two coordinates, one of them being the first coordinate; furthermore, $\hat{\sigma}_1$ is equal to $\sigma_1$ except that the order of the block of $a$, $b$ and $c$ might be rearranged and an additional alternative $d$ might be shifted arbitrarily; and in the other coordinate, call it $j$, $\hat{\sigma}_j$ is equal to $\sigma_j$ except perhaps $a$, $b$ and $c$ are shifted arbitrarily.
\end{lemma}

\begin{proof}
Let $\pi$ be the ranking profile such that $\left( \sigma, \pi \right) \in \partial_e \left( T_1 \left( a, b , c \right) \right)$, let $j$ be the coordinate in which they differ, and let $z$ be the adjacent transposition in which they differ, i.e., $\pi = z_j \sigma$. Since $\pi \notin T_1 \left( a, b, c \right)$, there exists a reordering of the block of $a$, $b$, and $c$ at the top of $\pi_1$ such that the outcome of $f$ is not the top ranked alternative in coordinate 1. Call the resulting vector $\pi'_1$, and let $\pi':= \left( \pi'_1, \pi_{-1} \right)$. W.l.o.g.\ let us assume that $\pi'_1 \left( 1 \right) = a$. Let us also define $\sigma' := z_j \pi'$. We distinguish two cases: $j=1$ and $j\neq 1$.

If $j=1$ (in which case we must have $k \geq 5$), $\pi'$ is a 2-manipulation point, since $f\left( \sigma' \right) = a$.

If $j\neq 1$, then there are various cases to consider. If the adjacent transposition $z$ does not move $a$, then either $\pi'$ or $\sigma'$ is a 2-manipulation point. So let us suppose that $z = \left[a:d\right]$ for some $d \neq a$.

Clearly we must have $f\left( \pi' \right) = d$, or else $\pi'$ or $\sigma'$ is a 2-manipulation point. Suppose first that $d\in \left\{ b, c \right\}$. W.l.o.g.\ suppose that $d=b$.

Then take alternative $c$ in coordinate $j$ of both $\sigma'$ and $\pi'$, and bubble it to the block of $a$ and $b$ simultaneously in the two ranking profiles. If along the way the value of the outcome of the SCF $f$ changes from $a$ or $b$, respectively, then we have a 2-manipulation point by Lemma~\ref{lem:nonManipBoundary}. Otherwise, we now have $a$, $b$, and $c$ adjacent in both coordinates 1 and $j$. Now rearranging the order of the blocks of $a$, $b$, and $c$ in these two coordinates (which can be done using adjacent transpositions), we either get a 2-manipulation point by Lemma~\ref{lem:nonManipBoundary}, or we can define a new SCF on two voters and three alternatives, $a$, $b$, and $c$. This SCF takes on three values and it is also not hard to see that the outcome is not only a function of the first coordinate, so by the Gibbard-Satterthwaite theorem we know that this SCF has a manipulation point, which is a 3-manipulation point of the original SCF $f$.

Now let us look at the case when $d \notin \left\{b,c\right\}$. In this case we do something similar to what we just did in the previous paragraph. In both $\sigma'$ and $\pi'$, first bubble up alternative $d$ in coordinate 1 up to the block of $a$, $b$, and $c$, and then bubble $b$ and $c$ in coordinate $j$ to the block of $a$ and $d$. All of this using adjacent transpositions. If the value of the outcome of the SCF $f$ changes from $a$ or $d$, respectively, at any time along the way, then we have a 2-manipulation point by Lemma~\ref{lem:nonManipBoundary}. Otherwise, we now have $a$, $b$, $c$ and $d$ adjacent in both coordinates 1 and $j$, and we can apply the same trick to find a 4-manipulation point, using the Gibbard-Satterthwaite theorem.
\end{proof}

The next corollary puts together Corollary~\ref{cor:many_dict_at_top} and Lemmas~\ref{lem:lg_bdry_for_loc_dict} and~\ref{lem:manip_pts_on_bdry_of_loc_dict}.

\begin{corollary}\label{cor:manip_by_bdry_loc_dict}
Suppose \eqref{eq:loc_dict_abc_on_top} holds. Then if for every $c \notin \left\{a,b\right\}$ we have $\p_1^{\left(a,b,c\right)} \left( \sigma \in T_1 \left( a,b, c \right) \right) \leq 1 - \frac{\eps}{100 k}$, then
\[
\p \left( \sigma \in M_4 \right) \geq \frac{\gamma \eps^2}{345600 n^4 k^{22}}.
\]
\end{corollary}

\begin{proof}
We know that \eqref{eq:loc_dict_abc_on_top} implies
\[
\sum_{c \notin \left\{a,b \right\}} \p_1^{a,b,c} \left( \sigma \in T_1 \left( a,b,c\right) \right) \geq \frac{\gamma \eps}{576 n^2 k^{10}}.
\]
Now then using the assumptions, Lemma~\ref{lem:lg_bdry_for_loc_dict} with $\delta = \frac{\eps}{100k}$ and Lemma~\ref{lem:manip_pts_on_bdry_of_loc_dict}, we have
\begin{align*}
\p \left( \sigma \in M_4 \right) &\geq \sum_{c \neq \left\{a,b \right\}} \frac{1}{k^3} \p_1^{\left(a,b,c \right)} \left( \sigma \in M_4 \right) \geq \sum_{c\notin \left\{a,b\right\}} \frac{1}{6nk^8} \p_1^{\left( a, b, c \right)} \left( \sigma \in \partial \left( T_1 \left( a,b,c \right) \right) \right)\\
&\geq \sum_{c \notin \left\{a,b\right\}} \frac{\eps}{600 n^2 k^{12}} \p_1^{\left( a, b , c \right)} \left( \sigma \in T_1 \left( a, b, c \right) \right) \geq \frac{\gamma \eps^2}{345600 n^4 k^{22}}. \qedhere
\end{align*}
\end{proof}

So again we are left with one case to deal with: if there exists an alternative $c \notin \left\{ a, b \right\}$ such that $\p_1^{\left(a,b,c\right)} \left( \sigma \in T_1 \left( a,b, c \right) \right) > 1 - \frac{\eps}{100 k}$. Define a subset of alternatives $K \subseteq \left[k \right]$ in the following way:
\[
K := \left\{a,b \right\} \cup \left\{ c \in \left[k \right] \setminus \left\{a,b \right\} : \p_1^{\left(a,b,c\right)} \left( \sigma \in T_1 \left( a,b, c \right) \right) > 1 - \frac{\eps}{100 k} \right\}.
\]
In addition to $a$ and $b$, $K$ contains those alternatives that whenever they are at the top of coordinate 1 with $a$ and $b$, they form a local dictator with high probability.

So our assumption now is that $\left| K \right| \geq 3$.

Our next step is to show that unless we have many manipulation points, for any alternative $c \in K$, conditioned on $c$ being at the top of the first coordinate, the outcome of $f$ is $c$ with probability close to 1.

\begin{lemma}\label{lem:cond_on_top}
Let $c \in K$. Then either
\begin{equation}\label{eq:cond_on_top1}
\p_1^{\left( c \right)} \left( f \left( \sigma \right) = c \right) \geq 1 - \frac{\eps}{50 k},
\end{equation}
or
\begin{equation}\label{eq:else_2manip}
\p \left( \sigma \in M_2 \right) \geq \frac{\eps}{100 k^4}.
\end{equation}
\end{lemma}

\begin{proof}
Just like the proof of Lemma~\ref{lem:cond_on_top_1vot}.
\end{proof}

We now deal with alternatives that are not in $K$: either we have many manipulation points, or for any alternative $d \notin K$, the outcome of $f$ is \emph{not} $d$ with probability close to 1.

\begin{lemma}\label{lem:d_notin_K_ref}
Let $d \notin K$. If $\p \left( f \left( \sigma \right) = d \right) \geq \frac{\eps}{4k}$, then
\[
\p \left( \sigma \in M_4 \right) \geq \frac{\eps^2}{10^6 n^2 k^{13}}.
\]
\end{lemma}

\begin{proof}
The proof is very similar to that of Lemma~\ref{lem:d_notin_K_1vot}: we do the same steps in the first coordinate as done in the proof of Lemma~\ref{lem:d_notin_K_1vot}, and the fact that we have $n$ coordinates only matters at the very end.

Let $\sigma$ be such that $f\left( \sigma \right) = d$. We will keep coordinates 2 through $n$ to be fixed as $\sigma_{-1}$ throughout the proof. By bubbling alternatives $d$, $a$, and $b$ in the first coordinate, we can define $\sigma'$, $\sigma^{\left(d, b, a \right)}$, $\sigma^{\left(d, a, b \right)}$, $\sigma^{\left( a, b, d \right)}$, $\sigma^{\left( a, d, b \right)}$, $\sigma^{\left( b, a, d \right)}$, and $\sigma^{\left( b, d, a \right)}$ just as in  the proof of Lemma~\ref{lem:d_notin_K_1vot}.  Again, we can show that either
\[
\p \left( \sigma \in M_2 \right) \geq \frac{\eps}{1600 k^3},
\]
in which case we are done, or
\begin{equation}\label{eq:loc_dict_abd}
\p_1^{\left( a, b, d \right)} \left( \sigma^{\left(a,b,d\right)} \in LD_1^{\left\{a,b,d\right\}} \right) = \p \left( \sigma: \sigma^{\left( a, b, d \right)} \in LD_1^{\left\{a,b,d\right\}} \right) \geq \frac{\eps}{1600 k}.
\end{equation}
Define $G_{\left(a,b,d\right)}$ and $T_{\left( a, b, d \right)}$ analogously to $G_{\left( a, b, c \right)}$ and $T_{\left( a, b, c \right)}$, respectively.

Suppose that \eqref{eq:loc_dict_abd} holds. We also know that $d\notin K$, so Lemma~\ref{lem:lg_bdry_for_loc_dict} applies, and then Lemma~\ref{lem:manip_pts_on_bdry_of_loc_dict} shows us how to find manipulation points. We can put these arguments together, just like in the proof of Corollary~\ref{cor:manip_by_bdry_loc_dict}, to show what we need:
\begin{align*}
\p \left( \sigma \in M_4 \right) &\geq \frac{1}{k^3} \p_1^{\left(a,b,d \right)} \left( \sigma \in M_4 \right) \geq \frac{1}{6nk^8} \p_1^{\left( a, b, d \right)} \left( \sigma \in \partial \left( T_1 \left( a,b,d \right) \right) \right)\\
&\geq \frac{\eps}{600 n^2 k^{12}} \p_1^{\left( a, b , d \right)} \left( \sigma \in T_1 \left( a, b, d \right) \right) \geq \frac{\eps^2}{10^6 n^2 k^{13}}. \qedhere
\end{align*}
\end{proof}

Putting together the results of the previous lemmas, there is only one case to be covered, which is covered by the following final lemma. Basically, this lemma says that unless there are enough manipulation points, our function is close to a dictator in the first coordinate, on the subset of alternatives $K$.
\begin{lemma}\label{lem:final_loc_dict}
Recall that we assume that $\Dist \left( f, \overline{\NONMANIP} \right) \geq \eps$. Furthermore assume that $\left| K \right| \geq 3$, for every $c \in K$ we have
\begin{equation}\label{eq:dict_cond_on_top}
\p_1^{\left( c \right)} \left( f \left( \sigma \right) = c \right) \geq 1 - \frac{\eps}{50k},
\end{equation}
and for every $d \notin K$ we have
\[
\p \left( f\left( \sigma \right) = d \right) \leq \frac{\eps}{4k}.
\]
Then
\begin{equation}\label{eq:final_manip}
\p \left( \sigma \in M_2 \right) \geq \frac{\eps}{4k^2}.
\end{equation}
\end{lemma}

\begin{proof}
Just like the proof of Lemma~\ref{lem:final_loc_dict}.
\end{proof}

To conclude the proof in the small fiber case, inspect all the lower bounds for $\p \left( \sigma \in M_4 \right)$ obtained in Section~\ref{sec:sm_fbr_ref}, and recall that $\gamma = \frac{\eps^3}{10^3 n^3 k^{24}}$.

\subsection{Large fiber case}\label{sec:lg_fbr_ref} 

We now deal with the large fiber case, when \eqref{eq:lg_fbr_ref} holds for both boundaries, i.e., when
\[
\p \left( \sigma \in \Lg \left( B_1^{a,b;\left[a:b\right]} \right) \right) \geq \frac{\eps}{nk^7}
\]
and
\[
\p \left( \sigma \in  \Lg \left( B_2^{c,d;\left[c:d\right]} \right)  \right) \geq \frac{\eps}{nk^7}.
\]
We differentiate between two cases: whether $d \in \left\{a,b\right\}$ or $d \notin \left\{a,b\right\}$. 

\subsubsection{Case 1}
Suppose $d \in \left\{a,b \right\}$, in which case w.l.o.g.\ we may assume that $d = a$. That is, in the rest of this case we may assume that
\begin{equation}\label{eq:lg_fbr_d=a_1}
\p \left( \sigma \in  \Lg \left( B_1^{a,b;\left[a:b\right]} \right) \right) \geq \frac{\eps}{nk^7}
\end{equation}
and
\begin{equation}\label{eq:lg_fbr_d=a_2}
\p \left( \sigma \in \Lg \left( B_2^{a,c;\left[a:c\right]} \right) \right) \geq \frac{\eps}{nk^7}.
\end{equation}

First, let us look at only the boundary between $a$ and $b$ in direction 1. Let us fix a vector $z_{-1}^{a,b}$ which gives a large fiber $B_1 \left( z_{-1}^{a,b} \right)$ for the boundary $B_1^{a,b;\left[a:b\right]}$, i.e., we know that
\begin{equation}\label{eq:lg_fbr_def2}
\p \left( \sigma \in B_1 \left( z_{-1}^{a,b} \right) \, \middle| \, \sigma \in \bar{F} \left( z_{-1}^{a,b} \right) \right) \geq 1 - \gamma.
\end{equation}

Our basic goal in the following will be to show that conditional on the ranking profile $\sigma$ being in the fiber $F \left( z_{-1}^{a,b} \right)$ (but not necessarily in $\bar{F} \left( z_{-1}^{a,b} \right)$), with high probability the outcome of the vote is $\tp_{\left\{a,b\right\}} \left( \sigma_1 \right)$, or else we have a lot of 2-manipulation points or local dictators on three alternatives in coordinate 1.

Our first step towards this is the following.
\begin{lemma}\label{lem:cond_ab_top_still_lg_fbr}
Suppose $z_{-1}^{a,b}$ gives a large fiber $B_1 \left( z_{-1}^{a,b} \right)$ for the boundary $B_1^{a,b;\left[a:b\right]}$. Then
\begin{equation}\label{eq:cond_ab_top_still_lg_fbr}
\p_1^{\left(a,b\right)} \left( \sigma \in B_1 \left( z_{-1}^{a,b} \right) \, \middle| \, \sigma \in F\left( z_{-1}^{a,b} \right) \right) \geq 1 - k \gamma.
\end{equation}
\end{lemma}

\begin{proof}
We know that
\[
\p \left( \left( \sigma_1 \left( 1 \right), \sigma_1 \left( 2 \right) \right) = \left( a, b \right) \, \middle| \, \sigma \in \bar{F} \left( z_{-1}^{a,b} \right) \right) = \frac{1}{k-1},
\]
and so
\begin{align*}
\p_1^{\left(a,b\right)} &\left( \sigma \notin B_1 \left( z_{-1}^{a,b} \right)\, \middle| \, \sigma \in F \left( z_{-1}^{a,b} \right) \right) = \p_1^{\left(a,b\right)} \left( \sigma \notin B_1 \left( z_{-1}^{a,b} \right)\, \middle| \, \sigma \in \bar{F}\left( z_{-1}^{a,b} \right) \right)\\
&= \left( k - 1 \right) \p \left( \sigma \notin B_1 \left( z_{-1}^{a,b} \right), \left( \sigma_1 \left( 1 \right), \sigma_1 \left( 2 \right) \right) = \left( a, b \right) \, \middle| \, \sigma \in \bar{F} \left( z_{-1}^{a,b} \right) \right) \leq \left( k - 1 \right) \gamma < k \gamma. \qedhere
\end{align*}
\end{proof}

The next lemma formalizes our goal mentioned above.
\begin{lemma}\label{lem:towards_inv_hyp_contr}
Suppose $z_{-1}^{a,b}$ gives a large fiber $B_1 \left( z_{-1}^{a,b} \right)$ for the boundary $B_1^{a,b;\left[a:b\right]}$. Then either
\begin{equation}\label{eq:f=top_ab_1}
\p \left( f\left( \sigma \right) = \tp_{\left\{a,b \right\}} \left( \sigma_1 \right) \, \middle| \, \sigma \in F \left( z_{-1}^{a,b} \right) \right) \geq 1 - 2 k \gamma
\end{equation}
or
\begin{equation}\label{eq:2manip_again}
\p \left( \sigma \in M_2 \, \middle| \, \sigma \in F \left( z_{-1}^{a,b} \right) \right) \geq \frac{\gamma}{2k}
\end{equation}
or
\begin{equation}\label{eq:loc_dict_again}
\p \left( \sigma \in LD_1 \left(a,b\right) \, \middle| \, \sigma \in F \left( z_{-1}^{a,b} \right) \right) \geq \frac{\gamma}{2k}.
\end{equation}
\end{lemma}

\begin{proof}
The proof of this lemma is essentially the same as that of Lemma~\ref{lem:lg_fbr_final_1vot}, there are only two slight differences. First, we use Lemma~\ref{lem:cond_ab_top_still_lg_fbr} to know that \eqref{eq:cond_ab_top_still_lg_fbr} holds. Second, we take $\sigma \in F \left( z_{-1}^{a,b} \right)$ to be uniform, and we stay on the fiber $F \left( z_{-1}^{a,b} \right)$ throughout the proof: we modify only the first coordinate throughout the proof, in the same way as we did for Lemma~\ref{lem:lg_fbr_final_1vot}. We omit the details.
\end{proof}

Now this lemma holds for all vectors $z_{-1}^{a,b}$ which give a large fiber $B_1 \left( z_{-1}^{a,b} \right)$ for the boundary $B_1^{a,b;\left[a:b\right]}$. By \eqref{eq:lg_fbr_d=a_1} we know that
\[
\p \left( \sigma : B_1 \left( x_{-1}^{a,b} \left( \sigma \right) \right) \text{ is a large fiber} \right)\geq \frac{\eps}{nk^7}.
\]
Now if \eqref{eq:2manip_again} holds for at least a third of the vectors $z_{-1}^{a,b}$ that give a large fiber $B_1 \left( z_{-1}^{a,b} \right)$, then it follows that
\[
\p \left( \sigma \in M_2 \right) \geq \frac{\gamma \eps}{6n k^8}
\]
and we are done. If \eqref{eq:loc_dict_again} holds for at least a third of the vectors $z_{-1}^{a,b}$ that give a large fiber $B_1 \left( z_{-1}^{a,b} \right)$, then similarly we have
\[
\p \left( \sigma \in LD_1 \left( a, b \right) \right) \geq \frac{\gamma \eps}{6nk^8},
\]
which means that \eqref{eq:many_loc_dict} also holds, and so we are done by the argument in Section~\ref{sec:loc_dict}.

So the remaining case to consider is when \eqref{eq:f=top_ab_1} holds for at least a third of the vectors $z_{-1}^{a,b}$ that give a large fiber $B_1 \left( z_{-1}^{a,b} \right)$.

We can go through this same argument for the boundary between $a$ and $c$ in direction 2 as well, and either we are done because
\[
\p \left( \sigma \in M_2 \right) \geq \frac{\gamma \eps}{6n k^8}
\]
or
\[
\p \left( \sigma \in LD_2 \left( a, c \right) \right) \geq \frac{\gamma \eps}{6nk^8},
\]
or for at least a third of the vectors $z_{-2}^{a,c}$ that give a large fiber $B_2 \left( z_{-2}^{a,c} \right)$ we have
\[
\p \left( f \left( \sigma \right) = \tp_{\left\{a,c \right\}} \left( \sigma_2 \right) \, \middle| \, \sigma \in F \left( z_{-2}^{a,c} \right) \right) \geq 1 - 2 k \gamma.
\]

So basically our final case is if
\begin{equation}\label{eq:lg_fbr_final_case_1}
\p \left( \sigma \in F_1^{a,b} \right) \geq \frac{\eps}{3 n k^7}
\end{equation}
and also
\begin{equation}\label{eq:lg_fbr_final_case_2}
\p \left( \sigma \in F_2^{a,c} \right) \geq \frac{\eps}{3 n k^7}.
\end{equation}
Notice that being in the set $F_1^{a,b}$ only depends on the vector $x^{a,b} \left( \sigma \right)$ of preferences between $a$ and $b$, and similarly being in the set $F_2^{a,c}$ only depends on the vector $x^{a,c} \left( \sigma \right)$ of preferences between $a$ and $c$. We know that $\left\{\left( x_i^{a,b} \left( \sigma \right), x_i^{a,c} \left( \sigma \right) \right) \right\}_{i=1}^{n}$ are independent, and for any given $i$ we know that $\left| \E \left( x_i^{a,b} \left( \sigma \right) x_i^{a,c} \left( \sigma \right) \right) \right| = \frac{1}{3}$. Hence we can apply reverse hypercontractivity (Lemma~\ref{lem:invhypcontr}), to get the following result.

\begin{lemma}\label{lem:applying_inv_hyp_contr}
If \eqref{eq:lg_fbr_final_case_1} and \eqref{eq:lg_fbr_final_case_2} hold, then also
\begin{equation}\label{eq:result_of_inv_hyp_contr}
\p \left( \sigma \in F_1^{a,b} \cap F_2^{a,c} \right) \geq \frac{\eps^3}{27 n^3 k^{21}}.
\end{equation}
\end{lemma}

\begin{proof}
See above.
\end{proof}

The next and final lemma then concludes that we have lots of manipulation points.
\begin{lemma}\label{lem:inv_hyp_contr_end}
Suppose \eqref{eq:result_of_inv_hyp_contr} holds. Then
\begin{equation}\label{eq:final_manip2}
\p \left( \sigma \in M_3 \right) \geq \frac{\eps^3}{54n^3 k^{27}} - \frac{9 \gamma}{k^3}.
\end{equation}
\end{lemma}

\begin{proof}
First let us define two events:
\begin{align*}
I_1 &:= \left\{ \sigma : f \left( \sigma \right) = \tp_{\left\{a,b \right\}} \left( \sigma_1 \right) \right\}\\
I_2 &:= \left\{ \sigma : f \left( \sigma \right) = \tp_{\left\{a,c \right\}} \left( \sigma_2 \right) \right\}.
\end{align*}
Using similar estimates as previously in Lemma~\ref{lem:lg_fbr_1}, we have
\begin{multline*}
\p \left( \sigma \in I_1 \cap I_2 \cap F_1^{a,b} \cap F_2^{a,c} \right) \geq \p \left( \sigma \in F_1^{a,b} \cap F_2^{a,c} \right)\\
 - \p \left( \sigma \notin I_1, \sigma \in F_1^{a,b} \cap F_2^{a,c} \right) - \p \left( \sigma \notin I_2, \sigma \in F_1^{a,b} \cap F_2^{a,c} \right).
\end{multline*}
The first term is bounded below via \eqref{eq:result_of_inv_hyp_contr}, while the other two terms can be bounded using the definition of $F_1^{a,b}$ and $F_2^{a,c}$, respectively:
\[
\p \left( \sigma \notin I_1, \sigma \in F_1^{a,b} \cap F_2^{a,c} \right) \leq \p \left( \sigma \notin I_1, \sigma \in F_1^{a,b} \right) \leq \p \left( \sigma \notin I_1 \, \middle| \, \sigma \in F_1^{a,b} \right) \leq 2 k \gamma,
\]
and similarly for the other term. Putting everything together gives us
\[
\p \left( \sigma \in I_1 \cap I_2 \cap F_1^{a,b} \cap F_2^{a,c} \right) \geq \frac{\eps^3}{27 n^3 k^{21}} - 4k \gamma.
\]
If $\sigma \in I_1 \cap I_2 \cap F_1^{a,b} \cap F_2^{a,c}$, then clearly we must have $f\left( \sigma \right) = a$, and therefore $x_1^{a,b} \left( \sigma \right) = 1$ and $x_2^{a,c} \left( \sigma \right) = 1$. Now define $\sigma'$ from $\sigma$ by bubbling up $b$ in coordinate 1 to just below $a$, and bubbling up $c$ in coordinate 2 to just below $a$. Either we encounter a 2-manipulation point along the way, or the outcome is still $a$: $f\left( \sigma' \right) = a$. If we encounter a 2-manipulation point along the way for at least half of such ranking profiles, then we are done:
\[
\p \left( \sigma \in M_2 \right) \geq \frac{1}{k^2} \left( \frac{\eps^3}{54 n^3 k^{21}} - 2k \gamma \right) = \frac{\eps^3}{54 n^3 k^{23}} - \frac{2 \gamma}{k}.
\]
Otherwise, we may assume that
\[
\p \left( \sigma \in I_1 \cap I_2 \cap F_1^{a,b} \cap F_2^{a,c}, f\left( \sigma' \right) = a  \right) \geq \frac{\eps^3}{54 n^3 k^{21}} - 2k \gamma.
\]
In this case define $\tilde{\sigma}' := \left[a:b\right]_1 \sigma'$ and $\tilde{\sigma}'' := \left[a:c\right]_2 \sigma'$.
If $f \left( \tilde{\sigma}' \right) \notin \left\{a,b \right\}$ or $f \left( \tilde{\sigma}'' \right) \notin \left\{a,c \right\}$, then we automatically have that one of $\sigma', \tilde{\sigma}', \tilde{\sigma}''$ is a 2-manipulation point. If $f \left( \tilde{\sigma}' \right) = b$ and $f\left( \tilde{\sigma}'' \right) = c$, then by Lemma~\ref{lem:nonManipTriple} we know that there exists a 3-manipulation point $\hat{\sigma}$ which agrees with $\sigma$ except perhaps $a$, $b$, and $c$ could be arbitrarily shifted in the first two coordinates. The final case is when $a \in \left\{ f\left( \tilde{\sigma}' \right), f \left( \tilde{\sigma}'' \right) \right\}$. But we now show that this has small probability, and therefore \eqref{eq:final_manip2} follows.

First let us look at the case of $f \left( \tilde{\sigma}' \right) = a$. We have
\begin{multline*}
\p \left( \sigma \in I_1 \cap I_2 \cap F_1^{a,b} \cap F_2^{a,c}, f\left( \sigma' \right) = a, f\left( \tilde{\sigma}' \right) = a  \right)\\
\begin{aligned}
&= \sum_{z_{-1}^{a,b} : F \left( z_{-1}^{a,b} \right) \subseteq F_1^{a,b}} \p \left( \sigma \in I_1 \cap I_2 \cap F \left( \left( 1, z_{-1}^{a,b} \right) \right) \cap F_2^{a,c}, f\left( \sigma' \right) = a, f\left( \tilde{\sigma}' \right) = a  \right)\\
&= \sum_{z_{-1}^{a,b} : F \left( z_{-1}^{a,b} \right) \subseteq F_1^{a,b}} \p \left( \sigma \in I_1 \cap I_2 \cap F_2^{a,c}, f\left( \sigma' \right) = a, f\left( \tilde{\sigma}' \right) = a  \, \middle| \,  \sigma \in F \left( \left( 1, z_{-1}^{a,b} \right) \right) \right) \p \left( \sigma \in  F \left( \left( 1, z_{-1}^{a,b} \right) \right) \right)\\
&\leq \sum_{z_{-1}^{a,b} : F \left( z_{-1}^{a,b} \right) \subseteq F_1^{a,b}} \p \left( \sigma :  f\left( \tilde{\sigma}' \right) = a  \, \middle| \,  \sigma \in F \left( \left( 1, z_{-1}^{a,b} \right) \right) \right) \p \left( \sigma \in  F \left( \left( 1, z_{-1}^{a,b} \right) \right) \right).
\end{aligned}
\end{multline*}
Now we know that $\tilde{\sigma}' \in F \left( \left( -1, z_{-1}^{a,b} \right) \right) \subseteq F_1^{a,b}$, and we also know that
\[
\p \left( f \left( \sigma \right) \neq b \, \middle| \, \sigma \in F \left( \left( -1, z_{-1}^{a,b} \right) \right) \right) \leq 4k \gamma.
\]
The number of $\sigma$'s that give the same $\tilde{\sigma}'$ is at most $k^2$, and so we can conclude that
\[
\p \left( \sigma \in I_1 \cap I_2 \cap F_1^{a,b} \cap F_2^{a,c}, f\left( \sigma' \right) = a, f\left( \tilde{\sigma}' \right) = a  \right) \leq 4 k^3 \gamma,
\]
and similarly
\[
\p \left( \sigma \in I_1 \cap I_2 \cap F_1^{a,b} \cap F_2^{a,c}, f\left( \sigma' \right) = a, f\left( \tilde{\sigma}'' \right) = a  \right) \leq 4 k^3 \gamma,
\]
which shows that
\[
\p \left( \sigma \in M_3 \right) \geq \frac{1}{k^6} \left( \frac{\eps^3}{54 n^3 k^{21}} - 2k \gamma - 8 k^3 \gamma \right) \geq \frac{\eps^3}{54 n^3 k^{27}} - \frac{9 \gamma}{k^3}. \qedhere
\]
\end{proof}
To conclude the proof in this case, recall that we have chosen $\gamma = \frac{\eps^3}{10^3 n^3 k^{24}}$.

\subsubsection{Case 2}

First, as in the previous case, we can look at simply the boundary between $a$ and $b$ in direction 1, and conclude that either there are many manipulation points, or there are many local dictators, or \eqref{eq:lg_fbr_final_case_1} holds. This holds similarly for the boundary between $c$ and $d$ in direction 2. Finally, just as in Section~\ref{sec:lg_fbr_gen_case2}, we can show that \eqref{eq:lg_fbr_final_case_1} and \eqref{eq:lg_fbr_final_case_2} cannot hold at the same time. We omit the details.

\subsection{Proof of Theorem~\ref{thm:k_refined} concluded} 

\begin{proof}[Proof of Theorem~\ref{thm:k_refined}]

Our starting point is Lemma~\ref{lem:boundaries2}, which directly implies Lemma~\ref{lem:cases_ref} (unless there are many 2-manipulation points, in which case we are done). We then consider two cases, as indicated in Section~\ref{sec:cases_refined}.

We deal with the small fiber case in Section~\ref{sec:sm_fbr_ref}. First, Lemmas~\ref{lem:comparable_bdries_ref},~\ref{lem:manip_on_bdry_of_bdry_in_dir!=1_ref}, and~\ref{lem:bdry_of_bdry_in_coord1_ref}, and Corollaries~\ref{cor:comparable_bdries_ref},~\ref{cor:cases_bdry_dir_ref},~\ref{cor:lg_bdry_in_dir!=1_gives_manip},  and~\ref{cor:sm_fbr_ref_last_cor} imply that either there are many 3-manipulation points, or there are many local dictators on three alternatives in coordinate 1. We then deal with the case of many local dictators in Section~\ref{sec:loc_dict}. Lemma~\ref{lem:loc_dict_abc_to_top}, Corollary~\ref{cor:many_dict_at_top}, Lemmas~\ref{lem:lg_bdry_for_loc_dict},~\ref{lem:manip_pts_on_bdry_of_loc_dict}, Corollary~\ref{cor:manip_by_bdry_loc_dict}, and Lemmas~\ref{lem:cond_on_top},~\ref{lem:d_notin_K_ref}, and~\ref{lem:final_loc_dict} together show that there are many 4-manipulation points if there are many local dictators on three alternatives, and the SCF is $\eps$-far from the family of 
nonmanipulable functions.

We deal with the large fiber case in Section~\ref{sec:lg_fbr_ref}. Here Lemmas~\ref{lem:cond_ab_top_still_lg_fbr},~\ref{lem:towards_inv_hyp_contr},~\ref{lem:applying_inv_hyp_contr}, and~\ref{lem:inv_hyp_contr_end} show that if there are not many local dictators on three alternatives, then there are many 3-manipulation points. In the case when there are many local dictators, we refer back to Section~\ref{sec:loc_dict} to conclude the proof.
\end{proof}

\section{To truly nonmanipulable functions}\label{sec:TRUENONMANIP} 


\begin{proof}[Proof of Theorem~\ref{thm:TRUENONMANIP}]
Our assumption means that there exists a SCF $g \in \overline{\NONMANIP}$ such that $\Dist \left( f, g \right) \leq \alpha$. We distinguish two cases: either $g$ is a function of one coordinate, or $g$ takes on at most two values.

\textbf{Case 1. $g$ is a function of one coordinate.} In this case we can assume w.l.o.g.\ that $g$ is a function of the first coordinate, i.e., there exists a SCF $h : S_k \to \left[k\right]$ on one coordinate such that for every ranking profile $\sigma$, we have $g \left( \sigma \right) = h \left( \sigma_1 \right)$.

We know from the quantitative Gibbard-Satterthwaite theorem for one voter that for any $\beta$ either $\Dist \left( h, \NONMANIP \left( 1, k \right) \right) \leq \beta$, or $\p \left( \sigma \in M_3 \left( h \right) \right) \geq \frac{\beta^3}{10^5 k^{16}}$.

In the former case, we have that 
\[
\Dist \left( g, \NONMANIP \left( n, k \right) \right) \leq \Dist \left( h, \NONMANIP \left(1, k \right) \right) \leq \beta,
\]
and so consequently
\[
\Dist \left( f, \NONMANIP \left(n, k \right) \right) \leq \alpha + \beta.
\]

In the latter case, we have that
\[
\p \left( \sigma \in M_3 \left( g \right) \right) = \p \left( \sigma \in M_3 \left( h \right) \right) \geq \frac{\beta^3}{10^5 k^{16}},
\]
and so consequently
\[
\p \left( \sigma \in M_3 \left( f \right) \right) \geq \frac{\beta^3}{10^5 k^{16}} - 6 n k \alpha,
\]
since changing the outcome of a SCF at one ranking profile can change the number of 3-manipulation points by at most $6nk$. Now choosing $\beta = 100 n k^6 \alpha^{1/3}$ shows that either \eqref{eq:true_NONMANIP} or \eqref{eq:many_manip} holds.

\textbf{Case 2. $g$ is a function which takes on at most two values.} W.l.o.g.\ we may assume that the range of $g$ is $\left\{a,b\right\} \subset \left[k\right]$, i.e., for every ranking profile $\sigma \in S_k^n$ we have $g\left( \sigma \right) \in \left\{a,b\right\}$.

There is one thing we have to be careful about: even though $g$ takes on at most two values, it is not necessarily a Boolean function, since the value of $g \left( \sigma \right)$ does not necessarily depend only on the Boolean vector $x^{a,b} \left( \sigma \right)$.

We now define a function $h : S_k^n \to \left\{a,b \right\}$ that is close in some sense to $g$ and which can be viewed as a Boolean function $h : \left\{a,b \right\}^n \to \left\{a,b \right\}$ because $h\left( \sigma \right)$ depends on $\sigma$ only through $x^{a,b} \left( \sigma \right)$. (The vector $x^{a,b} \left( \sigma \right) \in \left\{ -1, 1 \right\}^{n}$ encodes which of $a$ and $b$ is preferred in each coordinate, and a vector in $\left\{a,b\right\}^{n}$ can encode the same information.) For a given ranking profile $\sigma$, let us consider the fiber on which it is on, $F \left( x^{a,b} \left( \sigma \right) \right)$, and let us define $g|_{F \left( x^{a,b} \left( \sigma \right) \right)}$ to be the restriction of $g$ to ranking profiles in the fiber $F \left( x^{a,b} \left( \sigma \right) \right)$. Then define (see Definition~\ref{def:maj})
\[
h \left( \sigma \right) := \Maj \left( g|_{F\left( x^{a,b} \left( \sigma \right) \right)} \right).
\]
By definition, $h\left( \sigma \right)$ depends on $\sigma$ only through $x^{a,b} \left( \sigma \right)$, so we may also view $h$ as a Boolean function $h : \left\{a,b \right\}^n \to \left\{a,b \right\}$.

For any given $0 < \delta < 1$, we either have $\Dist \left( g, h \right) \leq \delta$, in which case $\Dist \left( f, h \right) \leq \alpha + \delta$, or if $\Dist \left( g, h \right) > \delta$, then we show presently that
\begin{equation}\label{eq:many_manip_h}
\p \left( \sigma \in M_2 \left( f \right) \right) \geq \frac{\delta}{4 n k^5} - n k \alpha.
\end{equation}
Choosing $\delta = 8 n^2 k^6 \alpha$ then shows that either \eqref{eq:many_manip} holds, or $\Dist \left( f, h \right) \leq 9 n^2 k^6 \alpha$.

Let us now show \eqref{eq:many_manip_h}. We use a canonical path argument again, but first we divide the ranking profiles according to the fibers with respect to preference between $a$ and $b$.

Let us consider an arbitrary fiber $F\left( z^{a,b} \right)$, and divide it into two disjoint sets: into those ranking profiles for which the outcome of $g$ and $h$ agree, and those for which these outcomes are different. I.e.,
\[
F\left( z^{a,b} \right) = F^{\maj} \left( z^{a,b} \right) \cup F^{\min} \left( z^{a,b} \right),
\]
where
\begin{align*}
F^{\maj} \left( z^{a,b} \right) &= \left\{ \sigma \in F \left( z^{a,b} \right) : g \left( \sigma \right) = h \left( \sigma \right) \right\},\\
F^{\min} \left( z^{a,b} \right) &= \left\{ \sigma \in F \left( z^{a,b} \right) : g \left( \sigma \right) \neq h \left( \sigma \right) \right\}.
\end{align*}
By construction, we know that
\[
\left| F^{\min} \left( z^{a,b} \right) \right| \leq \frac{1}{2} \left| F \left( z^{a,b} \right) \right| = \frac{1}{2} \left( \frac{k!}{2} \right)^n.
\]
Now for every pair of profiles $\left( \sigma, \sigma' \right) \in F^{\min} \left( z^{a,b} \right) \times F^{\maj} \left( z^{a,b} \right)$ define a canonical path from $\sigma$ to $\sigma'$ by applying a path construction in each coordinate one by one, and then concatenating these paths. In each coordinate we apply the path construction of~\cite[Proposition 6.6.]{IsKiMo:12}: we bubble up everything except $a$ and $b$, and then finally bubble up the last two alternatives as well.

For a given edge $\left( \pi, \pi' \right) \in F^{\min} \left( z^{a,b} \right) \times F^{\maj} \left( z^{a,b} \right)$ there are at most $2 k^4 \left( \frac{k!}{2} \right)^n$ possible pairs $\left( \sigma, \sigma' \right) \in F^{\min} \left( z^{a,b} \right) \times F^{\maj} \left( z^{a,b} \right)$ such that the canonical path between $\sigma$ and $\sigma'$ defined above passes through $\left( \pi, \pi' \right)$. (This can be shown just like in the previous lemmas, e.g., Lemma~\ref{lem:comparable_bdries_ref}.) Consequently we have
\[
\left| \partial_e \left( F^{\min} \left( z^{a,b} \right) \right) \right| \geq \frac{\left| F^{\min} \left( z^{a,b} \right) \right| \left| F^{\maj} \left( z^{a,b} \right) \right|}{2 k^4 \left( \frac{k!}{2} \right)^n} \geq \frac{\left| F^{\min} \left( z^{a,b} \right) \right|}{4 k^4},
\]
where the edge boundary $\partial_e \left( F^{\min} \left( z^{a,b} \right) \right)$ is defined via the refined rankings graph restricted to the fiber $F \left( z^{a,b} \right)$. Summing this over all fibers we have that
\begin{equation}\label{eq:sum_fbrs}
\sum_{z^{a,b}} \left| \partial_e \left( F^{\min} \left( z^{a,b} \right) \right) \right| \geq \sum_{z^{a,b}} \frac{\left| F^{\min} \left( z^{a,b} \right) \right|}{4 k^4} \geq \frac{\delta}{4 k^4} \left( k! \right)^n,
\end{equation}
using the fact that $\Dist \left( g, h \right) > \delta$.

Now it is easy to see that if $\left( \sigma, \sigma' \right) \in \partial_e \left( F^{\min} \left( z^{a,b} \right) \right)$ for some $z^{a,b}$, then either $\sigma$ or $\sigma'$ is a 2-manipulation point for $g$. In the refined rankings graph every vertex (ranking profile) has $n \left( k - 1 \right) < nk$ neighbors, so each 2-manipulation point can be counted at most $nk$ times in the sum on the left hand side of \eqref{eq:sum_fbrs}, showing that
\[
\p \left( \sigma \in M_2 \left( g \right) \right) \geq \frac{\delta}{4nk^5},
\]
from which \eqref{eq:many_manip_h} follows immediately, since changing the outcome of a SCF at one ranking profile can change the number of 2-manipulation points by at most $nk$.

So either we are done because \eqref{eq:many_manip} holds, or $\Dist \left( f, h \right) \leq 9 n^2 k^6 \alpha$; suppose the latter case. Our final step is to look at $h$ as a Boolean function, and use a result on testing monotonicity~\cite{GGLRS:00}.

Denote by $\tilde{\Dist}$ the distance of $h$ when viewed as a Boolean function from the set of monotone Boolean functions. Let $0 < \eps < 1$ be arbitrary. Then either $\tilde{\Dist} \leq \eps$, in which case $\Dist \left(h, \NONMANIP \right) \leq \tilde{\Dist} \leq \eps$ and therefore $\Dist \left( f, \NONMANIP \right) \leq 9 n^2 k^6 \alpha + \eps$, or $\tilde{\Dist} > \eps$. In the latter case we show that then
\begin{equation}\label{eq:manip_eps}
\p \left( \sigma \in M_2 \left( f\right) \right) \geq \frac{2\eps}{nk} - 9 n^3 k^7 \alpha.
\end{equation}
Choosing $\eps = 5 n^4 k^8 \alpha$ then shows that either \eqref{eq:true_NONMANIP} or \eqref{eq:many_manip} holds.

Let us now show \eqref{eq:manip_eps}. Let us view $h$ as a Boolean function, and denote by $p\left( h \right)$ the fraction of pairs of strings, differing on one coordinate, that violate the monotonicity condition. Goldreich, Goldwasser, Lehman, Ron, and Samorodnitsky showed in~\cite[Theorem 2]{GGLRS:00} that $p \left( h \right) \geq \frac{\tilde{\Dist}}{n}$.

Now going back to viewing $h$ as a SCF on $k$ alternatives, this tells us that there are at least $\frac{\eps}{2} 2^n$ pairs of fibers, which differ on one coordinate, that violate monotonicity. For each such pair of fibers, whenever $a$ and $b$ are adjacent in the coordinate where the two fibers differ, we get a 2-manipulation point. Such a 2-manipulation point can be counted at most $n$ times in this way (since there are $n$ coordinates where $a$ and $b$ can be adjacent). Consequently, we have
\[
\left| M_2 \left( h \right) \right| \geq \frac{\eps}{2} \cdot 2^n \cdot 2 \left( k - 1 \right)! \left( \frac{k!}{2} \right)^{n-1} \cdot \frac{1}{n} = \frac{2\eps}{nk}\left( k! \right)^n,
\]
i.e.,
\[
\p \left( \sigma \in M_2 \left( h \right) \right) \geq \frac{2 \eps}{nk},
\]
from which \eqref{eq:manip_eps} follows immediately, since changing the outcome of a SCF at one ranking profile can change the number of 2-manipulation points by at most $nk$.
\end{proof}

\begin{proof}[Proof of Theorem~\ref{cor:k_refined_truenonmanip}]
First we argue without specific bounds. Suppose on the contrary that our SCF $f$ does not have many 4-manipulation points. Then $f$ is close to $\overline{\NONMANIP}$ by Theorem~\ref{thm:k_refined}. Consequently, by Theorem~\ref{thm:TRUENONMANIP}, $f$ is close to $\NONMANIP$, which is a contradiction.

Now we argue with specific bounds. Assume on the contrary that
\[
\p \left( \sigma \in M_4 \left( f \right) \right) < \frac{\eps^{15}}{10^{39} n^{67} k^{166}}.
\]
Then by Theorem~\ref{thm:k_refined} we have that $\Dist \left( f, \overline{\NONMANIP} \right) < \frac{\eps^3}{10^6 n^{12} k^{24}}$, and consequently by Theorem~\ref{thm:TRUENONMANIP} we have $\Dist \left( f, \NONMANIP \right) < \eps$, which is a contradiction.
\end{proof}


\chapter[Aggregation]{Aggregation power of Boolean Functions}

In this chapter we will review some of the theory of Boolean functions from the perspective of Condorcet Jury Theorem. 

\section{Influences and Aggregation} 

We now take a modern view of Condorcet Jury Theorem. First, recall the setting. There are two alternatives denoted $+$ and $-$, which are apiori 
equally likely to be the preferable alternative, and that each voter independently receives the correct information with probability $p > 1/2$ and incorrect information with probability $1-p$.
We now denote the $n$ signals by $x_1,\ldots,x_n$ and aggregate them via a boolean function 
$f : \{-1,1\}^n \to \{-1,1\}$. The following fact is a generalization of the second part of the Jury Theorem. 
\begin{theorem} \label{thm:neyman_pearson}
Let $0.5 < p \leq 1$. Then $f : \{-1,1\}^n \to \{-1,1\}$ maximizes $P[f(x) = s]$ if and only if 
$f(x) = +$ for all $x$ such that $\sum_{i=1}^n x_i > 0$ and $f(x) = -$ for all $x$ such that $\sum x_i < 0$. 

In particular when $n$ is odd Majority is the only function that maximizes $P[f(x) = s]$. 
\end{theorem}
This is a special case of the Neyman-Pearson lemma for hypothesis testing in statistics~\cite{NeymanPearson:33}. 
 In our case, the proof immediately follows from the fact that
\[
\frac{P[s = +| x_1,\ldots,x_n]}{P[s = -| x_1,\ldots,x_n]} = 
\frac{P[ x_1,\ldots,x_n | s = +]}{P[ x_1,\ldots,x_n | s = -]} = \left( \frac{p}{1-p} \right)^{n_+(x) - n_-(x)}
\]
where $n_+(x)$ is the number of $+$s in $x$ and similarly $n_-$.
 
Assume for a moment that $s = +$. Therefore, our interest is in the quantity $\IP[f = +]$. 
If $f$ is a monotone function, we have the following Russo's formula~\cite{Margulis:74,Russo:82}, where $\IP_p$ denotes the product measure with marginals $\IP_p[x_i = +] = p$: 
\begin{proposition} \label{prop:russo}
  Let $f :  \{-1,1\}^n \to \{-1,1\}$ be monotone then 
\[
\frac{d}{dp}\bigg\rvert_{p=1/2} \IP_p[f = 1] = \sum_{i=1}^n I_i(f).
\]
\end{proposition} 
One short proof of Proposition~\ref{prop:russo} is to consider an $n$ variable function $P_{(p_1,\ldots,p_n)}[f=1]$ and applying the chain rule. 
In fact, if we define for Boolean functions
\[
J_{p,i}(f) = \IE_p[\IP_p[f(x_{-i},+) \neq f(x_{-i},-)]],
\]
then the same proof yields more generally that for monotone $f :  \{-1,1\}^n \to \{-1,1\}$ it holds that 
\begin{equation} 
\frac{d}{dp}\bigg\rvert_{p} \IP_p[f = 1] = \sum_{i=1}^n J_{p,i}(f)= 
\frac{1}{4 p(1-p)}  \sum_{i=1}^n I_{p,i}(f),
\end{equation}
where 
\[
I_{p,i}(f) = \IE[Var[f(x) | x_{-i}]],
\]

\begin{lemma} In the setting of Condorcet voting with uniform prior on $s$, 
  among all monotone functions majority maximizes
\[
\frac{d}{dp}\bigg\rvert_{p=1/2} \IP_p[f=s] = \sum_{i=1}^n I_i(f)
\]
\end{lemma} 

\begin{proof}
  Consider jurors who receive the correct signal with probability $p$. Then for any
  monotone function $f$ we have that 
\[
\IP_p[f = s] = 0.5(\IP_p[f = +] + \IP_{1-p}[f = -]) = 0.5 + 0.5(\IP_p[f = +] - \IP_{1-p}[f = +]). 
\]
We know that for every $p$, among all monotone functions, the pan-ultimae expression above is maximized for the Majority function. 
Taking the derivative of the ultimate expression with respect to $p$ concludes the proof.
\end{proof} 
By examining the proof of Theorem~\ref{thm:neyman_pearson} more carefully, we see that the Majorities are the only maximizers of the influence sum. 

Interestingly, dictators are the minimizers in their aggregation power. In particular we have:

\begin{proposition}
For any Boolean function with  $f : \{-1,1\}^n \to \{-1,1\}$ it holds that $\sum I_i(f) \geq Var[f]$. 
Equality holds if and only if $f$ is a dictator or a constant. 
\end{proposition}

\begin{proof}
This follows from the fact that 
\[
\sum_i I_i(f) = \sum_S |S| \hat{f}^2(S), \quad Var[f] = \sum_{S \neq \emptyset} \hat{f}^2(S). 
\]
\end{proof} 

Again, in fact the claim is true for every $0 < p < 1$: 
\begin{proposition}
For any Boolean function $f : \{-1,1\}^n \to \{-1,1\}$ it holds that 
\[
\sum I_{p,i}(f) \geq Var_p[f]
\] 
Equality holds if and only if $f$ is a dictator.
\end{proposition}

\begin{exercise}
Prove it.
\end{exercise}

In fact more is true: 

\begin{proposition}
Among all $f : \{-1,1\}^n \to \{-1,1\}$ which are monotone and satisfy $\IE_{1/2}[f] = 0$, dictators are the only minimizers of 
$\IE_p[f]$ for all $0.5 < p < 1$.
\end{proposition} 

\begin{proof}
Assume by contradiction the statement is false. Then there exists a non-dictator function $f$ and $1>r > 1/2$ such that $\IP_r[f=1] = r$. Let 
$q = \inf(r > 1/2 : \IP_r[f=1] = r)$. Since $\IP_r[f=1]$ is a polynomial in $r$, it follows that 
$\IP_q[f = 1] = q$. Moreover, by Poincare's inequality at $1/2$, it follows that $q > 1/2$. 
Since $\IP_r[f=1] \geq r$ for $1/2 \leq r \leq q$ it follows that
$d/dr_{|r=q} \IP_r[f=1] \leq 1$. However, by Poincare's inequality at $q$ it is strictly greater than $1$. The proof follows.
\end{proof}

\section{The minimal Influence Problem and Coalitions}

Assume that in a binary vote $f : \{-1,1\}^n \to \{-1,1\}$ it is possible to buy a certain number of the votes. 
How much can this influence the outcome of the election. This question and related questions were studied 
in~\cite{BenorLinial:90} which is one of the foundational papers of analysis of Boolean functions. 

Poincare's inequality: 
\[
\sum_{i=1}^n I_i(f) = \sum_{S} |S| \hat{f}^2(S) \geq \Var[f]
\]
implies that any function $f : \{-1,1\}^n \to \IR$ has a variable $i$ with $I_i(f) \geq \Var[f]/n$. 
This in turn implies the following: 

\begin{lemma} \label{lem:coal1} 
If $f : \{-1,1\}^n \to \{-1,1\}$ is a monotone function with $\IE[f] = 0$, then for any $1>\eps > 0$, there exists  a set 
$|S| \leq \eps n$ such that $\IE[f | \forall i \in S, x_i = 1] \geq \eps/4$. 
\end{lemma} 

That is, a small linear fraction of the voters can change the expected value of the function noticeably. 
For the proof the following lemma is useful to note that if $f : \{-1,+1\}^n \to \{-1,+1\}$ is monotone then: 
\[
\hat{f}(i) = \IE[f x_i] = \IE [\IE[f x_i | x_{-i}]] =  \IE_{x_{-i}}[Var[f | x_{-i}]] = I_i(f),
\]
since conditioned on $x_{-i}$, either $f$ is a constant or $f$ is the identity and in both cases, $\IE[f x_i | x_{-i}] = Var[f | x_{-i}]$. 

\begin{proof} 
  Construct $S$ iteratively. Let $f_1 = f$ and let $S_0= \emptyset$. At stage $i$ of the construction, let $j$ be the variable with most influence of the function $f_i$ and let $S_i = S_{i-1} \cup \{ j \}$. Let $f_{i+1}$ be defined as $f_i$ conditioned on $x_j = 1$. If $\IE[f_{i+1}] \geq \eps$ we stop. Otherwise, we continue. 
  Note that at each step we have $\IE[f_{i+1} - f_i] \geq I_j(f_i) \geq 0.25/n$. The proof follows.
\end{proof} 

It is natural to ask if the claim and proof are tight. 
We first ask if one can improve the bound $I_i(f) \geq \Var[f]/n$.
For general functions the answer to the later question is no, as the function $f = n^{-1/2} \sum x_i$ shows. 
For Boolean functions, Ben-Or and Linial~\cite{BenorLinial:90} found an example that is almost tight i.e., the tribes function.

\begin{example}
Let $m = 2^r$ and let $n = m \log_2 m = m r$ and define the tribes function $f : \{0,1\}^n \to \{0,1\}$ by 
\[
f(x) = (x_1 \wedge \ldots \wedge x_r) \vee (x_{r+1} \wedge \ldots \wedge x_{2r}) \vee \cdots \vee (x_{n-r+1} \wedge \ldots \wedge x_n)
\]
This can be described in words as follows. The function is $1$ if there is a decision to go to war and $0$ otherwise. 
In order to go to war there has to be at least one tribe where all members decide to go to war, i.e, there has to exist a $k$ so that 
$x_{k r+1} = \ldots = x_{(k+1) r} = 1$. 
It is easy to see that the number of tribes that go to war is $\mathrm{Bin}(m,1/m)$ which is approximately Poisson. In particular, 
\[
\IE[f] = \IP[\mathrm{Bin}(m,1/m) \geq 1] \in [0.1,0.9],
\]
 for $r \geq 1$. 
 For $\partial_1 f = 1$, i.e., that the first coordinate is influential,
 we need that none of the tribes but the first tribe wants to go to war and that in tribe $1$, everyone but the first voter wants to go to war.
Thus 
\[
I_i(f) = \IP[\partial_1 f = 1] = 2^{-r+1} \IP[\mathrm{Bin}(m-1,1/m) = 0] \leq 2^{-r+1} = \frac{2}{m} \leq \frac{2 \log_2(n)}{n}.
\]
It is also easy to see from the equation above that in fact $I_i(f)$ is of order $\frac{\log_2(n)}{n}$. 
\end{example}

Thus, we already know that the minimal influence is of order at least $\Var[f]/n$ and we have an example where it is of order 
$\Var[f] \log n / n$. The KKL Theorem~\cite{KaKaLi:88} by Kahn, Kalai and Linial which will be proved  in the next section shows that the $\log n$ factor is in fact necessary. 

\begin{theorem} \label{thm:KKL}
There exists a constant $c$ such that for all $f : \{-1,1\}^n \to \{-1,1\}$ it holds that 
\[
\min_i I_i(f) \geq c \frac{\log n}{n} \Var[f]
\]
\end{theorem}

Repeating the proof of Lemma~\ref{lem:coal1} we obtain the following: 
\begin{lemma} \label{lem:coal2} 
Let $f : \{-1,1\}^n \to \{-1,1\}$ be a monotone function with $\IE[f] = 0$, then for any $\eps > 0$, there exists  a set 
$|S| \leq c(\eps) n / \log n$ such that $\IE[f | \forall i \in S, x_i = 1] \geq 1 - \eps$. 
\end{lemma} 

Interestingly, the tribes example is far from tight for Lemma~\ref{lem:coal2} as it suffices to take $S$ of size $\log n$ (e.g., a  tribe) to fix the value of the function to $1$. There is a famous example by Ajtai and Linial where a coalition of size of order $n/\log^2 n$ is needed to shift the expected value of $f$~\cite{AjtaiLinial:93} by a constant.

\section{Semi-Group proof of the KKL Theorem}

In this section we will write $P_t = T_{e^{-t}}$ so that $P_t$ is a  {\em semi-group}, so $P_{s+t} = P_s P_t$. 

\begin{lemma}
\[
P_t f = \sum_{S} \hat{f}(S) e^{-t |S|} x_S, \quad 
(\grad_i f)(x) = \frac{1}{2} (f(x_{-i},1)-f(x_{-i},-1)) = \sum_{S : i \in S} \hat{f}(S) x_{S \setminus \{i\}}
\]
Note that 
\[
\grad_i (P_t f) = \sum_{S : i \in S} \hat{f}(S) e^{-t|S|} x_{S \setminus \{i\}} = 
e^{-t} P_t \grad f_i.
\]
and that 
\[
\frac{d}{dt} P_t f = -\sum_{S} |S| \hat{f}(S) e^{-t |S|} x_S = 
-\sum_{i=1}^n x_i \grad_i (P_t f) = -e^{-t} \sum_{i=1}^n x_i P_t \grad_i f
\]
so  
\[
\frac{d}{dt} \IE[f P_t g] = -e^{-t} \sum_{i=1}^n \IE[f x_i P_t \grad_i g] = 
-e^{-t} \sum_{i=1}^n \IE[\grad_i f P_t \grad_i g] 
\]
\end{lemma} 

The computation above allows to prove the following~\cite{KaKeMo:16}:  
\begin{claim}
If $f$ and $g$ are monotone functions then $\IE[f P_t g]$ is monotonically decreasing in $t$. 
\end{claim}

\begin{proof}
Assume $f$ and $g$ are increasing then the proof follows since $\grad_i f$ and $\grad_i g$ are always non-negative.
\end{proof} 

\begin{claim}
Let $f_1,\ldots,f_k$ be positive and monotone increasing then 
$\E[P_t f_1 P_t f_2 \ldots P_t f_k]$ is monotone decreasing. 
\end{claim}
  
\begin{proof}
Write the derivative as a sum of partial derivatives.
\end{proof} 

\begin{remark}
The proofs are so easy that the claims can be generalized to a more general Markov chain setting. 
\end{remark}

We now turn to applications a la Cordero-Ersquin and Ledoux~\cite  {CorderoLedoux:12} of the derivative formula which can be thought of as an integration by parts formula.

\begin{claim}
\[
Cov(f,g) = \int_0^{\infty} e^{-t} \sum_{i=1}^n \IE[\grad_i f P_t \grad_i g] dt 
\] 
\end{claim} 

\begin{proof}
\[
Cov(f,g) = \IE[f P_0 g] - \IE[f P_{\infty} g] = -\int_{0}^{\infty} \frac{d}{dt} \IE[f P_t g] dt  = \int_0^{\infty} e^{-t} \sum_{i=1}^n \IE[\grad_i f P_t \grad_i g] dt 
\]
as needed. 
\end{proof} 
We can now prove some useful corollaries. 
\begin{claim}
\[
|Cov(f,g)| \leq \sum_{i=1}^n \sqrt{I_i(f) I_i(g)}
\]
\end{claim}

\begin{proof}
Note that 
\[
|\E[\grad_i f P_t \grad_i g]| \leq \| \grad_i f \|_2 \| P_t \grad_i g \|_2 \leq \| \grad_i f\|_2 \| \grad_i g \|_2 
= \sqrt{I_i(f) I_i(g)}.
\]
The proof follows by integration. 
\end{proof}  

\begin{exercise}
Provide an alternative proof using directly the Fourier expansion. Show moreover that if 
$\hat{f}(S)\hat{g}(S) \neq 0$ if and only if $S = \emptyset$ of $|S| \geq k$, then the results can be improved to:
\[
|Cov(f,g)| \leq \frac{1}{k} \sum_{i=1}^n \sqrt{I_i(f) I_i(g)}
\]
\end{exercise} 

If instead of the contractive estimate above we use the hyper-contractive inequality: 
\[
\|E[\grad_i f P_t \grad_i g]| \leq \| \grad_i f \|_{1+e^{-t}} \| \grad_i g \|_{1+e^{-t}}
\]
We get the following bound:
\[
|Cov(f,g)| 
\leq 
 \int_0^{\infty} e^{-t} \sum_{i=1}^n \| \grad_i f \|_{1+e^{-t}} \| \grad_i g \|_{1+e^{-t}} dt = 
 \int_{1}^2 \sum_{i=1}^n \| \grad_i f \|_v \| \grad_i g \|_v dv 
 \]
 Now By Holder inequality 
 \[
 \| f \|_v \leq \| f \|_2^{1-\theta(v)} \| f \|_1^{\theta(v)} = \| f \|_2 \left( \frac{\| f\|_1}{\| f \|_2} \right)^{\theta(v)},
 \]
 where $v \theta + 0.5 (1-\theta)v = 1$ so  
 $0.5 \theta v = 1 - 0.5 v$ so $\theta = 2/v - 1$. 
 Writing 
 \[
 r_i = \frac{\|\grad_i f\|_1 \| \| \grad_i g \|_1}{\| \grad_i f \|_2 \| \grad_i g \|_2},
 \]
 we see that 
 \[
 |Cov(f,g)| \leq \sum_{i=1}^n \| \grad_i f \|_2 \| \grad_i g \|_2 \int_1^2 r_i^{\theta(v)} dv 
 \]
 Note that $\theta(v) = 2/v-1 \geq 1-v/2$ so 
 \[
\int_1^2 r_i^{\theta(v)} dv \leq  \int_1^2 r_i^{1-v/2} dv = 2 \int_0^{1/2} r_i^v dv = 
\frac{1}{\ln r_i}(r_i-1) \leq \frac{5}{1+\ln(1/r_i)}
\]
We thus obtain

\begin{claim}
\[
|Cov(f,g)| \leq 5 \sum_{i=1}^n \frac{\sqrt{I_i(f) I_i(g)}}{1 + \ln\left( \| \grad_i f \|_2 \| \grad_i g \|_2/ 
(\| \grad_i f \|_1 \grad \|_i g \|_2) \right)}
\]
\end{claim}


\begin{remark}
This statement is a  slight generalization of the proof by~\cite{CorderoLedoux:12} of Talagrands's result~\cite{Talagrand:94}.
\[
Var[f]  \leq 5 \sum_{i=1}^n \frac{I_i(f)}{1 + 2 \ln\left( \| \grad_i f \|_2/ 
(\| \grad_i f \|_1) \right)}
\]
We claim that furthermore the last inequality implies the KKL Theorem that for Boolean functions 
$f: \{-1,1\}^n \to \{0,1\}$ there exists an $i$ such that $I_i(f) \geq c Var[f] \log n / n$. This is true since for Boolean functions $\grad_i f \in \{1,0,-1\}$ and therefore 
\[
\| \grad_i f \|_2 = \sqrt{I_i(f)}, \quad \| \grad_i f \|_1 = I_i(f).
\]
So for such function we get that 
\[
Var[f] \leq 10 \sum_{i=1}^n \frac{I_i(f)}{1 - \ln I_i(f)}
\]
So the there exist an $i$ such that 
\[
\frac{I_i(f)}{1 - \ln I_i(f)} \geq c Var[f]/n
\]
and therefore $I_i(f) \geq c Var[f] \log n / n$. 
\end{remark}


\section{Larger Domains}
\begin{remark}
The KKL Theorem and its proof immediately extend from the measure
$2^{-n} (\delta_{-1}+\delta_1)^{\otimes n}$ to any measure of the form
$  (p \delta_{-1} + (1-p)\delta_1)^{\otimes n}$ for any fixed $p$, since hyper-contractive estimates are known for such measures. 
\end{remark}

Allowing varying values of $p$ in the KKL Theorem has interesting implications for aggregation of information. 

We call a function $f : \{-1,1\}^n \to \{-1,1\}^n$ {\em symmetric}, if there exists a group $\Pi$ who acts transitively on $[n]$ such that $f(\pi(x)) = f(x)$ for all $\pi \in \Pi$, where
$\pi(x) = (x_{\pi(1)},\ldots,x_{\pi(n)})$. Recall that $\Pi$ acts transitively on $[n]$ if for all $i,j \in [n]$, there exists $\pi \in \Pi$ with $\pi(i) = j$. 

\begin{corollary} 
Let $f : \{-1,1\}^n \to \{-1,1\}$ be monotone and symmetric and assume $f(-x) = -f(x)$ for all $x$. Then for any $\eps > 0$, there exists a $c(\eps) > 0$ such that if 
\[
\IE_{1/2 + c/ \log n}[f] \geq 1-\eps.
\]
\end{corollary}

\begin{proof}
Use Russo's formula and the KKL Theorem.
\end{proof}

Of course there are examples, where the theorem is far from tight. For example for the majority function, we know that the statement of the theorem holds when $p = 1/2 + c/\sqrt{n}$. 

\begin{exercise}
Consider the tribes functions and show that the statement of the theorem cannot be improved.
\end{exercise}

\begin{exercise}
Show that for the recursive Majority of threes functions, the statement of the theorem holds for $p = 1/2 + n^{-\delta}$ for some $\delta > 0$. 
\end{exercise} 

These results can be further generalized to scenarios involving multiple candidates following~\cite{KalaiMossel:15}.

\subsection{Monotonicity and Aggregation for multiple alternatives }

Let $A$ be a finite set. Let $X = A^n$. For $\sigma \in S(n)$, a
permutation on $n$ elements and $x
\in A^n$ we denote by $y = x_{\sigma}$ the vector satisfying
$y_i = x_{\sigma(i)}$ for all $i$. For $\sigma \in S(A)$ we write
$y = \sigma(x)$ for the vector satisfying $y_i = \sigma(x_i)$ for all
$i$. For $a \in A$ and $x,y \in X$ we
write $x \leq_a y$ if $\{i : x_i = a\} \subset \{i : y_i = a\}$ and for all $i$ such that $y_i \neq a$ it
holds that $x_i = y_i$. In other words if $x \leq_a y$ then for all $i$ if $x_i \neq y_i$ then $y_i = a$.
It is easy to see that $\leq_a$ defines a partial order on $X$.

\begin{definition} \label{def:basic}
We say that $f : X = A^n \to A$ is {\em monotone} if
for all $a \in A$ and $x,y \in X$ such that $x \leq_a y$ it holds that
$f(x) = a$ implies that $f(y) = a$.
We say that $f$ is {\em symmetric} if there exists a transitive
group $\Sigma \subset S(n)$ such that $f(x_{\sigma}) = f(x)$ for all
$x \in X$ and $\sigma \in \Sigma$.
We say that $f$ is {\em fair} if for all $\sigma \in S(A)$ and all
$x \in X$ it holds that $f(\sigma(x)) = \sigma(f(x))$.

Let $\Delta[A]$ denote the simplex of probability measures on $A$ and
let $\gamma$ denote the standard probability measure on $\Delta[A]$.
For
$\mu \in \Delta[A]$ denote by $\P_{\mu}$ the measure
$\mu^{\otimes n}$ on $X$. We denote by $\E_{\mu}$ the expected value
according to the measure $\P_{\mu}$.
For any measure $\mu \in \Delta[A]$, we write $\mu^{\ast}$ for the minimal probability $\mu$ assigns to any of
the atoms in $A$.
\end{definition}

It is not too hard to generalize the proof of the KKL Theorem to obtain the following theorems, see cite{} 

\begin{theorem} \label{thm:kalai_mossel}
There exists an absolute constant $C = C(|A|)$ such
that if $f$ is symmetric and monotone then for
any $a \in A$ and $1/2 > \eps > 0$ it holds that
\[
\gamma[\mu : \eps \leq P_{\mu}[f = a] \leq 1-\eps] \leq
C (\log(1-\eps) - \log(\eps)) \frac{\log \log n}{\log n}.
\]
\end{theorem}

\begin{theorem}
\label {t:kalai_aggreagtion}
For every $q$ there exists a constant $C = C(q)$ for which the following holds for every $0 < \eps < 1/3$.
Let $\mu \in \Delta(q)$ and $i \in [q]$ satisfy that
\[
\mu(i) > \max_{j \neq i} \mu(j) + C (\log(1-\eps) - \log(1/q)) \frac{\log \log n}{\log n}.
\]
Then for every fair monotone symmetric function $f : [q]^n \to [q]$ it holds that
\[
\mu[f = i] \geq 1-\eps.
\]
\end{theorem}

The results above establishes sharp threshold for symmetric functions
as it shows that for almost all probability measures $f$ takes one
specific value with probability at least $1-\eps$.
\begin{remark}
It is interesting to compare the results established here to those
of~\cite{FriedgutKalai:96}. For $|A|=2$ our results give a threshold
interval of length $O(\log \log n/\log n)$ compared to the results
of~\cite{FriedgutKalai:96} which give threshold interval of length
$O(1/\log n)$. The example of the tribes function shows that a better bound is impossible. 
It is natural to conjecture that the threshold is always of measure $O(1/\log n)$.
\end{remark}

\subsection{Applications to Majority/Plurality Dynamics}
An interesting application of the theorems above is to Majority and Plurality dynamics \cite{MoNeTa:14}. We will concentrate on Majority dynamics.

Consider a graph $G$ where all the degree are odd and an assignment of opinions 
$X_v(0)$ for nodes of the graph. Consider the following dynamics, where
$X_v(t+1) = maj(X_w(t) : w \sim v)$. 

It is natural to ask a few questions:
{\em Does this process converge?} It is easy to see that the answer to this question is no. 
This can be seen for example in the case where $G$ is bi-partite graph where the initial opinions on one side are $+$ and on the other $-$. 

Interestingly a finer notion of convergence is known to hold \cite{GolesOlivos:80} 
\begin{proposition}
For each $v$, the sequence $X_{v}(2t)$ converges. 
\end{proposition}

\begin{proof}
Let 
\[
J_v(t) = (X_v(t+1)-X_v(t-1)) \sum_{w \sim v} X_w(t), \quad J(t) = \sum_v J_t(v)
\]
We first note that $J_v(t) \geq 0$ and $J_v(t) = 0$ if and only if $X_v(t+1)=X_v(t-1)$.
From this we conclude that $J_t \geq 0$ and $J_t = 0$ if and only if 
$X_v(t+1)=X_v(t-1)$ for all $v$.
Thus it suffices to show that $J_t =0$ from some point on. Also note that if $J_t = 0$ then 
$X_v(t+2i-1) = X_v(t-1)$ for all $i$. 

For this let $L_t = \sum_{u \sim v} (X_{v}(t+1)-X_u(t))^2$ be a measure of disagreement. 
We claim that $L_t - L_{t-1} = -J_t$. Note that this implies that $L_t$ is decreasing and therefore at some point we will have to have $J_t = 0$ as needed. 
It is straightforward to check that $L_t - L_{t-1} = -J_t$.
\end{proof} 

Let us write $X_v = \lim_{t \to \infty} X_{v}(2t)$. 
Given that this process converges, it is natural to ask if it retains information in the following sense: Suppose that $X_v(0)$ are chosen i.i.d. $p \delta_1 + (1-p) \delta_{-1}$ for $p > 1/2$. What is the probability that the information is retained, i.e.,
what is 
\[
\IP_p[maj(X_v) = +]?
\]
KKL Theorem allows us to answer this question in the special case where the graph is transitive.

\section{Aggregation for non-product measures}

In this section we ask the following question: can aggregation results such as the Jury Theorem and the KKL theorem be extended to non-product measures. This section provides some answers to this question. The section is based on ~\cite{HaKaMo:06} 
 
To describe a more general settings consider the following
framework. Let $f:\{0,1\}^n \to \{0,1\}$ be a Boolean function.
We will assume that $f$ is 
\begin{itemize}
\item
monotone non-decreasing, i.e., 
\[
\left( \forall i: \,\,x_i \geq y_i \right) 
\implies f(x_1,\ldots,x_n) \geq f(y_1,\ldots,y_n), 
\]
\item
and anti-symmetric, i.e., 
$$f(1-x_1,1-x_2,\dots,1-x_n)=1-f(x_1,x_2,\dots,x_n).$$
\end{itemize} 

The purpose of this section is to study extensions of the 
weak law of large numbers 
in the context of general probability distributions. 
Let $\mu$ be a probability distribution on $\cube$.
When $\mu$ is not a product measure
the notion of influence can be extended in different way compared to
the above.
Define the {\it effect} of the $k$'th variable on the Boolean function
$f$ as the difference between the expected value of 
$f(x_1,\dots,x_n)$ conditioned on $x_k=1$ 
and the expected value of $f(x_1,\dots,x_n)$ conditioned on $x_k=0$, and 
denote by $e_k^{\mu}(f)$ the effect of the $k$'th variable for
the Boolean function $f$, w.r.t. the distribution $\mu$. 
More precisely,
\begin{equation} \label{eq:efdef}
e_k^{\mu}(f) =  
\mu[f(X_1,\ldots,X_n) | X_k = 1] - \mu[f(X_1,\ldots,X_n) | X_k = 0].
\end{equation}
The effect is undefined if the probability for $X_k=1$ is $1$ or $0$.
Writing $\mu[X_k] = p$ and $Y_k = X_k - p$, we get
\begin{eqnarray*}
\Cov_{\mu}[f(X_1,\ldots,X_n),X_k] &=& \mu[f(X_1,\ldots,X_n) Y_k] \\ &=&
p \mu[(1-p) f | X_k = 1] + (1-p) \mu[-p f | X_k = 0] \\ &=& 
p(1-p) e_k^{\mu}(f)
\end{eqnarray*} 
so that the effect may be interpreted as a normalized form of 
the correlation between the individual vote and the election's outcome.

When $\mu$ represent a product probability measure, the effect (\ref{eq:efdef}) and the 
influence  coincide, but in general this is not the case.
For instance, for general $\mu$ the effect may be negative (an example will appear later) while the influence is of course always non-negative. 

It is not true that for general probability distributions
and general $f$, small influences
implies aggregation of information. Our main result is that small effects
implies aggregation of information for 
the particular case of weighted majority functions.
Moreover, the bounds for weighted majorities are rather realistic. 

We call monotone antisymmetric function $f$ a {\em weighted majority} function  
if there exists  non-negative weights $w_1,\dots,w_n$, not all zero such 
that $f(x_1,\dots,x_n)=1$ if $\sum_{i=1}^n w_i (2 x_i - 1) > 0$ and 
and $f(x_1,\dots,x_n)= 0$ if $\sum_{i=1}^n w_i (2 x_i - 1) < 0$.
If $n$ is odd and $w_i=1$ for every $i$, $f$ is called the majority 
function (or simple majority). 

Note that in our definition of a weighted majority function, 
if $\sum w_i (2 x_i - 1) = 0$ then 
the value of $f(x)$ may be either $0$ or $1$ as long as $f$ is monotone and 
anti-symmetric. This is different from the traditional definition of a 
weighted 
majority (or threshold) 
function where $f(x)=1$ iff $\sum w_i (2 x_i - 1) > 0$ 
and $f(x) = 0$ iff $\sum w_i (2x_i - 1) < 0$.

Thus for example, any monotone anti-symmetric 
function $f : \cube \to \{0,1\}$ 
satisfying $f(x) = 1$ when $x_1 = x_2 = 1$ and 
$f(x) = 0$ when $x_1 = x_2 = 0$ is a weighted majority function (taking
$w_1 = w_2 = 1$ and $w_3 = \cdots = w_n  = 0$) according to our definition.

The above example demonstrates that under our definition of weighted majority 
functions, there are at least $2^{2^{n-2}}$ weighted majority functions. 
Under the traditional definition the number of weighted majority 
functions is at most $2^{n^2}$ \cite{CoatesLewis:64,RSOK:91}.   

Of particular interest are voting schemes 
where all the voters have the same power. 
One such case is when $f$ is invariant under a transitive group of 
permutations. In other words there exists a group of permutation 
$\Gamma \subset S_n$ such that 
$f(x_1,\ldots,x_n) = f(x_{\sigma(1)},\ldots,x_{\sigma(n)})$ for all 
$\sigma \in \Gamma$ and for all $1 \leq i,j \leq n$ there exists 
$\sigma \in \Gamma$ such that $\sigma(i) = j$; here $S_n$ denotes
the full permutation group on $n$ elements. 
One instructive example is 
the simple majority function when $n$ is odd which is invariant under $S_n$;
another is the recursive majority function $RM_{k,\ell}$
which is defined for $n=k^{\ell}$ where $k$ is odd. The definition is by 
induction. $RM_{k,1}$ is just the majority function on $k$ bits and 
\[
RM_{k,\ell+1}(x_1,\ldots,x_{k^{\ell+1}}) = 
RM_{k,1} \left(RM_{k,\ell}(x_1,\ldots,x_{k^{\ell}}),\ldots,
RM_{k,\ell}(x_{k^{\ell} - k^{\ell-1}+1},\ldots,x_{k^{\ell}}) \right).
\] 

\begin {theorem}
\label {t:hkm}
$\mbox{ }$
\begin{description}
\item{\rm(a)}
For every $p > \frac{1}{2}, \epsilon >0$ 
there is $\delta = \delta (p,\epsilon) > 0$ 
such that for every weighted 
majority function $f$ and any distribution 
$\mu$ on $\cube$, if
$e_k^{\mu}[f] \leq \delta$ and $\mu[X_k = 1] \geq p$ for all $k$ 
then $\mu[f] \geq 1 - \eps$. 

In other words, if the effect of each
variable is at most $\delta$ and the probability that each variable is
$1$ is at least $p$, then $f=1$ with
$\mu$-probability at least $1-\eps$.  
\item{\rm(b)}
If $f$ is a monotone anti-symmetric function but not a 
weighted majority function, 
then there exists
a probability distribution $\mu$ such that $\mu[X_k = 1]>1/2$ 
for all $k$, yet $\mu[f] = 0$ and $e_k^{\mu}(f)=0$ for all $k$.

In other words, if $f$ is not a weighted majority function, then there
is a probability measure $\mu$ for which $f=0$ with $\mu$-probability
$1$, yet $\mu[X_k = 1] > \frac{1}{2}$ for all $k$. (Since $f$ is constant
according to the measure $\mu$, all the effects are $0$ in this case.)
\item{\rm(c)}
If $f$ is monotone anti-symmetric and invariant under a transitive group,
but is not the (simple) majority function, then then there exists 
a probability distribution $\mu$ such that $\mu[X_k = 1]>1/2$ 
for all $k$, yet $\mu[f] = 0$ and $e_k^{\mu}(f)=0$ for all $k$.
\end{description}
\end {theorem}

The rest of this section is organized as follows.
In subsection \ref {subsec:hkm_d} we will discuss the notions of aggregation of 
information, 
influences and effects for general probability distributions on $\cube$. 
We will try to examine what aggregation of information means 
when we do not suppose that the probability distribution for the 
voter's behavior is a product distribution. We also examine to 
what extent our technical notion of ``effects'' represent real 
influence in the non-technical sense of the words. 
Subsection \ref {subsec:hkm_p} contains the proof of our theorem  
and in subsection \ref {subsec:hkm_e} we present 
several natural problems as well as an 
example showing lack of aggregation for 
monotone functions even for the restricted class of FKG-distributions.

\subsection {Voting games, information aggregation and notions of influence}
\label {subsec:hkm_d}

Consider 
the following scenario. Every agent $k$ receives a single bit of
information $s_i$ which is either `Vote for Alice' or `Vote for Bob' and
these signals are independent. When Alice is the better candidate 
the probability of receiving  the signal `Vote for Alice' is
$p>1/2$. 
Condorcet's Jury Theorem deals
with the case that the voters vote precisely as the signal dictates and the
decision is made according to  the simple majority rule. It asserts that
for every $p>1/2$  the better candidate will be elected with 
probability tending to one. Thus the majority rule allows to reveal 
the actual state of the world from rather weak individual signals.

A major problem in the economic and political
interpretation of Condorcet's Jury Theorem and its extensions
arises from the fact that
the basic assumption of probability
independence among voters
is quite unrealistic.
Without the assumption of independence, Condorcet's Jury Theorem as stated
is no longer true, and it will no longer be the case that
when each individual votes for Alice with probability $p>1/2$,
Alice will win with a high probability.

To see this, consider the following
example. As before, we have an election between Alice and Bob and
Alice is the superior candidate. The distribution of signals
$s_1,s_2, \dots , s_n$ will be biased towards Alice as follows:
Let $p=1/2+\epsilon/2$, where $\epsilon$ is small.
First choose at random a number $t$ uniformly
in the interval $[\epsilon,1]$.
Then, independently for each $i$, 
choose  the $i$'th voter  signal $s_i$ to be `1' with
probability $t$ and `0' with probability $1-t$. Voters with $s_i=1$
will vote for Alice.
In this case, the probability
for each individual signal $s_i$ being `1' is $p$
but the individual signals are not independent. The probability that
Alice will win is below $\frac{1}{2(1-\eps)}$  for any
number of voters. This is because we can think of $t$ being chosen in two 
stages. First we toss a coin which is 'H' with probability $\epsilon/(1-\epsilon)$. If the 
coin is 'H', then $t$ is chosen uniformly in the interval $[1-\epsilon,1]$. 
This contributes to the probability that Alice wins at most 
$\epsilon/(1-\epsilon)$. If the coin is 'T' then 
$t$ is chosen uniformly in the 
interval $[\epsilon,1-\epsilon]$. Here by symmetry, Alice and Bob have the 
same chance of winning. Thus the contribution to the probability that Alice 
will win from this case is $\frac{1-2\epsilon}{2(1-\epsilon)}$. Thus the 
overall probability that Alice will win is at most 
$\frac{1-2\epsilon}{2(1-\epsilon)} + \frac{\epsilon}{1-\epsilon} = 
\frac{1}{2(1-\eps)}$.

An even more extreme example is the case in which all voters vote in the same 
way: With probability $p$ they all vote for Alice and with probability $1-p$ 
they all vote for Bob. Alice
will be elected with probability $p$ regardless of the number of
voters when the election is based on simple majority and for every
other simple game.

These simple examples will help us to
examine the notions of information aggregation and influence 
in the case when the assumption of probability independence is dropped.
The problem in these examples is not
in the way information aggregates but in
the quality of the information to start with.
This assertion can be formalized as follows.
Suppose that Alice and Bob are given an a-priori probability $1/2$ of being 
the superior candidate. We assume that the distribution of voters for Bob 
given Bob is the superior candidate is the same as the distribution of 
voters for Alice given that Alice is the superior candidate. Thus in the 
first example above if Bob is superior then we choose $t$ uniformly in 
$[\eps,1]$ and then each voter votes for Bob independently with 
probability $t$. In the second example if Bob is superior, then all voters
will vote for Bob with probability $p$.
 
We now wish to decide between the hypothesis that Alice is the superior 
and the hypothesis that Bob is the superior candidate  
given the entire vector of individual signals. 
It is intuitively clear and easy to prove using the Neyman-Pearson Lemma 
that in both cases described above 
one should guess that Alice is superior to Bob exactly when the majority 
of voters voted for Alice. However, in both examples above the probability 
that the majority will vote Alice when Alice is superior 
is bounded away from $1$ and
tends to $1/2$ as $p$ does.

When we consider general distributions, the issue is to
understand what information we can derive on the superior alternative
from knowing the signals of all individuals and how the voting mechanism
extracts this information. Note that in the examples we considered 
above the individual effects are large while the individual 
influences are small. This is most transparent in the second example where 
if $f$ is the majority function and $n \geq 3$, 
then all of the influences are $0$, while the effect of 
all voters are $1$. 
Theorem \ref {t:hkm} asserts that for the weighted 
majority voting rule (and only for these rules) 
for every probability measure on $\cube$,
small individual effects 
implies asymptotically complete aggregation of information. 

Let us next consider the notion of influence without probability independence.
The notion of pivotal variables (or players) and 
influence is of important technical importance in various 
areas of mathematics, computer science and economics. 
This notion is also of a considerable conceptual importance. 
The voting power index of Banzhaf is 
based on measuring the influence with respect to the uniform 
distribution. The Shapley--Shubik power index can also be based on 
the influence with respect to another distribution.
Conceptual understanding of voting power in situations
where the voters' behavior is not independent is of great interest.
In \cite {GeKaTu:02}  the authors propose to define the voting power
as the probability to be pivotal 
based on realistic assumptions on individual voting distributions.
We make the following remarks on the notion of individual effects which is 
quite a different extension of influence and voting power measures
to arbitrary probability distributions. 




\begin{description}
\item{(i)} 
For general distributions, the effect of an agent can be negative. 
This will be the case for a voter who always votes  for
the candidate who is the  underdog in the election polls and also for a 
committee member who antagonizes the other members of the committee.
(On the other hand, the influence of an agent is always nonegative,
because it is defined as a represents a probability.) 
\item{(ii)} 
A dummy (a voter $k$ which is never pivotal)
has zero influence (with respect to every probability distribution). 
He may nevertheless have a large effect, such as if 
he always votes  for the candidate who is expected
to win  according to election polls. 
In real life, this will also be the case
for an observer on a committee
without the right to vote but who is likely to convince the
committee of  his opinion. Note that in
the first case we do  not attribute to that player real
``influence'' in the (non-technical) English sense of the word,
while  in the second case we would consider him ``influential''.
The uncertainty in  interpreting  effects as real
``influences'' is genuine. 
\item{(iii)} 
What is the 
motivation for a voter to vote, given the small probability for 
him to be pivotal? This is a social dilemma, related to, e.g., 
the so-called {\em tragedy of the commons}, and has been extensively 
discussed in the political science and philosophy literature. (Sometimes
the term {\em voting paradox} has been used for this dilemma, but
may cause some confusion as
the same term is used also for Condorcet's famous observation that
when three or more choices are available, the majority preference
between them need not be transitive.)
A possible solution to the dilemma
may lie in the fact that in real-life elections, individual effects 
tend to be large, namely bounded away from zero regardless of the 
size of the society. 
The uncertainty in 
regarding effects as real ``influence'' may suggest that it is 
the effect of an agent rather than his influence which is related
to his ``satisfaction'' with  the social decision process and his
ability to identify with the collective choice. 
\end{description}

\subsection{Proof of Theorem \ref {t:hkm}}
\label {subsec:hkm_p}

\subsubsection{Part (a) of the theorem}

We begin this section by providing a probabilistic proof of the
following result, which clearly implies 
Theorem \ref{t:hkm} (a). 

\begin{lemma} \label{lem:prob}
Let $(w_i)_{i=1}^n$ be non-negative weights which are not all $0$, 
let $0 < q < 1$, and let $f : \cube \to \bits$ be a function which satisfies 
\[
f = \left\{
\begin{array}{ll}
1 & \mbox{if } \sum_{i=1}^n (2 x_i - 2 q) w_i > 0 \\
0 & \mbox{if } \sum_{i=1}^n (2 x_i - 2 q) w_i < 0.
\end{array}  \right. 
\]
Write $W = \sum_{i=1}^n w_i$.
Suppose furthermore that $p > q$ and that $\mu$ is a 
probability measure satisfying 
$\mu[X_i] = p_i$ and
\begin{equation}  \label{eq:first_cond}
\sum_{i=1}^n w_i p_i = p W
\end{equation}
as well as
\begin{equation}  \label{eq:second_cond}
\sum_{i=1}^n w_i p_i (1-p_i) e_i^{\mu}[f] \leq p (1-p) \delta W . 
\end{equation}
Then 
\[
\mu[f] \geq 1 - \frac{\delta p(1-p)}{p-q}.
\]
\end{lemma}
(Note that (\ref{eq:first_cond}) holds if $\mu[X_i] = p$ for all $i$,
and that (\ref{eq:second_cond}) holds if 
$\mu[X_i] = p$ and $e_i^{\mu}[f] \leq \delta$ for all
$i$, so that indeed Theorem \ref{t:hkm} (a) follows.)

\begin{proof}
Let $X = \sum_{i=1}^n (2 X_i - 2 q) w_i$.
We start by noting that $\mu[X] = (2p-2q)W$.

We let $g = 1 - f$ and $Y_i = p_i - X_i$, so
that 
\[
\mu[Y_i g] = \Cov_{\mu}[f,X_i] = p_i(1-p_i) e_k^{\mu}[f]. 
\]

Note that conditioned on $g = 1$, 
$\sum_{i=1}^n (2 X_i - 2 q) w_i \leq 0$ and therefore 
$\sum_{i=1}^n w_i Y_i \geq  (p-q) W$. 
It follows that 
\begin{eqnarray}
\nonumber
\mu\left[\left(\sum_{i=1}^n w_i Y_i\right) g(X_1,\ldots,X_n)\right] 
& \geq & (p-q) W \mu[g] \\
& = & (p-q) W (1 - \mu[f]).   \label{eq:g1}
\end{eqnarray}
On the other hand,
\begin{eqnarray}
\nonumber
\mu\left[\left(\sum_{i=1}^n w_i Y_i\right) g(X_1,\ldots,X_n)\right] 
& = & \sum_{i=1}^n w_i \mu[Y_i g(X_1,\ldots,X_n)] \\
\nonumber
& = & \sum_{i=1}^n w_i p_i (1-p_i) e_i^{\mu}[f] \\
& \leq  & p(1-p) \delta W. 
 \label{eq:g2}
\end{eqnarray}
Combining (\ref{eq:g1}) and (\ref{eq:g2}), we get that  
\begin{eqnarray*}
\mu[f] & \geq & 1 - \frac{\sum_{i=1}^n w_i p_i (1-p_i) e_i^{\mu}[f]}{(p-q) W} 
\\ & \geq &  1 - \frac{\delta p(1-p)}{p-q}.
\end{eqnarray*}
\end{proof}

\subsubsection{Parts (b) and (c) of the theorem}

We note that part (c) of Theorem \ref{t:hkm} follows immediately from part
(b), because the only weighted majority function that is
invariant under a transitive group, is simple majority. 
Let us nevertheless begin by giving an 
independent and simple proof of part (c).  
Note that if $f$ is not the majority function 
then there is a vector $(x_1,x_2,\dots,x_n) \in \{0,1\}^n$ such that 
$f(x)=0$ and $x_1+x_2+\dots+x_n>n/2$. Then we can simply take $\mu$ to be 
uniform probability distribution on the orbit of $x$ under $\Gamma$.
It is then easy to see that $\mu[X_k] > 1/2$ for all $k$ and that 
$\mu[f=0]=1$. 

We now turn to the proof of Theorem \ref{t:hkm} (b).
We will show that if $f$ is not a weighted majority function, then 
there exists a measure $\mu$ satisfying $\mu[X_k] > 1/2$ for all $k$ 
and $\mu[f=0]=1$.  


Define $[n]=\{1,2,\dots,n\}$.
For $S \subset [n]$ put $x_S=(x_1,x_2\dots,x_n)$ 
where $x_i=1$ if and only if $i \in S$. 
Let $H$ be a hypergraph whose set of vertices is 
$[n]$ and whose edges are subsets $S$ of $[n]$ such that $f(x_S)=0$. 
Let $\tau^* = \tau^*(H)$ be the fractional cover number of $H$, i.e., 
the infimum over all $\nu:\cube \to \R$ of 
$\sum_{S \in H} \nu[x_S]$, under the condition that 
$\nu(x_S) \ge 0$ for every $S \in H$ and 
$\sum_{S \in H, k \in S} \nu[x_S] \ge 1$ for all $k$. 
We get $\tau^{\ast} = \infty$ 
if there are no $\nu$ satisfying the two conditions above (note that this is 
the case if $f(x) = x_1$, say).


If $\tau^* < 2$, then we can define $\mu(S)=0$ if $f(S)=1$ and
$\mu(S)=\nu(S)/\tau^*$ when $f(S)=0$. The probability measure $\mu$ satisfies
that 
\[
\sum_{S:k \in S, f(S)=0} \mu (x_S)\geq 1/\tau^* > 1/2
\]
for every $k$ and $\mu[f=0]=1$ as stated in the theorem. 
Therefore, in order to prove part (b) of the theorem, 
it only remains to analyze the case $\tau^* \geq 2$. 
  
A well known equivalent (by linear programming duality) 
definition of of $\tau^*$ is as the supremum of $\sum_{i=1}^n w_i$ under 
the condition that $w_k \ge 0$ for $k=1,2,\dots,n$ and 
$\sum \{w_i: i \in S \} \le 1$ for every $S \in H$. 

Assume first that $\tau^{\ast} > 2$. In this case we can find $w_i$'s such 
that $\sum_i w_i > 2$ and $f(x_1,\dots,x_n) = 1$ if $\sum w_i x_i > 1$. 
By slightly perturbing the $w_i$ we may assume that for all $x \in \cube$ 
it holds that $\sum_i w_i x_i \neq \frac{1}{2}\sum_i w_i$ in addition to the 
properties that $\sum_i w_i > 2$ and $f(x_1,\dots,x_n) = 1$ if 
$\sum w_i x_i > 1$. 
Let $g(x) = 1$ if $\sum_i w_i x_i > \frac{1}{2}\sum_i w_i$ and 
    $g(x) = 0$ if $\sum_i w_i x_i < \frac{1}{2}\sum_i w_i$.
Then $g$ is 
anti-symmetric and $f=0 \implies g = 0$. It follows that $f=g$ so that 
$f$ is a weighted majority function as needed. 

The remaining case is where $\tau^{\ast} = \sum w_i=2$. 
We obtain that $f(x_1,\dots,x_n) = 1$ if $\sum w_i x_i > 1$.
Since $f$ is anti-symmetric it follows that  
$f(x_1,\dots,x_n) = 0$ if $\sum w_i x_i < 1$. 
The result follows. 
{\hfill $\square$}

\subsection {Problems and an additional example}
\label {subsec:hkm_e}

The following problems naturally suggest themselves at this point:

\begin{description}
\item{(1)} 
For which class of distributions is it the case that for simple majority 
small voting power implies asymptotically complete aggregation of information?
\item{(2)} 
For which class of distributions is it the case that for every monotone 
Boolean function small voting power implies 
asymptotically complete aggregation of information?
\item{(3)} 
For which class of distributions is it the case that for every monotone 
Boolean function small individual effects 
implies asymptotically complete aggregation of information?
\end{description}


A natural condition to impose on the 
distribution $\mu$ which is realistic in various economic 
situations is the FKG
condition (see \cite {Liggett:85}).
For $x=(x_1,\dots,x_n)$ and $y=(y_1,\dots,y_n)$, define 
\[
\max (x,y)=(\max (x_1,y_1),\dots, \max (x_n,y_n))
\]
and
\[
\min (x,y)=(\min (x_1,y_1),\dots,\min (x_n,y_n)) . 
\]
One definition of FKG measure on $\{0,1\}^n$
goes as follows:
A distribution $\mu$ on $\{0,1\}^n$ (or on $\R^n$)
is called an FKG measure if
for every $x,y \in \{0,1\}^n$ we have
\[
\mu(x)\mu(y) \leq \mu(\max (x,y)) \mu (\min (x,y)).
\]
The FKG property is a profound notion
of non-negative correlations between agents' signals. It implies
(but is strictly stronger than) the following condition
(known as {\it non-negative association}, see \cite {MilgromWeber:82}):
For all increasing
real functions $f$ and $g$, it is the case that $E[fg] \geq E[f] E[g]$.
This is equivalent to the condition that for all increasing events $A$ and
$B$ we have that $P[AB] \ge P[A]P[B]$. Under the FKG property
if the simple game is monotone,  all effects are
non-negative. This form of non-negative correlation is a plausible
assumption to make in various contexts of collective choice.
It is easy to see that under the condition of 
non-negative association all individual effects are non-negative.

\begin{description}
\item{(4)} 
For which class of monotone Boolean functions does small individual effects 
imply asymptotically complete aggregation of information?
\end{description}

In the following subsection, 
we present an example of an FKG measure and a 
monotone Boolean function such that the 
individual effects are small and yet there is no asymptotically 
complete aggregation of information. In this example both the voting scheme 
and the measure $\mu$ are invariant under a transitive group of permutations. 

\subsection{Example: FKG without aggregation}
{\bf The measure $\mu$.}
We start by describing the measure $\mu$. The measure is given 
by a Gibbs measure for the Ising model on the $3$-regular tree.
See e.g. \cite{Georgii:88}.  
The measure is defined as follows. Let $T_r = (V_r,E_r)$ 
be the $r$-level $3$-regular tree.  
This is a rooted tree where each internal nodes has exactly $3$ children 
and all the leaves are at distance exactly $r$ from the root $\rho$. 
Let $L_r$ be the set of leaves of that tree. Note that $|L_r| = 3^r$.  

We first define a measure $\nu$ on the tree $\{0,1\}^{V_r}$. 
In this measure the probability of $x$ is given by 
\[
\nu[x] = \frac{1}{2} 
\prod_{(u,v) \in E_r} \left(
(1-\epsilon) 1_{\{x_u = x_v\}} + \epsilon 1_{\{x_u \neq x_v\}}
\right).
\]
In words, this means
that in the measure $\nu$ the sign of the root $x_{\rho}$ is chosen 
to be $0$ or $1$ with probability $1/2$. Then each vertex
inherits its parents label with probability $\theta=1-2\eps$ and is
chosen independently otherwise. 

Our measure $\mu$ is defined on $\{0,1\}^{L_r}$ (so that
the voters are the leaves of the tree) 
as follows. 
\[
\mu[x] = \sum_{y : y | L_r \leq x} 
\nu[y] \delta^{|\{i : x_i = 1, y_i = 0\}|}.
\]   
In other words, a configuration of votes according to $\mu$ may be obtained 
by drawing a configuration $x$ according to $\nu$ and looking at $x | L_r$. 
Then for each of the coordinates of $i \in L_r$ independently, the vote at 
$x$ re-sampled to have the value $1$ with probability $\delta$. Below we will 
sometime abuse notation and write $\mu$ for the joint probability 
distribution of $x$ and $y$. 

Standard results for the Ising model (see, e.g., \cite{Georgii:88}) imply 
that $\mu$ is an FKG measure. 
Moreover it is easy to see that the measure is 
invariant under a transitive group and  
that $\mu[x_i] = (1+\delta)/2$ for all $i$.

{\bf The function $m$.} 
The function $m$ is given by the recursive majority function 
$m=RM_{3,r}$. Clearly, $m$ is monotone, anti-symmetric and invariant under 
a transitive group. 

\begin{claim}
If $\eps = \delta \leq 0.01$ then  
$\mu[m] \leq 1/2 + \delta/2$ for $m=RM_{3,r}$ and all $r$. 
\end{claim}  

\begin{proof}
The proof below is similar to arguments in \cite{Mossel:98,Mossel:01}. 
Let $(y_v : v \in V_r)$ 
be chosen according to the measure $\nu$. Let $(x_v : v \in L_r)$ be 
obtained from $y_v$ by re-sampling each of the coordinates of 
$(y_v : v \in L_r)$ to $1$ with probability $\delta$. Let $(m_v : v \in V_r)$
denote the value of the recursive majority of all $(x_w : w \in L_r(v))$, 
where $L_r(v)$ are all the leaves of $T$ below $v$.
We will show that $\mu[m = m_{\rho} = 0 |  y_{\rho}=0] \geq 1-\delta$. 
Since $\mu[y_{\rho}=0] = 1/2$,  
we conclude that 
$\mu[m] \leq 1/2 + \mu[m | x_{\rho} = 0]/2 \leq (1+\delta)/2$, as needed. 

We are interested in the probability that $m_v = 0$ conditioned 
on $y_v = 0$. It is easy to 
see that this probability only depends on the height of $v$, i.e., on the
distance between $v$ and the set of leaves. We let $p(k)$ denote the 
probability that $m_v = 0$ conditioned on $y_v = 0$ 
for a vertex $v$ of height $k$. 

Clearly, $p(0) = 1-\delta$. We want to prove by induction that 
$p(k) \geq 1-\delta$ for all $k$. Let $v$ be a node of height $k+1$ and $w$ a 
child of $v$. Note that conditioned on $x_v = 0$ the probability that 
$m_w = 0$ is at least $(1-\eps)p(k)$ which is at least 
$t = (1-\eps)(1-\delta)$ by the induction hypothesis. 
Moreover, noting        
that the values of the
majorities of the children of the node $v$ are conditionally independent
given that $m_v = 0$, we conclude that 
\[
p(k) \geq t^3 + 3t^2(1-t) = 3t^2 - 2t^3 = t^2(3-2t).
\]
We need that $t^2(3 - 2t) \geq 1-\delta$ or recalling that $\eps=\delta$: 
$(1-\eps)^4 (3 - 2(1-\eps)^2) \geq (1-\eps)$. This in turn is equivalent to 
$(1-\eps)^3(3-2(1-\eps)^2) \geq 1$. The function $h(\eps) = 
(1-\eps)^3(3-2(1-\eps)^2)$ has $h'(\eps) = 10(1-\eps)^4 - 9(1-\eps)^2 = 
(1-\eps)^2(10(1-\eps)^2 - 9)$. Therefore $h$ is increasing in the interval 
$[0,0.01]$. Since $h(0)=1$ it follows that $h(\eps) \geq 1$ for all 
$\eps \leq 0.01$ as needed.
\end{proof}

Our next objective is to bound the effect of a voter at level $r$. 
We will prove the following: 

\begin{claim} \label{cl:1}
The measure $\mu$ on $T_r$ and the function $m=RM_{3,r}$ 
satisfy that the effect of each voter is at most 
$(1-\eps/2)^{(r-1)/2} + 2^{-(r-1)/2}$.
\end{claim}

\begin{proof}
The argument here is similar in spirit to an argument in \cite{BeKeMoPe:05}. 
Let $t+s=r$ where $t \geq (r-1)/2$ and $s \geq (r-1)/2$. 
Fix a leaf voter $i$. We want to estimate 
$\mu[m = 1 | y_i = 1] - \mu[m = 1 | y_i = 0]$. Let's denote 
by $\mu_0$ the measure $\mu$ conditioned on $y_i=0$ and by $\mu_1$ the
measure $\mu$ conditioned on $y_i=1$. 

Let $i=v_0,v_1,\ldots,v_r = \rho$ denote the path from $i$ to the root. 
We first claim that the measures $\mu_0,\mu_1$ and $\mu$  
may be coupled in such a way 
that except with probability $(1-2\eps)^t$ the only disagreements between 
$\mu_0, \mu_1$ and $\mu$ are on vertices below $v_t$.

The follows immediately from the random cluster representation of the model.
In this representation we declare and edge $(u,v)$ open with probability 
$(1-2\eps)$ and closed with probability $2\eps$. If the edge $(u,v)$ 
is open then $y_u = y_v$, otherwise, 
the two labels are independent. It is then clear that 
we may couple the two measure $\mu_0,\mu_1$ and $\mu$ 
below $v_t$ as long as the 
path from $i$ to $v_t$ contains at least one closed edge. The probability that 
such an edge does not exist is at most $(1-2\eps)^t$. The proof of the first 
claim follows. 

For each $j$ denote by $u_j$ and $w_j$ the siblings of $v_j$. We assume  
that the measures $\mu_0,\mu_1$ and $\mu$ are coupled in such a way 
that the only disagreements between them are on vertices below $v_t$. 
Note that if this is the case, then if the values of $m$ under 
$\mu_0$ and $\mu_1$ are different 
then for all $r \geq j \geq t$ it holds that $m_{u_j} \neq m_{w_j}$. 
We wish to bound the $\mu$ probability that $m_{u_j} \neq m_{w_j}$ 
for $r \geq j \geq t$. We will bounds this probability conditioned 
on the values $(y_{v_j})_{j=t}^r$. Conditioned on $(y_{v_j})_{j=t}^r$ 
the event $m_{u_j} \neq m_{w_j}$ are independent for different $j$'s. 
Moreover, by the Markov property, 
$\mu[m_{u_j} \neq m_{w_j} | (y_{v_h})_{h=t}^r] = 
\mu[m_{u_j} \neq m_{w_j} | y_{v_{j-1}}]$. Finally note that conditioned 
on $y_{v_{j-1}}$ the random variables $m_{u_j},m_{w_j}$ are identically 
distributed and independent. Therefore 
\[
\mu[m_{u_j} \neq m_{w_j} | y_{v_{j-1}}] \leq 
\max_{p \in [0,1]} 2p(1-p) \leq 1/2.
\]
We thus obtain that the $\mu$ probability that $m_{u_j} \neq m_{w_j}$ 
for $r \geq j \geq t$ is at most $2^{-s}$.
\end{proof}

\chapter[Very Recent Progress]{Recent Progress}
Some exciting developments have taken place since the course these lecture notes are based upon was taught. 
We briefly discuss some of these. 

\section{Plurality is Stablest}
The Majority is Stablest Theorem played a major role in these lectures, see theorems~\ref{thm:MISTsimple}, \ref{thm:MIST}, \ref{thm:MIST_res} etc. 
It is natural to ask if a similar statement holds for low-influence or resilient functions $f : [q]^n \to [q]$ for $q \geq 3$, where the conjectured most-stable function is now a plurality function. This was first asked in~\cite{KKMO:07}. Similarly to the case of $q=2$, there is an equivalent Gaussian question~\cite{IsakssonMossel:12}, which was conjectured to be true in~\cite{IsakssonMossel:12}. The shape of the conjectured optimal partition has a number of names, including the Gaussian Double-Bubble, the Standard $Y$, and the Peace-Sign. 

For the isoperimetric problem, corresponding to $\rho \to 1$, the optimality of such partition for $q=3$ was obtained in~\cite{CCHSXADLV:08} under mid conditions and for all $q$ and no additional condition in recent work~\cite{MilmanNeeman:18a,MilmanNeeman:18b}. 

The case of noise-stability, i.e., constant $\rho$, turned out to be much more subtle. 
Recall that in the binary case, low influence functions $f : \{0,1\}^n \to \{0,1\}$ cannot be asymptotically more stable than majorities with the same expectation.
~\cite{HeMoNe:15a} showed that for any probability measure $\mu  = (\mu_1,\ldots,\mu_q) \neq (1/q,\ldots,1/q)$, that has full support 
there exists low-influence functions that are more stable than the pluralities functions with the same expected values. Thus low influence functions $f : [q]^n \to [q]$ that are not balanced can be asymptotically more stable all pluralities with the same expectation. On the other hand, in the balanced case, for $q=3$,  
and for small positive values of $\rho$, the conjectured Gaussian Double-Bubble was very recently shown to be optimal~\cite{HeilmanTarter:20}. 

\section{Judgment Aggregation} 
In the legal literature, Kornhauser and Sager~\cite{KornhauserSager:86} discusses
a situation where three cases $A,B,C$ are considered in court, and by law, one should rule against $C$ if and only if 
there is a ruling against both $A$ and $B$.
When several judges are involved, their opinions should be aggregated using a function $f : \{0,1\}^n \to \{0,1\}$ that preserves this law, that is, satisfies 
\begin{equation} \label{eq:and}
f(x \land y) = f(x) \land f(y), 
\end{equation} 
where $(x_1,\ldots,x_n) \land (y_1,\ldots,y_n) = (x_1 \land y_1,\ldots, x_n \land y_n)$.
List and Pettit~\cite{ListPettit:02,ListPettit:04} showed that the only non-constant aggregation functions that satisfy (\ref{eq:and}) are the AND functions, known in the social choice literature as \emph{oligarchies}, i.e. functions of the form $f(x_1,\ldots,x_n) = x_{i_1} \land \ldots \land x_{i_r}$ for some $1 = i_1 < \ldots i_r \leq n$. 
Recent work~\cite{FLMM:20} established quantitative versions of the result by List and Pettit by showing that if $f$ is $\eps$-close to satisfying judgment aggregation, then it is $\delta(\eps)$-close to an oligarchy, improving on prior work by Nehama~\cite{Nehama:2013}
in which $\delta$ decays polynomially with $n$. These results are based on the analysis of a variant of the noise-operator, named the one-sided noise operator.

\bibliographystyle{abbrv}
\bibliography{all,my}

\end{document}